\documentclass[10pt]{extarticle}
\usepackage{psfrag}
\usepackage[parfill]{parskip}
\usepackage{float, graphicx}
\usepackage[]{epsfig}
\usepackage{amsmath, amsthm, amssymb,}
\usepackage{epsfig}
\usepackage{verbatim}
\usepackage{multicol}
\usepackage{url}
\usepackage{appendix}
\usepackage[colorlinks, bookmarks=true]{hyperref}
\usepackage{enumerate}
\usepackage[normalem]{ulem}
\usepackage{bm}
\usepackage{dsfont}
\usepackage{stmaryrd}
\usepackage{enumitem}
\usepackage{mathrsfs}
\usepackage{cite}
\usepackage{color}
\usepackage{xcolor}
\usepackage{authblk}
\usepackage{geometry}
 \usepackage{graphicx}
\usepackage{tikz}
\usetikzlibrary{positioning}

\usepackage{imakeidx}
\makeindex
\newcommand{\xindex}[2]{\expandafter\index\expandafter{\pageref{#1}#2}}
\makeatletter

\def\@settitle{\begin{center}%
    \bfseries
 \normalfont\LARGE\@title
  \end{center}%
}
\def\@setauthors{\begin{center}%
 \normalsize\@author
  \end{center}%
}

\makeatother

\numberwithin{equation}{section}

\def\cA{{\mathcal A}}
\def\cB{{\mathcal B}}

\def\cH{{\mathcal H}}
\def\cA{{\mathcal A}}

\def\sB{{\sf B}}
\def\sC{{\sf C}}

\renewcommand{\frak}{\mathfrak}
\newcommand{\fa}{{\frak a}}
\newcommand{\fb}{{\frak b}}
\newcommand{\fc}{{\frak c}}
\newcommand{\fd}{{\frak q}}

\newcommand{\fR}{{\frak R}}

\newcommand{\cT}{{\mathcal T}}

\newcommand{\bFS}{{\bf{S}}}

\newcommand{\al}{\alpha}

\newcommand{\be}{\begin{equation}}
\newcommand{\ee}{\end{equation}}

\newcommand{\cX}{{\cal X}}
\newcommand{\cG}{{\cal G}}

\newcommand{\cK}{{\cal K}}
\newcommand{\bL}{{\mathbb L}}

\newcommand{\cY}{{\cal Y}}

\newcommand{\bW}{\mathbb W}
\newcommand{\bC}{{\mathbb C}}

\newcommand{\wM}{\widetilde{M}}

\newcommand{\T}{\mathbb T}

\newcommand{\bA}{ {\mathbb  A}}

\newcommand{\bT}{\T}

\newcommand{\bB}{{\mathbb B}}

\newcommand{\bX}{{\mathbb X}}
\newcommand{\bG}{{\mathbb G}}

\newcommand{\bH}{{\mathbb H}}
\newcommand{\bV}{{\mathbb V}}
\newcommand{\bK}{\mathbb{K}}
\newcommand{\bY}{\mathbb{Y}}

\newcommand{\id}{\mspace{2mu}\mathrm{i}\mspace{-0.6mu}\mathrm{d}} 
\renewcommand{\epsilon}{\varepsilon}
\renewcommand{\leq}{\leqslant}
\renewcommand{\geq}{\geqslant}

\renewcommand{\le}{\leq}
\renewcommand{\ge}{\geq}
\renewcommand{\P}{\mathbb{P}}
\newcommand{\E}{\mathbb{E}}
\newcommand{\R}{\mathbb{R}}
\newcommand{\C}{\mathbb{C}}
\newcommand{\N}{\mathbb{N}}
\newcommand{\Z}{\mathbb{Z}}

\def\bC{{\mathbb C}}

\newcommand{\cM}{{\cal M}}

\newcommand{\bfS}{{\bf S}}

\newcommand{\qq}[1]{[\![{#1}]\!]}

\DeclareMathOperator{\diag}{diag}

\DeclareMathOperator{\supp}{supp}

\DeclareMathOperator{\re}{Re}
\DeclareMathOperator{\im}{Im}

\DeclareMathOperator{\OO}{O}
\DeclareMathOperator{\oo}{o}

\DeclareMathOperator{\diam}{diam}
\DeclareMathOperator{\anc}{\operatorname{anc}}

\newcommand{\del}{\partial}

\newcommand{\dist}{\mathrm{dist}}

\newcommand{\bS}{\mathbb{S}}
\newcommand{\cGT}{\cal G^{(\T)}}

\newcommand{\tcG}{\tilde{\cal G}}

\newcommand{\tcGT}{\widetilde{\cal G}^{(\T)}}

\newcommand{\GT}{G^{(\T)}}
\newcommand{\tG}{\tilde{G}}

\newcommand{\sG}{{\sf G}}
\newcommand{\sS}{{\sf S}}

\newcommand{\GNH}{\sG_{N}(\cH)}

\renewcommand{\cal}{\mathcal}

\newcommand{\rn}[1]{%
       \lowercase\expandafter{\romannumeral#1}%
}

\newcommand{\RN}[1]{%
       \uppercase\expandafter{\romannumeral#1}%
}

\renewcommand{\Im}{{\mathrm{Im}}}
\renewcommand{\Re}{\mathrm{Re}}

\theoremstyle{plain} 

\makeatletter
\def\thm@space@setup{%
  \thm@preskip=4pt plus 2pt minus 2pt
}
\makeatother

\newtheorem{theorem}{Theorem}[section]
\newtheorem*{theorem*}{Theorem}
\newtheorem{lemma}[theorem]{Lemma}
\newtheorem*{lemma*}{Lemma}

\newtheorem*{corollary*}{Corollary}
\newtheorem{proposition}[theorem]{Proposition}
\newtheorem*{proposition*}{Proposition}

\newtheorem*{assumption*}{Assumption}
\newtheorem{claim}[theorem]{Claim}

\newtheorem{definition}[theorem]{Definition}
\newtheorem*{definition*}{Definition}

\newtheorem*{example*}{Example}
\newtheorem{remark}[theorem]{Remark}

\newtheorem*{remark*}{Remark}
\newtheorem*{remarks*}{Remarks}

\theoremstyle{definition}

\newcommand{\TE}{\mathrm{TE}}
\newcommand{\Ext}{\mathrm{Ext}}

\newcommand{\Ba}{{\mathsf B}_1}
\newcommand{\Bb}{{\mathsf B}_2}
\newcommand{\Bc}{{\mathsf B}_3}

\begin{document}
\title{Local Law and Eigenvector Delocalization for Random Lifts}

\author{
Jiaoyang Huang\textsuperscript{\(\dagger\)}\qquad
Ruizhe Xu\textsuperscript{\(\ddagger\)}
}

\date{}

\maketitle
\thispagestyle{empty}

\begingroup
\renewcommand{\thefootnote}{\fnsymbol{footnote}}

\footnotetext[2]{Courant Institute of Mathematical Sciences, New York University. Email:
\texttt{jh4427@nyu.edu}.}

\footnotetext[3]{Department of Mathematics,
The University of Hong Kong. Email:
\texttt{ruizhexu.math@gmail.com}.}

\endgroup

\maketitle                                                                                     

\begin{abstract}
We establish a quantitative local law for random $N$-lifts of a
fixed finite connected simple graph with minimum degree at least
three. Away from the singular energies of the universal-cover
spectrum, we obtain high-probability entrywise Green's function
estimates down to spectral scales $(\log N)^C/N$, uniformly in the
bulk and near regular edges and cusps. As a consequence, we establish complete eigenvector
delocalization: every normalized
eigenvector $u$ whose eigenvalue stays a small but fixed positive
distance away from the singular energies satisfies
$\|u\|_\infty\leq(\log N)^C/\sqrt{N}$ with high probability.
\end{abstract}

\section{Introduction and main results}

\subsection{Motivation}\label{subsec:motivation}

Let $\cH$ be a fixed finite connected
simple base graph. An $N$-lift of $\cH$ is a graph $\cG$ obtained by
replacing each vertex of $\cH$ with a fiber of $N$ vertices and, for each
edge of $\cH$, joining the corresponding fibers by perfect matching.
We denote by $\GNH$ the collection of all such $N$-lifts.
A random $N$-lift $\cG$ of $\cH$ is obtained by sampling
an element uniformly at random from $\GNH$.
The local limit of random $N$-lifts of $\cH$ is described by the universal covering tree $\cX$ of $\cH$. More precisely, $\cX$ is a tree equipped with a covering map $\pi:\cX\to\cH$ that induces a bijection between the neighbors of each vertex $x$ in $\cX$ and those of $\pi(x)$. Intuitively, $\cX$ is obtained by unfolding the cycles of $\cH$, preserving its local adjacency structure while eliminating cycles. The covering map $\pi$ also assigns each vertex $x$ in $\cX$ a type $\pi(x)$ in $\cH$, which determines its degree and the types of its neighbors. 
In particular, random lifts of a $d$-regular base graph are themselves $d$-regular graphs with the homogeneous $d$-regular tree as their universal cover, whereas a general base graph may give rise to a nonregular universal cover with several geometrically distinct local types.

The study of random lifts is motivated by several complementary perspectives, ranging from sparse random matrix theory and quantum ergodicity to spectral geometry and theoretical computer science. As sparse random matrix models, random lifts have prescribed, potentially nonhomogeneous local geometry, complementing the extensively studied Erd\H{o}s--R\'enyi and random regular graph models; For further discussion, see \hyperref[sec:related-works]{\textcolor{darkgray}{Related works}}
below and \cite{HuangSurvey2026}. They also have a natural connection with quantum ergodicity, which studies the spatial distribution of eigenfunctions on large finite systems whose local geometry approaches an infinite limiting space. For random $N$-lifts of a fixed base graph $\cH$, this limiting geometry is precisely the universal cover $\cX$ of $\cH$. Random lifts therefore provide a natural discrete model for investigating how the geometry of the universal cover governs not only the spectrum, but also the spatial behavior of eigenfunctions on large random coverings \cite{Anantharaman2017,AnantharamanLeMasson2015,AnantharamanSabri2019,AnantharamanSabri2019Recent}. From the perspective of spectral geometry, random lifts may be viewed as discrete analogues of random finite-sheeted coverings of a fixed hyperbolic surface. In both settings, one constructs increasingly large random coverings of a fixed base object while preserving its universal cover, which is a tree in the graph setting and the hyperbolic plane in the surface setting. This parallel leads naturally to related spectral questions, including the behavior of new eigenvalues, spectral gaps, and the influence of the universal cover on the spectrum of large finite coverings \cite{MageeNaudPuder2022,MageePuderVanHandel2025}. Finally, random lifts also play an important role in theoretical computer science and coding theory, providing a probabilistic way to construct large sparse graphs from prescribed finite templates, with well-studied connectivity and expansion properties \cite{AmitLinial2002,AmitLinial2006}.

A central question concerns the eigenvalues of the adjacency matrix of random lifts.
The spectrum of a lift naturally separates into two parts. Functions that
are constant on each fiber form an invariant subspace on which the
adjacency operator is identified with that of $\cH$. Its eigenvalues,
counted with multiplicity, are the \emph{old eigenvalues}. The eigenvalues
on the orthogonal complement are the \emph{new eigenvalues}. Since the
base graph is fixed, the old spectrum has bounded dimension, whereas the
new eigenvalues account for almost all of the spectrum as $N\to\infty$.
The principal spectral questions therefore concern the location and
local distribution of these new eigenvalues.  

The question of spectral location is closely related to the \emph{Ramanujan property}, classically studied in the special case of $d$-regular graphs. A $d$-regular graph is Ramanujan if all eigenvalues other than
the trivial eigenvalues $\pm d$ lie in
$[-2\sqrt{d-1},2\sqrt{d-1}]$, the spectrum of the infinite $d$-regular tree
\cite{LubotzkyPhillipsSarnak1988}. Friedman's proof of Alon's conjecture
shows that random $d$-regular graphs are asymptotically Ramanujan, and
Bordenave subsequently gave another proof and an extension to random
lifts \cite{Friedman2008,Bord15}. More recently, Huang, McKenzie, and Yau
proved edge universality for every fixed $d\geq 3$: after centering and
an appropriate degree-dependent $N^{2/3}$ rescaling, the extreme
nontrivial eigenvalues have Tracy--Widom$_1$ limits. The joint behavior of
the two spectral edges implies that asymptotically about $69\%$ of random
$d$-regular graphs are Ramanujan \cite{HuangMcKenzieYau2024}.

For lifts of a general base graph, the corresponding relative question
asks whether all new eigenvalues lie in $[-\rho,\rho]$, where $\rho$ is the spectral radius of the universal cover.
Early bounds and approaches to this problem were developed by Friedman,
Linial and Puder, Lubetzky, Sudakov, and Vu, and Puder
\cite{Friedman2003,LinialPuder2010,LubetzkySudakovVu2011,Puder2015}.
A major breakthrough was achieved by Bordenave and Collins, who proved that, for general fixed base graphs, the new spectrum converges in probability, in Hausdorff distance, to the spectrum of the universal cover \cite{BordenaveCollins2019}.  This also rules out new eigenvalues at a fixed
positive distance inside gaps of the limiting spectrum, not merely
outside $[-\rho,\rho]$.
More recently, a new approach to strong convergence of random matrices based on the polynomial method has been developed in \cite{ChenGarzaVargasTroppVanHandel2026,ChenGarzaVargasVanHandel2026} by Chen, Garza-Vargas, Tropp, and  van Handel, yielding quantitative bounds on the convergence of extreme eigenvalues.

An equally fundamental question concerns \emph{eigenvector delocalization}: to what extent can the mass of an eigenvector concentrate on a small part of the graph? For locally tree-like regular graphs, non-concentration estimates were established by Brooks and Lindenstrauss \cite{BrooksLindenstrauss2013}, while quantum ergodicity results provide averaged forms of spatial delocalization, including for random lifts under suitable assumptions \cite{AnantharamanLeMasson2015,Anantharaman2017,AnantharamanSabri2019,AnantharamanSabri2019Recent}. Stronger results have been obtained for fixed-degree random regular graphs, where complete eigenvector delocalization provides uniform bounds on the coordinates of all normalized eigenvectors \cite{Bauer2019Local,Huang2024Spectrum}. These developments naturally motivate the question of whether similarly strong delocalization bounds hold for random lifts with a prescribed nonhomogeneous local structure.

A local law provides a common framework for studying the local eigenvalue
density and eigenvector delocalization. Its basic object is the Green's
function
$G(z):=(A-z)^{-1}$ with $z=E+\mathrm i\eta\in\mathbb C_+,$
where $A$ denotes the adjacency matrix of the graph. The parameter $\eta$ determines the spectral resolution of the Green's function: as $\eta$ decreases, $G(z)$ probes increasingly fine spectral scales. This allows Green's function estimates to provide information about both eigenvalues and eigenvectors.  
In the special case of random $d$-regular graphs, local laws were first established in growing-degree regimes \cite{BauerschmidtKnowlesYau2017} and subsequently in the fixed-degree setting \cite{Bauer2019Local,Huang2024Spectrum}. These results naturally motivate the extension of local laws to random lifts with potentially nonhomogeneous universal covers.

The purpose of the present paper is to establish a local law and to
prove complete eigenvector delocalization for random lifts throughout
the regular spectral regime. Under the assumptions stated below, we
obtain quantitative Green's function estimates down to spectral scales
of order $(\log N)^C/N$, for a sufficiently large constant $C$,
uniformly in the stable bulk and in neighborhoods of all regular edges
and regular cusps. Thus, apart from the singular energies of the
universal-cover spectrum, our local law covers the full regular
spectrum. As consequences, it identifies the local
eigenvalue density with that of the universal cover in the bulk and
implies complete delocalization of all eigenvectors whose eigenvalues
stay away from the singular energies. To the best of our knowledge,
these are the first local law and complete eigenvector delocalization
results covering the bulk, regular edges, and regular cusps for random
lifts of general fixed base graphs.

\paragraph{Related works.}\label{sec:related-works}
Strong convergence provides a global spectral approach to random
lifts, whose adjacency matrices can be represented as matrix-valued
polynomials in random permutation matrices. Strong convergence and
quantitative norm estimates for such polynomials were established
and further developed in
\cite{BordenaveCollins2019,BordenaveCollins2023Norm}.
More recent approaches based on polynomial and rational trace
expansions yield quantitative strong convergence
\cite{ChenGarzaVargasTroppVanHandel2026,ChenGarzaVargasVanHandel2026}.
Related spectral bounds have been obtained through non-backtracking
methods \cite{BenaychGeorgesBordenaveKnowles2020} and
matrix concentration and free-probability techniques
\cite{BandeiraBoedihardjoVanHandel2023,BrailovskayaVanHandel2024}.
These results provide global spectral control, but do not
directly yield Green's function estimates on nearly microscopic
scales.

For random regular graphs, bulk universality was first established
in growing-degree regimes
\cite{BauerschmidtHuangKnowlesYau2017}, and has recently been
proved for every fixed degree $d\ge 3$
\cite{he2026bulkuniversalityrandomregular}.
Beyond eigenvalue statistics, the eigenvector moment flow
\cite{BourgadeYau2017} has yielded Gaussian eigenvector statistics
for sparse random graphs in growing-degree regimes
\cite{BourgadeHuangYau2017}.
In the fixed-degree setting, Gaussian-wave limits have been
established for almost eigenvectors and edge eigenvectors
\cite{BackhauszSzegedy2019,HeHuangYau2025}.
Related questions of quantum ergodicity and delocalization
are discussed in
\cite{Anantharaman2022,AnantharamanSabri2017Anderson}.

For Erd\H{o}s--R\'enyi graphs, local laws and complete
eigenvector delocalization have been established in various
sparsity regimes, extending to logarithmic expected degrees
\cite{ErdosKnowlesYauYin2013,HeKnowlesMarcozzi2019,AltDucatezKnowles2022}.
Bulk and edge universality were developed in
\cite{ErdosKnowlesYauYin2012,HuangLandonYau2015}.
At greater sparsity, edge statistics require deterministic
and random corrections
\cite{LeeSchnelli2018,HuangLandonYau2020,HeKnowles2021,HuangYauSparse2026},
while degree fluctuations can produce localized or
semilocalized edge eigenvectors
\cite{AltDucatezKnowles2021,AltDucatezKnowles2023}.
Gaussian statistics of extremal eigenvectors have also been
established for sparse random matrices
\cite{HeHuangWang2026}.

\subsection{Main results: Local law and Delocalization}
To state our main results, we first introduce the spectral structure of
the universal cover of the base graph, which describes the limiting
spectral behavior of random lifts. A detailed account of the relevant
results is given in Section~\ref{sec:TE}; here we briefly summarize only
those facts needed for the formulation of our main theorems.

Let $\cX$ be the universal cover of the base graph $\cH$.
Its density-of-states measure is encoded by the diagonal Green
functions of $\cX$.  More precisely, for each vertex type
$i\in\bH$, let $m_i(z)$ denote the diagonal Green function
at a vertex of type $i$, and set
\[
m_{\cH}(z)
:=
\frac{1}{|\bH|}
\sum_{i\in\bH}m_i(z).
\]
By the spectral theorem, there is a probability measure
$\nu_{\cH}$ on $\R$ such that
\[
m_{\cH}(z)
=
\int_{\R}\frac{1}{\tau-z}\,\nu_{\cH}(\mathrm d\tau).
\]
The measure $\nu_{\cH}$ consists of an absolutely continuous part,
with density denoted by $\rho_{\cH}$, together with possibly finitely
many atoms.  Thus, schematically,
\begin{equation}\label{def.sum-measure}
\nu_{\cH}(\mathrm d\tau)
=
\rho_{\cH}(\tau)\,\mathrm d\tau
+
\sum_{\tau_j}
w_{\tau_j}\,\delta_{\tau_j}(\mathrm d\tau),
\end{equation}
where $\{\tau_j\}$ is the finite set of atomic energies and
$w_{\tau_j}>0$ denotes the corresponding atomic weight. See Lemma \ref{lem:NB-Dyson-existence} and Remark \ref{rem:vertex-from-edge} for details.

The spectral measure of a general universal cover need not be regular
at every energy.  Besides the atoms, there may be exceptional energies
at which the absolutely continuous density becomes singular, for
instance by diverging as the energy approaches the exceptional point.
We refer to these exceptional points collectively as the
\emph{singular energies} and exclude them from the spectral regimes
considered below.  Away from the singular energies, the absolutely
continuous spectrum decomposes into the standard regular regimes:
\emph{regular bulk}, where $\rho_{\cH}$ is positive and regular; \emph{regular edges}, which mark the endpoints of spectral bands; and \emph{regular cusps}, where the endpoints of two neighboring spectral bands meet as the gap between them closes.  See Figure \ref{fig:spectral-density-singularities} and Definition \ref{def:NB-bulk-edge-cusp} for details.

\begin{figure}[htbp]
    \centering
    \includegraphics[scale=0.2]{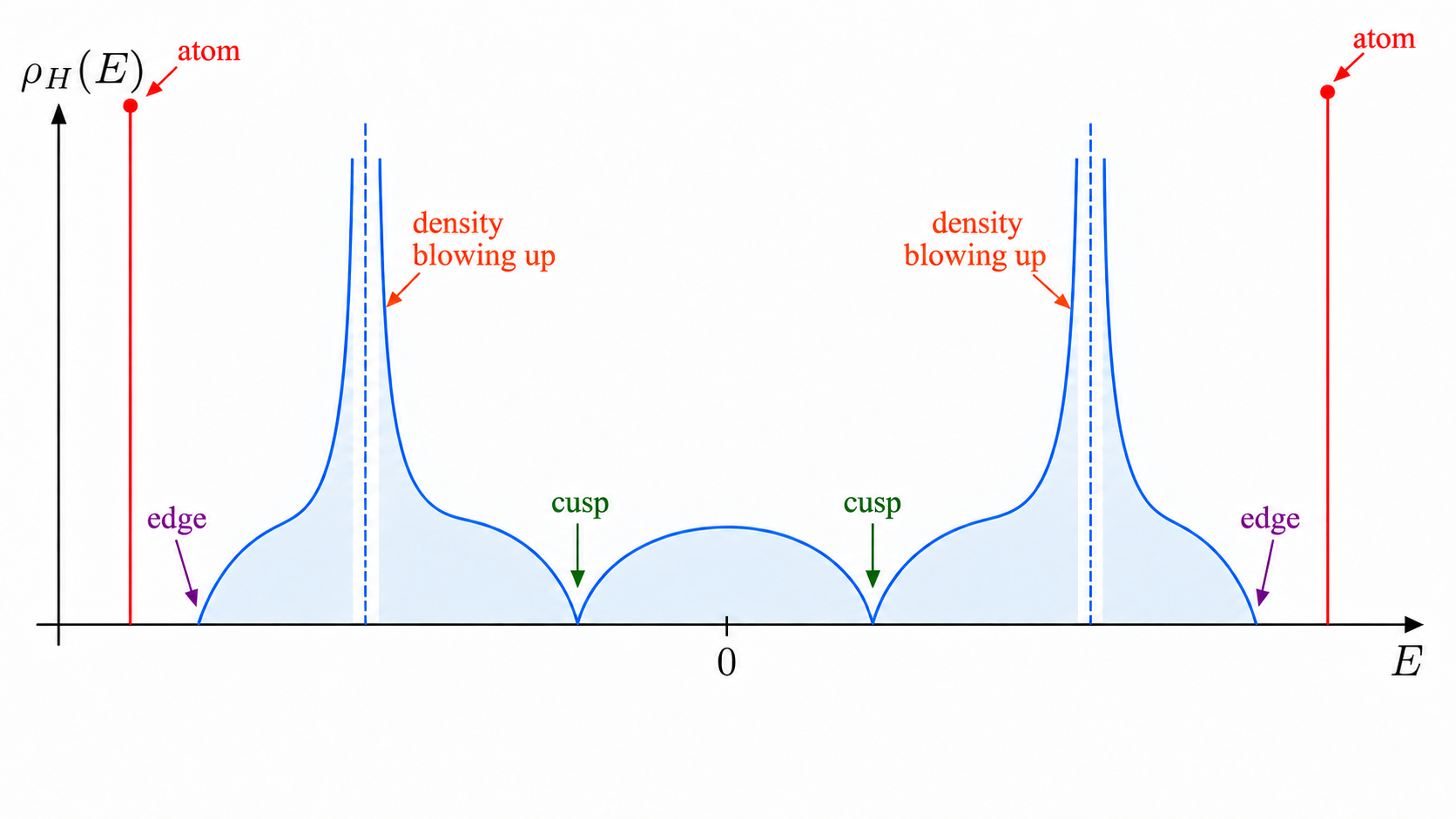}
    \caption{Classification of the local behaviors of the limiting spectral density.}
    \label{fig:spectral-density-singularities}
\end{figure}

To compare the Green's function of a finite lift with that of the
universal cover, we need the local geometry of the lift to remain
tree-like.  We therefore introduce the following
deterministic class $\bar\Omega$ of lifts with this typical local
geometry. In Proposition~\ref{prop:structure} below, we show
that a uniformly random $N$-lift belongs to $\bar\Omega$ with high
probability.
Throughout this paper, we fix a simple base graph $\cH$ whose minimum
degree is at least $3$, and denote its maximum degree by $d_{\rm max}$.
Fix a constant $\fc\in(0,1)$ and set
$
\fR:=(\fc/4)\log_{d_{\rm max}-1}N.
$
For notational simplicity, throughout the paper we write $\log N$ for
$\log_{d_{\rm max}-1}N$.
\begin{definition}\label{def:barOmega}
Given a base graph $\cH$, let $\bar\Omega\subset \GNH$ denote the set of graphs that are tree-like up to radius $\fR$ in the following sense:
\begin{itemize}
    \item the radius-$\fR$ neighborhood of every vertex contains at most one cycle;
    \item the number of vertices whose radius-$\fR$ neighborhoods contain a cycle is at most $N^{\fc}$.
\end{itemize}
\end{definition}

\medskip

Our main result is a quantitative local law for random lifts, stated
precisely in Theorem~\ref{thm.main}.  The quantitative estimates hold
uniformly in the regular bulk and near all regular edges and regular
cusps.  Since the full statement requires several additional notion
and parameters, we postpone its precise statement and first
present the following simpler consequence in the regular bulk.

\begin{theorem}[Bulk eigenvalue counting]
\label{thm.main.intro}
Fix any sufficiently large $\fd\in\N_+$ and let $\cG$ be sampled uniformly from $\bar\Omega$. For every sufficiently large
$N\in\N_+$, with probability
$1-\OO(N^{-\fd})$, the following holds:

Uniformly for all intervals $I$ in the regular bulk that stay a fixed
positive distance away from all edges, cusps, and singular energies and satisfy
$
|I|\geq (\log N)^C/N
$
for some sufficiently large constant $C>0$, we have
\[
\#\{k:\lambda_k(\cG)\in I\}
=
N|\bH|\int_I \rho_{\cH}(E)\,\mathrm dE
+
o(N|I|),
\]
where $\{\lambda_k(\cG)\}$ are the eigenvalues of the adjacency matrix of
$\cG$, counted with multiplicity.
\end{theorem}
\medskip

The quantitative local law Theorem \ref{thm.main} further implies delocalization of
eigenvectors throughout the entire regular spectral regime.  More
precisely, away from the singular energies of the universal-cover
spectrum, all normalized eigenvectors satisfy the following uniform
$\ell_\infty$ bound with high probability.

\begin{theorem}[Eigenvector delocalization]
\label{thm.delocalization.intro}
Fix any sufficiently large $\fd\in\N_+$ and let $\cG$ be sampled uniformly from $\bar\Omega$. For every sufficiently large
$N\in\N_+$, with probability
$1-\OO(N^{-\fd})$, the following holds:

For every normalized eigenvector $u_k$ of the adjacency matrix of $\cG$
whose eigenvalue $\lambda_k(\cG)$ stays a fixed positive distance away from the
singular energies of the universal-cover spectrum, we have
\[
|u_k|_\infty
\leq
\frac{(\log N)^C}{\sqrt{N}},
\]
where $C>0$ is a sufficiently large constant and $|\cdot|_\infty$ denotes the $\ell_\infty$ norm.

\end{theorem}

\subsection{Key ideas, challenges, and new contributions}
\label{subsec:key-ideas}
In Section \ref{sec.proof-strategy}, we present the overall architecture and main flow of the proof, which follows the classical bootstrap and continuation scheme for Green's function local laws; at its core are the two key propositions stated in Section \ref{sec.two-key-proposition}, which are proved in Section \ref{sec:stability} and Section \ref{sec:concentration}, respectively.

Our proof exploits the fact that random lifts are locally tree-like.
Since Green's functions on trees admit exact recursive descriptions, see Proposition \ref{greentree},
our basic strategy is to reduce the global resolvent problem to
calculations on local tree-like neighborhoods.  More precisely, we
approximate the Green's function near a given vertex, or pair of
vertices, by that of a suitable local neighborhood, while replacing
the exterior of the graph by effective boundary weights.  In this way,
the local geometry is retained and the distant environment enters only
through the boundary conditions.  The essential question is
therefore how these boundary weights should be chosen.

\paragraph{Boundary self-energies and empirical approximation.}
Localizing the graph removes the connections between the chosen
neighborhood and its exterior. These missing connections must be
accounted for, since the Green's function of the isolated neighborhood
does not capture the influence of the surrounding graph.
A natural first choice is to complete each missing boundary edge by the
corresponding truncated universal-cover branch.  The motivation is that,
for large $N$, a random lift is locally modeled by the universal cover of
the base graph.  On the universal cover itself, if one first localizes to
a finite neighborhood and then attaches back the missing branches of the
appropriate types, one exactly reconstructs the original universal-cover
tree.  Thus the corresponding universal-cover branches provide the most
natural boundary completion for the localized finite graph. See Figure \ref{fig:boundary-tree-completion}.

We denote the set of oriented edges of the base graph $\cH$ by $\vec E_\cH$. 
Hence, the universal cover contains $|\vec E_{\cH}|$ distinct types of branches, indexed by the directed edges of $\cH$.
For $e=(i,j)\in\vec E_\cH$, this branch, denoted by $\cY_e$, is rooted at a vertex of
type $j$, with the branch leading back to type $i$ removed.
Its root Green's function is denoted by $m_e(z)$, and attaching it along a missing
edge of type $e$ contributes the boundary self-energy $m_e(z)$.
For convenience, we will also write $m_{ij}$ in place of
$m_e$ when $e=(i,j)$, especially when the quantity appears as a
subscript.
The recursive structure of these branches gives the directed-edge
Dyson equation
\begin{equation}\label{eq.DE}
-\frac{1}{m_{ij}(z)}
=
z+\sum_{k\sim j, k\neq i}m_{jk}(z).
\vspace{-2mm}
\end{equation}

Although these deterministic weights give the correct limiting
boundary condition, replacing the exterior directly by them is not
sufficient for the improvement argument near regular edges and cusps.
In these regions, the Dyson equation becomes unstable, so a small
error in the equation can produce a substantially larger deviation
of its solution. A direct comparison with the deterministic boundary
condition does not separate this collective deviation from the local
fluctuations that can be reduced by averaging.
Following the principle introduced by Huang and Yau in
\cite{Huang2024Spectrum}, we therefore first approximate the exterior
using empirical boundary weights computed from the same graph. See Figure \ref{fig:boundary-self-energy}.
The resulting comparison retains the collective fluctuation in the
boundary condition and allows us to derive a self-consistent equation
for it. Only after controlling this equation do we replace the
empirical weights by the deterministic solution $m(z)$.

\begin{figure}[htbp]
    \centering
    \begin{minipage}{0.48\textwidth}
        \centering
        \includegraphics[width=\linewidth]{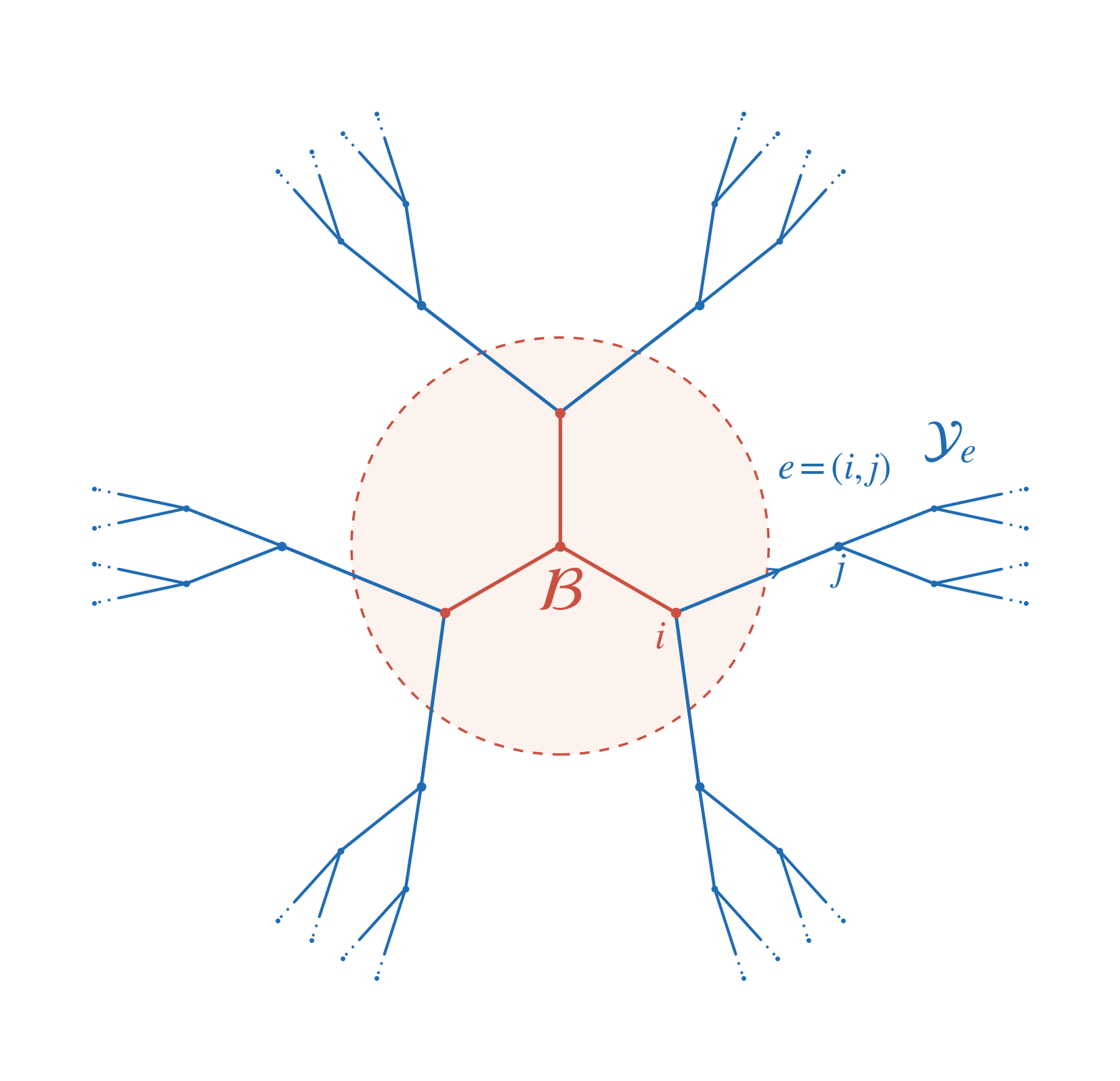}
        \caption{Boundary completion using universal-cover branches.}
        \label{fig:boundary-tree-completion}
    \end{minipage}
    \hfill
    \begin{minipage}{0.48\textwidth}
        \centering
        \includegraphics[width=\linewidth]{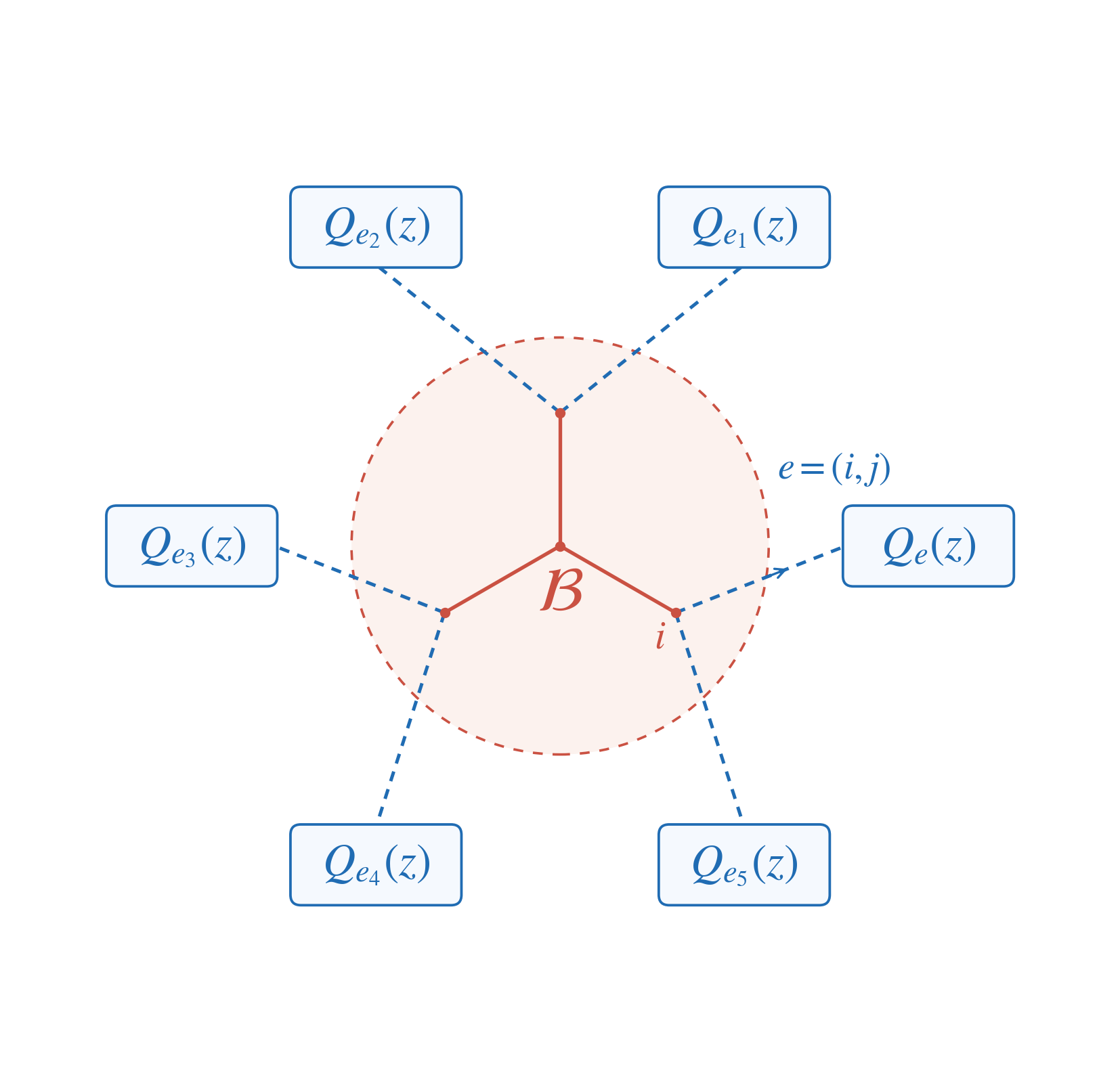}
        \caption{Boundary completion using empirical self-energies.}
        \label{fig:boundary-self-energy}
    \end{minipage}
\end{figure}

\paragraph{Choosing the empirical boundary weights.}
The first difficulty is to identify empirical quantities that are
compatible with both the tree recursion and the resolvent estimates.
In the regular setting, Huang and Yau use an average of
vertex-deleted Green's function entries
\cite{Huang2024Spectrum}. A direct use of such quantities in the
present setting would require control of additional operators obtained
by deleting vertices. This control cannot be inferred merely from
boundedness of the original Green's function.
Indeed, deleting a vertex changes the boundary conditions and the
spectrum of the auxiliary operator. Its Green's function may consequently
become large near an eigenvalue of the modified operator, even when the
original Green's function remains bounded at the same energy.
Thus the use of vertex-deleted observables introduces spectral
regularity questions that are not part of the original problem.
This consideration also guides the graph operations throughout our analysis: we carefully avoid deleting vertices, since doing so could introduce artificial singularities that are difficult to control. In particular, the localization and switching procedures developed later are designed to preserve all vertices while modifying their connections, thereby avoiding the need for separate spectral estimates on vertex-deleted operators.

We instead construct the empirical weights entirely from Green's
function entries of the unmodified graph.

\begin{definition}[Empirical vertex and directed-edge quantities]
\label{def:empirical-Q-p}
For each vertex type $i\in\bH$, define
\begin{equation}
\label{e:def-Q-ver}
Q^{\rm ver}(z):=(Q_i(z))_{i\in\bH},
\qquad
Q_i(z):=
\frac1N
\sum_{x\in\bG, \pi(x)=i}G_{xx}(z).
\vspace{-2mm}
\end{equation}
For a directed edge $e=(i,j)\in\vec E_{\cH}$ and a vertex $x$ with
$\pi(x)=i$, let $x_e$ denote the unique neighbor of $x$ corresponding
to the lift of $e$. Define
\begin{equation}
\label{e:def-p-edge}
p(z):=(p_e(z))_{e\in\vec E_{\cH}},
\qquad
p_e(z):=
\frac1N
\sum_{x\in\bG, \pi(x)=i}G_{x x_e}(z).
\vspace{-2mm}
\end{equation}
The empirical directed-edge weights are
\vspace{-3mm}
\begin{equation}
\label{e:def-Q-edge-new}
Q(z):=(Q_e(z))_{e\in\vec E_{\cH}},
\qquad
Q_e(z):=-\frac{p_e(z)}{Q_{t(e)}(z)},
\vspace{-4mm}
\end{equation}
where $t(e)$ denotes the initial vertex of $e$.
\end{definition}

This choice is motivated by the corresponding identity on the
universal cover. Let $G$ be the Green's function of the universal cover tree. For adjacent vertices $x,y$ with directed-edge
type $e$, we have
$
G_{xy}(z)=-m_e(z)G_{xx}(z);
$
see \eqref{e:Gtreemkm}. Thus $Q_e=-p_e/Q_{t(e)}$ is the empirical
analogue of the deterministic branch Green's function $m_e$.

An important advantage of this definition is that the empirical
quantities satisfy an exact vertex--edge relation.
Averaging the diagonal entries of $(A-z)G=I$ over vertices of type $i$
gives
\begin{equation}
\label{e:averaged-row-identity}
1+zQ_i
=
\sum_{j\sim i}p_{ij}.
\vspace{-5mm}
\end{equation}
Substituting $p_e=-Q_{t(e)}Q_e$, we obtain
\vspace{-3mm}
\begin{equation}
\label{e:new-Q-vertex-reconstruction}
-\frac1{Q_i}
=
z+\sum_{j\sim i}Q_{ij}.
\vspace{-2mm}
\end{equation}
Hence the empirical vertex quantities are reconstructed from the
empirical edge weights in exactly the same way as the corresponding
Green's functions on the universal cover. No additional error is introduced at this stage.

However, the directed-edge Dyson equation is only approximately
satisfied. For a vector of nonzero boundary weights $q=(q_e)_{e\in\vec E_\cH}$, we measure its defect by the Dyson residual
\begin{equation}\label{eq.Dyson-residual}
R(q,z)=(R_e(q,z))_{e\in\vec E_\cH},\qquad
R_{ij}(q,z)
:=
-\frac{1}{q_{ij}}
-z
-\sum_{k\sim j, k\neq i}q_{jk}.
\vspace{-2mm}
\end{equation}
Since $R(m,z)=0$ where $m$ is the solution in \eqref{eq.DE}, quantitative control of this residual, together
with stability and the appropriate a priori control, provides
estimates for the difference between $Q$ and $m$.  In this sense,
the Dyson residual is the natural quantity that measures the
deviation of $Q$ from the deterministic solution $m$ throughout the
entire regular spectral regime: in the bulk, near regular edges,
and near regular cusps.  The only difference between these regimes
lies in the stability with which a bound on $R(Q,z)$ can be converted
into a bound on $Q-m$.

\paragraph{Deterministic stability of the Dyson equation.}
The spectral properties of the universal cover tree are governed by
the directed-edge Dyson equation \eqref{eq.DE}.
Unlike the scalar equation in the regular case, this equation is
vector-valued and involves the generally non-symmetric
non-backtracking matrix $B_{\cH}$ of the base graph.
Our analysis therefore requires a detailed understanding of this equation,
including the spectral decomposition of its solution $m(z)$,
its boundedness, regularity, and singular behavior,
as well as the quantitative stability of the equation,
which allows us to convert bounds on the Dyson residual $R(Q,z)$
into estimates for $Q-m$.
The necessary theory was developed in our preceding work
\cite{huang2026nonsymmetricvectordysonequations};
see Section~\ref{sec.non-backtracking} for details.

\paragraph{Anisotropic propagation and elliptic neighborhoods.}
Another main difficulty comes from the loss of homogeneity of the
universal cover.
On a regular tree, Green's function propagation along a path is
governed by the same factor at every step. Consequently, vertices
at a common distance from the root have comparable influence on
its Green's function.
For a general base graph, the propagation factor depends on the
directed-edge type. Products along paths of the same length can
therefore differ exponentially in that length.
An ordinary graph-distance sphere need not provide a balanced
boundary for the averaging argument: a small collection of directions
may carry a disproportionately large part of the contribution.

To address this problem, we replace ordinary balls by
\emph{elliptic neighborhoods} adapted to Green's function propagation.
Their construction uses Perron--Frobenius normalized path weights
associated with the weighted non-backtracking operator.
Unlike the unnormalized products of the $m_e$, these weights decrease
along each path, and their squares satisfy a conservation relation
under branching.
We stop a path when its normalized weight falls below a prescribed
threshold, of order $2^{-\ell}$, rather than when it reaches a
prescribed graph distance.
The resulting boundary is balanced in terms of propagation weights.
This gives uniform control of root-to-boundary contributions and of
the boundary sums appearing in the resolvent comparison. See Definition \ref{def:elliptic-neighborhood} and Proposition \ref{prop:elliptic-neighborhood-estimates} for details.

\paragraph{Woodbury-based resolvent comparison.}
A further simplification in our argument is that we organize the
resolvent computations systematically through the Woodbury formula; see section \ref{sec:woodbury-switching} for details.
Since the switchings and local modifications that appear in the proof
are finite-rank perturbations, the Woodbury formula gives direct and
transparent expressions for the corresponding changes of the Green's
function.  This allows many resolvent comparisons to be treated within
a single framework and avoids repeated Schur-complement expansions.
As a result, the main estimates and cancellations become considerably
more explicit, and the overall proof is substantially shorter and
cleaner.

\subsection{Preliminaries}
In this subsection, we collect the notation and conventions used throughout the paper.

\paragraph{Notation} 
We denote by $\R^d$ and $\C^d$ the spaces of real- and complex-valued
vectors of dimension $d$, and by $\R^{d\times d}$ and $\C^{d\times d}$
the corresponding spaces of real and complex matrices.  We write
$|\cdot|_p$ for the $\ell_p$ norm on $\R^d$ or $\C^d$, and
$\|\cdot\|_p$ for the induced matrix $p\to p$ operator norm.
We use the notation
\[
a\vee b:=\max\{a,b\},
\qquad
a\wedge b:=\min\{a,b\},
\qquad
\qq{a,b}:=[a,b]\cap\Z.
\]
The complex upper half-plane is denoted by $\bC_+$, and the spectral
radius of a matrix by $\rho(\cdot)$.  For $z\in\bC_+$, we write
$
\eta:=\im z.
$

For positive quantities $f$ and $g$, we write $f\lesssim g$ if
$f\leq Cg$ for some constant $C>0$ depending only on fixed model
parameters and absolute constants. We write $f\gtrsim g$ if
$g\lesssim f$, and $f\sim g$ if both $f\lesssim g$ and
$f\gtrsim g$. In particular, $f\sim 1$ means that $f$ is bounded
above and below by positive constants depending only on fixed model
parameters and absolute constants. We write $f=\OO(g)$ if
$|f|\leq Cg$, and $f=\oo(g)$, or equivalently $f\ll g$, if
$|f|/g\to0$ as $N\to\infty$.

Throughout this paper, we always focus on simple graphs $\cG$, which means that they have no self-loops or multiple edges. Let $\cG=\left(\bG,E_\cG\right)$ be a connected undirected simple graph, where $\bG$ denotes the set of vertices and $E_\cG$ denotes the set of edges. The set of directed edges of graph $\cG$ is denoted by $\vec E_\cG$.
The degree of a vertex in the graph $\cG$ is denoted by $\deg_{\cG}(\cdot)$. The distance between two vertices, defined as the length
of a shortest path between them, is denoted by  $\dist_{\cG}(\cdot, \cdot)$. 
The diameter $\diam(\cG)$ of the graph $\cG$ is the largest distance between any pair of vertices. If two vertices $i,j\in \bG$ are adjacent, we write $i\sim j$.
For a finite graph $\cG$, we define its excess by $\operatorname{excess}(\cG) := |E_\cG|-|\bG|+\operatorname{comp}(\cG), $ where $\operatorname{comp}(\cG)$ denotes the number of connected components of $\cG$. Equivalently, $\operatorname{excess}(\cG)$ is
the number of independent cycles in $\cG$.

For any subset $\bT\subset\bG$, let $\cG^{(\bT)}$ denote the graph obtained by removing the vertices in $\bT$ together with all incident edges, and let $G^{(\bT)}$ denote its Green's function.
For any subgraph $\cal K \subset \cG$, we denote by $\del \cal K=\{v\in \bG: \dist_{\cG}(v,\cal K)=1\}$
the vertex boundary of $\cal K$ in $\cG$.
Given a subset $\bT$ of the vertex set of a graph $\cG$ and an integer $r>0$,
we denote by $\cB_r(\bT,\cG):=\{j \in \bG: \dist_{\cal G}(\bT,j) \leq r\}$  the radius-$r$ neighborhood of $\bT$ in $\cG$. In particular, $\cB_r(i,\cG)$ is the radius-$r$ neighborhood of the vertex $i$.

For a matrix $A$, we write $A|_{\bG}$ for the restriction of both row
and column indices to $\bG$, $A|_{\cdot,\bG}$ for the restriction of
the column indices to $\bG$, and $A|_{\bG,\cdot}$ for the restriction
of the row indices to $\bG$.

\paragraph{Conventions}
We reserve letters in mathfrak mode, e.g. $\fa, \fb, \fc$, to represent universal constants, and use $C,C_1, C_2,\cdots$, to represent large universal constants, which may vary from line to line. We use letters in mathcal mode, e.g. $\cB, \cG, \cT$, to represent graphs, or subgraphs, and letters in mathbb mode, e.g. $\bB, \bG, \bT$, to represent sets of vertices. We use letters in mathsf mode, e.g. $ \mathsf A, \mathsf B, \mathsf U, \mathsf W$ to represent index sets.

Products, quotients, and powers of vectors are understood componentwise. Thus, for vectors $u,v\in\mathbb C^d$ with nonzero entries of $v$,
$
uv=(u_1v_1,\dots, u_dv_d)$
and
$
u/v=\left(u_1/v_1,\dots,u_d/v_d\right).
$
Similarly, scalar functions applied to vectors are understood componentwise. Thus,
for \(v\in\mathbb C^d\), quantities such as
$
  \re v,\im v
$
and $|v|$
are defined by applying the corresponding scalar operation to each
component. For example,
$
  \im v=(\im v_1,\dots,\im v_d).
$
Inequalities between vectors or matrices are also understood componentwise. 

When a directed edge is used as a subscript, we write \(ij\) for the directed edge \((i,j)\); for instance, \(m_{ij}\) denotes \(m_{(i,j)}\).

\paragraph{Auxiliary results for Green's functions}
In Appendix~\ref{app:Green}, several well-known properties of Green's function are summarized;
they will be used throughout the paper.

\paragraph{AI Statement}
This project began in September 2025, during Ruizhe Xu's visit to Jiaoyang Huang at the University of Pennsylvania. 
In the course of addressing the technical challenges of this problem, we also developed a separate paper, Non-symmetric Dyson equations, which provides part of the mathematical foundation for the present work. This research history provides a traceable account of the development of our results.

The use of AI tools in preparing this manuscript was limited to improving the presentation and language and generating figures (all figures in this paper were created using AI). All mathematical ideas, proof strategies, and scientific contributions originated from the authors.

\paragraph{Acknowledgements}
We thank Charles Bordenave and Wenbo Li for helpful notes. Ruizhe Xu is grateful to Zhonggen Su for his support and to the University of Pennsylvania for its hospitality and support during his visit.
The research of J.H. was supported in part by NSF Grant DMS-2246664, DMS-2331096 and a Sloan Research Fellowship.

\section{Tree-like Graphs}\label{sec.tree-like}
Recall the constants $0<\fc<1$ and $\fR=(\fc/4)\log N$.
The following proposition is the basic geometric input for our
spectral analysis. It shows that a uniformly random $N$-lift belongs,
with high probability, to the class $\bar\Omega$ of radius-$\fR$
tree-like graphs introduced in Definition~\ref{def:barOmega}.
Thus, the local law established in this paper on $\bar\Omega$ describes the
spectral behavior of a typical random lift.
\begin{proposition} \label{prop:structure}
Let $\cH$ be a base graph with minimum degree at least three. For all sufficiently large $N$, a uniformly random $N$-lift $\cG\in\GNH$ satisfies
\begin{equation}\label{eq.tree-like}
  \mathbb{P}\bigl(\cG\notin\bar\Omega\bigr)\lesssim N^{-(1-\fc)}.
\end{equation}
\end{proposition}

\begin{proof}
By \cite[Lemma 27]{Bord15}, the probability that every radius-$\fR$
neighborhood in $\cG$ contains at most one cycle is at least
$
1-\OO((d_{\rm max}-1)^{4\fR}/N)
=
1-\OO(N^{-(1-\fc)}).$
Moreover, a slight adaptation of the argument proving
\cite[Proposition 4.1, (4.3)]{Bauer2019Local} to the permutation model of random lifts gives
\[
\mathbb{P}\left(
 \left|
 \left\{
 x\in\bG:
 \cB_{\fR}(x,\cG)\text{ contains a cycle}
 \right\}
 \right|
 >
 (d_{\rm max}-1)^{4\fR}
\right)
\lesssim N^{-(1-\fc)}.
\]
Here, the configuration-model exploration in
\cite{Bauer2019Local} is replaced by the corresponding sequential
exposure of the permutations associated with the edges of the fixed
base graph $\cH$.
Combining the two estimates with a union bound yields \eqref{eq.tree-like}.
\end{proof}

\subsection{Quadratic vector equation and universal cover trees}\label{sec:TE}
In this section, we introduce the quadratic vector equation corresponding to the non-backtracking matrix of the base graph $\cH=(\bH,E_\cH)$ and list some properties of its solution. Then, we calculate the Green's functions of the universal cover of the base graph $\cH$.

\subsubsection{The base graph and the quadratic vector equation}\label{sec.non-backtracking}

The \emph{adjacency matrix} of \( \cH \) is denoted by \( H = (H_{ij})_{i,j \in \bH} \), where
$
H_{ij} = 1 $ if $\{i,j\}\in E_\cH$; $H_{ij}=0$ otherwise.
In addition to the adjacency matrix, it is often convenient to encode the local structure of \( \cH \) 
using the \emph{non-backtracking matrix} (also called the Hashimoto matrix). 
Recall that $\vec E_{\cH}$ denotes the set of directed edges of $\cH$.
For a directed edge \(e=(i,j)\), we write
$
t(e)=i$ and $h(e)=j
$
for its tail and head, respectively, and denote the reversed edge by
$
\bar e=(j,i).
$
For oriented edges \(e,f\in \vec E_{\cH}\), we write
$
e\rightsquigarrow f
$
if \(f\) follows \(e\) without backtracking, that is,
$
h(e)=t(f)
$
and
$
f\neq \bar e.
$
For vertices \(i,j\in\bG\) where $\cG$ is a lift of $\cH$, we use
$
i\leadsto j
$
to denote a non-backtracking path in \(\cG\) from \(i\) to \(j\).
The non-backtracking matrix of \(\cH\) is the
$
|\vec E_{\cH}|\times |\vec E_{\cH}|
$
matrix \(B_{\cH}=(B_{ef})_{e,f\in\vec E_\cH}\), indexed by oriented edges, defined by
$
B_{ef}
=\mathbf{1}_{e\rightsquigarrow f}.
$
Thus, \(B_{\cH}\) records the allowed one-step transitions of non-backtracking walks on \(\cH\). We shall also use the shorthand \(B:=B_{\cH}\).

\medskip

For \( z \in \mathbb{C}_+\), we consider a vector \( m(z): = (m_e)_{e \in \vec{E}_\cH} \in \mathbb{C}_+^{\vec{E}_\cH} \) that satisfies the 
\emph{Dyson equation} \eqref{eq.DE}.
The vector \( m(z) \) provides an analytic description of the resolvent of the adjacency operator on the forward universal-cover branches, 
and plays a central role in the study of local spectral statistics and the limiting spectral distribution 
of random lifts of \( \cH \). Besides $m_{ij}(z)$, we also define
\begin{equation}\label{eq.mii}
  m^{\rm ver}=(m_i)_{i\in\bH},\qquad\mbox{where~}m_{i}(z):=\frac{-1}{z+\sum_{j\sim i}m_{ij}(z)}, \qquad i\in \bH.
\end{equation}

Some properties of the solution to the Dyson equation \eqref{eq.DE} are listed below. Before stating the results, we introduce the following notation.
For a set $I$ contained in $\R$, set $ \mathbb C_+(I) := \{z\in\mathbb C_+:\Re z\in I\}$.
The first lemma shows the basic analytic properties of the solution to the Dyson equation, together with the structure of its singular set and regular spectral support.

\begin{lemma}[Theorem 2.1 and Theorem 2.3 of \cite{huang2026nonsymmetricvectordysonequations}]\label{lem:NB-Dyson-existence}
For every $z\in\mathbb C_+$, equation \eqref{eq.DE} has a unique
solution
$
m(z)\in\mathbb C_+^{\vec E_{\mathcal H}},
$
and the map $z\mapsto m(z)$ is analytic. 
There exists a finite set 
\begin{equation}
  \mathcal A_\cH:=\left\{\tau\in\R: \lim_{\eta\downarrow 0}|m_e(\tau+i\eta)|=\infty~\text{for some $e\in \vec E_\cH$}\right\}
\end{equation}
such that the boundary value
$
m(\tau)
:=
\lim_{\eta\downarrow0}m(\tau+\mathrm i\eta)
$
exists and is finite for all $\tau\in\mathbb R\setminus\mathcal A_\cH$. 
For every
$e\in\vec E_{\mathcal H}$, there exists a unique probability measure
$\nu_e$ such that
\[
m_e(z)
=
\int_{\mathbb R}\frac{\nu_e(\mathrm d\tau)}{\tau-z}.
\]
The measure $\nu_e$ admits the decomposition
\[
\nu_e(\mathrm d\tau)
=
\rho_e(\tau)\,\mathrm d\tau
+
\sum_{\tau_j}
w_{e,\tau_j}\,\delta_{\tau_j}(\mathrm d\tau),
\vspace{-3mm}
\]
where $\tau_j\in\mathcal A_\cH$ is an atomic energy,
$w_{e,\tau_j}>0$ are the corresponding atomic weights, and
$\rho_e$ denotes the density of the absolutely continuous part. 
All these measures are supported in the compact interval $[-\Sigma,\Sigma]$ where $\Sigma:=2\|B_\cH\|_\infty^{1/2}\le 2\sqrt{d_{\rm max}-1}$.

For a sufficiently small parameter $\delta>0$, we define the regular support away from the singular points by 
\[ \mathcal S_\delta := \left\{ \tau\in\R: \dist(\tau,\mathcal A_\cH)>\delta~\mbox{and}~\sum_{e\in\vec E_\cH}\rho_e(\tau)>0 \right\}.\]
Then $\mathcal S_\delta$ is a finite union of bounded open intervals. 
\end{lemma}

\begin{remark}\label{rem:vertex-from-edge}
The preceding theorem concerns the directed-edge
quantities. By \eqref{eq.mii}, the vertex Green
functions $m^{\rm ver}=(m_i)_{i\in\bH}$, and hence
their associated spectral measures $\nu_i$ and
densities $\rho_i$, are determined by the
directed-edge Green functions
$(m_e)_{e\in\vec E_{\cH}}$.

In particular, the density-of-states measure
$\nu_{\cH}$ defined in \eqref{def.sum-measure}
satisfies
\[
\nu_{\cH}
=\frac{1}{|\bH|}\sum_{i\in\bH}\nu_i,
\qquad
\rho_{\cH}
=\frac{1}{|\bH|}\sum_{i\in\bH}\rho_i.
\]
Moreover, it follows from Lemma \ref{lem:NB-Dyson-regularity} that
\[
\rho_{\cH}(\tau)
\sim
\sum_{e\in\vec E_{\cH}}\rho_e(\tau)
\]
throughout the regular spectral region.
Consequently, the vertex and directed-edge
densities have the same positivity set, and
the corresponding regularity and boundary
properties of the vertex quantities follow
from those of the directed-edge quantities.
\end{remark}

\begin{remark}
The purpose of introducing $\mathcal S_\delta$ is to exclude the singular
points at which the solution blows up, since such blow-up destroys most of
the regularity and stability properties required in the subsequent analysis.
As $\delta$ decreases to zero, the set $\mathcal S_\delta$ approaches
the positive-density region of the absolutely continuous spectrum away from the singular points. Thus,ww a smaller choice of $\delta$ retains a
larger portion of the real line and, in this sense, yields a more precise
description of the solution. On the other hand, every regular edge or cusp
that does not coincide with an atomic singularity admits a neighborhood on
which the solution remains uniformly bounded. Hence, by choosing
$\delta$ sufficiently small, removing the $\delta$-neighborhoods of
the singular points does not exclude such regular edges or cusps.
\end{remark}

Since the spectrum of an \(N\)-lift need not be contained in \(\mathcal S_\delta\), while it is always contained in \([-d_{\max},d_{\max}]\), we introduce the larger spectral domain
$$
\mathcal I_\delta
:=
\left\{
\tau\in[-d_{\max},d_{\max}]
:
\dist(\tau,\mathcal A_{\mathcal H})\ge \delta
\right\}.
$$ 
Clearly, it follows that \(\mathcal S_\delta\subset \mathcal I_\delta\).
All subsequent analysis in this paper will be carried out on the domain $\bC_+(\mathcal I_\delta)$.
Moreover, by Theorem 2.1 and Theorem 2.3 of \cite{huang2026nonsymmetricvectordysonequations}, $m(z)$ is uniformly bounded on $\bC_+(\mathcal I_\delta)$:
$
\sup_{z\in\mathbb C_+(\mathcal I_\delta)}|m(z)|_2\le\Phi_\delta<\infty,
$
where $\Phi_\delta$ depends on the model parameters and $\delta$.

We next classify the regular spectral points according to the local structure of the support.

\begin{definition}[Spectral regimes]
\label{def:NB-bulk-edge-cusp}
Fix the parameter $\delta>0$ as in Lemma
\ref{lem:NB-Dyson-existence}. A point $\tau\in\mathcal I_\delta$ is called
\begin{itemize}
    \item an \emph{off-support point} if
    $
    \tau\in \mathcal I_\delta\setminus \overline{\mathcal S_\delta}.
    $
    \item a \emph{regular bulk point} if $\tau\in\mathcal S_\delta$;
    \item a \emph{regular edge} if $\tau\in\partial\mathcal S_\delta$, $\rho_\cH(\tau)=0$, and $\tau$ is the
    endpoint of exactly one connected component of $\mathcal S_\delta$. We denote by $\mathcal S_{\rm edge}$ the set of all regular edges.
    \item a \emph{regular cusp} if $\tau\in\partial\mathcal S_\delta$, $\rho_\cH(\tau)=0$, and $\tau$ is the
    common endpoint of two connected components of $\mathcal S_\delta$. We denote by $\mathcal S_{\rm cusp}$ the set of all regular cusps.
\end{itemize}
\end{definition}

We next record the uniform bounds and regularity of the Dyson solution, together with the characteristic square-root and cubic-root behavior at regular edges and cusps.
The proof of the following lemma is in the Appendix \ref{a:ext}.

\begin{lemma}[Boundedness and regularity] \label{lem:NB-Dyson-regularity} 
  For every $e\in\vec E_{\mathcal H}$ and $i\in \bH$ and $z=E+i\eta\in\mathbb C_+(\mathcal I_\delta)$, we have
  \begin{equation}\label{eq.mbound-comparable}
    |m_e(z)|\sim|m_{i}(z)| \sim \frac{1}{1+|z|}, \quad \Im m_e(z)\sim\Im m_{i}(z)\sim\begin{cases}
|z-\tau|^{1/2},
&
\begin{array}{l}
z\text{ is close to a regular edge }\tau,\\[-1mm]
E\text{ lies on the spectral side of }\tau,
\end{array}
\\[4mm]
\dfrac{\eta}{|z-\tau|^{1/2}},
&
\begin{array}{l}
z\text{ is close to a regular edge }\tau,\\[-1mm]
E\text{ lies on the off-support side of }\tau,
\end{array}
\\[5mm]
|z-\tau|^{1/3},
&
z\text{ is close to a regular cusp }\tau,
\\[2mm]
(1+|z|)^{-1},
&
z\text{ is in the stable bulk},
\\[2mm]
\dfrac{\eta}{(1+|z|)^2},
&
z\text{ is in the stable off-support regime}.
\end{cases}
  \end{equation}
Moreover, the solution $m_e(z)$ and density $\rho_e(\tau)$ are uniformly $\alpha$-H\"older continuous on $\overline{\mathbb C_+(\mathcal I_\delta)}$ and $\mathcal S_\delta$ respectively, where $\alpha$ is a finite positive exponent depending only on the model parameters.  
\end{lemma}

We shall repeatedly encounter approximate solutions of the Dyson
equation. For a vector of boundary weights $q=(q_e)_{e\in\vec E_\cH}$ with nonzero
entries, recall the Dyson residual $R(q,z)$ defined in
\eqref{eq.Dyson-residual}.

Since $R(m(z),z)=0$, writing
$
q(z)=m(z)+u(z)
$
and subtracting the Dyson equation \eqref{eq.DE} for $m$ yields the exact identity
\begin{equation}
\label{eq:Dyson-residual-expansion}
(I-m^2B)u
=
m^2R(q,z)+\frac{u^2}{q}.
\end{equation}
Thus $I-m^2B$ describes the linear stability of the Dyson
equation, while the second term on the right-hand side of
\eqref{eq:Dyson-residual-expansion} becomes relevant precisely
when the linear stability degenerates.

Away from the spectral singularities, the linear term is stable and
a small Dyson residual directly implies that $q$ is close to $m$.
Near a regular edge or cusp, however, the inverse of
$I-m^2B$ becomes singular, and the quadratic and cubic structure of
the Dyson equation provides the additional stability required to
control the unstable direction. 
The following lemma collects these stability estimates.

\begin{lemma}[Linear and nonlinear stability, Theorem 2.9 of \cite{huang2026nonsymmetricvectordysonequations}]
\label{lem.stability2}
Fix $z\in\bC_+(\mathcal I_\delta)$ and a height $\eta_0>0$. 
Let $q(z)=(q_e(z))_{e\in\vec E_\cH}$ have nonzero entries, and its
Dyson residual $R(q,z)$ is defined as \eqref{eq.Dyson-residual}. 
Define
the vertex quantities $q^{\rm ver}(z)=(q_i(z))_{i\in\bH}$ associated with $q(z)$ by
\begin{equation}\label{eq:q-vertex-reconstruction}
q_i(z)
:=
-\frac1{z+\sum_{j:j\sim i}q_{ij}(z)}.
\end{equation}
There exists a sufficiently small constant $c_0>0$ and $\eta_0$, depending only on the model parameters, such that, whenever  $|q(z)-m(z)|_\infty\le c_0$ and $|R(q,z)|_\infty\le c_0$, the following estimates hold.
\begin{itemize}

\item
\emph{Stable bulk and off-support.}
If $\cH$ is non-bipartite, or if $\cH$ is bipartite and $z$ stays
away from $0$, then, uniformly for $z$ in the stable bulk and stable
off-support, 
\begin{equation}
\label{eq:nonlinear-bulk-stability}
|q(z)-m(z)|_\infty
\lesssim
|R(q,z)|_\infty.
\end{equation}

If $\cH$ is bipartite and $0$ is a regular bulk point, let
$\sigma_i=\pm1$ on the two bipartition classes of $\bH$.
Then, for $z$ sufficiently close to $0$, provided that
$
\left|
\sum_{i\in\bH}\sigma_i
\bigl(q_i(z)-m_i(z)\bigr)
\right|
\le c_0,
$
we have
\begin{equation}
\label{eq.bulk-0-stability}
|q(z)-m(z)|_\infty
\lesssim
|R(q,z)|_\infty
+
\left|
\sum_{i\in\bH}\sigma_i
\bigl(q_i(z)-m_i(z)\bigr)
\right|.
\end{equation}

\item
\emph{Regular edges and cusps.}
Let $\tau$ be a regular edge or cusp, and let
$z=E+\mathrm i\eta$ be sufficiently close
to $\tau$, with $\eta<\eta_0$.
Assume further that there exists a continuous family $(q_t)_{\eta\le t\le\eta_0}$ satisfying
$q_\eta=q(z)$ and
\begin{equation}\label{eq.continuity}
\sup_{\eta\le t\le\eta_0}
|q_t-m(E+\mathrm i t)|_\infty\le c_0,
\qquad
\sup_{\eta\le t\le\eta_0}
|R(q_t,E+\mathrm i t)|_\infty\le\varepsilon.
\end{equation}
Then, for $|z-\tau|$ sufficiently small,
\begin{equation}\label{eq:nonlinear-boundary-stability}
|q(z)-m(z)|_\infty
\lesssim
\begin{cases}
\displaystyle
\frac{\varepsilon}
     {(\varepsilon+|z-\tau|)^{1/2}},
&
\tau\in\mathcal S_{\rm edge},
\\[3mm]
\displaystyle
\frac{\varepsilon}
     {(\varepsilon+|z-\tau|)^{2/3}},
&
\tau\in\mathcal S_{\rm cusp}.
\end{cases}
\end{equation}

\end{itemize}
\end{lemma}

\begin{remark}[Selection of the branch]
\label{rem:physical-branch}
The continuity assumption \eqref{eq.continuity} in the edge and cusp regimes is used to
select the branch of the Dyson equation.
Indeed, near a regular edge or cusp the linearized Dyson operator
becomes singular, and the corresponding nonlinear equation may have
several local solution branches which coalesce at edges and cusps.
Consequently, the pointwise assumptions
$
|q-m(z)|_\infty\le c_0$
and
$
|R(q,z)|_\infty\le\varepsilon,
$
do not by themselves guarantee that $q$ lies on the branch issued
from the true solution $m(z)$: as $z$ approaches the edge or cusp,
a different branch may enter any fixed $c_0$-neighborhood of $m(z)$.

The family $(q_t)_{\eta\le t\le\eta_0}$ rules out this ambiguity.
The upper endpoint $E+\mathrm i\eta_0$ is chosen in a uniformly stable
region, where the solution close to $m(E+\mathrm i\eta_0)$ is locally
unique. The continuity of $q_t$, together with the uniform bound on
its Dyson residual, then ensures that $q_t$ follows the branch
continuously connected to the true solution as $t$ decreases from
$\eta_0$ to $\eta$. In the stable bulk no such condition is needed,
since the linearized Dyson operator is uniformly invertible and hence
already gives local uniqueness in a fixed neighborhood of $m$.
\end{remark}

\subsubsection{Green's function on universal cover trees}
In this subsection, since the universal cover of the base graph $\cH$ is an infinite tree, we can explicitly compute its Green's function. 

In the remainder of the paper, we extend the quantities $m_{ij}(z)$ to coverings of the base graph by interpreting the indices in terms of vertex types. 
Specifically, for adjacent vertices $i,j$ in the covering graph, we define
\begin{equation} \label{eq:mij-extension}
    m_{ij}(z) := m_{\pi(i)\,\pi(j)}(z)
\end{equation}
where $\pi$ is the covering map. 
In other words, the value of $m_{ij}(z)$ on the universal cover depends only on the types of $i$ and $j$ in the base graph.

For a given path
$
\gamma=(i=i_0,i_1,\ldots,i_\ell=j),
$
we define its path weight by
\begin{equation}\label{e:path-weight}
    \mathcal M(\gamma;z)
    :=
    \prod_{k=0}^{\ell-1}
    \bigl(-m_{i_k i_{k+1}}(z)\bigr),
\end{equation}
where $m_{i_k i_{k+1}}(z)$ denotes the Green function associated
with the directed edge type of $i_k\to i_{k+1}$ in the base graph
$\cH$. For a path of length zero, we set $\mathcal M(\gamma;z)=1$.

In particular, if $\cG$ is a tree, then any two vertices $i,j\in\bG$
are connected by a unique simple path, denoted by
$
[i,j]=(i=i_0,i_1,\ldots,i_\ell=j).
$
We therefore define
\begin{equation}\label{e:tree-path-weight}
    \mathcal M_{\cG}(i,j;z)
    :=
    \mathcal M([i,j];z)
    =
    \prod_{k=0}^{\ell-1}
    \bigl(-m_{i_k i_{k+1}}(z)\bigr).
\end{equation}
When the underlying tree is clear from the context, we simply write
$\mathcal M(i,j;z)$.

\begin{proposition}\label{greentree}
Let $\cX$ be the universal cover of the base graph $\cH$. For all $z\in\mathbb{C}^+$ and $i,j\in\bX$, the Green's function of $\cX$ is
\begin{equation}\label{e:Gtreemkm}
    G_{ij}(z)
    =
    m_i(z)\mathcal M_\cX(i,j;z).
\end{equation}

Fix a directed edge $e=(k,l)\in\bH$. Let $\cY_{e}$ be the truncated universal cover rooted at vertex $o$ of type $l$, obtained by deleting the edge adjacent to $o$ of type $\{l,k\}$.
For every $i\in\bY_e$, we have
\begin{equation}\label{e:Gtreemsc2}
    G_{oi}(z)
    =
    m_{e}(z)\mathcal M_{\cY_{e}}(o,i;z).
\end{equation}
In particular, $G_{oo}(z)=m_e(z)$.
\end{proposition}

\begin{proof}
For an oriented edge $(x,y)$ of
$\cX$, by the same argument in \cite[Proof of Proposition 5.1]{Bauer2019Local}, the Schur complement formula \eqref{e:Schur} together with the Dyson equation \eqref{eq.DE}
gives
$
G_{xx}=m_{x}
$
and
$
G_{xy}=-m_{xy}G_{xx}.
$
Iterating the same identity along the unique path
$
[i,j]=(i=i_0,i_1,\ldots,i_\ell=j)
$
in $\cX$, we obtain
\[
G_{ij}
=
m_i\prod_{r=0}^{\ell-1}(-m_{i_ri_{r+1}})
=
m_i\cM_{\cX}(i,j),
\]
which proves \eqref{e:Gtreemkm}.

For $\cY_e$, the same argument applies verbatim. By the Schur
complement formula and the directed-edge Dyson equation, the removal
of the branch of type $\{l,k\}$ gives
$
G_{oo}=m_e.
$
Iterating the identity
$
G_{xy}=-m_{xy}G_{xx}
$
along the unique path from $o$ to $i\in\bY_e$ therefore yields
$
G_{oi}
=
m_e\mathcal M_{\cY_e}(o,i),
$
which proves \eqref{e:Gtreemsc2}.
\end{proof}

Next we introduce two types of weighted non-backtracking matrices.
For $z\in\mathbb C^+$, define
\[
M(z):=B_\cH\operatorname{diag}\bigl(m_e(z)^2\bigr)_{e\in\vec E_\cH},
\qquad
\wM(z)
:=
B_\cH\operatorname{diag}\bigl(|m_e(z)|^2\bigr)_{e\in\vec E_\cH}.
\]
More precisely, 
\[
M_{ef}(z)=\mathbf 1_{\{e\rightsquigarrow f\}}m_f(z)^2,
\qquad
\wM_{ef}(z)
=\mathbf 1_{\{e\rightsquigarrow f\}}|m_f(z)|^2.
\]
The matrices $M(z)$ and $\wM(z)$ provide a finite-dimensional way to
encode weighted non-backtracking propagation over long distances in the universal cover or a large lift of the base graph $\cH$. Indeed, for every
$\ell\geq 1$,
\[
\bigl(M^\ell(z)\bigr)_{ef}
=
\sum_{\substack{
e=e_0\rightsquigarrow e_1\rightsquigarrow\cdots
\rightsquigarrow e_\ell=f}}~
\prod_{r=1}^{\ell}m_{e_r}(z)^2.
\]
Consequently, $\bigl(M^\ell(z)\bigr)_{ef}$ is the sum of the weights of
all $\ell$-step non-backtracking walks in the universal cover that
start from a fixed directed edge of type $e$ and end at a directed edge
of type $f$. The weight of such a walk is
$
\prod_{r=1}^{\ell} m_{e_r}(z)^2;
$
that is, every directed edge entered after the initial edge contributes
the factor $m_{e_r}(z)^2$. The same interpretation applies to
non-backtracking walks in a large lift of $\cH$.

Likewise,
\begin{equation}\label{eq.wM-ell}
\bigl(\wM^\ell(z)\bigr)_{ef}
=
\sum_{\substack{
e=e_0\rightsquigarrow e_1\rightsquigarrow\cdots
\rightsquigarrow e_\ell=f}}~
\prod_{r=1}^{\ell}|m_{e_r}(z)|^2,
\end{equation}
and therefore
\[
\left|\bigl(M^\ell(z)\bigr)_{ef}\right|
\leq
\bigl(\wM^\ell(z)\bigr)_{ef}.
\]
Thus, although $M$ and $\wM$ are indexed only by the directed edges of
the base graph, their powers encode the contributions of
non-backtracking walks over arbitrarily long distances in the
universal cover or in a large lift. The matrix $M^\ell$ retains the
complex phases and possible cancellations among the walk weights,
whereas $\wM^\ell$ provides a nonnegative majorant suitable for
long-range decay and summability estimates.

\begin{proposition}\label{pro.largest-eigenvalue}
The largest eigenvalue $\lambda_1(z)$ of $\wM(z)$ is positive and real. Moreover, we have
\[
1-\lambda_1(z)\gtrsim \frac{\eta}{1+\eta}, \qquad z\in\bC_+(\mathcal I_\delta).
\]
Furthermore, \(\widetilde M(z)\) admits a positive eigenvector
\(\boldsymbol v(z)=(v_e(z))_{e\in\vec E_{\mathcal H}}\) satisfying
$
\widetilde M(z)\boldsymbol v(z)=\lambda_1(z)\boldsymbol v(z),
$
and for every \(e,f\in\vec E_{\mathcal H}\) and \(z\in\bC_+(\mathcal I_\delta)\),
\[
v_e(z)\sim v_f(z).
\]
\end{proposition}

\begin{proof}
Since $m_e(z)\neq 0$ for every $e\in\vec E_\cH$ and $z\in\mathbb C^+(\mathcal I_\delta)$,
the matrix
\[
\wM(z)
=
B_\cH\operatorname{diag}\bigl(|m_e(z)|^2\bigr)_{e\in\vec E_\cH}
\]
has the same zero pattern as $B_\cH$. Hence, by the irreducibility of
$B_\cH$, the matrix $\wM(z)$ is irreducible and has nonnegative
entries. The Perron--Frobenius theorem therefore implies that
$\lambda_1(z)=\rho(\wM(z))$ is a positive real eigenvalue and $\boldsymbol v$ is the corresponding positive eigenvector. 
Moreover, $v_e\sim v_f$ follows directly from $|m_e|\sim |m_f|$ and the equation $
\widetilde M(z)\boldsymbol v(z)=\lambda_1(z)\boldsymbol v(z).
$

Fix $e=(k,l)\in\vec E_\cH$, and consider the Green function $G(z)$ of the
truncated universal cover $\cY_{e}$ rooted at $o$. By the Green
function calculation \eqref{e:Gtreemsc2} and the path interpretation \eqref{eq.wM-ell} of
$\wM^\ell$, for every $\ell\geq 0$,
\[
\sum_{\substack{x\in\bY_{e}\\ \dist(o,x)=\ell}}
|G_{ox}(z)|^2
=
|m_e(z)|^2\bigl(\wM^\ell(z)\mathbf 1\bigr)_e.
\]
Indeed, after fixing the initial directed edge type $e$, each vertex
\smallskip
at distance $\ell$ corresponds to a unique $\ell$-step
non-backtracking walk $e=e_0\rightsquigarrow e_1\rightsquigarrow\cdots
\rightsquigarrow e_\ell$, and its squared Green function weight is
$
|m_e(z)|^2\prod_{r=1}^{\ell}|m_{e_r}(z)|^2.
$

Summing over $\ell$ and applying the Ward identity \eqref{e:Ward}, we obtain
\[
|m_e(z)|^2
\sum_{\ell=0}^{\infty}
\bigl(\wM^\ell(z)\mathbf 1\bigr)_e
=
\sum_{x\in\cY_{e}}|G_{ox}(z)|^2
=
\frac{\Im G_{oo}(z)}{\eta}.
\]
Since $G_{oo}(z)=m_e(z)$, it follows that
\[
\sum_{\ell=0}^{\infty}
\bigl(\wM^\ell(z)\mathbf 1\bigr)_e
=
\frac{1}{\eta}
\frac{\Im m_e(z)}{|m_e(z)|^2}.
\]
Moreover, taking the imaginary parts in the Dyson equation \eqref{eq.DE} gives
\[
\frac{\Im m_e(z)}{|m_e(z)|^2}
=
\eta+\sum_{f:\,e\rightsquigarrow f}\Im m_f(z).
\]
Since the base graph is fixed and $\Im m_f(z)\le|m(z)|_2\lesssim 1$ by \eqref{eq.mbound-comparable}, we conclude that
\[
\sup_{e\in\vec E_\cH}
\sum_{\ell=0}^{\infty}
\bigl(\wM^\ell(z)\mathbf 1\bigr)_e
\lesssim 1+\frac{1}{\eta}.
\]

Let $\mathbf v=\mathbf v(z)>0$ be a left Perron eigenvector of
$\wM(z)$, normalized so that
$
\mathbf v^{\mathsf T}\mathbf 1=1.
$
Thus, it follows that
$
\mathbf v^{\mathsf T}\wM(z)
=
\lambda_1(z)\mathbf v^{\mathsf T}.
$
Then, for every $L\geq 0$, we have
\[
\begin{aligned}
\sum_{\ell=0}^{L}\lambda_1(z)^\ell=
\mathbf v^{\mathsf T}
\left(\sum_{\ell=0}^{L}\wM^\ell(z)\right)\mathbf 1
\leq
\max_{e\in\vec E_\cH}
\sum_{\ell=0}^{\infty}
\bigl(\wM^\ell(z)\mathbf 1\bigr)_e
\lesssim 1+\frac{1}{\eta}.
\end{aligned}
\]
The above bound is uniform in $L$, and therefore $\lambda_1(z)<1$. Letting $L\to\infty$ yields
$
(1-\lambda_1(z))^{-1} \lesssim 1+\eta^{-1}.
$
Consequently,
$
1-\lambda_1(z)\gtrsim\eta/(1+\eta),
$
as claimed.
\end{proof}

We next introduce a normalized version of the path weight which is
adapted to the Perron--Frobenius structure of the weighted
non-backtracking matrix $\wM(z)$.
Recall from Proposition \ref{pro.largest-eigenvalue} that
$$
\wM(z)\boldsymbol v(z)=\lambda_1(z)\boldsymbol v(z),\qquad\mbox{where~}\boldsymbol v(z)=(v_e(z))_{e\in\vec E_{\cH}}>0~\mbox{and}~v_e(z)\sim v_f(z).
$$

\begin{definition}[PF-normalized path weight]
\label{def:PF-path-weight}
For $k\geq1$,
let
$
\gamma=(x_0,x_1,\ldots,x_k)
$
be a simple path in a lift of $\cH$, and let
$
e_s=
\bigl(\pi(x_{s-1}),\pi(x_s)\bigr)\in\vec E_\cH.
$
Define
\begin{equation}
\label{eq:PF-path-weight}
\mathcal M^{\rm PF}(\gamma;z)
:=
\lambda_1(z)^{-|\gamma|/2}\mathcal M(\gamma;z)
\sqrt{\frac{v_{e_k}(z)}{v_{e_1}(z)}}.
\end{equation}
For the trivial path consisting of a single vertex, we set
$
\mathcal M^{\rm PF}(\gamma;z)=1.
$
\end{definition}

Then, it follows from $v_e\sim v_f$ and $\lambda_1\le 1$ that
\begin{equation}
\label{eq:PF-raw-comparison}
\left|\mathcal M(\gamma;z)\right|\lesssim\left|\mathcal M^{\rm PF}(\gamma;z)\right|.
\end{equation}

The normalization in \eqref{eq:PF-path-weight} is useful for the following reasons. 
The factors $m_e$ appearing in the original path weight are
not necessarily bounded in modulus by one; in particular, it may happen
that $|m_e|>1$.
Thus the weight along a fixed path may fluctuate locally, and an endpoint condition does not by itself
give a monotone stopping rule.
The PF-normalized weight gives a stronger local structure. Indeed, if
$\gamma$ ends with a directed edge of type $e$ and $\gamma f$ denotes
an admissible non-backtracking extension of $\gamma$ by an edge of type
$f$, then
\begin{equation}
\label{eq:PF-one-step-ratio}
\frac{|\mathcal M^{\rm PF}(\gamma f;z)|}
     {|\mathcal M^{\rm PF}(\gamma;z)|}
=
\lambda_1(z)^{-1/2}|m_f(z)|
\sqrt{\frac{v_f(z)}{v_e(z)}}.
\end{equation}
By the Perron--Frobenius equation, we have
\begin{equation}
\label{eq:PF-submarkov}
\sum_{f:\,e\rightsquigarrow f}
|m_f(z)|^2\frac{v_f(z)}{v_e(z)}
=
\lambda_1(z)\Longrightarrow \sum_{f:e\rightsquigarrow  f}
\frac{|\mathcal M^{\rm PF}(\gamma f;z)|^2}
     {|\mathcal M^{\rm PF}(\gamma;z)|^2}
=1.
\end{equation}
Hence $|\mathcal M^{\rm PF}|$ is non-increasing along every fixed
non-backtracking path.  Moreover,
\eqref{eq:PF-submarkov} shows that the squared PF-normalized weights define a conserved mass on the unfolded non-backtracking path tree.

\begin{proposition}\label{pro.path-weight-decay}
There exists a constant $\beta\in(0,1)$, depending only on the model parameters, such that for every simple path
$
\gamma
$ from $i$ to $j$
in a lift $\cG$ of the base graph $\cH$,
\[
\left|\mathcal M^{\rm PF}(\gamma;z)\right|\lesssim \beta^{|\gamma|}\le\beta^{\dist_\cG(i,j)},\qquad z\in\bC_+(\mathcal I_\delta).
\]
Consequently, we have $|\cM(\gamma;z)|\lesssim \beta^{|\gamma|}\le\beta^{\dist_\cG(i,j)}$ for all $z\in\bC_+(\mathcal I_\delta)$.
\end{proposition}

\begin{proof}
Let
$
\gamma=(e_1\rightsquigarrow e_2\rightsquigarrow\cdots
\rightsquigarrow e_k)
$
be the directed edge-type sequence associated with $\gamma$.  For
every admissible transition $e\rightsquigarrow f$, define
\[
p_{ef}(z)
:=
\frac{|m_f(z)|^2v_f(z)}
{\lambda_1(z)v_e(z)}.
\]
The Perron--Frobenius equation implies that
$
\sum_{f:e\rightsquigarrow f}p_{ef}(z)=1.
$
Moreover, using $\wM\boldsymbol{v}=\lambda_1\boldsymbol{v}$ yields
\[
p_{ef}(z)
=
\frac{|m_f(z)|^2v_f(z)}
{\sum_{g:e\rightsquigarrow g}|m_g(z)|^2v_g(z)}.
\]
Since Lemma \ref{lem:NB-Dyson-regularity} and Proposition \ref{pro.largest-eigenvalue} give
$
|m_f(z)|\sim |m_g(z)|
$
and 
$
v_f(z)\sim v_g(z)
$
uniformly in $z$, there exists a constant $c_*>0$ such that
\begin{equation}\label{pef}
p_{ef}(z)\geq c_*
\end{equation}
for every admissible transition $e\rightsquigarrow f$.

Since the minimum degree of the base graph $\cH$ is at least 3, every directed edge type $e$ has at least two admissible successors,
then $p_{ef}\ge c_*$ and
$\sum_f p_{ef}=1$ imply
\begin{equation}
\label{eq:PF-transition-contraction}
p_{ef}(z)\leq 1-c_*
\end{equation}
for every admissible successor $f$ of $e$.

On the other hand, by Definition~\ref{def:PF-path-weight}, we have
\[
|\mathcal M^{\rm PF}(\gamma;z)|^2
=
\lambda_1(z)^{-k}
\prod_{s=1}^k |m_{e_s}(z)|^2
\frac{v_{e_k}(z)}{v_{e_1}(z)}
=
\lambda_1(z)^{-1}|m_{e_1}(z)|^2
\prod_{s=2}^k
p_{e_{s-1}e_s}(z).
\]
It follows from $|m_{e_1}(z)|^2/\lambda_1(z)\lesssim1$ that
$
|\mathcal M^{\rm PF}(\gamma;z)|
\lesssim 
\beta^{|\gamma|}\le \beta^{\dist_\cG(i,j)},
$
where $\beta=\sqrt{1-c_*}$.

Finally, since $v_e\sim v_f$ and $\lambda_1(z)\leq1$,
Definition~\ref{def:PF-path-weight} gives
$
|\mathcal M(\gamma;z)|
\lesssim
|\mathcal M^{\rm PF}(\gamma;z)|\lesssim 
\beta^{|\gamma|}\le \beta^{\dist_\cG(i,j)}.
$
\end{proof}

\subsection{Tree extension and Elliptic neighborhood}\label{sc:ext}
\subsubsection{Tree extension}
\begin{definition}[Missing type]\label{def.deficit}
Let $\widehat{\cG}$ be a lift of the base graph $\cH$ with covering
map $\pi:\widehat{\cG}\to\cH$, and let $\cG\subset\widehat{\cG}$ be a
subgraph. We use the restriction of $\pi$ to $\cG$ as its type map.
For \(v\in\bG\), denote the set of directed edge types emanating from \(\pi(v)\) by
\[
\vec E_{\cH}(\pi(v))
:=
\bigl\{
e\in \vec E_{\cH}:
\operatorname{tail}(e)=\pi(v)
\bigr\},
\]
and the set of directed edge types present at \(v\) in \(\cG\) by
\[
\vec E_{\cG}(v)
:=
\bigl\{
\pi(e):
e\in \vec E_{\cG},\
\operatorname{tail}(e)=v
\bigr\}
\subset
\vec E_{\cH}(\pi(v)).
\]
The set of missing types at $v$ is
\[
\mathscr M_{\cG}(v)
:=
\vec E_{\cH}(\pi(v))\setminus \vec E_\cG(v).
\]
\end{definition}

\begin{definition}[Tree extension]\label{def.tree-extension}
Let $\cG$ be a subgraph of a lift of the base graph $\cH$, and let $v\in\bG$.
If
$
\mathscr M_\cG(v)\neq\varnothing,
$
then, for each directed missing edge type
$
e\in\mathscr M_\cG(v),
$
we attach to $v$ a fresh copy of the truncated universal cover branch
$\cY_e$ by joining $v$ to the root of $\cY_e$ with an edge of directed
type $e$.
All the attached copies of $\cY_e$ are taken to be mutually disjoint
apart from their attachment to $v$.
By performing the above operation simultaneously at every
vertex $v\in\bG$, 
we obtain the \emph{tree extension} of $\cG$, denoted by
$
\TE(\cG).
$

\end{definition}

The next proposition shows that the Green's function of the tree extension of a graph with at most one cycle can be bounded by that of the universal cover.
\begin{proposition}[Green's function estimates for tree extensions]\label{p:tree-extension-bound}
Let $\cG$ be a connected subgraph of a lift of $\cH$, containing at most one cycle. 
Then for all $z\in\bC^+(\mathcal I_\delta)$ and all $i,j\in\bG$,
\begin{equation}\label{e:TE-offdiagonal-bound}
\begin{aligned}
\left|
G_{ij}\bigl(\TE(\cG),z\bigr)
\right|
\lesssim\max_{\gamma: i\leadsto j}|\cM(\gamma;z)|\lesssim
\beta^{\dist_{\cG}(i,j)},\qquad
\left|
\Im G_{ij}\bigl(\TE(\cG),z\bigr)
\right|
\lesssim |\im m^{\rm ver}|_\infty,
\end{aligned}
\end{equation}
where $\beta\in(0,1)$ is the parameter defined in Proposition \ref{pro.path-weight-decay}. 
\end{proposition}

\subsubsection{Elliptic neighborhoods}
\label{sec:elliptic-neighborhood}
We now define the elliptic neighborhood.  Let
$\Gamma_r(o,x)$ denote the set of simple paths from $o$ to $x$
contained in $\cB_r(o,\cG)$ with $r=\fR/50$. Hence $\cB_r$ is a tree-like neighborhood by Definition \ref{def:barOmega}, which contains at most one cycle.

\begin{definition}[Elliptic neighborhood]
\label{def:elliptic-neighborhood}
For $z\in\bC_+(\mathcal I_\delta)$ and $\ell\sim\log_4\log N$, define
\begin{equation}
\label{eq:elliptic-neighborhood}
\cB_\ell^{\rm ell}(o,\cG;z)
:=
\left\{
x\in\cB_r(o,\cG):
\text{ there exists }\gamma\in\Gamma_r(o,x)
\text{ such that }
|\mathcal M^{\rm PF}(\gamma;z)|\geq2^{-\ell}
\right\}.
\end{equation}
\end{definition}
By $|\cM^{\rm PF}|$ is non-increasing, if a path
$\gamma$ is admissible in
\eqref{eq:elliptic-neighborhood}, then every prefix path of $\gamma$ is
admissible as well.  In particular,
$\cB_\ell^{\rm ell}(o,\cG;z)$ is connected.

\begin{remark}[Locality of the elliptic neighborhood]\label{rem:elliptic-locality}
The definition \eqref{eq:elliptic-neighborhood} is insensitive to
changes strictly outside the elliptic neighborhood.
More precisely, let
$
\cT:=\cB_\ell^{\rm ell}(o,\cG;z).
$
By Definition~\ref{def:elliptic-neighborhood}, $\cT$ depends only on
the edge types along admissible paths. Therefore, changing edges
entirely outside $\cT$, or rewiring edges between $\cT$ and its
exterior while preserving their edge types, does not affect the
elliptic neighborhood.
\end{remark}

We next collect the main estimates for the elliptic neighborhood. 
To formulate them conveniently, we first introduce the elliptic level of a vertex. For $x\in\cB_\ell^{\rm ell}(o,\cG;z)$, define its elliptic level by
\begin{equation}
\label{eq:elliptic-level}
\ell_x
:=
\min\left\{
k\geq0:
x\in\cB_k^{\rm ell}(o,\cG;z)
\right\}.
\end{equation}

\begin{proposition}[Estimates for the elliptic neighborhood]
\label{prop:elliptic-neighborhood-estimates}
Let $z\in\bC_+(\mathcal I_\delta)$ and $\ell=\OO(\log\log N)$. Then the following estimates hold.

\begin{itemize}
\item[(i)] The outer vertex boundary of the elliptic neighborhood
satisfies
\begin{equation}
\label{eq:elliptic-boundary-cardinality}
\left|
\del\cB_\ell^{\rm ell}(o,\cG;z)
\right|
\sim
4^\ell.
\end{equation}
\item[(ii)] For every
$j\in\cB_\ell^{\rm ell}(o,\cG;z)$,
\begin{equation}
\label{eq:elliptic-boundary-L2}
\sum_{y\in\del\cB_\ell^{\rm ell}(o,\cG;z)}
|G_{jy}(\TE(\cB_r(o,\cG)),z)|^2
\lesssim1.
\end{equation}

\item[(iii)] For every
$j\in\cB_\ell^{\rm ell}(o,\cG;z)$,
\begin{equation}
\label{eq:elliptic-interior-L2}
\sum_{y\in\cB_\ell^{\rm ell}(o,\cG;z)}
|G_{jy}(\TE(\cB_r(o,\cG)),z)|^2
\lesssim\ell.
\end{equation}

\item[(iv)] If
$j\in\cB_\ell^{\rm ell}(o,\cG;z)$ has elliptic level $\ell_j$, then
\begin{equation}
\label{eq:elliptic-boundary-L1}
\sum_{y\in\del\cB_\ell^{\rm ell}(o,\cG;z)}
|G_{jy}(\TE(\cB_r(o,\cG)),z)|
\lesssim
(1+\ell_j) 2^{\ell-\ell_j}.
\end{equation}

\item[(v)] If
$i,j\in\cB_\ell^{\rm ell}(o,\cG;z)$ have elliptic levels
$\ell_i$ and $\ell_j$, respectively, then
\begin{equation}
\label{eq:elliptic-boundary-pair}
\sum_{y\in\del\cB_\ell^{\rm ell}(o,\cG;z)}
|G_{iy}(\TE(\cB_r(o,\cG)),z)|\,|G_{yj}(\TE(\cB_r(o,\cG)),z)|
\lesssim
\bigl(1+|\ell_i-\ell_j|\bigr)
2^{-|\ell_i-\ell_j|}.
\end{equation}
\end{itemize}
\end{proposition}

\begin{proof}
For simplicity, write
$
G_{ij}:=G_{ij}(\TE(\cB_r(o,\cG)),z),
$
and suppress $(o,\cG;z)$ in the notation $\cB_\ell^{\rm ell}$. By $\OO(\log_4\log N)=\ell\ll r=\OO(\log N)$, all
relevant stopping paths and boundary vertices lie inside
$\cB_r(o,\cG)$.

Since $\cB_r(o,\cG)$ contains at most one cycle, there are
at most two simple paths between any two vertices.
Thus Proposition~\ref{p:tree-extension-bound} and
\eqref{eq:PF-raw-comparison} give
\begin{equation}
\label{eq.ell-G-MPF}
|G_{ij}|
\lesssim
\max_{\gamma:i\to j}
|\mathcal M^{\rm PF}(\gamma;z)|,
\end{equation}
where the maximum is over simple paths in $\cB_r(o,\cG)$.
Throughout the proof, $x\leadsto y$ denotes a path attaining
this maximum.

We shall use the multiplicative comparison
\[
|\mathcal M^{\rm PF}(\gamma_1\gamma_2;z)|
\sim
|\mathcal M^{\rm PF}(\gamma_1;z)
  \mathcal M^{\rm PF}(\gamma_2;z)|
\]
whenever $\gamma_1$ and $\gamma_2$ can be joined to form a simple path.
Furthermore, for $x,y\in\cB_\ell^{\rm ell}$, we have
\begin{equation}\label{eq:PF-level-comparison}
|\mathcal M^{\rm PF}(x\leadsto y;z)|
\lesssim
2^{-|\ell_x-\ell_y|}.
\end{equation}
\medskip
\noindent
\textit{The tree case.}
We first assume that $\cB_r(o,\cG)$ is a tree.
All path-mass calculations are performed on its full tree
extension $\TE(\cB_r(o,\cG))$.

Let $\mathscr A_1(o)$ denote the one-step paths starting from
$o$. For every nonempty path $\gamma$ starting from $o$, define
\begin{equation}
\label{eq.defmu}
\mu(\gamma)
:=
Z_o(z)^{-1}|\mathcal M^{\rm PF}(\gamma;z)|^2,
\qquad
Z_o(z)
:=
\sum_{\gamma\in\mathscr A_1(o)}
|\mathcal M^{\rm PF}(\gamma;z)|^2.
\end{equation}
Set $\mu((o)):=1$ for the trivial path.
For any other starting vertex $j$, define $Z_j$ and $\mu_j$
analogously.

The PF equation and the uniform comparison of the cavity
moduli and PF eigenvector entries give
$\lambda_1(z)\sim|m_e(z)|^2$.
Consequently, the one-step PF weights have modulus comparable
to one, and $Z_j\sim1$ uniformly in $j$.
For a nonempty path $\gamma$ ending with edge type $e$,
\[
\frac{|\mathcal M^{\rm PF}(\gamma f;z)|^2}
     {|\mathcal M^{\rm PF}(\gamma;z)|^2}
=
p_{ef}(z),
\qquad
c\leq p_{ef}(z)\leq1-c.
\]
Here the lower bound follows from \eqref{pef}, while the
upper bound uses the minimum-degree assumption and
$\sum_f p_{ef}=1$.

Thus the PF identity \eqref{eq:PF-submarkov} gives
\[
\sum_{\gamma'\in\operatorname{Ch}(\gamma)}
\mu_j(\gamma')=\mu_j(\gamma),
\]
where $\operatorname{Ch}(\gamma)$ denotes the set of all admissible
non-backtracking one-step extensions of $\gamma$.
A family is called \emph{prefix-free} if no member is a proper
prefix of another member.
For every such family $\mathscr P$ starting from $j$,
mass conservation implies
\begin{equation}
\label{eq:PF-prefix-free-bound}
\sum_{\gamma\in\mathscr P}\mu_j(\gamma)\leq1,
\qquad
\sum_{\gamma\in\mathscr P}
|\mathcal M^{\rm PF}(\gamma;z)|^2\lesssim1.
\end{equation}
The first inequality is an equality when $\mathscr P$ is a
finite cutset meeting every branch.

We first prove \eqref{eq:elliptic-boundary-cardinality}.
Let $\Gamma_\ell$ be the family of first-stopping paths:
\[
\Gamma_\ell
:=
\left\{
\gamma:
|\mathcal M^{\rm PF}(\gamma^-;z)|\geq2^{-\ell},
\quad
|\mathcal M^{\rm PF}(\gamma;z)|<2^{-\ell}
\right\}.
\]
Here $\gamma^-$ is the parent path.
The one-step lower bounds imply
\[
|\mathcal M^{\rm PF}(\gamma;z)|\sim2^{-\ell},
\qquad
\mu(\gamma)\sim4^{-\ell},
\qquad \gamma\in\Gamma_\ell.
\]
By Proposition~\ref{pro.path-weight-decay}, every branch
reaches this stopping level within $O(1+\ell)$ steps.
Thus $\Gamma_\ell$ is a finite cutset.
Since root paths are unique, its endpoints are in bijection
with $\del\cB_\ell^{\rm ell}$, and therefore
\[
1
=
\sum_{\gamma\in\Gamma_\ell}\mu(\gamma)
\sim
4^{-\ell}|\Gamma_\ell|
=
4^{-\ell}|\del\cB_\ell^{\rm ell}|.
\]
This proves \eqref{eq:elliptic-boundary-cardinality}.
We also obtain
\begin{equation}
\label{eq:PF-boundary-root-weight}
|\mathcal M^{\rm PF}(o\leadsto y;z)|
\sim2^{-\ell},
\qquad y\in\del\cB_\ell^{\rm ell}.
\end{equation}

We next prove the two-point estimate at the level of path
weights.
Fix $i,j\in\cB_\ell^{\rm ell}$.
For $y\in\del\cB_\ell^{\rm ell}$, let $b(i,j,y)$ denote the
common vertex of the three pairwise paths joining $i,j,y$.
For a fixed $b\in i\leadsto j$, the paths from $b$ to boundary
vertices satisfying $b(i,j,y)=b$ form a prefix-free family,
since $b\in i\leadsto j\subset\cB_\ell^{\rm ell}$.
Applying \eqref{eq:PF-prefix-free-bound} with root $b$ gives
\[
\sum_{\substack{
y\in\del\cB_\ell^{\rm ell}:\\ b(i,j,y)=b}}
|\mathcal M^{\rm PF}(b\leadsto y;z)|^2
\lesssim1.
\]
By the multiplicative comparison,
\begin{equation}
\label{eq:PF-pair-fixed-branch}
\sum_{\substack{
y\in\del\cB_\ell^{\rm ell}:\\ b(i,j,y)=b}}
|\mathcal M^{\rm PF}(i\leadsto y;z)
  \mathcal M^{\rm PF}(j\leadsto y;z)|
\lesssim
|\mathcal M^{\rm PF}(i\leadsto b;z)
  \mathcal M^{\rm PF}(j\leadsto b;z)|.
\end{equation}

There are only $O(1)$ vertices of each dyadic elliptic level
on $i\leadsto j$.
Indeed, this path splits at $\anc(i,j)$ into two root-directed
segments, and the bound $p_{ef}\leq1-c$ gives geometric
contraction along each segment away from $o$, apart from the
harmless first step.
Summing \eqref{eq:PF-pair-fixed-branch} and using
\eqref{eq:PF-level-comparison}, we obtain
\begin{equation}
\label{eq:PF-pair-main-levels}
\begin{aligned}
\sum_{y\in\del\cB_\ell^{\rm ell}}
|\mathcal M^{\rm PF}(i\leadsto y;z)
  \mathcal M^{\rm PF}(j\leadsto y;z)|
\lesssim
\sum_{b\in i\leadsto j}
2^{-|\ell_i-\ell_b|-|\ell_j-\ell_b|}
\lesssim
\sum_{t\geq0}2^{-|\ell_i-t|-|\ell_j-t|}
\lesssim
(1+|\ell_i-\ell_j|)2^{-|\ell_i-\ell_j|}.
\end{aligned}
\end{equation}
The last step separates the levels between $\ell_i$ and
$\ell_j$ from the geometric tails.
Together with \eqref{eq.ell-G-MPF}, this proves
\eqref{eq:elliptic-boundary-pair}.

Taking $i=j$ in \eqref{eq:PF-pair-main-levels} gives
\eqref{eq:elliptic-boundary-L2}.
Taking $i=o$ and using \eqref{eq:PF-boundary-root-weight}
gives
\[
\begin{aligned}
\sum_{y\in\del\cB_\ell^{\rm ell}}|G_{jy}|
\lesssim
2^\ell
\sum_{y\in\del\cB_\ell^{\rm ell}}
|\mathcal M^{\rm PF}(o\leadsto y;z)
  \mathcal M^{\rm PF}(j\leadsto y;z)|\lesssim
(1+\ell_j)2^{\ell-\ell_j},
\end{aligned}
\]
which proves \eqref{eq:elliptic-boundary-L1}.

Finally, Proposition~\ref{pro.path-weight-decay} implies
$\cB_\ell^{\rm ell}\subset\cB_{C(1+\ell)}(o,\cG)$.
Increasing $C$ if necessary, every path from
$j\in\cB_\ell^{\rm ell}$ to another vertex of this subtree
has length at most $C(1+\ell)$.
For each fixed length, the paths starting from $j$ form a
prefix-free family.
Hence \eqref{eq:PF-prefix-free-bound} and $Z_j\sim1$ give
\[
\begin{aligned}
\sum_{y\in\cB_\ell^{\rm ell}}|G_{jy}|^2
\lesssim
\sum_{y\in\cB_\ell^{\rm ell}}
|\mathcal M^{\rm PF}(j\leadsto y;z)|^2\lesssim
\sum_{k=0}^{C(1+\ell)}
\sum_{\substack{|\gamma|=k:\\\gamma\text{ starts at }j}}
\mu_j(\gamma)
\lesssim1+\ell
\lesssim\ell.
\end{aligned}
\]
This proves \eqref{eq:elliptic-interior-L2}.

\medskip
\noindent
\textit{The one-cycle case.}
Removing cycle-closing successors
changes mass conservation into an inequality, but preserves
\eqref{eq:PF-prefix-free-bound}.
Thus the first-stopping family still satisfies
$|\Gamma_\ell|\lesssim4^\ell$.
Consequently,
$|\del\cB_\ell^{\rm ell}|\lesssim4^\ell$.
The one-step lower bound also proves
\eqref{eq:PF-boundary-root-weight} for the actual boundary.

For the reverse cardinality bound, choose an edge out of $o$
whose forward component does not meet the cycle.
Such an edge exists by the minimum-degree assumption.
This branch carries a uniformly positive initial PF mass,
and its vertices have unique simple paths from $o$.
The tree stopping argument therefore produces
$\gtrsim4^\ell$ distinct genuine outer-boundary vertices.
This proves \eqref{eq:elliptic-boundary-cardinality}.

To transfer the summation estimates, fix
$i,j\in\cB_\ell^{\rm ell}$ and let $\mathscr C$ consist of the
unique cycle together with the tree paths connecting $o,i,j$
to it.
Every component outside $\mathscr C$ is a tree attached to
$\mathscr C$ at a unique vertex.
For $y\notin\mathscr C$, let $b(y)$ denote this attachment
point, and set $b(y)=y$ for $y\in\mathscr C$.
For a fixed attachment point $b$, every path from $o,i,j$ to
a vertex $y$ with $b(y)=b$ has the same terminal tree segment
$b\leadsto y$. Hence, by the multiplicative comparison of the
PF weights,
\[
|\mathcal M^{\rm PF}(i\leadsto y;z)
  \mathcal M^{\rm PF}(j\leadsto y;z)|
\lesssim
|\mathcal M^{\rm PF}(i\leadsto b;z)
  \mathcal M^{\rm PF}(j\leadsto b;z)|
|\mathcal M^{\rm PF}(b\leadsto y;z)|^2 .
\]
Moreover, the elliptic neighborhood intersects each attached
tree in an initial subtree. Therefore the paths from $b$ to
its outer boundary form a prefix-free family, and the tree
estimate gives
\[
\sum_{\substack{
y\in\del\cB_\ell^{\rm ell}:\\ b(y)=b}}
|\mathcal M^{\rm PF}(b\leadsto y;z)|^2
\lesssim1.
\vspace{-3mm}
\]
Thus the fixed-$b$ estimate
\[
\sum_{\substack{
y\in\del\cB_\ell^{\rm ell}:\\ b(y)=b}}
|\mathcal M^{\rm PF}(i\leadsto y;z)
  \mathcal M^{\rm PF}(j\leadsto y;z)|
\lesssim
|\mathcal M^{\rm PF}(i\leadsto b;z)
  \mathcal M^{\rm PF}(j\leadsto b;z)|
\vspace{-2mm}
\]
is exactly the same as in the tree case. The corresponding
interior estimate follows from the same length-layer argument.

It remains only to sum over the possible attachment points.
We claim that there are $\OO(1)$ vertices of $\mathscr C$ in each dyadic PF
level. Indeed, the tree part of $\mathscr C$ consists of a
bounded number of paths along which the PF weights vary
geometrically. On the cycle, there are two simple paths from
the attachment point of $o$ to a given vertex, and
\[
|\mathcal M^{\rm PF}(o\leadsto b;z)|
\sim
\max\{
|\mathcal M^{\rm PF}(\gamma_+(b);z)|,\,
|\mathcal M^{\rm PF}(\gamma_-(b);z)|
\},
\]
where
$\gamma_+(b)$ and
$\gamma_-(b)$ denote the two simple paths from $o$ to $b$,
obtained by following the two different directions around
the cycle.
As $b$ moves around the cycle, the two directional path weights
vary geometrically in opposite directions, since each additional
step multiplies the squared PF weight by $p_{ef}$ with
$c\le p_{ef}\le1-c$. Hence, for every $t$,
\[
\#\left\{
b\in\mathscr C:
|\mathcal M^{\rm PF}(o\leadsto b;z)|\sim 2^{-t}
\right\}
\lesssim1.
\]
Therefore the summation over $b$ reduces to exactly the same
dyadic-level summation as in the tree case.

Using \eqref{eq:PF-level-comparison}, we may therefore repeat
the summation in \eqref{eq:PF-pair-main-levels}, with
$b\in\mathscr C$ in place of $b\in i\leadsto j$.
This proves the same two-point path-weight bound, from which
\eqref{eq:elliptic-boundary-pair},
\eqref{eq:elliptic-boundary-L2}, and
\eqref{eq:elliptic-boundary-L1} follow exactly as above.
For the interior bound, each relevant terminal tree segment
has length $O(\ell)$, since it belongs to a root path of
weight at least $2^{-\ell}$.
The same length-layer estimate consequently gives
\[
\begin{aligned}
\sum_{y\in\cB_\ell^{\rm ell}}|G_{jy}|^2
\lesssim
\ell\sum_{b\in\mathscr C}
|\mathcal M^{\rm PF}(j\leadsto b;z)|^2\lesssim
\ell\sum_{t\geq0}4^{-|\ell_j-t|}
\lesssim\ell.
\end{aligned}
\]
This proves \eqref{eq:elliptic-interior-L2} and completes
the proof.
\end{proof}

\subsection{Extension with weights}\label{sec.extension weight}

\begin{definition}[Extension with weight $q(z)$]
\label{def:ext}
Let
$
q(z)
=
\bigl(q_e(z)\bigr)_{e\in\vec E_{\cH}}
$
be a collection of weights indexed by the directed edge types of
$\cH$.
The extension of $\cG$ with weight $q(z)$,
denoted by
$
\Ext(\cG,q(z)),
$
is obtained by assigning every vertex $v\in\bG$
the weight
$
-\sum_{e\in\mathscr M_\cG(v)}q_{ e}(z)
$
.
We regard $\Ext(\cG,q(z))$ as a virtual graph and write
$
G(\Ext(\cG,q(z)),z)
$
for its Green's function.
\end{definition}

Equivalently, if $A$ denotes the adjacency matrix of $\cG$, then the
matrix associated with the extension is
\[
A-
\sum_{v\in\bG}
\sum_{e\in\mathscr M_\cG(v)}
q_{e}(z)e_{vv},
\]
where
$
(e_{vv})_{ij}
=
\delta_{vi}\delta_{vj}.
$
Hence its Green's function is
\begin{align}
G\bigl(\Ext(\cG,q(z)),z\bigr)
=
\left(
A-z
-
\sum_{v\in\bG}
\sum_{e\in\mathscr M_\cG(v)}
q_{ e}(z)e_{vv}
\right)^{-1}.
\label{e:formula}
\end{align}

The above extension with weight can be viewed as a general tree extension. Indeed, recall that
$
m(z)
=
\bigl(m_e(z)\bigr)_{e\in\vec E_{\cH}}
$
is the solution to the Dyson equation \eqref{eq.DE}. Proposition \ref{greentree} implies that, for $e=(i,j)$, $m_e$ is the diagonal Green function at a root of type $j$ after deleting its edge toward type $i$. Combining this observation with \eqref{e:formula}, one can see that for every $i,j\in\bG$,
\begin{equation*}
  G_{ij}(\Ext(\cG,m(z)),z)=G_{ij}(\TE(\cG),z).
\end{equation*}
Thus, the tree extension is a special case of the extension with weight, corresponding to the choice
$q(z)=m(z).
$
This naturally suggests that if $q(z)$ is sufficiently close to $m(z)$, then
$
G_{ij}\bigl(\Ext(\cG,q(z)),z\bigr)
$
should be close to the Green function of the tree extension. We therefore turn to the regime where $q(z)$ is a small perturbation of $m(z)$, i.e., $|q(z)-m(z)|_\infty\ll 1$.

\begin{lemma}\label{p:subgraph}
Let $\cG$ be a finite connected subgraph of a lift of $\cH$, containing at most one cycle.
If the weight $q(z)$ satisfies $(1+\diam(\cG))|q(z)-m(z)|_\infty\ll1$, then for every $i,j\in \bG$ 
\begin{align}\begin{split} \label{e:smalldiff}
\left|G_{ij}(\Ext(\cG, q(z)), z)-G_{ij}(\Ext(\cG, m),z)\right|
\lesssim (1+\dist_{\cG}(i,j))\max_{\gamma: i\leadsto j}|\cM(\gamma;z)||q(z)-m(z)|_\infty.
\end{split}\end{align}
As a consequence, the first inequality in Proposition \ref{p:tree-extension-bound} and Proposition \ref{prop:elliptic-neighborhood-estimates} also hold for $G(\Ext(\cG,q(z)),z)$.
\end{lemma}

\begin{remark}\label{r:subgraph}
Under the same assumptions as in Lemma \ref{p:subgraph}, if two weights $q_1(z)$ and $q_2(z)$ satisfy
$(1+\diam(\cG))|q_1(z)-m(z)|_\infty, (1+\diam(\cG))|q_2(z)-m(z)|_\infty\ll1$, then by the same argument as in Lemma \ref{p:subgraph} we have 
\begin{align*}
\left|G_{ij}(\Ext(\cG, q_1(z)),z)-G_{ij}(\Ext(\cG,q_2(z)),z)\right|
\lesssim (1+\dist_{\cG}(i,j))\max_{\gamma: i\leadsto j}|\cM(\gamma;z)||q_1(z)-q_2(z)|_\infty.
\end{align*}
\end{remark}

If $\cG$ is a finite rooted tree, then the tree extension of $\cG$ is the universal cover tree with root vertex $o$. Therefore, $m(z)$ is a fixed point: $G_{oo}(\Ext(\cG, m(z)),z)=m_{o}(z)=m_{io}/(1-m_{oi}m_{io})$ for any $i\sim o$. The following lemma gives its stability estimates, i.e. if $q(z)$ is sufficiently close to $m(z)$ then $G_{oo}(\Ext(\cG,q(z)), z)$ is sufficiently close to $q_{io}(z)/(1-q_{oi}(z)q_{io}(z))$.

\begin{lemma}[Fixed-point estimate]
\label{p:fixpoint}
Fix $z\in\bC_+(\mathcal I_\delta)$ with $|z|\le 2 d_{\rm max}$.
Let $\cG$ be a finite connected subgraph of a lift of $\cH$,
and assume that $\cG$ is a tree with root $o$.
If the weight satisfies
$
(1+\diam(\cG))|q-m|_\infty\ll1,
$
then for every directed base edge $e$ with $h(e)=\pi(o)$,
\[
\left|
G_{oo}(\Ext(\cG,q),z)
-
\frac{q_e}{1-q_eq_{\bar e}}
\right|
\lesssim
(1+\diam(\cG))|R(q,z)|_\infty,
\]
where the Dyson-residual $R$ is defined in \eqref{eq.Dyson-residual}.
\end{lemma}

The localization principle shows that, up to a controlled error, the Green function \(G_{ij}\) depends only on a local neighborhood of \(i\) and \(j\), so that the ambient graph \(\cG\) may be replaced by any subgraph \(\cK\) containing \(\cB_r(i,j,\cG)\). 
In particular, it allows us to enlarge \(\cB_r(i,j,\cG)\) by including additional vertices that are relevant to the argument and then work on the resulting localized subgraph; this procedure will be used repeatedly in Section~\ref{sec:stability}.

\begin{lemma}[Localization principle]\label{p:localization}
Fix $z\in\bC_+(\mathcal I_\delta)$ with $|z|\le 2 d_{\rm max}$.
Let $\cG$ be a finite connected subgraph of a lift of $\cH$,
containing at most one cycle.
If
$
(1+\diam(\cG))|q-m|_\infty\ll1,
$
then, for every $i,j\in\bG$ and every subgraph $\cK$ of $\cG$
containing $\cB_r(\{i,j\},\cG)$,
\[
\left|
G_{ij}(\Ext(\cG,q),z)
-
G_{ij}(\Ext(\cK,q),z)
\right|
\lesssim
(1+\diam(\cG))|R(q,z)|_\infty+\beta^{2r}.
\]
\end{lemma}

\begin{remark}
The localization principle also applies to the
mixed extensions arising in the switching arguments
in Section \ref{sec:stability} and Section \ref{sec:concentration}. Exact universal-cover branches attached to
the retained subgraph $\cK$ are kept intact, while
those attached outside $\cK$ are separated from
the vertices under consideration. By the resolvent
identity and the $\beta$-decay of the corresponding
Green functions, their total contribution is
absorbed into the existing localization error.
The same convention is preserved before and after
switching, with exact branches reattached intact.
\end{remark}

Since the proofs of all the above lemmas consist of direct computations and closely follow the corresponding arguments in \cite[Section 2.3]{Huang2024Spectrum}, we defer them to Appendix \ref{a:ext}.

\section{Local Resampling and Geometric Control}\label{sec:local-resampling}
In this section, we introduce a local resampling procedure adapted to
the elliptic neighborhood introduced in
Definition~\ref{def:elliptic-neighborhood}.
The construction is the random-lift analogue of the local resampling
for random $d$-regular graphs in \cite{Huang2024Spectrum}, with the metric neighborhood replaced by
the elliptic neighborhood and with all switchings required to preserve
the oriented edge types of the base graph.

\subsection{Elliptic Local Resampling}
\begin{definition}[Type-preserving switching]
\label{def:type-preserving-switching}
Let $ e_1=(v_1,v_2)$ and $ e_2=(v_3,v_4)$ be two disjoint oriented lift edges with the same type.
The switching encoded by
$
    \vec S=\{e_1, e_2\}
$
replaces the unoriented edges
$\{v_1,v_2\}$ and $\{v_3,v_4\}$
by $\{v_1,v_4\}$ and $\{v_3,v_2\}$.
We denote the resulting lift by
$
    T_{\vec S}(\cG),
$
and define the switched data by
$
    T(\vec S)
    :=
    \{(v_1,v_4),(v_3,v_2)\}.
$
\end{definition}
Since the two original edges have the same oriented base-edge type,
the two new edges have the same type as well. Thus the switching
preserves the covering map and \(T_{\vec S}(\cG)\in\GNH\).

Our local resampling procedure is specified by a fixed center vertex $o$, a parameter
$z\in\bC_+(\mathcal I_\delta)$ and a radius $\ell\ge 0$.
Given a lift \(\cG\in\GNH\), we write
$
    \cT
    :=
    \cB_\ell^{\rm ell}(o,\cG;z)
$, which may not be a tree.
Following our previous convention, we denote its vertex set by $\bT$ and define its oriented edge boundary set by 
$
    \partial_{\vec E}\cT
    :=
    \left\{
    (x,y)\in\vec E_\cG:
    x\in\bT,\ y\notin\bT
    \right\}.
$
We then enumerate
$
    \partial_{\vec E}\cT
    =
    \{e_1,\dots, e_\mu\},
$
where
$
     e_\alpha=(l_\alpha,a_\alpha)
$
with $l_\alpha\in\bT$ and $
    a_\alpha\notin\bT.
$
For an oriented base-edge type
\( e\in\vec E_\cH\), define
\[
    \vec E^{ e}
    _{\cG^{(\bT)}}
    :=
    \left\{
    (b,c)\in\vec E_\cG:
    b,c\notin\bT,\ 
    (\pi(b),\pi(c))=e
    \right\}.
\]
Conditionally on \(\cG\), for each
\(\alpha\in\llbracket1,\mu\rrbracket\), choose
$
    (b_\alpha,c_\alpha)
$
uniformly from
$
    \vec E^{e_\alpha}_{\cG^{(\bT)}},
$
independently for different \(\alpha\)'s.
Repetitions are allowed.
Then, we define
\begin{equation}
\label{eq:def-elliptic-Salpha}
    \vec S_\alpha
    :=
    \left\{
    (l_\alpha,a_\alpha),
    (b_\alpha,c_\alpha)
    \right\},
    \qquad
    {\bf S}
    :=
    (\vec S_1,\dots,\vec S_\mu).
\end{equation}
The collection \({\bf S}\) will be called the
\emph{resampling data} for \(\cG\).

For
\(\alpha\in\llbracket1,\mu\rrbracket\), define the indicator function 
$
    I_\alpha(\cG,{\bf S})=1
$
if 
the local neighborhood $\cB_{\fR/4}(\{a_\alpha,b_\alpha,c_\alpha\},\cG^{(\bT)})$ after adding
the edge \(\{a_\alpha,b_\alpha\}\) is a tree.
Otherwise we set \(I_\alpha=0\).
We further define
$
    J_\alpha(\cG,{\bf S})=1
$
if
$
\dist_{\cG^{(\bT)}}
\left(
\{a_\alpha,b_\alpha,c_\alpha\},
\{a_\beta,b_\beta,c_\beta\}
\right)
>
\fR/4
$
for every
$
    \beta\in
    \llbracket1,\mu\rrbracket
    \setminus\{\alpha\},
$
and \(J_\alpha=0\) otherwise.

The two indicator functions encode two different geometric requirements.
First, the condition \(I_\alpha=1\) imposes both a local tree condition
and a separation condition. Indeed, it implies that
$
    \cB_{\fR/4}(a_\alpha,\cG^{(\bT)})
$
and
$
    \cB_{\fR/4}(\{b_\alpha,c_\alpha\},\cG^{(\bT)})
$
are both trees. Moreover, \(a_\alpha\) has to be well separated from
the sampled edge \(\{b_\alpha,c_\alpha\}\). In particular, we have
$
    \dist_{\cG^{(\bT)}}
    \bigl(a_\alpha,\{b_\alpha,c_\alpha\}\bigr)
    >
    \fR/2.
$
Indeed, if the
distance were at most \(\fR/2\), a shortest path from \(a_\alpha\)
to either \(b_\alpha\) or \(c_\alpha\) would be entirely contained in
$
    \cB_{\fR/4}
    (\{a_\alpha,b_\alpha,c_\alpha\},\cG^{(\bT)}).
$
Together with the edge \(\{b_\alpha,c_\alpha\}\), 
adding the new edge \(\{a_\alpha,b_\alpha\}\) would then create a
cycle, contradicting \(I_\alpha=1\).

On the other hand, \(J_\alpha\) is an isolation condition between
different switching pairs.  If \(J_\alpha=1\), then for every
\(\beta\ne\alpha\), the support
\(\{a_\beta,b_\beta,c_\beta\}\) of the \(\beta\)-th switching lies
outside the radius-\(\fR/4\) neighborhood used to define
\(I_\alpha\).  Hence, performing the $\beta$-th switching does not alter the radius-$\fR/4$ neighborhood associated with any other index $\alpha$, and in particular preserves the corresponding tree condition $I_\alpha$.  This separation
prevents interactions between distinct simple switchings and, in
particular, guarantees that the admissibility condition is stable
under the local resampling. See Figure \ref{fig:switching} for details.

\begin{figure}[ht]
    \centering
    \begin{minipage}[c]{0.49\textwidth}
        \centering
        \includegraphics[width=\linewidth]{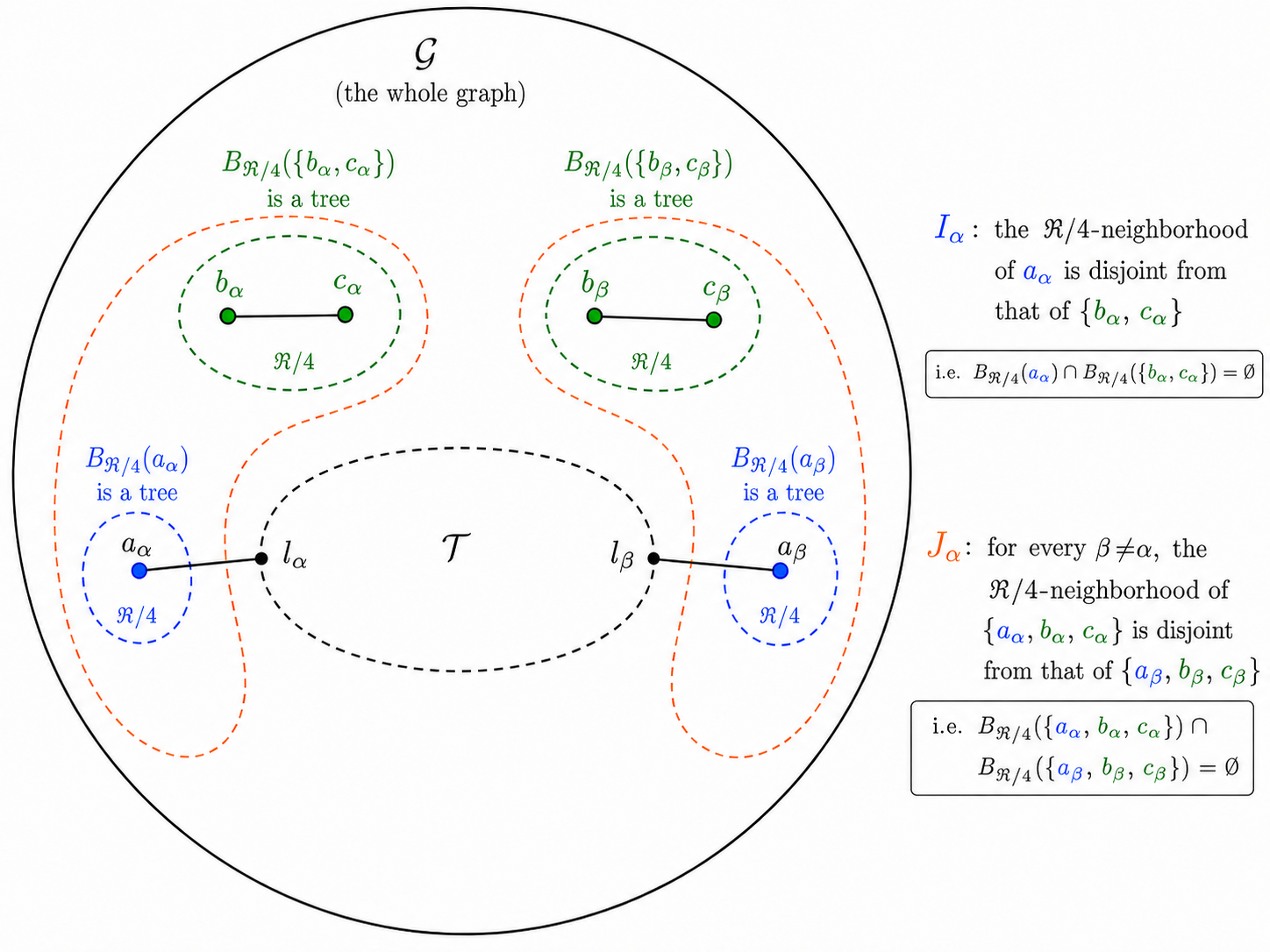}
    \end{minipage}
    \hfill
    \begin{minipage}[c]{0.49\textwidth}
        \centering
        \includegraphics[width=\linewidth]{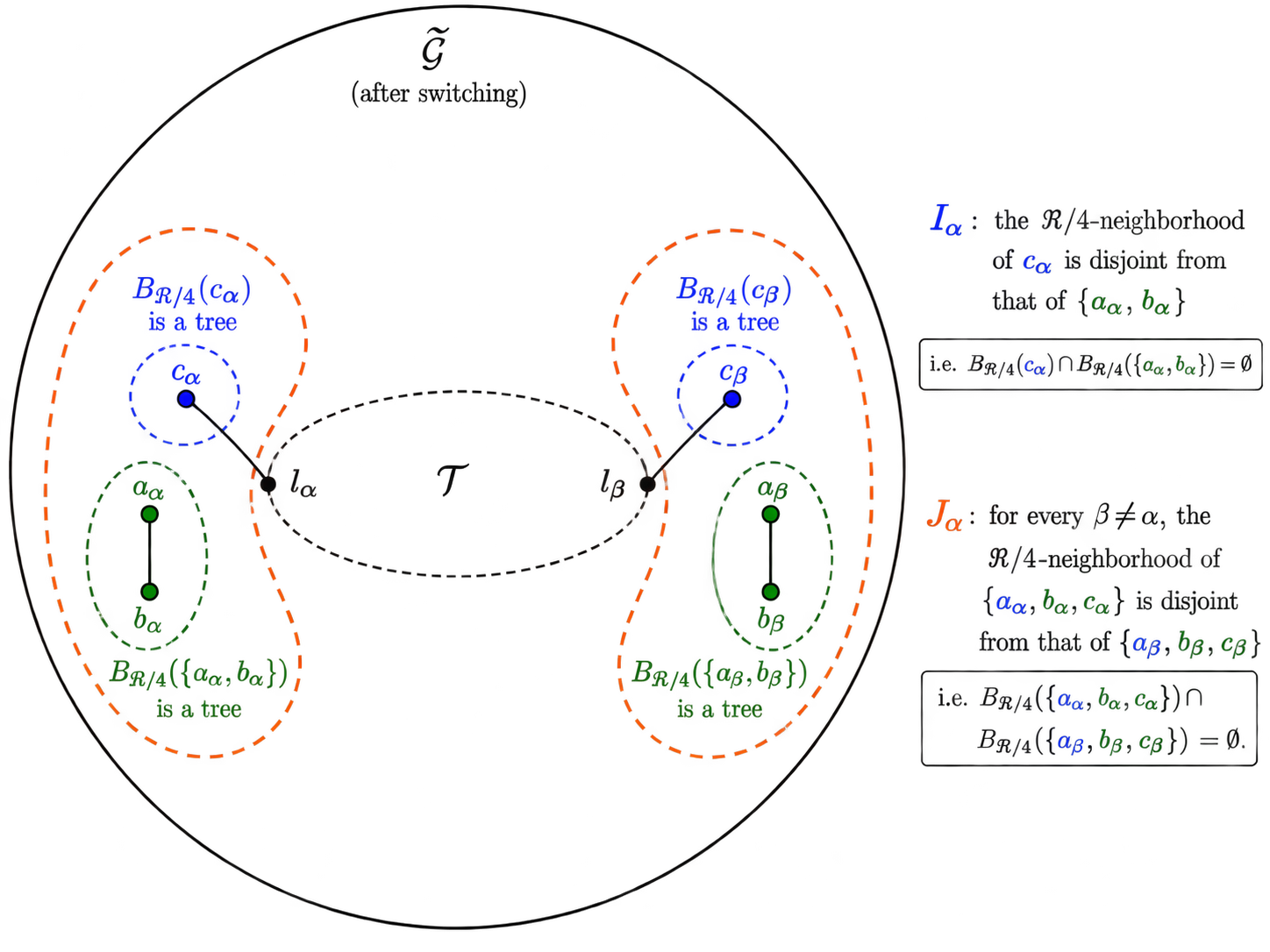}
    \end{minipage}
    \caption{As is clear from the figure, if $a_\alpha,b_\alpha,c_\alpha$ satisfy the conditions $I_\alpha$ and $J_\alpha$ before the switching, then they satisfy the corresponding conditions in the switched graph as well, since the support of every other switching lies outside the radius-$\fR/4$ neighborhood associated with the $\alpha$-th switching and hence does not affect the relevant local configuration.}
    \label{fig:switching}
\end{figure}

Then the admissible set is defined by
\begin{equation}
\label{eq:elliptic-Wdef}
    \mathsf W_{\bf S}
    :=
    \left\{
    \alpha\in\llbracket1,\mu\rrbracket:
    I_\alpha(\cG,{\bf S})
    J_\alpha(\cG,{\bf S})
    =1
    \right\}.
\end{equation}
We say that \(\alpha\) is \emph{switchable} if
$
    \alpha\in\mathsf W_{\bf S}.
$
Let
$
    \nu:=|\mathsf W_{\bf S}|,
$
and let
$
    \alpha_1,\dots,\alpha_\nu
$
be any enumeration of \(\mathsf W_{\bf S}\).
We define the switched graph by
\begin{equation}
\label{eq:elliptic-Tdef1}
T_{\bf S}(\cG)
:=
\left(
T_{\vec S_{\alpha_1}}
\circ\cdots\circ
T_{\vec S_{\alpha_\nu}}
\right)(\cG).
\end{equation}
and define the switching data by
\begin{equation}
\label{eq:elliptic-Tdef2}
T({\bf S})
:=
\bigl(
T_1(\vec S_1),\dots,T_\mu(\vec S_\mu)
\bigr),\qquad\mbox{where}~T_\alpha(\vec S_\alpha)
:=
\begin{cases}
T(\vec S_\alpha),
&
\alpha\in\mathsf W_{\bf S},
\\[1mm]
\vec S_\alpha,
&
\alpha\notin\mathsf W_{\bf S}.
\end{cases}
\end{equation}

To make the structure more clear, we introduce an enlarged probability space.
For
\(\alpha\in\llbracket1,\mu\rrbracket\), define
\begin{equation*}
\label{eq:elliptic-Salpha-space}
\mathsf S_\alpha(\cG)
:=
\left\{
\left\{
e_\alpha,e
\right\}:
 e\in
\vec E^{e_\alpha}_{\cG^{(\bT)}}
\right\}.
\end{equation*}
Then
$
    \vec S_\alpha
$
is uniformly distributed on
\(\mathsf S_\alpha(\cG)\), and
$
    {\bf S}
$
is uniformly distributed on
\begin{equation}\label{def.S}
  \mathsf S(\cG)
    :=
    \mathsf S_1(\cG)
    \times\cdots\times
    \mathsf S_\mu(\cG).
\end{equation}

\begin{definition}\label{def:elliptic-enlarged-space}
For any lift \(\cG\in\GNH\), define the fibre
$
    \iota(\cG)
    :=
    \{\cG\}\times\mathsf S(\cG)
$
(with respect to root $o$, parameter $z$ and elliptic radius $\ell$).
The enlarged probability space is
$$
    \mathsf G_N(\cH)'
    :=
    \bigsqcup_{\cG\in\GNH}
    \iota(\cG).
$$
For
$
    (\cG,{\bf S})\in\mathsf G_N(\cH)',
$
define
$
    \P'(\cG,{\bf S})
    :=
    \P(\cG)\P_\cG({\bf S})
$
where \(\P_\cG\) denotes the uniform probability measure on
\(\mathsf S(\cG)\).
\end{definition}

Therefore, on the enlarged probability space, we introduce the following maps to encode the local resampling procedure:
\begin{align}
\label{eq:elliptic-Ttilde}
&\widetilde T
:
\mathsf G_N(\cH)'
\longrightarrow
\mathsf G_N(\cH)',
\qquad
\widetilde T(\cG,{\bf S})
:=
\bigl(
T_{\bf S}(\cG),
T({\bf S})
\bigr),
\\
\label{eq:elliptic-T-map}
&T
:
\mathsf G_N(\cH)'
\longrightarrow
\GNH,
\qquad
T(\cG,{\bf S})
:=
T_{\bf S}(\cG).
\end{align}

We  record the basic structural properties of the local resampling map.
The proposition below shows that the switching preserves the geometry relevant
to the elliptic neighborhood and the separation of the switching cells, while
also preserving the admissible graph class. Moreover, the reverse switching is
well defined and the corresponding transformations preserve the underlying
sampling measures. These facts will be used repeatedly when transferring
estimates between the original and resampled graphs.

\begin{proposition}\label{prop:elliptic-reverse}
  For any switching data $\bfS\in \sS(\cG)$, we have
\begin{enumerate}
  \item The elliptic neighborhood is invariant under the local resampling:
$
    \cB_\ell^{\rm ell}(o,T_{\bf S}(\cG);z)
    =
    \cB_\ell^{\rm ell}(o,\cG;z).
$
  \item The radius-$\fR/8$ neighborhoods of the root before and after
the switching are isomorphic as rooted graphs:
\[
    \cB_{\fR/8}(o,\cG)
    \simeq
    \cB_{\fR/8}(o,T_{\bf S}(\cG)).
\]
\item For every $\alpha\in\mathsf W_{\bf S}$, the switching cell
$\{a_\alpha,b_\alpha,c_\alpha\}$ remains separated from all the other
switching cells by distance larger than $\fR/4$ both before and after the
switching:
\[
\dist_{\cG^{(\bT)}}\big(
    \{a_\alpha,b_\alpha,c_\alpha\},
    \bigcup_{ \beta\neq\alpha}
    \{a_\beta,b_\beta,c_\beta\}
\big)>\frac{\fR}{4},
\quad
\dist_{T_{\bf S}(\cG)^{(\bT)}}\big(
    \{a_\alpha,b_\alpha,c_\alpha\},
    \bigcup_{\beta\neq\alpha}
    \{a_\beta,b_\beta,c_\beta\}
\big)
>
\frac{\fR}{4}.
\]
\item For every $\cG\in\GNH$, $\widetilde T$ is an involution: $\widetilde T\circ\widetilde T=\id$. 
\item The maps \(\widetilde T\) and \(T\) are measure preserving :
$\P'\circ\widetilde T^{-1}=\P'$ and $\P'\circ T^{-1} =\P$.
\end{enumerate}
\end{proposition}

\begin{proof}
Write $\widetilde\cG=T_{\bf S}(\cG)$ and
$\bC_\alpha=\{a_\alpha,b_\alpha,c_\alpha\}$.
The invariance property follows directly from
Remark~\ref{rem:elliptic-locality}.

For the second assertion, each admissible switching replaces the
branch attached through $\{l_\alpha,a_\alpha\}$ by the branch
attached through $\{l_\alpha,c_\alpha\}$.
The condition $I_\alpha=1$ ensures that both branches are trees
on the relevant scale. Since the switching preserves edge types,
these branches have the same rooted type-labelled structure.
Moreover, the separation condition \(J_\alpha=1\), 
which places distinct switching cells at distance larger than \(\fR/4\), guarantees that their portions visible inside 
the radius-\(\fR/8\) neighborhood of \(o\) cannot meet: otherwise two such portions would connect the corresponding switching cells by a path of length at most \(2(\fR/8)=\fR/4\).
Replacing these branches by their type-preserving counterparts
therefore gives a rooted isomorphism
$
\cB_{\fR/8}(o,\cG)
\simeq
\cB_{\fR/8}(o,\widetilde\cG),
$
which fixes $\cT$.

The separation in the original graph follows from $J_\alpha=1$.
In the exterior graph, every modified edge has both endpoints
in a single switching cell.
If a path of length at most $\fR/4$ in
$\widetilde\cG^{(\bT)}$ connected $\bC_\alpha$ to another cell,
the segment between its last departure from $\bC_\alpha$ and its
first visit to another cell would contain only unchanged edges.
This would contradict the original separation, proving the
third assertion.
For the involution and measure-preserving properties, the
arguments in
\cite[Proposition 3.3 and Proposition 3.5]{Huang2024Spectrum}
apply without essential modification.
Indeed, the admissibility conditions are preserved by reversal,
and the numbers of available exterior edges of each type are
determined by the fixed type-labelled neighborhood $\cT$.
\end{proof}

\begin{remark}
\label{rem:z-dependent-local-resampling}
The local resampling constructed above is defined for a fixed
\(z\), since the elliptic neighborhood
$
    \cB_\ell^{\rm ell}(o,\cG;z)
$
depends on $z$.
Thus, more precisely, the maps above should be regarded as
$
    \widetilde T_z
$
and
$
    T_z.
$
However, for simplicity, we suppress the dependence on z, since it plays no essential role in the proof.
\end{remark}

In Section \ref{sec.proof-strategy}, our argument requires two successive local resampling procedures. For the first resampling, the definition of $\bar\Omega$ in Definition~\ref{def:barOmega} provides the required control of the local geometry of the original graph $\cG$. 
The second resampling, however, is performed on the already resampled graph $T_{\mathbf S}\cG$. It is therefore necessary to understand how the geometric properties encoded in $\bar\Omega$ are affected by one local resampling. In the rest of this subsection, we show that, although some of these properties may deteriorate slightly after resampling, the resulting graph still retains sufficiently good local geometry for the second resampling step.
To formulate the weaker geometric conditions that will be used for the resampled graphs, we introduce the following definition.

\begin{definition}\label{def:barOmega1}
Given a base graph $\cH$ and a large constant $\fd$, let $\bar\Omega^+\subset \GNH$ denote the set of graphs:
\begin{itemize}
\item
The radius-$ \fR/8$ neighborhood of every vertex contains at most one cycle.
\item
The radius-$\fR/2$ neighborhood of every vertex has excess at most $C\fd$, where $C$ is a universal constant.
\item
The number of vertices that have a radius-$\fR$ neighborhood that contains a cycle is at most $2N^{\fc}$.
\end{itemize} 
\end{definition}

Notice that the sets $\bar\Omega$ defined in Definition \ref{def:barOmega} and  $ \bar\Omega^+$ are subsets of $\GNH$ defined by deterministic properties. It is  clear that 
 $\bar\Omega \subset  \bar\Omega^+$.  
Compared with $\bar\Omega$, the set $\bar\Omega^+$ should be viewed as a relaxed version of the good geometric set. Moreover,
for any graph in $\bar\Omega^+$, the radius-$\fR/8$ neighborhood of every vertex contains at most one cycle. Therefore, Proposition \ref{p:tree-extension-bound} applies to the tree extension of the radius-$\fR/8$ neighborhood of every vertex in such a graph.

\begin{remark*} 
  The second condition in Definition~\ref{def:barOmega1} serves a different purpose from the first one. If a neighborhood has excess larger than one, the tree-extension estimates established above are in general no longer available for controlling the corresponding Green function. Thus, the radius-$\fR/2$ excess bound is not imposed in order to obtain Green function estimates on that scale. Rather, this condition will be used in the geometric counting arguments below. 
  In addition to the Green function estimates on radius-$\fR/8$ neighborhoods, we shall need to control the number of atypical vertices whose radius-$\fR/2$ neighborhoods fail to have the desired tree-like structure. The uniform bound on the excess of radius-$\fR/2$ neighborhoods provides the geometric input needed for this counting argument. 
\end{remark*}

The next proposition shows that a typical local resampling of a graph $\cG\in\bar\Omega$ belongs to $\bar\Omega^+$.
The proof is purely geometric and is the same as \cite[Proposition 5.12]{Huang2024Spectrum}; we therefore omit it.

\begin{proposition}
  For every large constant $\fd$ and every graph $\cG\in\bar\Omega$, there exists a event $F_0(\cG)\subset\sS(\cG)$ with probability $\P_\cG(F_0(\cG))=1-\OO(N^{-\fd})$ such that, for every switching data $\mathbf S\in F_0(\cG)$, $T_{\mathsf S}(\cG)\in\bar\Omega^+$.
\end{proposition}

\subsection{Geometry of Admissible Switchings}
We first analyze the geometric properties of the boundary vertices
\(\{a_\alpha\}_{\alpha=1}^{\mu}\). This part is deterministic, since
\(a_\alpha\) is fixed once \(\cG\) and \(\bT\) are fixed; the randomness
of the local sampling enters only through the subsequently sampled
vertices \(b_\alpha\) and \(c_\alpha\).

\begin{lemma}[Proposition 5.2 of \cite{Huang2024Spectrum}]\label{GTstructure}
Let $\cG\in\GNH$ and recall that $\cT=\cB^{\rm ell}_\ell(o,\cG)$ is the elliptic neighborhood defined in Definition \ref{def:elliptic-neighborhood}. Assume that $\ell\ll R$ and $\cB_{R}(o,\cG)$ has excess at most $\omega$. 
Then the following statements hold.
\begin{itemize}
\item
After removing $\T$, most boundary vertices of $\cT$ are far away from the other boundary vertices:
\begin{align}\label{fewclose}
|\{\al\in \qq{1,\mu}: \exists \beta\in \qq{1,\mu} \setminus \{\al\}, \dist_{\cGT}(a_\al, a_\beta)\leq R/2\}|\leq 2\omega.
\end{align}
\item
After removing $\T$, any vertex $x\in \bG \setminus \T$ can only be close to at most $\omega+1$  boundary vertices of $\bT$:
\begin{align}
\label{distfinitedega}
|\{\al\in\qq{1,\mu}: \dist_{\cGT}(x,a_\al)\leq {R}/2\}|\leq \omega+1.
\end{align}
\end{itemize}
\end{lemma}

Let $\cA$ denote the annular region obtained by removing $\bT$ from
$\cB_R(o,\cG)$, namely,
$
\cA=\cB_R^{(\bT)}(o,\cG).
$
In general, $\cA$ may have several connected components.
We partition $\qq{1,\mu}$ into disjoint subsets
$
\qq{1,\mu}
=
\sC_1\sqcup\sC_2\sqcup\cdots\sqcup\sC_n
$
according to the connectivity of the vertices $a_\alpha$ in $\cA$:
for $\alpha,\beta\in\qq{1,\mu}$,
\[
\alpha,\beta\in\sC_s
\text{ for some } s
\quad\Longleftrightarrow\quad
a_\alpha \text{ and } a_\beta
\text{ belong to the same connected component of }\cA.
\]

\begin{lemma}[Lemma 5.4 of \cite{Huang2024Spectrum}] \label{lem:beta2n}
Under the same assumptions as in Lemma \ref{GTstructure}, we have
\begin{equation}\label{eq:sum-As}
\sum_{s=1}^{n}
\mathbf 1_{\{|\sC_s|\ge 2\}}\,|\sC_s|
\le 2\omega.
\end{equation}
Moreover, for every $1\le s\le n$, the connected component of $\cA$
containing $\{a_\alpha:\alpha\in\sC_s\}$ has the excess at most $\omega-|\sC_s|+1$.
\end{lemma}
\medskip

The following proposition shows that the tree condition at \(a_\alpha\)
required in the definition of the local sampling indicator \(I_\alpha\)
is satisfied for all but three boundary indices \(\alpha\).

\begin{proposition}[Tree neighborhoods of the boundary vertices]
\label{prop:boundary-tree-neighborhood}

Under the same assumptions as in Lemma \ref{GTstructure}, we have
\[
\left|
\left\{
\alpha\in\qq{1,\mu}:
\cB_{R/2}(a_\alpha,\cGT)
\text{ is not a tree}
\right\}
\right|
\le 3\omega.
\]
\end{proposition}

\begin{proof}
Lemma~\ref{lem:beta2n} implies that at most \(2\omega\) indices \(\alpha\) belong to a connected
component of \(\cA\) containing more than one boundary vertex
\(a_\alpha\).
It remains to consider the singleton components.  Suppose that
\(\sC_s=\{\alpha\}\) and that
$
\cB_{R/2}(a_\alpha,\cGT)
$
is not a tree.  Since \(R\gg2\ell\), this neighborhood is contained in
the connected component of \(\cA\) containing \(a_\alpha\).
Consequently, this component contains a cycle and hence has excess
at least one.
The singleton components are pairwise disjoint, and excess is
additive over connected components.  Moreover, 
Lemma~\ref{lem:beta2n} gives
$
\operatorname{excess}(\cA)\le \omega.
$
Therefore, there are at most \(\omega\) singleton components for which
\(\cB_{R/2}(a_\alpha,\cGT)\) is not a tree.

Combining the two cases gives
\[
\left|
\left\{
\alpha\in\qq{1,\mu}:
\cB_{R/2}(a_\alpha,\cGT)
\text{ is not a tree}
\right\}
\right|
\le 2\omega+\omega
=3\omega,
\]
which proves the claim.
\end{proof}

We next turn to the geometric properties involving
\(\{b_\alpha,c_\alpha\}_{\alpha=1}^{\mu}\).
Unlike the boundary vertices \(a_\alpha\), these vertices are chosen
randomly by the local sampling. We therefore begin with the following
elementary approximate-uniformity estimate, which will be used repeatedly
to control the probability of geometric coincidences.

\begin{lemma} \label{lem:switchable}
Given $\cG\in\GNH$, for any vertex $x \in \bG\setminus \T$ and $\al\in\qq{1,\mu}$
\begin{equation} \label{e:approxunifbd}
  \P_{\cG}(b_\al= x)\lesssim N^{-1},\qquad \P_\cG(c_\al = x)
  \lesssim N^{-1}.
\end{equation}
\end{lemma}

\begin{proof}
We recall that, for any $\alpha\in\llbracket 1,\mu\rrbracket$,
the oriented edge $(b_\alpha,c_\alpha)$ is chosen uniformly from
the set $\vec E_{\cG^{(\bT)}}^{\,e_\alpha}$.
Since $\cG$ is an $N$-lift of the fixed base graph $\cH$, for each
oriented edge $e\in\vec E(\cH)$ there are exactly $N$ oriented edges
of $\cG$ projecting onto $e$. It follows from $\ell\sim\log_4\log N$ that removing the neighborhood $\bT$ excludes
at most $\OO((\log N)^\fa)$ such edges. Hence, we have
$
    \bigl|\vec E_{\cG^{(\bT)}}^{\,e_\alpha}\bigr|
    \gtrsim
    N
$, which directly gives \eqref{e:approxunifbd}.
\end{proof}

We next introduce the basic good event on which the local resampling
will be performed. Since the edges \(\{b_\alpha,c_\alpha\}\) are chosen
randomly in \(\cG^{(\bT)}\), the required separation and tree-neighborhood
conditions are not deterministic. The following proposition shows that,
after excluding an event of probability \(\OO(N^{-\fd})\), these geometric
conditions fail for only \(\OO(\fd)\) indices. In particular, on this good
event, all but a bounded number of indices are admissible for the local
sampling. This event will serve as the geometric foundation for the
subsequent local resampling argument. The proof of Proposition \ref{prop:distdeg} is almost identical to that of \cite[Proposition~5.10 and Proposition 5.11]{Huang2024Spectrum}, and we therefore omit the details.
\begin{proposition} \label{prop:distdeg}
For any large $\fd\in\N_+$ and a graph $\cG \in \bar\Omega$ or $\bar\Omega^+$, there exists an event $F_1(\cG)\subset\mathsf S(\cG)$ with probability $\P_\cG(F_1(\cG))=1-\OO(N^{-\fd})$ such that the following estimates hold:
\begin{itemize}
\item
Any vertex $x\in \bG\setminus\T$ is far away from most vertices in the set $\{a_1, b_1, c_1,\dots, a_\mu, b_\mu, c_\mu\}$:
\begin{align}
{\label{distfinitedegb}} |\{\al\in\qq{1,\mu}:\dist_{\cGT}(x,\{a_\al, b_\al, c_\al\})\leq {\fR/4}\}|\lesssim\fd.
\end{align}
\item
Most indices $\al \in \qq{1,\mu}$ are good in the following sense: 
\begin{align}
\label{distdega}
&|\Ba|\lesssim\fd,
\text{with } \Ba =\{\al\in \qq{1,\mu}: {\dist_{\cGT}(a_\al, \{b_\al,c_\al\})\leq \fR/2+2}\},
\\
\label{distdegb}
&|\Bb|\lesssim\fd, 
\text{with }\Bb =\{\al\in \qq{1,\mu}: \exists \beta\neq \al, {\dist_{\cGT}(\{a_\al, b_\al,c_\al\}, \{a_\beta, b_\beta, c_\beta\})\leq {\fR}/4}\},
\\
\label{treeneighbor}
&|\Bc|\lesssim\fd,
\text{with } \Bc =\{\al\in \qq{1, \mu}: \text{$\cB_{\fR/4}(\{b_\al,c_\al\}, \cGT)$ is not a tree}\}.
\end{align}
\item Estimates \eqref{distfinitedegb}  and \eqref{distdegb} also hold for the switched graph $\tcGT:=T_{\mathbf S}(\cG)^{(\bT)}$:
\begin{align}
&|\{\al\in\qq{1,\mu}: \dist_{\tcGT}(x,\{a_\al,b_\al,c_\al\})\leq \fR/4\}|\lesssim\fd,\qquad x\in \bG\setminus \T \label{e:lessshortdist}\\
&\dist_{\tcGT}(\{a_\al,b_\al, c_\al\}, \{a_\beta,b_\beta,c_\beta: \beta\in \qq{1,\mu}\setminus \{\al\}\})>\fR/4,\qquad \al\in \qq{1,\mu}\setminus \Bb \label{tcGTdist}
\end{align}
\end{itemize}
\end{proposition}

\begin{remark*}
  Proposition \ref{prop:boundary-tree-neighborhood}, together with
\eqref{distdega} and \eqref{treeneighbor}, verifies the geometric
conditions entering the indicator \(I_\alpha\), while
\eqref{distdegb} gives the mutual-separation condition required for
\(J_\alpha\). Moreover, it follows from 
\eqref{eq:elliptic-boundary-cardinality} that
$
\mu\sim 4^\ell\sim(\log N)^C.
$
Hence, on the good event \(F_1(\cG)\), at most \(\OO_\fd(1)\) indices are
excluded from the local resampling, which is only a negligible fraction
of all boundary indices.
\end{remark*}

\section{Quantitative Local Law and Proof Strategy}\label{sec.proof-strategy}
In this section, we formulate the quantitative local law that will be proved in the remainder of the paper and describe the main structure of its proof. We first collect the parameters and error scales used throughout the argument. 

\subsection{Choice of parameters}\label{sec:parameters}
\begin{definition}\label{def:parameters}
Fix parameters $\delta,\beta,\fc\in(0,1), \fa>6, \fb\ge 16\fa$ and a large constant $\fd$. After enlarging  $\beta\in(0,1)$ if necessary, assume that $N^{-1+c}\le\beta^{2r}$. 
The radius parameters are defined by: $\fR=(\fc/4)\log N, r=(\fc/200)\log N$ and $\ell=\fa\log_4\log N$.
We denote by
$
\mu
=
\bigl|\partial \cB^{\rm ell}_\ell\bigr|
$
the cardinality of the boundary of the elliptic ball of radius $\ell$.
We restrict to the spectral domain $z\in\bC_+(\mathcal I_\delta)$ with $\eta=\im[z]\ge(\log N)^{\fb}/N$.
With this choice of parameters, we have 
\begin{equation}
  \fR/50\ge r=\OO(\log N)\gg\ell=\OO(\log_4\log N),\qquad \mu\sim 4^\ell\sim(\log N)^{\fa},
\end{equation}
where $\mu\sim 4^\ell$ follows from \eqref{eq:elliptic-boundary-cardinality}.
Define the baseline error by
\begin{equation}
\label{e:def-epsilon-zero}
\varepsilon(z)
:=
4^{6\ell}
\left(
\beta^r
+
\sqrt{\frac{|\im m^{\rm ver}(z)|_\infty}{N\eta}}
+\varepsilon_{\rm edge}+\varepsilon_{\rm cusp}
\right),
\end{equation}
where $|\Im m^{\rm ver}|_\infty$ satisfies the bound from Lemma~\ref{lem:NB-Dyson-regularity} and
\begin{equation}
\label{e:def-epsilon-edge}
\begin{aligned}
\varepsilon_{\rm edge}(z)
:=
\max_{\tau\in\mathcal S_{\rm edge}}
\frac{1}
{N\eta
 \bigl(|z-\tau|+(N\eta)^{-2/3}\bigr)^{1/2}},
\qquad
\varepsilon_{\rm cusp}(z):=
\max_{\tau\in\mathcal S_{\rm cusp}}
\frac{1}
{N\eta
 \bigl(|z-\tau|+(N\eta)^{-3/5}\bigr)^{2/3}}.
\end{aligned}
\end{equation}
The maximum over an empty set is understood to be zero.
Let $\varepsilon_*(z)$ be the stability error scale defined in
\eqref{e:def-epsilon-star}. Finally, define
\begin{equation}
\label{e:def-varphi}
\varphi(z)
:=
4^\ell
\sqrt{
\frac{
|\im m^{\rm ver}(z)|_\infty
+
\varepsilon_*(z)
+
(\log N)\varepsilon(z)
}
{N\eta}
},
\end{equation}
where 
\begin{equation}
\label{e:def-epsilon-star}
\varepsilon_*(z)
:=
\begin{cases}
\displaystyle
\frac{\varepsilon(z)}
{\bigl(\varepsilon(z)+|z-\tau|\bigr)^{1/2}},
&
\text{if $z$ lies near a regular edge $\tau$},\\[3mm]
\displaystyle
\frac{\varepsilon(z)}
{\bigl(\varepsilon(z)+|z-\tau|\bigr)^{2/3}},
&
\text{if $z$ lies near a regular cusp $\tau$},\\[3mm]
\varepsilon(z),
&
\text{otherwise}.
\end{cases}
\end{equation}
\end{definition}

\medskip
With these choices, uniformly on the spectral domain under
consideration,
\begin{equation}
\label{e:parameter-basic-bounds}
\varepsilon(z)
\lesssim
4^{-2\ell},
\qquad
\varphi(z)
\lesssim
(\log N)^{1/2}4^{-2\ell}\varepsilon(z).
\end{equation}
These estimates provide all the error absorptions needed in the
remaining proof.

We next explain the roles and choices of the above parameters. 
The parameter $\delta>0$ is the regularity parameter introduced in Lemma~\ref{lem:NB-Dyson-existence}, which keeps the spectral domain away from the singular set while retaining the regular edges and cusps under consideration.
The constant $\beta\in(0,1)$ is the path-decay rate from Proposition~\ref{pro.path-weight-decay}, and determines the localization error $\beta^r$.

The three radii $\fR$, $r$, and $\ell$ serve different purposes.
The radius $\fR$ controls the local tree-like geometry.
For graphs in either $\bar\Omega$ or $\bar\Omega^+$, every
radius-$\fR/8$ neighborhood contains at most one cycle.
This is the geometric condition needed for the deterministic
Green-function estimates in Lemma~\ref{p:subgraph} and the
localization principle in Lemma~\ref{p:localization}.
In particular, the localized Green function can be controlled
through the corresponding tree-extension model, with the
possible cycle retained in the finite core.
The radius $r$ determines the size of the neighborhood used
in this localization.
We choose $r\leq\fR/50$ to leave room for the local enlargements
appearing below: each relevant connected component of the
unions of $r$-neighborhoods must remain inside a
radius-$\fR/8$ neighborhood.

The elliptic radius $\ell$ determines the size of the resampling
boundary. We have
$
\mu\sim 4^\ell\sim(\log N)^\fa
$
The boundary averaging in Proposition~\ref{prop:lift-concentration}
produces a gain of order $2^{-\ell}$. Hence $\ell$ must be large
enough so that this gain dominates the logarithmic losses appearing
in the concentration and reconstruction estimates. This is the
reason for choosing $\fa>6$. 
On the other hand, the resampling boundary should not grow too fast. Indeed, the switching argument samples vertices attached to the boundary, and its error estimates involve unions and sums over boundary indices. Keeping $ 4^\ell=(\log N)^\fa $ polylogarithmic ensures that collisions among the sampled vertices, intersections with exceptional local neighborhoods, and the losses from summing over the boundary remain negligible.

The parameter $\varepsilon_*$ is the stability-adjusted error scale.
Indeed, the bootstrap first yields
$
|R(Q,z)|_\infty\lesssim \varepsilon(z),
$
and Lemma~\ref{lem.stability2} then gives
$
|Q-m(z)|_\infty\lesssim \varepsilon_*(z).
$
In the stable bulk there is no stability loss, so
$\varepsilon_*\sim\varepsilon$. Near a regular edge or cusp, however,
the Dyson equation becomes degenerate, and the resulting error
may deteriorate to $\varepsilon^{1/2}$ and $\varepsilon^{1/3}$
respectively. Therefore, close to an edge or cusp,
$\varepsilon_*$ can be much larger than the baseline error
$\varepsilon$.

We next explain the two error parameters $\varphi$ and
$\varepsilon$. The quantity $\varphi$ is not an independent error
scale. It is chosen directly from the Ward-identity estimate \eqref{eq.imG_xx} so that, for
each fixed vertex, only an $\OO(4^{-2\ell})$ proportion of the other
vertices can have Green-function entries larger than $\varphi$. See section \ref{sec.cell} for details.

The main quantity that has to be chosen is $\varepsilon$.
Its first contribution, $\beta^r$, is the deterministic localization
error coming from replacing a Green function by the corresponding
radius-$r$ local model. See Proposition \ref{prop:bootstrap-start} for details.

The term
$
\sqrt{
|\Im m^{\rm ver}(z)|_\infty/(N\eta)
}
$
is the natural fluctuation scale coming from the Ward identity.
However, the diagonal Green function is not approximated by
$m^{\rm ver}$ with error $\varepsilon$ itself, but rather with the
stability error $\varepsilon_*$. Near a regular edge or cusp,
$\varepsilon_*$ is larger because the Dyson equation becomes less
stable. Substituting this weaker control back into the Ward estimate
produces the additional errors
$\varepsilon_{\rm edge}$ and $\varepsilon_{\rm cusp}$.
Their definitions are chosen so that they reproduce the natural
self-consistent scales $(N\eta)^{-2/3}$ near a regular edge and
$(N\eta)^{-3/5}$ near a regular cusp, while retaining the stronger
bounds $(N\eta)^{-1}$ away from the corresponding regular edge or cusp.

Finally, the prefactor $4^{6\ell}$ in
\eqref{e:def-epsilon-zero} is a bookkeeping buffer for the switching
argument. Several estimates below involve sums over one or two
boundary indices, and therefore lose powers of $4^\ell$. Enlarging
the basic error by $4^{6\ell}$ ensures that, after passing through
the Ward estimate and the boundary summations, the resulting
quantity $\varphi$ is still much smaller than $\varepsilon$.

An important feature of the error parameter \(\varepsilon\) is its
monotonicity along vertical lines. For fixed $E$, the error $\varepsilon(E+\mathrm i\eta)$ is non-increasing in
$\eta$. Indeed, the Ward term has this property by the Stieltjes transform
representation, and the edge and cusp correction terms in
\eqref{e:def-epsilon-edge} are also non-increasing in $\eta$.
Hence, for $t\ge\eta$,
$
\varepsilon(E+\mathrm i t)
\le
\varepsilon(E+\mathrm i\eta).
$
Therefore the residual along the whole vertical continuation segment can be
controlled by the single endpoint error $\varepsilon(E+\mathrm i\eta)$, as
required in the edge and cusp stability Lemma \ref{lem.stability2}.

The lower bound
$
N\eta\geq(\log N)^\fb
$
has the following role: it guarantees that the Ward fluctuations and
the edge/cusp corrections remain sufficiently small even after the
factor $4^{6\ell}$ is inserted. Since
$4^\ell\sim(\log N)^\fa$, the condition
$\fb\geq16\fa$ leaves enough room for all the required
polylogarithmic losses.

\subsection{Quantitative Local Law and Continuation Argument}
We first state the quantitative local law, which holds throughout the regular spectral regime away from the singular energies, including the stable bulk, regular edges, and regular cusps.
\begin{theorem}\label{thm.main}
    Fix any sufficiently large $\fd\in\N_+$ and let $\cG$ be sampled uniformly from $\bar\Omega$ defined in Definition \ref{def:barOmega}. For every sufficiently large
$N\in\N_+$, with probability
$1-\OO(N^{-\fd})$, the Green's function of $\cG\in\bar\Omega$ satisfies, uniformly for $z\in\bC_+(\mathcal I_\delta)$ with $\eta=\im[z]\ge(\log N)^\fb/N$,
  \begin{equation}\label{eq.main}
    \big|G_{ij}(\cG,z)-G_{ij}(\Ext(\cB_r(\{i,j\},\cG),m(z)),z)\big|\lesssim\varepsilon_*(z),
  \end{equation}
  and the vector of type-averaged diagonal Green functions satisfies
  \begin{equation}\label{eq.main 2}
    |Q^{\rm ver}(z)-m^{\rm ver}(z)|_\infty\lesssim\epsilon_*(z),
  \end{equation}
  where $\epsilon_*$ is defined in \eqref{e:def-epsilon-star}.
\end{theorem}

Theorem~\ref{thm.main.intro} and
Theorem~\ref{thm.delocalization.intro} are direct consequences of
Theorem~\ref{thm.main}: the former follows by the standard Stieltjes
inversion argument, while the latter follows by the standard
Green-function argument; see, for instance,
\cite[Section~18.5]{ErdosYau2017}.

The remainder of this subsection presents the proof scheme for Theorem~\ref{thm.main}.
We begin by propagating the local estimates from the large-$|z|$ regime, where they follow directly from a Neumann expansion, to the whole spectral domain under consideration. 
For this purpose, we introduce two nested good events $ \Omega^-(z)\subset\Omega(z)$. 
The event $\Omega(z)$ records the estimates that are stable under a small variation of the spectral parameter, 
whereas $\Omega^-(z)$ contains an additional margin in the error bounds. The main one-step improvement shows that, up to an arbitrarily small probability, membership in $\Omega(z)$ can be upgraded to membership in $\Omega^-(z)$. 
Starting from large $|z|$ and iterating this improvement along a sufficiently fine spectral lattice then yields the desired estimates uniformly in $z$.

We first introduce the basic good event $\Omega(z)$.
\begin{definition}[The set $\Omega(z)$]
\label{def:Omega}
For $z\in\mathbb C_+(\mathcal I_\delta)$, we define the set of spectral regular graphs
$\Omega(z)\subset \bar\Omega$ to be the set of graphs $\cG$ such that
\begin{align}
|Q-m|_\infty\le\epsilon_*(z),\quad
|R(Q(\cG,z),z)|_\infty
+\max_{i,j\in\bG}\left|
G_{ij}(z)
-
G_{ij}\left(
\Ext\bigl(
\cB_r(\{i,j\},\cG),Q(\cG,z)
\bigr),z
\right)
\right|
\leq
\varepsilon(z).
\end{align}
\end{definition}

To carry out the iteration, we also define the event
$\Omega^-(z)$, which imposes a stronger error bound.
\begin{definition}[The set $\Omega^{-}(z)$]
\label{def:Omega-minus}
For $z\in\mathbb C_+(\mathcal I_\delta)$, we define
$\Omega^{-}(z)\subset\Omega(z)\subset \bar\Omega$ to be the set of graphs $\cG$ such that
\begin{align}\label{eq.Omega-minus}
|Q-m|_\infty\le\frac{\epsilon_*(z)}{2},\quad|R(Q(\cG,z),z)|_\infty
+\max_{i,j\in\bG}\left|
G_{ij}(z)
-
G_{ij}\left(
\Ext\bigl(
\cB_r(\{i,j\},\cG),Q(\cG,z)
\bigr),z
\right)
\right|
\leq
\frac{\varepsilon(z)}{2}.
\end{align}
\end{definition}
Compared with $\Omega(z)$, the event $\Omega^-(z)$ improves the defining error bound from $\varepsilon(z)$ to $\varepsilon(z)/2$. This margin is essential for the continuation argument: after moving the spectral parameter to a nearby point, Lipschitz continuity may consume part of the available error, but the graph still remains in $\Omega(z)$. The one-step improvement below can then be applied once again to recover membership in $\Omega^-(z)$.

The following proposition provides the starting point of the continuation argument. It shows that, when $|z|$ is sufficiently large, every graph in $\bar\Omega$ already satisfies the stronger estimates defining $\Omega^-(z)$.
\begin{proposition}\label{prop:bootstrap-start}
  We have $\bar\Omega\subset\Omega^-(z)$ uniformly for every $z\in\bC_+(\mathcal I_\delta)$ with $|z|\ge 2d_{\rm max}$. Here $\bar\Omega$ and $\Omega^-(z)$ are defined in Definition \ref{def:barOmega} and Definition \ref{def:Omega-minus} respectively.
\end{proposition}

\begin{proof}
Fix $\cG\in\bar\Omega$ and $|z|\ge 2d_{\max}$.
Enlarging $\beta$ in Proposition~\ref{pro.path-weight-decay}
if necessary, we may assume that $\beta\ge 1/4$.
Since $|z|\ge 2d_{\max}$, the Neumann expansions for $\cG$
and its tree extensions converge absolutely.
Any walk from $i$ to $j$ which distinguishes $\cG$ from its
radius-$r$ local extension must leave
$\cB_r(\{i,j\},\cG)$ and return, and hence has length at least
$2r+2$.
Consequently,
\begin{equation}\label{e:bootstrap-start-truncation}
  \left|
G_{ij}(z)
-
G_{ij}\left(
\Ext\bigl(\cB_r(\{i,j\},\cG),m(z)\bigr),z
\right)
\right|
\le
2\sum_{k\ge 2r+2}\frac{d_{\max}^k}{|z|^{k+1}}
\lesssim
\frac{\beta^r}{|z|^3}.
\end{equation}

Recall that $1\le r\le\fR$.
If $x$ has a radius-$\fR$ tree neighborhood, then we have
$
G_{xx}\bigl(\Ext(\cB_r(x,\cG),m),z\bigr)
=m_x(z)$
and
$
G_{x x_e}\bigl(\Ext(\cB_r(x,\cG),m),z\bigr)
=-m_{x}(z)m_e(z)$
where $x_e$ denotes the unique neighbor of $x$ such that the edge $(x,x_e)$ is of type $e$.
Since $(A^k)_{x x_e}$ counts the walks of length $k$ from $x$ to
$x_e$, and $\cG$ agrees with its universal-cover extension inside
$\cB_r(x,\cG)$, the two Neumann expansions agree for all
$k<2r+1$. Indeed, any walk from $x$ to $x_e$ that leaves
$\cB_r(x,\cG)$ has length at least $(r+1)+r=2r+1$.
Hence
\[
\left|G_{x x_e}(z)+m_{\pi(x)}(z)m_e(z)\right|
\le
2\sum_{k\ge 2r+1}\frac{d_{\max}^k}{|z|^{k+1}}
\lesssim
\frac{\beta^r}{|z|^3}.
\]

There are at most $\OO(N^\fc)$ exceptional vertices.
For these vertices, the Neumann expansion, together with the
vanishing diagonal and unit adjacency entries along edges, gives
\[
\begin{aligned}
G_{xx}(z)
&=-\frac1z+\OO(|z|^{-3}),
&
m_{x}(z)
&=-\frac1z+\OO(|z|^{-3}),
\\
G_{x x_e}(z)
&=-\frac1{z^2}+\OO(|z|^{-3}),
&
-m_{x}(z)m_e(z)
&=-\frac1{z^2}+\OO(|z|^{-4}).
\end{aligned}
\]
Averaging these estimates over each fiber, and using
\eqref{e:bootstrap-start-truncation} for the nonexceptional
diagonal entries, we obtain
$
|Q^{\rm ver}-m^{\rm ver}|_\infty
+
\max_{e\in\vec E_\cH}
\left|p_e+m_{t(e)}m_e\right|
\lesssim
|z|^{-3}(\beta^r+N^{-1+\fc}).
$
The Neumann expansion also gives
$|Q_{t(e)}|,|m_e|\sim  |z|^{-1}$ uniformly in this region.
Using $Q_e=-p_e/Q_{t(e)}$ from
\eqref{e:def-Q-edge-new}, we have the exact identity
$
Q_e-m_e
=
-(p_e+m_{t(e)}m_e)/Q_{t(e)}
-m_e\bigl(Q_{t(e)}-m_{t(e)}\bigr)/Q_{t(e)}.
$
Therefore,
\begin{equation}\label{e:bootstrap-start-Q}
\begin{aligned}
|Q^{\rm ver}-m^{\rm ver}|_\infty
\lesssim
\frac{\beta^r+N^{-1+\fc}}{|z|^3}
\lesssim \frac{\beta^r}{|z|^3},\qquad
|Q-m|_\infty
\lesssim
\frac{\beta^r+N^{-1+\fc}}{|z|^2}
\lesssim \frac{\beta^r}{|z|^2}\lesssim\beta^r,
\end{aligned}
\end{equation}
where we used the parameter condition
$N^{-1+\fc}\lesssim\beta^r$.

In particular, $|Q_e|,|m_e|\sim  |z|^{-1}$ for sufficiently
large $N$.
The identity
$
R(Q,z)=(Q-m)/(Qm)-B(Q-m)
$
therefore implies
\begin{equation}\label{e:bootstrap-start-residual}
|R(Q,z)|_\infty
\lesssim
(1+|z|^2)|Q-m|_\infty
\lesssim
\beta^r.
\end{equation}

Finally, Lemma~\ref{p:subgraph} implies that replacing $m$ by $Q$
changes the boundary self-energy by a matrix of operator norm
$\OO(|Q-m|_\infty)$.
The exact tree-extension resolvent, restricted to the local core,
has operator norm at most $(|z|-d_{\max})^{-1}$.
Thus a Neumann perturbation in the boundary weights yields
\[
\left|
G_{ij}\bigl(\Ext(\cB_r(\{i,j\},\cG),Q),z\bigr)
-
G_{ij}\bigl(\Ext(\cB_r(\{i,j\},\cG),m),z\bigr)
\right|
\lesssim
\frac{|Q-m|_\infty}{|z|^2}
\lesssim
\beta^r.
\]
Combining this with
\eqref{e:bootstrap-start-truncation}--\eqref{e:bootstrap-start-residual},
we conclude that
\[
|R(Q,z)|_\infty
+
\max_{i,j\in\bG}
\left|
G_{ij}(z)
-
G_{ij}\left(
\Ext\bigl(\cB_r(\{i,j\},\cG),Q\bigr),z
\right)
\right|
\lesssim
\beta^r.
\]
By the parameter condition
$\beta^r\ll\varepsilon\lesssim\varepsilon_*$ uniformly in the
region under consideration, the above inequality and \eqref{e:bootstrap-start-Q} imply $\cG\in\Omega^-(z)$.
\end{proof}

The following proposition is the key one-step improvement. It states that, apart from an arbitrarily small exceptional probability, every graph satisfying the estimates of $\Omega(z)$ in fact satisfies the stronger estimates of $\Omega^-(z)$ at the same spectral parameter. Consequently, once the weaker estimates have been propagated to a nearby value of $z$ by continuity, this proposition restores the extra margin needed for the next step of the iteration.
\begin{proposition}
\label{prop:bootstrap}
For any large $\fd,N\in\N_+$, we have
\begin{equation}\label{eq.bootstrap}
  \P(\Omega(z)\setminus\Omega^-(z))=\OO(N^{-\fd+1}),
\end{equation}
uniformly for every $z\in\bC_+(\mathcal I_\delta)$ with $|z|\le 2d_{\rm max}$ and $\eta\ge(\log N)^\fb/N$. Here $\Omega(z)$ and $\Omega^-(z)$ are defined in Definition \ref{def:Omega} and \ref{def:Omega-minus} respectively.
\end{proposition}

The main ingredients of the proof are the local resampling
estimates, concentration of the boundary terms, and the resulting
self-consistent equation. We defer the proof to Section \ref{sec.two-key-proposition}.

The following proposition explains how the large-$|z|$ initial estimate and the one-step improvement yield the local law uniformly throughout the spectral domain.

\begin{proposition}\label{prop.bootstrap union}
 For any large $\fd>0$ and $N$ sufficiently large, with probability
$1-\OO(N^{-\fd})$ with respect to the uniform measure on $\bar\Omega$,
a graph $\cG\in\bar\Omega$ belongs to $\Omega^-(z)$ uniformly for all
$
z\in\C_+(\mathcal I_\delta)$ with
$
\eta\ge(\log N)^{\fb}/N.
$
\end{proposition}

\begin{proof}
The proof follows exactly the same argument as \cite[Proposition 4.9]{Huang2024Spectrum}; we therefore only indicate the main idea.
By Proposition~\ref{prop:bootstrap-start}, we have
$\bar\Omega\subset\Omega^-(z)$ for $|z|\ge 2d_{\rm max}$, which
provides the starting point of the bootstrap.
We then move from this region to
$\{z\in\C_+(\mathcal I_\delta):\eta\ge(\log N)^\fb/N\}$
along a sufficiently fine deterministic lattice in the spectral
parameter with mesh size
$N^{-1}$. By the Lipschitz continuity of the resolvent and of
$Q(\cG,z)$ in $z$, membership in $\Omega^-$ at one lattice point
implies membership in $\Omega$ at the next one, while
Proposition~\ref{prop:bootstrap} upgrades this to membership in
$\Omega^-$ with probability $1-\OO(N^{-\fd})$.

Iterating this argument over the lattice and using the same Lipschitz
continuity to pass from the lattice to all spectral parameters gives
\[
\P\left(
  \cG\in\Omega^-(z)
  \text{ uniformly for all }
  z\in\C_+(\mathcal I_\delta),\
  \eta\ge\frac{(\log N)^\fb}{N}
\right)
=
1-\OO(N^{-\fd}).
\]
Here we apply Proposition~\ref{prop:bootstrap} with a larger value of
$\fd$ if necessary, so that the polynomial loss from the union bound
over the spectral lattice is absorbed.
\end{proof}

We next turn to the proof of Theorem~\ref{thm.main}. The main difference between its conclusion and that of Proposition~\ref{prop.bootstrap union} is that the weights attached to the boundary in the local comparison graph are given by the deterministic Dyson solution $m$, rather than by the empirical weights $Q$. 

\begin{proof}[Proof of Theorem \ref{thm.main}]
Since we already have Proposition \ref{prop.bootstrap union}, it is sufficient to prove that, for every $\cG\in\Omega^-(z)$, its Green's function satisfies \eqref{eq.main} and \eqref{eq.main 2}. 
Since Lemma \ref{p:subgraph} implies that replacing $Q$ by $m$ gives an error $\OO(|Q-m|_\infty)$, it follows from \eqref{eq.Omega-minus} that
\[
\left|
G_{ij}(z)
-
G_{ij}\left(
\Ext\bigl(
\cB_r(\{i,j\},\cG),m(z)
\bigr),z
\right)
\right|
\lesssim
\epsilon+\epsilon_*\lesssim\epsilon_*,
\]
which proves \eqref{eq.main}.
Averaging the above estimate over the diagonal entries in each
fiber, and using that all but at most $\OO(N^\fc)$ vertices have a
radius-$\fR$ tree neighborhood, we obtain
\[
|Q^{\rm ver}(z)-m^{\rm ver}(z)|_\infty
\lesssim
\epsilon_*+N^{-1+\fc}
\lesssim
\epsilon_*,
\]
which proves \eqref{eq.main 2}.
\end{proof}

\medskip

It therefore remains only to prove Proposition~\ref{prop:bootstrap}. In the next subsection, we establish this result through two key propositions.

\subsection{Two key propositions}\label{sec.two-key-proposition}
We now turn to the proof of Proposition~\ref{prop:bootstrap}. The argument is based on two successive local resampling steps around a fixed root $o$. 
Roughly speaking, the first resampling transfers the global estimates encoded in $\Omega(z)$ to a local event $\Omega_o^+(z)$ that is stable under the filled switching construction. 
The second resampling produces the stronger rootwise event $\Omega_o^-(z)$, which contains the improved estimates needed to close the iteration.

Before defining the local event $\Omega_o^+(z)$, we introduce the
filled exterior associated with $\bT$. Recall that
$\cG^{(\bT)}$ denotes the graph obtained from $\cG$ by removing
$\bT$. We define
$
\cG^{\langle\bT\rangle}
:=
\TE\bigl(\cG^{(\bT)}\bigr),
$
where $\TE(\cdot)$ is the tree-extension operator defined in
Definition~\ref{def.tree-extension}. The construction is illustrated
in Figure~\ref{fig:filled_removal_definition}. We denote the Green function of
$\cG^{\langle\bT\rangle}$ by
$
G^{\langle\bT\rangle}(z).
$

\begin{figure}[H]
    \centering
    \includegraphics[width=0.50\linewidth]{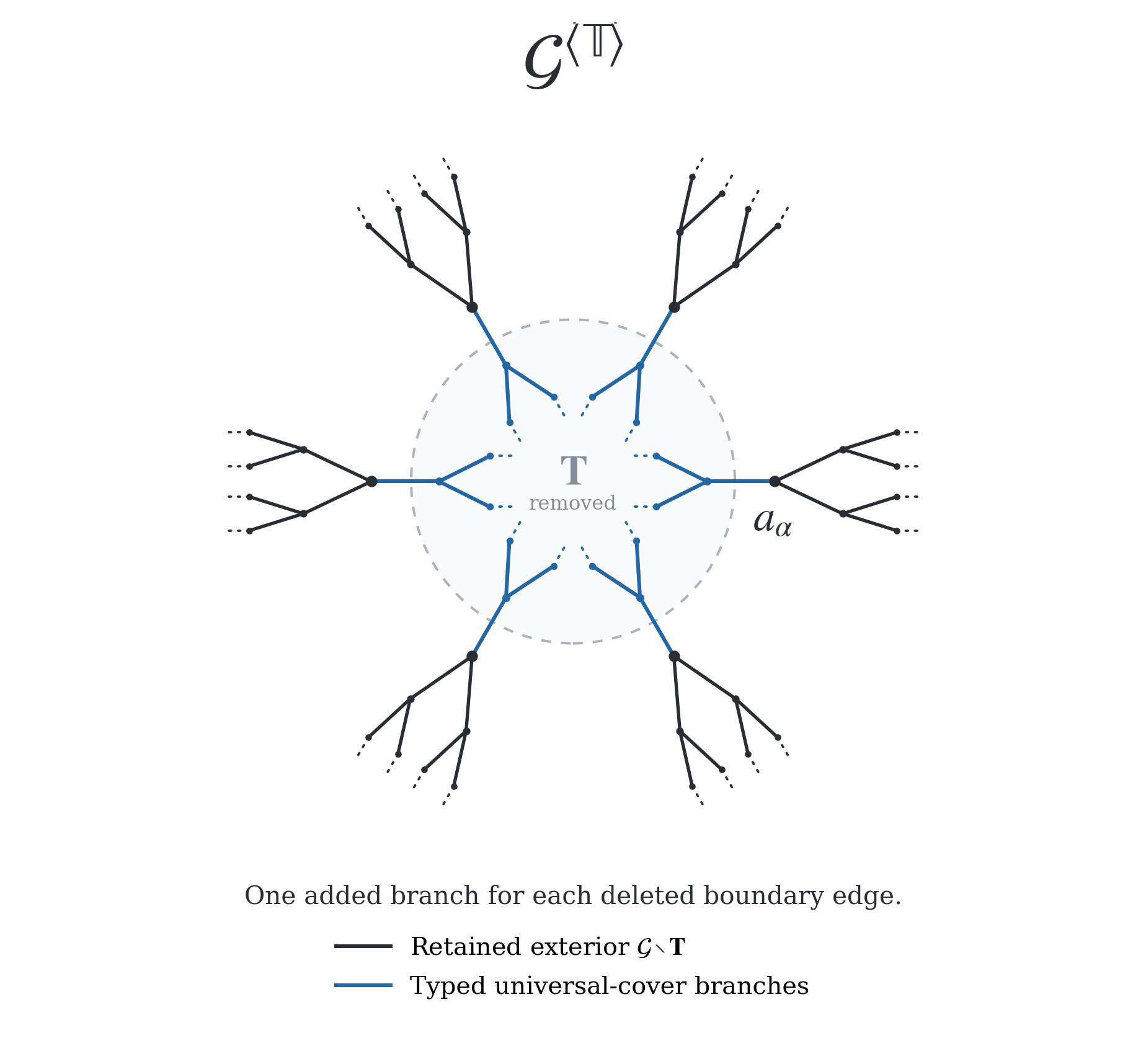}
    \caption{The filled exterior $\cG^{\langle\bT\rangle}$
    associated with $\bT$.}
    \label{fig:filled_removal_definition}
\end{figure}

\begin{definition}[The set $\Omega_o^{+}(z)$]
\label{def:Omega-o-plus}
For $z\in\mathbb C_+(\mathcal I_\delta)$ with
$
\eta\geq (\log N)^{\fb}/N
$,
let
$
\cT=\cB^{\rm ell}_\ell(o,\cG;z)
$
with fixed root $o\in\bG$.
We define $\Omega_o^{+}(z)\subset\bar\Omega^{+}$ to be the set of
graphs $\cG$ such that
\begin{align}
&|Q-m|_\infty\lesssim\epsilon_*(z),\qquad\qquad \max_{i,j\in\bG\setminus\bT}\left|
G^{\langle\bT\rangle}_{ij}(z)
-
G_{ij}\left(
\Ext\bigl(
\cB_r(\{i,j\},\cG^{\langle\bT\rangle}),
Q(\cG,z)
\bigr),z
\right)
\right|
\lesssim(\log N)\varepsilon(z),\label{e:Omega-o-plus-filled}\\[2mm]
&|R(Q(\cG,z),z)|_\infty
\lesssim\varepsilon(z),\qquad\max_{i,j\in\bG}\left|
G_{ij}(z)
-
G_{ij}\left(
\Ext\bigl(
\cB_r(\{i,j\},\cG),Q(\cG,z)
\bigr),z
\right)
\right|
\lesssim
(\log N)^3\varepsilon(z).\label{e:Omega-o-plus-G}
\end{align}

\end{definition}

The first key proposition shows that these additional local estimates are inherited from the global good event $\Omega(z)$ after one elliptic switching, except for a set of switching data of arbitrarily small probability. In this sense, $\Omega_o^+(z)$ is the intermediate event produced by the first local resampling.
\begin{proposition}\label{p:Omega+}
Let $z\in \bC_+(\mathcal I_\delta)$ with $\eta\geq (\log N)^\fb/N$. 
Then for any $\cG\in \Omega(z)$ and large $\fd,N\in\N_+$, there exists $F_2(\cG) \subset \sS(\cG)$, defined in Proposition \ref{prop:cell-sparsity}, with $\P_{\cG}(F_2(\cG))= 1-\OO(N^{-\fd})$ such that for every switching data $\bfS\in F_2(\cG)$, 
the switched graph $\widetilde{\cG}=T_{\bf S}(\cG)\in \Omega_o^+(z)$. 
Here, the sets $\Omega(z)$ and $\Omega_o^+(z)$ are defined in Definitions \ref{def:Omega} and \ref{def:Omega-o-plus}.
\end{proposition}

The proof of the above proposition is in Section \ref{sec:stability}.

The event $\Omega_o^+(z)$ is sufficient to control the switched filled resolvent, but it does not yet contain the improvement needed to close the iteration. We therefore introduce a second local event $\Omega_o^-(z)$. It strengthens the approximation at the fixed root $o$ by the averaging gain $2^{-\ell}$ and, when the root has a tree neighborhood, also records the three-term Green-function relation that will later yield the improved Dyson residual.
\begin{definition}[The set $\Omega_o^{-}(z)$]\label{def:Omega-o-minus}
For $z\in\mathbb C_+(\mathcal I_\delta)$ with
$
\eta\geq (\log N)^{\fb}/N
$,
let
$
\cT=\cB^{\rm ell}_\ell(o,\cG;z)
$
with fixed root $o\in\bG$.
We define $\Omega_o^{-}(z)$ to be the set of graphs $\cG$ such that,
\begin{align}
&|R(Q(\cG,z),z)|_\infty
\lesssim\varepsilon(z),\qquad|Q-m|_\infty\lesssim\epsilon_*(z)\label{eq.Omega-minus-RQ}\\
&\max_{i\in\bG}\left|
 G_{oi}
-
G_{oi}\left(
\operatorname{Ext}
\bigl(
\mathcal B_r(\{o,i\},\mathcal G),Q(z)
\bigr),z
\right)
\right|\lesssim
\frac{(\log N)^3\varepsilon(z)}{2^\ell}+
\log N\big|
R(Q(\cG,z),z)\big|_\infty.\label{e:Omega-o-minus-Q}
\end{align}
Moreover, if $o$ has a radius-$\fR$ tree neighborhood in $\cG$, then for every neighbor $v$ of $o$
\begin{equation}\label{e:Omega-o-minus-root}
\left|
 G_{ov}^2- G_{ov}
-
 G_{oo} G_{vv}
\right|
\lesssim
\frac{(\log N)^3\varepsilon(z)}{4^\ell}
\end{equation}
\end{definition}

The second key proposition is the corresponding improvement step. Starting from a graph in $\Omega_o^+(z)$, a second local resampling produces a graph in $\Omega_o^-(z)$ with overwhelming conditional probability.

\begin{proposition}\label{p:Omega-}
Let $z\in \bC_+(\mathcal I_\delta)$ with $\eta\geq (\log N)^\fb/N$. 
Then for any $\cG\in \Omega_o^+(z)$ and large $\fd,N\in\N_+$, there exists $F_3(\cG) \subset \sS(\cG)$, defined in Proposition \ref{prop:lift-concentration}, with $\P_{\cG}(F_3(\cG))= 1-\OO(N^{-\fd})$ such that for every switching data $\bfS\in F_3(\cG)$, 
the switched graph $\widetilde{\cG}=T_{\bf S}(\cG)\in \Omega_o^-(z)$. 
Here, the sets $\Omega_o^+(z)$ and $\Omega_o^-(z)$ are defined in Definitions \ref{def:Omega-o-plus} and \ref{def:Omega-o-minus}.
\end{proposition}

The proof of the above proposition is in Section \ref{sec:concentration}.

Let us summarize the relations among the good events introduced so far.
By definition and Proposition \ref{p:Omega+},
\[
\Omega^-(z)\subset \Omega(z)\subset\bar\Omega,
\qquad
\Omega_o^+(z)\subset \bar\Omega^+.
\]
The event $\Omega^-(z)$ is the strengthened version of $\Omega(z)$,
with the same type of global estimates but with an additional margin
in the error bounds.

The event $\Omega_o^+(z)$ is the local event associated with the first
resampling step. Compared with $\Omega(z)$, it allows logarithmic
losses in the original Green-function estimates, while introducing
the corresponding control for the filled exterior
$\cG^{\langle\bT\rangle}$. Thus it is adapted to the switching
configuration around the fixed root $o$.

The event $\Omega_o^-(z)$ is the local event produced after the second
resampling step. Its estimates are concentrated at the root $o$ and
contain the additional averaging gain coming from the elliptic
boundary. It also records the local three-term relation that will be
used to improve the Dyson residual.

Accordingly, the two local events are related to $\Omega(z)$ through
the successive high-probability transitions
\[
\Omega(z)
\;\xrightarrow{\text{first local resampling}}\;
\Omega_o^+(z)
\;\xrightarrow{\text{second local resampling}}\;
\Omega_o^-(z).
\]

The key mechanism in the proof of Proposition~\ref{prop:bootstrap} is to obtain $\Omega_o^-(z)$ simultaneously for all roots $o$ and then use the deterministic implication 
\[ \Omega(z)\bigcap_{o\in\bG}\Omega_o^-(z) \subset \Omega^-(z). \] 
The former follows from the two resampling propositions and a union bound over the roots, while the latter converts the resulting rootwise estimates into the desired global improvement. We now carry out this argument.

\begin{proof}[Proof of Proposition \ref{prop:bootstrap}]
For simplicity, we omit \(z\) from the notation throughout the proof.
We first claim that, for every fixed root $o\in\bG$,
\[
\P\bigl(\Omega\setminus\Omega_o^-\bigr)
=
\OO(N^{-\fd}).
\]
Let $\P'$ be the joint law from
Definition~\ref{def:elliptic-enlarged-space} and write
$\widehat A:=T^{-1}(A)$ where $T$ is defined in \eqref{eq:elliptic-T-map}.
Since Proposition \ref{prop:elliptic-reverse} implies $T$ is measure preserving, we have
$
\P\bigl(\Omega\setminus\Omega_o^-\bigr)
=
\P'\bigl(\widehat\Omega\setminus\widehat\Omega_o^-\bigr).
$
By Proposition \ref{prop:elliptic-reverse}, the switching map $\widetilde T$ defined in \eqref{eq:elliptic-Ttilde} is a measure-preserving involution.
Applying Proposition \ref{p:Omega+} and Proposition \ref{p:Omega-} therefore gives
\[
\P'\bigl(
\widehat\Omega\setminus\widetilde T(\widehat\Omega_o^+)
\bigr)
+
\P'\bigl(
\widetilde T(\widehat\Omega_o^+)\setminus\widehat\Omega_o^-
\bigr)
=
\P'\bigl(
\widetilde T(\widehat\Omega)\setminus\widehat\Omega_o^+
\bigr)
+
\P'\bigl(
\widetilde T(\widehat\Omega_o^+)\setminus\widehat\Omega_o^-
\bigr)=
\OO(N^{-\fd}).
\]
Hence
$
\P'\bigl(
\widehat\Omega\setminus\widehat\Omega_o^-
\bigr)
=
\OO(N^{-\fd}),
$
which proves the claim.

Taking a union bound over all roots $o\in\bG$, and recalling that
$|\bG|=\OO(N)$ since the base graph is fixed, we obtain
\begin{equation}
\label{e:bootstrap-union-root}
\P\left(
\Omega
\setminus
\bigcap_{o\in\bG}\Omega_o^-
\right)
=
\OO(N^{-\fd+1}).
\end{equation}

It remains to prove the inclusion
$
\Omega
\bigcap_{o\in\bG}\Omega_o^-
\subset
\Omega^-.
$
Let
$
\cG\in
\Omega
\bigcap_{o\in\bG}\Omega_o^-.
$
We first improve the Dyson residual.
For every directed edge $e=(i,j)\in\vec E_\cH$, recall
\[
Q_i
=
\frac1N
\sum_{\pi(x)=i}G_{xx},
\qquad
p_{ij}
=
\frac1N
\sum_{\pi(x)=i}G_{xx_e}.
\]
For every root $x$ having a radius-$\fR$ tree neighborhood,
Definition~\ref{def:Omega-o-minus} gives
\[
\left|
G_{xx_e}^2
-
G_{xx_e}
-
G_{xx}G_{x_ex_e}
\right|
\lesssim
\frac{(\log N)^3\varepsilon}{4^\ell}.
\]
Averaging this identity over the roots of type $i$, using
$\cG\in\Omega$ to replace the individual Green function entries by
their common local comparison values, and using that at most
$\OO(N^\fc)$ roots fail to have a radius-$\fR$ tree neighborhood and that the Green's function entries are uniformly bounded,
gives
\begin{equation}
\label{e:bootstrap-averaged-quadratic}
\left|
p_{ij}^2-p_{ij}-Q_iQ_j
\right|
\lesssim
\frac{(\log N)^3\varepsilon}{4^\ell}
+\varepsilon^2+
N^{-1+\fc}\lesssim \frac{(\log N)^3\epsilon}{4^\ell}.
\end{equation}

The exact vertex reconstruction identity \eqref{e:new-Q-vertex-reconstruction} and  the relations
$
p_{ij}=-Q_iQ_{ij}=-Q_jQ_{ji}
$
give
\begin{align}
R_{ij}(Q,z)
=
-\frac1{Q_{ij}}
-z
-\sum_{\substack{k:k\sim j\\ k\neq i}}Q_{jk}
=
-\frac1{Q_{ij}}
+\frac1{Q_j}
+Q_{ji}
=
-\frac{p_{ij}^2-p_{ij}-Q_iQ_j}{p_{ij}Q_j}.
\label{e:residual-three-term-relation}
\end{align}
Since $p_{ij}$ and $Q_j$ are uniformly bounded away from zero on
$\mathbb C_+(\mathcal I_\delta)$ with $|z|\le 2d_{\rm max}$ and the above results hold for every directed edge $(i,j)$, it follows from \eqref{e:bootstrap-averaged-quadratic} that
\begin{equation}\label{eq.R-improved}
|R(Q,z)|_\infty\lesssim 4^{-\ell}(\log N)^3\varepsilon\lesssim\varepsilon/(\log N)^3. 
\end{equation}
This also improves the Green-function approximation.
Fix arbitrary $i,j\in\bG$, and apply
Definition~\ref{def:Omega-o-minus} with root $o=i$.
Substituting \eqref{eq.R-improved} into
\eqref{e:Omega-o-minus-Q} yields
\begin{align}
&|R(Q,z)|_\infty
+
\max_{i,j\in\bG}
\left|
G_{ij}(z)
-
G_{ij}\left(
\Ext\bigl(\cB_r(\{i,j\},\cG),Q(\cG,z)\bigr),z
\right)
\right|
\lesssim
\frac{(\log N)^3\varepsilon(z)}{2^\ell}
+
\frac{\varepsilon(z)}{(\log N)^2}
\le \frac{\varepsilon(z)}2
\label{e:bootstrap-local-improvement}
\end{align}
for sufficiently large $N$.
Thus, to obtain $\cG\in\Omega^-(z)$, it remains to prove
\[
|Q(\cG,z)-m(z)|_\infty\le\frac{\varepsilon_*(z)}2.
\]

Next we use Lemma~\ref{lem.stability2} to control
$|Q-m|_\infty$ by the Dyson residual.
In the stable bulk and stable off-support regimes away from the possible
bipartite singularity at $0$, 
Lemma~\ref{lem.stability2} directly implies $|Q-m|_\infty\lesssim|R(Q,z)|_\infty\lesssim\varepsilon/(\log N)^3\le\varepsilon_*/2$ for sufficiently large $N$.

We next consider the possible bipartite instability at $z=0$.
Suppose that
\[
\bH=\bH_+\sqcup\bH_-,
\qquad
\sigma_i=
\begin{cases}
1,&i\in\bH_+,\\
-1,&i\in\bH_-.
\end{cases}
\]
By \eqref{e:new-Q-vertex-reconstruction}, the vertex quantities
associated with the directed-edge vector $Q$ in
Lemma~\ref{lem.stability2} coincide with the empirical quantities
$Q^{\rm ver}$.
Using $p_{\bar e}=p_e$, multiplying
\eqref{e:averaged-row-identity} by $\sigma_i$, and summing over
$i\in\bH$, we obtain
\begin{align}
\sum_{i\in\bH}\sigma_i
+
z\sum_{i\in\bH}\sigma_iQ_i
=
\sum_{e\in\vec E_\cH}\sigma_{t(e)}p_e=
\frac12
\sum_{e\in\vec E_\cH}
\left(
\sigma_{t(e)}p_e
+
\sigma_{h(e)}p_{\bar e}
\right)
=0,
\label{e:empirical-bipartite-balance-Q}
\end{align}
where we used $\sigma_{h(e)}=-\sigma_{t(e)}$.
Hence
$
\sum_{i\in\bH}\sigma_iQ_i
=
-z^{-1}\sum_{i\in\bH}\sigma_i.
$
The deterministic solution $(m,m^{\rm ver})$ satisfies the same identity, and therefore
\begin{equation}
\label{e:bipartite-balance-final}
\sum_{i\in\bH}
\sigma_i
\bigl(
Q_i-m_i
\bigr)
=0.
\end{equation}
Thus the exceptional bipartite mode in
\eqref{eq.bulk-0-stability} produces no additional error. The same conclusion holds as in the regular bulk and off-support regimes.

It remains to consider a spectral parameter
$z=E+\mathrm i\eta$ near a regular edge or cusp $\tau$.
We prove the required probability estimate by a downward
induction along the vertical line through $E$.
The case $|z|=2d_{\rm max}$ follows from the starting estimate,
so suppose that $|z|<2d_{\rm max}$.
Choose $\eta_0>\eta$ such that
$|E+\mathrm i\eta_0|=2d_{\rm max}$, and choose heights
\[
\eta_0>\eta_1>\cdots>\eta_K=\eta,
\qquad
\eta_{k-1}-\eta_k\le N^{-1},
\qquad K=\OO(N).
\]
Write $z_k:=E+\mathrm i\eta_k$.
For each $k$, define the exceptional event
$ B_k
:=
\Omega(z_k)\setminus
\bigcap_{o\in\bG}\Omega_o^-(z_k),
$
and set
$
\mathcal{D}
:=
\bar\Omega\setminus\bigcup_{k=0}^K B_k.
$
By \eqref{e:bootstrap-union-root} and a union bound over
the lattice heights,
\begin{equation}
\label{e:bootstrap-vertical-exception}
\P(\bar\Omega\setminus\mathcal D)
=
\OO(N^{-\fd+2}).
\end{equation}

Fix $\cG\in\mathcal D$.
We prove by induction that, for every $0\le k\le K$,
\begin{equation}
\label{e:bootstrap-vertical-induction}
\cG\in\Omega^-(z_k),
\qquad
|R(Q(\cG,z_k),z_k)|_\infty
\le C_0\frac{\varepsilon(z_k)}{(\log N)^3},
\end{equation}
where $C_0$ is a sufficiently large fixed constant.

At $k=0$, Proposition \ref{prop:bootstrap-start} implies that the large-$|z|$ starting estimate gives
$\cG\in\Omega^-(z_0)\subseteq\Omega(z_0)$.
Since $\cG\notin B_0$, we obtain
$\cG\in\Omega^-(z_0)\bigcap_o\Omega_o^-(z_0)$.
Thus \eqref{eq.R-improved} supplies the residual estimate
in \eqref{e:bootstrap-vertical-induction}, proving the base case.

Suppose now that \eqref{e:bootstrap-vertical-induction}
holds at all heights $j<k$.
We first transfer the weaker estimates to $z_k$.
More generally, let $s$ be a previously treated lattice height,
and let $\eta\le t\le s$ satisfy $s-t\le N^{-1}$.
The uniform diagonal Green-function bounds at $E+\mathrm i s$,
the resolvent identity and the Ward identity give, as in
\cite[Claim~4.13]{Huang2024Spectrum},
\begin{equation}\label{eq.G-downward}
\max_{x,y}
|G_{xy}(E+\mathrm i t)-G_{xy}(E+\mathrm i s)|
\lesssim\frac1{Nt}.
\end{equation}
The same bound holds for the vertex and adjacent-entry averages.
By the exact vertex reconstruction identity \eqref{e:new-Q-vertex-reconstruction}, the vertex
averages at $s$ are uniformly bounded away from zero,
and the preceding estimate preserves this property at $t$.
Using $Q_e=-p_e/Q_{t(e)}$, and the Stieltjes representation
of $m$, we obtain
\begin{equation}
\label{e:bootstrap-vertical-continuity}
|Q(\cG,E+\mathrm i t)-Q(\cG,E+\mathrm i s)|_\infty
+
|m(E+\mathrm i t)-m(E+\mathrm i s)|_\infty
\lesssim\frac1{Nt}.
\end{equation}

For the local comparison Green functions, write
$
P(u)
:=
G\bigl(\Ext(\cB_r(\{x,y\},\cG),Q(\cG,E+\mathrm i u)),E+\mathrm i u\bigr).
$
The resolvent identity and Cauchy--Schwarz therefore imply
\begin{align}\label{eq.P-downward}
|P_{xy}(t)-P_{xy}(s)|
\lesssim
(r+1)\left(
|t-s|
+
|Q(\cG,E+\mathrm i t)-Q(\cG,E+\mathrm i s)|_\infty
\right)
\lesssim\frac{\log N}{Nt}.
\end{align}
Since the current vertical line is near a regular edge or cusp,
the corresponding correction term in
\eqref{e:def-epsilon-zero} gives, 
$
\varepsilon(E+\mathrm i t)
\gtrsim 4^{6\ell}/(Nt)\sim(\log N)^{6\fa}/(Nt).
$
It follows directly that
$
\log N/(Nt)
=
o\left(\varepsilon(E+\mathrm i t)/(\log N)^3\right).
$
Taking $s=\eta_{k-1}$ and $t=\eta_k$, the preceding estimates \eqref{eq.G-downward}-\eqref{eq.P-downward} and the margins in $\Omega^-(z_{k-1})$
give
$
\cG\in\Omega(z_k).
$
Since $\cG\in\mathcal D$, it follows that
\[
\cG\in\Omega(z_k)\bigcap_{o\in\bG}\Omega_o^-(z_k).
\]
Thus \eqref{eq.R-improved} and
\eqref{e:bootstrap-local-improvement} hold at $z_k$.
In particular, choosing $C_0$ sufficiently large gives
the residual estimate in \eqref{e:bootstrap-vertical-induction}.

If $z_k$ lies in a stable regime, the preceding stability
argument already gives $\cG\in\Omega^-(z_k)$.
Otherwise, consider the continuous family
\[
q_t:=Q(\cG,E+\mathrm i t),
\qquad \eta_k\le t\le\eta_0.
\]
For each $t$ in this interval, choose a previously treated
lattice height $t'=\eta_j$, $j<k$, such that
$0\le t'-t\le N^{-1}$.
By \eqref{e:bootstrap-vertical-continuity} and the induction
hypothesis,
\[
|q_t-m(E+\mathrm i t)|_\infty
\le
\frac{\varepsilon_*(E+\mathrm i t')}{2}
+\frac{C}{Nt}
=o(1)\le c_0
\]
uniformly along the segment.
In particular, the entries of $q_t$ remain uniformly
bounded away from zero.
The definition of the Dyson residual and
\eqref{e:bootstrap-vertical-continuity} therefore give
$
|R(q_t,E+\mathrm i t)-R(q_{t'},E+\mathrm i t')|_\infty
\lesssim(Nt)^{-1}.
$
Using the residual estimate retained at all previous lattice
heights and the monotonicity of $\varepsilon$, we obtain
\begin{align*}
\sup_{\eta_k\le t\le\eta_0}
|R(q_t,E+\mathrm i t)|_\infty
\lesssim
\frac{\varepsilon(z_k)}{(\log N)^3}
+\frac1{N\eta_k}
\lesssim
\frac{\varepsilon(z_k)}{(\log N)^3}.
\end{align*}
Thus the continuation assumption in
Lemma~\ref{lem.stability2} holds with residual bound
$C\varepsilon(z_k)/(\looparrowright^3$.
Applying the corresponding stability estimate, and noting
that this improves the stability error by at least a factor
$(\log N)^{-1}$, yields
\[
|Q(\cG,z_k)-m(z_k)|_\infty
\lesssim
\frac{\varepsilon_*(z_k)}{\log N}
\le\frac{\varepsilon_*(z_k)}2
\]
for sufficiently large $N$.
Together with \eqref{e:bootstrap-local-improvement}, this
proves $\cG\in\Omega^-(z_k)$ and completes the induction.

Since $z_K=z$, we have shown that
$\mathcal D\subseteq\Omega^-(z)$.
Therefore,
\[
\P\bigl(\Omega(z)\setminus\Omega^-(z)\bigr)
\le
\P\bigl(\bar\Omega\setminus\mathcal D\bigr)
=
\OO(N^{-\fd+2}),
\]
by \eqref{e:bootstrap-vertical-exception}.
Since $\fd>0$ is arbitrary, replacing $\fd$ by $\fd+1$ yields \eqref{eq.bootstrap} and completes the proof.
\end{proof}

\section{Proof of Proposition \ref{p:Omega+}: Stability under local resampling}\label{sec:stability}
The purpose of this section is to prove
Proposition~\ref{p:Omega+}, namely the stability of the local law
under the first local resampling around a fixed root $o$.
The proof is organized around three operations: we first remove the
elliptic neighborhood $\bT$ and replace it by its filled exterior,
then perform the admissible switchings on the exterior graph, and
finally glue $\bT$ back. The main task is to show that the local
Green-function approximation remains stable throughout these
operations.

We first recall and summarize the graph notation used in this section.
The graph $\cG$ denotes the original lift, while
$
\cG^{\langle\bT\rangle}
:=
\TE\bigl(\cG^{(\bT)}\bigr)
$
is the filled exterior obtained by removing $\bT$ and applying the
tree extension. Given switching data $\mathbf S$, we denote by
$
\widetilde{\cG}
:=
T_{\mathbf S}(\cG)
$
the switched graph, and by
$
\widetilde{\cG}^{\langle\bT\rangle}
$
the graph obtained by performing the corresponding switchings on the
filled exterior $\cG^{\langle\bT\rangle}$.
Their Green functions are denoted respectively by
\[
G,
\qquad
G^{\langle\bT\rangle},
\qquad
\widetilde G,
\qquad
\widetilde G^{\langle\bT\rangle}.
\]
Thus, throughout this section, a tilde indicates the switched graph,
while the superscript $\langle\bT\rangle$ indicates the filled
exterior obtained after removing $\bT$.

Thus the graph transformations and the corresponding Green functions
may be summarized schematically as
\[
\begin{array}{ccccccc}
\cG
&
\xrightarrow{\ \text{filled removal of }\bT\ }
&
\cG^{\langle\bT\rangle}
&
\xrightarrow{\ \text{local switching}\ }
&
\widetilde{\cG}^{\langle\bT\rangle}
&
\xrightarrow{\ \text{glue back }\bT\ }
&
\widetilde{\cG}=T_{\mathbf S}(\cG)
\\[2mm]
G
&
\longrightarrow
&
G^{\langle\bT\rangle}
&
\longrightarrow
&
\widetilde G^{\langle\bT\rangle}
&
\longrightarrow
&
\widetilde G .
\end{array}
\]

\subsection{Woodbury Formula under Switching}
\label{sec:woodbury-switching}

In this subsection, we introduce a Woodbury expansion adapted to the switching
arguments used throughout the paper. The basic observation is that a switching
changes the adjacency matrix only through a finite-rank perturbation supported
on the switching vertices. By comparing the true Green function with a suitable
local reference Green function, the effect of switching can then be expressed
in terms of the pre-existing approximation error.

Throughout this subsection, by a switching we mean the following pair switching.
For vertices \(u,v\), write
\begin{equation}
\label{e:def-Delta}
\Delta_{uv}
:=
e_ue_v^{\mathsf T}
+
e_ve_u^{\mathsf T}.
\end{equation}
For each \(\alpha\in\qq{\mu}\), consider two pairs
$
(x_\alpha,y_\alpha)
$
and
$
(x_\alpha^\star,y_\alpha^\star),
$
and replace them by
$
(x_\alpha,y_\alpha^\star)
$
and
$
(x_\alpha^\star,y_\alpha).
$
We encode the corresponding change by
\begin{equation}
\label{e:def-xi-pair}
\xi_\alpha
:=
\Delta_{x_\alpha y_\alpha}
+
\Delta_{x_\alpha^\star y_\alpha^\star}
-
\Delta_{x_\alpha y_\alpha^\star}
-
\Delta_{x_\alpha^\star y_\alpha},\qquad 
\Xi
:=
\sum_{\alpha=1}^{\mu}\xi_\alpha.
\end{equation}
If \(H\) and \(\widetilde H\) denote the adjacency matrices before and after the switchings, respectively,
then we directly have
$
\widetilde H-H
=
-\Xi.
$
Introduce
\begin{equation}
\label{e:def-UV-pair}
U
:=
\bigl(
e_{x_1}-e_{x_1^\star},
\ldots,
e_{x_\mu}-e_{x_\mu^\star}
\bigr),
\qquad
V
:=
\bigl(
e_{y_1}-e_{y_1^\star},
\ldots,
e_{y_\mu}-e_{y_\mu^\star}
\bigr).
\end{equation}
Then a direct expansion gives
$
\Xi
=
UV^{\mathsf T}+VU^{\mathsf T}
=
\begin{pmatrix}V&U\end{pmatrix}
\begin{pmatrix}U&V\end{pmatrix}^{\mathsf T},
$

\subsubsection{Woodbury formula}

We recall the standard Woodbury identity. Let \(A\) be an invertible matrix
and let \(X,Y\) be rectangular matrices such that
$
\mathbb I+Y^{\mathsf T}A^{-1}X
$
is invertible. The Woodbury formula gives
\begin{equation}
\label{e:woodbury}
(A+XY^{\mathsf T})^{-1}
=
A^{-1}
-
A^{-1}X
\bigl(
\mathbb I+Y^{\mathsf T}A^{-1}X
\bigr)^{-1}
Y^{\mathsf T}A^{-1}.
\end{equation}

Let
$
G:=(H-z)^{-1}
$
and
$
\widetilde G=(\widetilde H-z)^{-1}.
$
Applying \eqref{e:woodbury} to
$
\widetilde{H}-H
=
-\begin{pmatrix}V&U\end{pmatrix}
\begin{pmatrix}U&V\end{pmatrix}^{\mathsf T}
$,
we obtain
\begin{equation}
\label{e:woodbury-G}
\widetilde G-G
=
G
F_G
G,\qquad \mbox{where~}F_G:=\begin{pmatrix}V&U\end{pmatrix}
\left(
\mathbb I
-
\begin{pmatrix}U&V\end{pmatrix}^{\mathsf T}
G
\begin{pmatrix}V&U\end{pmatrix}
\right)^{-1}
\begin{pmatrix}U&V\end{pmatrix}^{\mathsf T}.
\end{equation}

The matrix \(F_G\) also admits another representation. Indeed, we have
$
F_G
=
G^{-1}\widetilde G G^{-1}
-
G^{-1}.
$
Since
$
\Xi
=
G^{-1}-\widetilde G^{-1},
$
we obtain
\begin{align}
F_G
=
\bigl(
G^{-1}-\widetilde G^{-1}
\bigr)
\widetilde G G^{-1}
=
\Xi\widetilde G
\bigl(
\widetilde G^{-1}+\Xi
\bigr)
=
\Xi+\Xi\widetilde G\Xi.
\label{e:F-Xi-general}
\end{align}
This form will be particularly convenient for exploiting the geometry of the
switching vertices.

\begin{lemma}[Boundary factorization]
\label{lem:F-boundary-factorization}
The switching kernel \(F_G\) admits the factorization
\begin{equation}
\label{e:F-boundary-factorization}
F_G
=
\begin{pmatrix}
V&U
\end{pmatrix}
F_\partial
\begin{pmatrix}
V&U
\end{pmatrix}^{\mathsf T},\qquad 
\mbox{where~}F_\partial:
=
\begin{pmatrix}
U^{\mathsf T}\widetilde G U
&
I+U^{\mathsf T}\widetilde G V
\\[1mm]
I+V^{\mathsf T}\widetilde G U
&
V^{\mathsf T}\widetilde G V
\end{pmatrix}.
\end{equation}
Consequently, we have
\begin{equation}\label{eq.woodbury1}
  \widetilde G-G=G\begin{pmatrix}
V&U
\end{pmatrix}
F_\partial
\begin{pmatrix}
V&U
\end{pmatrix}^{\mathsf T}G.
\end{equation}
\end{lemma}

\begin{proof}
Expanding the right-hand side of
\eqref{e:F-boundary-factorization}, we obtain
\begin{align*}
&
\begin{pmatrix}
V&U
\end{pmatrix}
\begin{pmatrix}
U^{\mathsf T}\widetilde G U
&
I+U^{\mathsf T}\widetilde G V
\\
I+V^{\mathsf T}\widetilde G U
&
V^{\mathsf T}\widetilde G V
\end{pmatrix}
\begin{pmatrix}
V^{\mathsf T}\\
U^{\mathsf T}
\end{pmatrix}
\\
&
=
V(U^{\mathsf T}\widetilde G U)V^{\mathsf T}
+
V(I+U^{\mathsf T}\widetilde G V)U^{\mathsf T}
+
U(I+V^{\mathsf T}\widetilde G U)V^{\mathsf T}
+
U(V^{\mathsf T}\widetilde G V)U^{\mathsf T}
\\
&
=
UV^{\mathsf T}
+
VU^{\mathsf T}
+
\bigl(UV^{\mathsf T}+VU^{\mathsf T}\bigr)
\widetilde G
\bigl(UV^{\mathsf T}+VU^{\mathsf T}\bigr)
=
\Xi+\Xi\widetilde G\Xi
=
F_G.
\end{align*}
Substituting the above equality into \eqref{e:woodbury-G} yields \eqref{eq.woodbury1}.
\end{proof}

The factorization \eqref{e:F-boundary-factorization} is purely algebraic.
In particular, once the perturbation has the pair-switching form
\eqref{e:def-xi-pair}, the Woodbury kernel always reduces to the same
\(2\times2\) block structure.

\subsubsection{Expansion of the Green function}
In some cases, rather than using \eqref{eq.woodbury1} directly, we
compare \(G\) with a localized reference Green function.  More
precisely, let \(\mathcal K\subset\mathcal G\) be a subgraph containing
all vertices involved in the switching under consideration, and let
\(Q\) be the boundary weight used in the corresponding extension.  We
define the localized Green function on \(\mathcal K\) by
\[
G\bigl(\Ext(\mathcal K,Q),z\bigr).
\]
Since \(\bK\) is in general a proper subset of \(\bG\),
the matrix \(G\bigl(\Ext(\mathcal K,Q),z\bigr)\) and the original Green matrix \(G\) naturally
act on spaces of different dimensions.  For the purpose of the
Woodbury comparison below, we identify \(G\bigl(\Ext(\mathcal K,Q),z\bigr)\) with its
zero extension to \(\bG\times\bG\); namely, throughout this section we
write
\[
P_{xy}
:=
\begin{cases}
G_{xy}\bigl(\Ext(\mathcal K,Q),z\bigr),
& x,y\in\bK,\\
0,
& \text{otherwise}.
\end{cases}
\]
Thus \(P\) is regarded as a matrix on the same ambient vertex space
\(\bG\).  This convention is only an embedding of the
localized Green matrix into the ambient matrix space.

Let \(\widetilde{\mathcal K}\) be the graph obtained from
\(\mathcal K\) by performing the same pair switchings as in
\(\mathcal G\), and define the localized Green function after switching by
$
G\bigl(\Ext(\widetilde{\mathcal K},Q),z\bigr).
$
We again identify \(G\bigl(\Ext(\widetilde{\mathcal K},Q),z\bigr)\) with its zero extension to
\(\bG\times\bG\), and denote it by $\widetilde P$.  Since the switching vectors \(U\) and \(V\) are
supported entirely in \(\bK\), this zero-extension convention does not
affect any of the Woodbury products appearing below.

Applying Lemma \ref{lem:F-boundary-factorization} on the localized
vertex set \(\bK\), and then extending the resulting identity by zero
to \(\bG\times\bG\), gives
\begin{equation}
\label{e:woodbury-L}
\widetilde P-P
=
P F_P P,
\qquad\mbox{where~}
F_P
:=
\begin{pmatrix}V&U\end{pmatrix}
\begin{pmatrix}
U^{\mathsf T}\widetilde P U
&
I+U^{\mathsf T}\widetilde P V
\\[1mm]
I+V^{\mathsf T}\widetilde P U
&
V^{\mathsf T}\widetilde P V
\end{pmatrix}
\begin{pmatrix}V&U\end{pmatrix}^{\mathsf T}.
\end{equation}
Then the graph geometry enters through the four compressed Green function matrices
\begin{equation}
\label{e:four-F-blocks}
U^{\mathsf T}\widetilde P U,
\qquad
U^{\mathsf T}\widetilde P V,
\qquad
V^{\mathsf T}\widetilde P U,
\qquad
V^{\mathsf T}\widetilde P V.
\end{equation}
For the different switching configurations considered in the rest of the paper, the geometry of
the underlying graph often
allows us to simplify the boundary kernel \(F_\partial\) substantially. In
particular, some of the above blocks may vanish identically, become diagonal
or block diagonal, or admit sharper estimates due to separation and local
tree structure. These further simplifications depend on the specific
switching configuration and will therefore be carried out separately in the
corresponding applications.

Set
$
D
:=
G-P.
$
The quantity \(D\) measures the error of the local approximation before
switching. The following lemma gives an expansion of the difference $\widetilde G-G$ in terms of the localized kernel matrix $F_P$ and the local approximation error $D$.

\begin{lemma}[Woodbury expansion, Lemma 10.2 in \cite{huang2026lecture}]
\label{lem:woodbury-G-expansion}
Assume that the corresponding Neumann series converges. Then
\begin{equation}
\label{e:woodbury-G-expansion}
\widetilde G-G
=
\sum_{k=0}^{\infty}
GF_P(DF_P)^kG.
\end{equation}
\end{lemma}

Using the above lemma, we can obtain an expansion relating the Green function errors with respect to their local approximations before and after switching.

\begin{lemma}[Woodbury error expansion]
\label{lem:woodbury-error-expansion}
Assume that the Neumann series converges. Then
\begin{equation}
\label{e:woodbury-error-expansion}
\widetilde G-\widetilde P
=
(I+PF_P)
\left(
\sum_{k=0}^{\infty}
(DF_P)^kD
\right)
(I+F_PP).
\end{equation}
\end{lemma}

\begin{proof}
  Lemma \ref{lem:woodbury-G-expansion} gives $
\widetilde G-G
=
\sum_{k=0}^{\infty}
GF_P(DF_P)^kG.
$
Using \(G=P+D\) and the above equality, we obtain
\begin{align*}
\widetilde G-\widetilde P
=
G-P
+
(\widetilde G-G)
-
(\widetilde P-P)
=
D
+
\sum_{k\geq0}
(P+D)F_P(DF_P)^k(P+D)
-
PF_PP.
\end{align*}
Collecting the terms according to the number of occurrences of \(D\), we
obtain \eqref{e:woodbury-error-expansion}.
\end{proof}

In particular, the leading contribution is
$
(I+PF_P)D(I+F_PP)$
while every term with \(k\geq1\) contains at least one additional factor
\(DF_P\).

\subsection{Filled removal and the enlarged comparison graph}\label{subsec:filled-removal-enlarged-graph}
We first verify that the local Green-function approximation remains
stable for the graph $\cG^{\langle\bT\rangle}$. This is
the first step in passing from the global event $\Omega(z)$ to the
local event $\Omega_o^+(z)$.
\begin{proposition}[Stability under the filled removal of \(\bT\)]
\label{prop:woodbury-stability-filled-T}
Fix $z\in\bC_+(\mathcal I_\delta)$ with $\eta\ge(\log N)^\fb/N$ and let $\cG\in\Omega(z)$. Then for every $i,j\in\bG\backslash\bT$,
\[
\left|
G^{\langle \bT\rangle}_{ij}
-
G_{ij}\left(
\operatorname{Ext}
\bigl(
\mathcal B_r(\{i,j\},\mathcal G^{\langle \bT\rangle}), Q(z)
\bigr),z
\right)
\right|
\lesssim
(\log N)\varepsilon.
\]
In particular, we have $\Omega(z)\subset\Omega_o^+$, where $\Omega_o^+$ is defined in Definition \ref{def:Omega-o-plus}.
\end{proposition}

Next, we realize the passage from $\cG$ to
$\cG^{\langle\bT\rangle}$ through a suitable collection of switchings.
Recall that
$
\partial_{\vec E} \cT
=
\bigl\{
e_\alpha=(l_\alpha,a_\alpha):
1\leq \alpha\leq\mu
\bigr\}
$
where
$
l_\alpha\in \bT
$
and
$
a_\alpha\notin \bT.
$
Fix \(i,j\in \bG\setminus \bT\), and define the common local core
\[
\mathcal G_o
:=
\mathcal B_r\bigl(\bT\cup\{i,j\},\mathcal G\bigr),
\]
which contains \(\bT\), \(i,j\), and all the boundary vertices
\(l_\alpha,a_\alpha\). 
Define the auxiliary graphs associated with $\cG$ and $\cG_o$:
\[
\mathcal G^{\rm enl}
:=
\mathcal G
\coprod_{\alpha=1}^{\mu}\mathcal X_\alpha,
\qquad
\Ext(\mathcal{G}^{\rm enl}_o,Q)
:=
\Ext(\cG_o,Q)
\coprod_{\alpha=1}^{\mu}\mathcal X_\alpha,
\]
where $\cX_\alpha$ is a disjoint copy of the universal cover tree $\cX$.
Since components are disjoint, their Green matrices are
\[
G^{\rm enl}
=
G
\bigoplus_{\alpha=1}^{\mu}G_{\mathcal X_\alpha}
\quad
\mbox{and}
\quad
P^{\rm enl}
=
P
\bigoplus_{\alpha=1}^{\mu}G_{\mathcal X_\alpha},\qquad\mbox{where}~P
:=
G\bigl(\Ext(\mathcal G_o, Q(z)),z\bigr).
\]
Thus the two enlarged Green matrices agree exactly on all the added
universal-cover copies:
$D^{\rm enl}:=G^{\rm enl}-P^{\rm enl}=D\oplus0$ where $D:=G-P$.

We now describe how to use the switching action to obtain the graph $\cG^{\langle\bT\rangle}$. 
For every \(1\leq\alpha\leq\mu\), choose an edge $(l^\star_\alpha,a^\star_\alpha)$ in $\cX_\alpha$ having the same directed type as $(l_\alpha,a_\alpha)$ and perform the switching
\[
(l_\alpha,a_\alpha),
\quad
(l_\alpha^\star,a_\alpha^\star)
\longrightarrow
(l_\alpha,a_\alpha^\star),
\quad
(a_\alpha,l_\alpha^\star).
\]
After all the switchings, the enlarged graph splits into two components:
an exterior component consisting of $\bG\setminus\bT$ together with
the truncated branches rooted at $l_\alpha^\star$, and an interior
component consisting of $\bT$ together with those rooted at
$a_\alpha^\star$. The exterior
component of the switched actual graph is \(\mathcal G^{\langle \bT\rangle}\), and $G^{\langle\bT\rangle}$ is its Green's function.
We also denote the interior component by $\cG^{\rm in}$ and its Green matrix by $G^{\rm in}$.
The filled switching construction is illustrated in
Figure~\ref{fig:filled-switching-decomposition}.

Apply the same switchings to the enlarged graph $\Ext(\cG^{\rm enl}_o,Q)$. Since the common local core \(\mathcal G_o\)
contains \(\bT,i,j\) and all the boundary vertices, the exterior component
of the switched  graph is precisely
$
\Ext
\bigl(\mathcal G_o^{\langle \bT\rangle}, Q\bigr),
$
where \(\mathcal G_o^{\langle \bT\rangle}\) is obtained from \(\mathcal G_o\)
by the same procedure. We write
$
P^{\langle \bT\rangle}
:=
G\big(
\operatorname{Ext}
\bigl(\mathcal G_o^{\langle \bT\rangle}, Q\bigr),z
\big).
$

The interior components of the two enlarged graphs obtained after
switching are identical. Hence the restrictions of their Green
functions to the interior component are both given by
\(G^{\mathrm{in}}\). 
Let $G^{\rm aux}$ and $P^{\rm aux}$ denote the Green's functions
of the enlarged graphs obtained after switching from the original
graph and the localized graph, respectively.
Since the exterior and interior components are disconnected after switching, all Green-function blocks
between the two components vanish. 
Thus we can write
\begin{equation*}
  G^{\rm aux}=G^{\langle\bT\rangle}\oplus G^{\rm in}\qquad\mbox{and}\qquad P^{\rm aux}=P^{\langle \bT\rangle}\oplus G^{\rm in}.
\end{equation*}
It follows that the two switched
Green functions differ only on the exterior component, i.e., $G^{\rm aux}-P^{\rm aux}=(G^{\langle\bT\rangle}-P^{\langle \bT\rangle})\oplus 0$.
Accordingly, in the remainder of the argument we restrict all
Green-function differences and the corresponding estimates to the
exterior component. The common interior Green function
\(G^{\mathrm{in}}\) will enter only through the boundary blocks appearing
in the effective switching kernel.

\begin{figure}[H]
    \centering
    \includegraphics[width=1\linewidth]{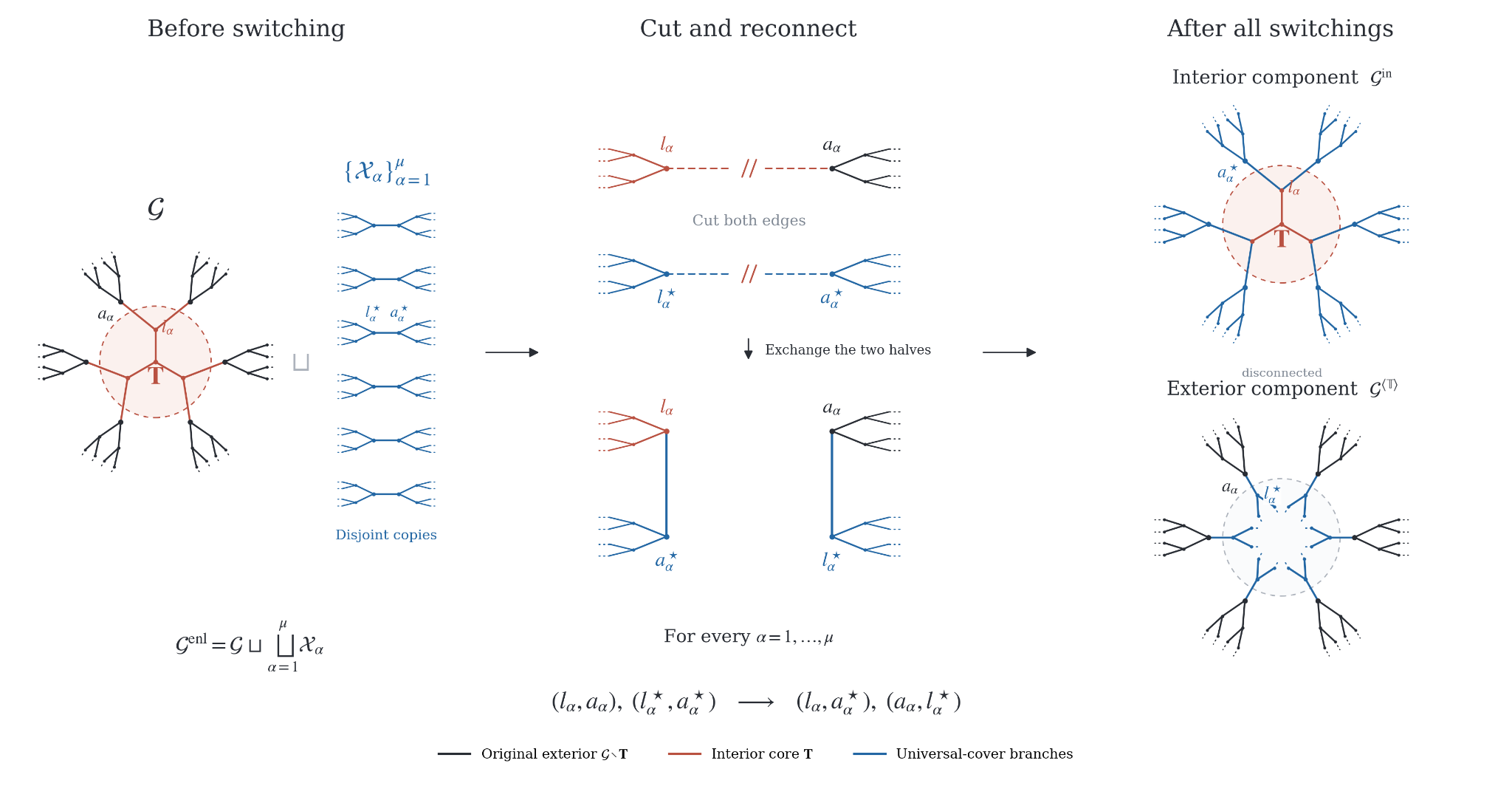}
    \caption{The filled switching construction and the resulting
    exterior--interior decomposition.}
    \label{fig:filled-switching-decomposition}
\end{figure}

We use the Woodbury formula to prove Proposition \ref{prop:woodbury-stability-filled-T}. 
Compared with the general Woodbury error expansion in
Lemma \ref{lem:woodbury-error-expansion}, the following lemma gives a more refined
representation tailored to
\(
G^{\langle\bT\rangle}_{ij}-P^{\langle\bT\rangle}_{ij}.
\)
The improvement comes from exploiting the specific geometry created by the
switching: the boundary kernel in Lemma \ref{lem:F-boundary-factorization} can be
further simplified according to the decomposition of the switched graph into
interior and exterior components and the resulting cell structure. As a
consequence, the full Woodbury correction can be compressed onto the switching
vertices, so that all interactions between the error matrix \(D\) and the
switching perturbation are transmitted through these boundary vertices only. 

For notational convenience, set
\[
\bA
:=
(a_\alpha)_{\alpha=1}^{\mu},
\qquad
\bL
:=
(l_\alpha)_{\alpha=1}^{\mu},
\qquad
\bA^\star
:=
(a_\alpha^\star)_{\alpha=1}^{\mu},
\qquad
\bL^\star
:=
(l_\alpha^\star)_{\alpha=1}^{\mu}.
\]

\begin{lemma}
\label{lem:compressed-woodbury-filled-T}
Under the same setting as Proposition \ref{prop:woodbury-stability-filled-T}, define the linear maps
\[
J_{\rm out},J_{\rm in}:\mathbb C^\mu\longrightarrow\mathbb C^{\bG},
\qquad
J_{\rm out}e_\alpha=e_{a_\alpha},
\qquad
J_{\rm in}e_\alpha=e_{l_\alpha},
\]
and set
$
J:=
\begin{pmatrix}
J_{\rm out}&J_{\rm in}
\end{pmatrix}
\in\mathbb C^{\bG\times 2\mu}.
$
Define the boundary kernel \(F_\partial\in\mathbb C^{2\mu\times2\mu}\) by
\begin{equation}\label{eq:F-boundary-matrix}
F_\partial
=
\begin{pmatrix}
G^{\rm in}_{\bL\bL}
+M_+
+M_+P^{\langle\bT\rangle}_{\bA\bA}M_+
&
I
+M_+P^{\langle\bT\rangle}_{\bA\bA}
+G^{\rm in}_{\bL\bL}M_-
\\[1mm]
I
+P^{\langle\bT\rangle}_{\bA\bA}M_+
+M_-G^{\rm in}_{\bL\bL}
&
P^{\langle\bT\rangle}_{\bA\bA}
+M_-
+M_-G^{\rm in}_{\bL\bL}M_-
\end{pmatrix},
\end{equation}
where 
$
M_+:=\operatorname{diag}
\bigl(
m_{a_\alpha l_\alpha}\bigr)_{\alpha=1}^\mu
$
and
$
M_-:=\operatorname{diag}
\bigl(
m_{l_\alpha a_\alpha}\bigr)_{\alpha=1}^\mu.
$
Moreover, let
$
W
\in
\mathbb C^{(\bG\setminus\bT)\times\bG}
$
be the matrix defined as 
\begin{equation}
\label{eq:transfer-entry-rewritten}
W_{wu}
=
\delta_{wu}
+
\sum_{\alpha=1}^{\mu}
P^{\langle\bT\rangle}_{wa_\alpha}
\left(
\mathbf 1_{\{u=l_\alpha\}}
+
m_{a_\alpha l_\alpha}
\mathbf 1_{\{u=a_\alpha\}}
\right).
\end{equation}
Then, for any \(i,j\in\bG\setminus\bT\), we have
\begin{equation}
\label{eq:compressed-woodbury-series-rewritten}
G^{\langle\bT\rangle}_{ij}
-
P^{\langle\bT\rangle}_{ij}
=
W_i
\sum_{k=0}^{\infty}
\left(
DJF_\partial J^{\mathsf T}
\right)^k
D
W_j^{\mathsf T},
\end{equation}
where \(W_i\)
denotes the row of \(W\) indexed by \(i\).
\end{lemma}

\begin{proof}[Proof of Lemma \ref{lem:compressed-woodbury-filled-T}]
Let \(H^{\rm enl}\) and \(\widetilde H^{\rm enl}\) denote the 
adjacency matrices of the enlarged graphs before and after switching. Then
\[
\widetilde H^{\rm enl}
-
H^{\rm enl}
=
-\sum_{\alpha=1}^\mu
\xi_\alpha,\qquad\mbox{where}~
\xi_\alpha
:=
\Delta_{a_\alpha l_\alpha}
+
\Delta_{a^\star_\alpha l^\star_\alpha}
-
\Delta_{a_\alpha l^\star_\alpha}
-
\Delta_{a^\star_\alpha l_\alpha}
\]
Lemma \ref{lem:F-boundary-factorization} and Lemma \ref{lem:woodbury-error-expansion} directly imply
\begin{equation}
\label{eq:woodbury-error-master}
G^{\rm aux}-P^{\rm aux}
=
(I+P^{\rm enl}F)\sum_{k=0}^\infty (D^{\rm enl} F)^k
D^{\rm enl}
(I+FP^{\rm enl}),
\end{equation}
where \begin{equation}
\label{eq:F-boundary-factorization}
F
=
\begin{pmatrix}V&U\end{pmatrix}
\begin{pmatrix}
U^{\mathsf T} P^{\rm aux} U
&
I+U^{\mathsf T} P^{\rm aux} V
\\[1mm]
I+V^{\mathsf T}P^{\rm aux} U
&
V^{\mathsf T}P^{\rm aux} V
\end{pmatrix}
\begin{pmatrix}V&U\end{pmatrix}^{\mathsf T},\qquad U=(e_{l_\alpha}-e_{l_\alpha^\star})_{\alpha=1}^\mu,\quad V=(e_{a_\alpha}-e_{a_\alpha^\star})_{\alpha=1}^\mu.
\end{equation}
As \(D^{\rm enl}=D\oplus 0\) vanishes on the added universal-cover copies, it follows that $(D^{\rm enl}F)^k
D^{\rm enl}=(D \left. F\right|_{\bG})^k D$, where $\left. F\right|_{\bG}$ means the restriction of $F$ to the vertices $\bG$.
Let \(W\) be the block of
\(I+P^{\rm enl}F\) whose rows are indexed by original vertices outside \(\bT\) and
whose columns are indexed by original graph vertices.
Since all Green matrices under consideration are complex symmetric, the
corresponding block of \(I+FP^{\rm enl}\) is \(W^{\mathsf T}\).
Consequently, for \(i,j\in\bG\setminus\bT\), we have
\begin{equation*}
G^{\langle \bT\rangle}_{ij}-P^{\langle \bT\rangle}_{ij}
=
W_i
\sum_{k=0}^{\infty}
\left(
D\left.F\right|_{\bG}
\right)^k
D
W^{\mathsf T}_j
\end{equation*}
provided the series converges, where \(W_i\) and \(W_j\) denote the corresponding rows of \(W\).

After switching, \(\bA\) and \(\bL^\star\) belong to the exterior
component, whereas \(\bL\) and \(\bA^\star\) belong to the interior component
containing \(\bT\). Hence all Green-function blocks joining these two
components vanish. Moreover, \eqref{e:Schur} and \eqref{e:Schur1} give
\begin{align}
P^{\rm aux}_{\bL^\star \bA}
&=-M_+P^{\langle\bT\rangle}_{\bA\bA},
&
P^{\rm aux}_{\bL^\star\bL^\star}
&=M_++M_+P^{\langle\bT\rangle}_{\bA\bA}M_+,
\label{eq:exterior-branch-blocks}
\\
P^{\rm aux}_{\bL \bA^\star}
&=-G^{\rm in}_{\bL\bL}M_-,
&
P^{\rm aux}_{\bA^\star \bA^\star}
&=M_-+M_-G^{\rm in}_{\bL\bL}M_-,
\label{eq:interior-branch-blocks}
\end{align}

We can now calculate the four blocks in
\eqref{eq:F-boundary-factorization}. For example,
\begin{align*}
U^{\mathsf T} P^{\rm aux} U
=
P^{\rm aux}_{\bL\bL}
-P^{\rm aux}_{\bL\bL^\star}
-P^{\rm aux}_{\bL^\star\bL}
+P^{\rm aux}_{\bL^\star\bL^\star}
=
G^{\rm in}_{\bL\bL}+M_++M_+P^{\langle\bT\rangle}_{\bA\bA}M_+,
\end{align*}
because the two mixed blocks vanish. Similarly,
\begin{align*}
I+U^{\mathsf T} P^{\rm aux} V
=
I+P^{\rm aux}_{\bL \bA}
- P^{\rm aux}_{\bL \bA^\star}
- P^{\rm aux}_{\bL^\star \bA}
+ P^{\rm aux}_{\bL^\star \bA^\star}
=
I+G^{\rm in}_{\bL\bL}M_-+M_+P^{\langle\bT\rangle}_{\bA\bA}.
\end{align*}
The remaining two blocks are obtained in the same way:
\[
I+V^{\mathsf T} P^{\rm aux} U
=
I+P^{\langle\bT\rangle}_{\bA\bA}M_++M_-G^{\rm in}_{\bL\bL}\qquad\mbox{and}\qquad V^{\mathsf T} P^{\rm aux} V
=
P^{\langle\bT\rangle}_{\bA\bA}+M_-+M_-G^{\rm in}_{\bL\bL}M_-.
\]
Since $J=(V~U)|_{\bG,\cdot}$, we have $F|_{\bG}=JF_\partial J^{\mathsf T}$
where
\begin{equation}
F_\partial
=
\begin{pmatrix}
G^{\rm in}_{\bL\bL}+M_++M_+P^{\langle\bT\rangle}_{\bA\bA}M_+
&
I+M_+P^{\langle\bT\rangle}_{\bA\bA}+G^{\rm in}_{\bL\bL}M_-
\\[1mm]
I+P^{\langle\bT\rangle}_{\bA\bA}M_++M_-G^{\rm in}_{\bL\bL}
&
P^{\langle\bT\rangle}_{\bA\bA}+M_-+M_-G^{\rm in}_{\bL\bL}M_-
\end{pmatrix}.
\end{equation}

It remains to calculate the factors \(W\) and \(W^{\mathsf T}\).  From
\(I+P^{\rm enl}F=I+P^{\rm aux}\Xi\) and applying $P^{\rm aux}_{wl_\alpha^\star}
=
-m_{a_\alpha l_\alpha} P^{\rm aux}_{wa_\alpha }$
and
$
P^{\rm aux}_{wa_\alpha}
=
P^{\langle \bT\rangle}_{wa_\alpha}
$, for \(w\in \bG\backslash\bT\) we obtain
\begin{equation}
  W_{wu}
=
\delta_{wu}
+
\sum_{\alpha=1}^{\mu}
P^{\rm aux}_{wa_\alpha}
\mathbf 1_{\{u=l_\alpha\}}
-
\sum_{\alpha=1}^{\mu}
P^{\rm aux}_{w l_\alpha^\star}
\mathbf 1_{\{u=a_\alpha\}}=\delta_{wu}
+
\sum_{\alpha=1}^{\mu}
P^{\langle \bT\rangle}_{wa_\alpha}
\left(
\mathbf 1_{\{u=l_\alpha\}}
+
m_{a_\alpha l_\alpha }\mathbf 1_{\{u=a_\alpha\}}
\right).  
\end{equation}
\end{proof}

To control the Woodbury expansion, we need uniform summation bounds on
the Green-function entries connecting the switching boundary. These
estimates are used to bound both the effective boundary kernel
$F_\partial$ and the transfer matrix $W$ appearing in
Lemma~\ref{lem:compressed-woodbury-filled-T}. The following lemma
provides the required bounds for the exterior and interior components.

\begin{lemma}\label{claim:boundary-sum}
  Under the same setting as Proposition \ref{prop:woodbury-stability-filled-T}, the exterior component Green functions and the interior component Green functions satisfy
  \begin{equation}\label{eq:blackbox-exterior-sum}
  \sum_{\alpha=1}^{\mu}
  \left|
  P^{\langle \bT\rangle}_{wa_\alpha}
  \right|
  \lesssim1,\qquad w\in\bG\backslash\bT\qquad\mbox{and}\qquad
  \sup_{1\leq\alpha\leq\mu}
\sum_{\beta=1}^{\mu}
\left|
G^{\mathrm{in}}_{l_\alpha l_\beta}
\right|
\lesssim\log_4\log N.
  \end{equation}
\end{lemma}

\begin{proof}
Fix \(w\in\bG\setminus\bT\). By Lemma \ref{GTstructure}, there are at most
\(\OO(\fd)\) indices \(\alpha\) such that
$
\dist_{\cGT}(w,a_\alpha)\leq \fR/4.
$
For these indices, the uniform Green function bound gives
\[
\sum_{\alpha:\,\dist_{\cGT}(w,a_\alpha)\leq \fR/4}
\left|P^{\langle\bT\rangle}_{wa_\alpha}\right|
\lesssim 1,
\]
since \(\fd\) is fixed.
On the other hand, if
$
\dist_{\cGT}(w,a_\alpha)>\fR/4,
$
then, since $r\le \fR/50$ is the localization radius in the definition of
$P^{\langle\bT\rangle}$, the vertices $w$ and $a_\alpha$ cannot belong
to the same localized component. Hence
$
P^{\langle\bT\rangle}_{wa_\alpha}=0.
$
This proves the first inequality.

For the second estimate, recall that the interior component is a
tree-like elliptic neighborhood with at most one cycle, and that the
vertices \(l_\beta\) belong to its boundary. Hence
Proposition~\ref{prop:elliptic-neighborhood-estimates}\,(iv), applied
with \(j=l_\alpha\), yields
\[
\sum_{\beta=1}^{\mu}
\left|G^{\mathrm{in}}_{l_\alpha l_\beta}\right|
\leq
\sum_{y\in\del\cB_\ell^{\rm ell}}
\left|G^{\mathrm{in}}_{l_\alpha y}\right|
\lesssim \ell\lesssim\log_4\log N.
\]
This completes the proof.
\end{proof}

\begin{proof}[Proof of Proposition \ref{prop:woodbury-stability-filled-T}]

Lemma \ref{claim:boundary-sum} implies that the absolute row sums of \(P^{\langle\bT\rangle}_{\bA\bA}\) are
uniformly bounded, those of \(G^{\rm in}_{\bL\bL}\) are \(\OO(\log_4\log N)\), and \eqref{eq.mbound-comparable} implies that
entries of \(M_\pm\) are \(O(1)\). It follows directly from
\eqref{eq:F-boundary-matrix} that
\begin{equation}
\label{eq:Ffrak-row-bound-rewritten}
\|F_\partial\|_{\infty}=\max_{1\le x\le 2\mu}\sum_{y=1}^{2\mu}|(F_\partial)_{xy}|
\lesssim
\log_4\log N.
\end{equation}
Therefore, it follows from \eqref{eq:blackbox-exterior-sum} and
\eqref{eq.mbound-comparable} that
\begin{equation}
\label{eq:transfer-row-bound-rewritten}
\begin{aligned}
\sum_u |W_{wu}|
\leq
1+
\sum_{\alpha=1}^{\mu}
\left|
P^{\langle\bT\rangle}_{w a_\alpha}
\right|
\bigl(
1+|m_{a_\alpha l_\alpha}|
\bigr)
\lesssim
1,
\qquad
w\notin\bT.
\end{aligned}
\end{equation}

We next study the effective bound of the error matrix $D$. Although \(D\) is defined on the whole vertex set \(\bG\), the
Woodbury expansion \eqref{eq:compressed-woodbury-series-rewritten} implies the effective support in this expansion is $\bS_i\times \bS_j$, where $\bS_i:=({i}\cup\bA\cup\bL)$ and $\bS_j:=({j}\cup\bA\cup\bL)$.
Since $\cG\in\Omega(z)$, we have $|G_{uv}-G_{uv}(\Ext(\cB_r(\{u,v\},\cG),Q),z)|\lesssim\epsilon$ for every $u\in\bS_i$ and $v\in\bS_j$. By Lemma \ref{p:localization}, we have $|P_{uv}-G_{uv}(\Ext(\cB_r(\{u,v\},\cG),Q),z)|\lesssim(\log N)\epsilon$.
Then it follows from the triangle inequality that $|D_{uv}|=|G_{uv}-P_{uv}|\lesssim (\log N)\epsilon$. 
With the above estimate at hand, throughout the following analysis we can always regard \(D\) as satisfying the entrywise bound
$
|D|\lesssim (\log N)\epsilon.
$

Using the estimate $D\lesssim(\log N)\epsilon$ together with \eqref{eq:Ffrak-row-bound-rewritten} implies
\[
\left\|
D J F_\partial J^{\mathsf T}
\right\|_{\infty}
\lesssim
\mu(\log N)\varepsilon
\| F_\partial\|_{\infty}
\lesssim
4^\ell(\log N)^2\varepsilon
\lesssim(\log N)^{\fa+2}\varepsilon
\ll1,
\]
where the last inequality follows from the choice of parameters.
Hence the series is absolutely convergent. Next, we analyze each term in the series.
The term \(k=0\) in \eqref{eq:compressed-woodbury-series-rewritten}
satisfies, by $D\lesssim(\log N)\epsilon$ entrywise and
\eqref{eq:transfer-row-bound-rewritten},
\begin{equation}\label{eq.W-sum}
\left|
\bigl(WD W^{\mathsf T}\bigr)_{ij}
\right|
\leq
(\log N)\varepsilon
\left(\sum_u|W_{iu}|\right)
\left(\sum_v|W_{jv}|\right)
\lesssim
(\log N)\varepsilon.
\end{equation}
For \(k\geq1\), the elementary inequality $|x^{\mathsf T}Ay|\le|x|_1\|A\|_\infty|y|_\infty$ together with \eqref{eq.W-sum} and \eqref{eq:transfer-row-bound-rewritten} imply 
\begin{equation}
  |W_i(DJ F_\partial J^{\mathsf T})^kDW_j^{\mathsf{T}}|\le |W_i|_1\|DJ F_\partial J^{\mathsf T}\|^k_\infty|DW^{\mathsf{T}}_j|_\infty\lesssim (\log N)\epsilon((\log N)^{\fa+2}\epsilon)^k.
\end{equation}
 Summing the geometric series in \eqref{eq:compressed-woodbury-series-rewritten}, we conclude that
$
\left|
G^{\langle \bT\rangle}_{ij}
-
P^{\langle \bT\rangle}_{ij}
\right|
\lesssim
(\log N)\varepsilon.
$
Finally, Lemma~\ref{p:localization} applies, since the attached infinite tree branches are exact universal-cover branches and can be represented by the boundary weights $m$ in the extension of the finite core.
\begin{equation}
\left|
P^{\langle \bT\rangle}_{ij}
-
G_{ij}\left(
\operatorname{Ext}
\bigl(
\mathcal B_r(\{i,j\},\mathcal G^{\langle \bT\rangle}),Q(z)
\bigr),z
\right)
\right|
\leq
(\log N)\varepsilon,
\end{equation}
which implies 
\[
\left|
G^{\langle \bT\rangle}_{ij}
-
G_{ij}\left(
\operatorname{Ext}
\bigl(
\mathcal B_r(\{i,j\},\mathcal G^{\langle \bT\rangle}),Q(z)
\bigr),z
\right)
\right|
\leq
\left|
G^{\langle \bT\rangle}_{ij}
-
P^{\langle \bT\rangle}_{ij}
\right|
+
(\log N)\varepsilon
\lesssim
(\log N)\varepsilon.
\]
Hence we conclude the proof.
\end{proof}

\subsection{Cell structure}\label{sec.cell}
The purpose of the cell structure is to isolate the possible interactions
between different switching pairs. Although the sampled edges
$\{b_\alpha,c_\alpha\}$ are typically far apart, geometric separation alone
is not sufficient for the subsequent Green function expansions, since two
distant switching pairs may still interact through a relatively large
off-diagonal Green function entry. We therefore group together precisely
those switching supports which are either geometrically close or strongly
coupled through the Green function. The main point is that such interactions
are sparse with overwhelming probability, so that all but a small number of
sampled pairs form isolated cells and may be treated essentially independently.

The sparsity of this cell structure is controlled by the Ward identity.
Indeed, after obtaining a uniform bound on
$\Im G^{\langle\bT\rangle}_{xx}$, the Ward identity gives
\[
\sum_{y\in\bG\setminus\bT}
\left|G^{\langle\bT\rangle}_{xy}\right|^2
\lesssim
\frac{\Im G^{\langle\bT\rangle}_{xx}}{\eta},
\]
and hence only a small number of vertices $y$ can satisfy
$|G^{\langle\bT\rangle}_{xy}|\ge\varphi$. Combined with the
approximate uniformity of the sampled vertices $b_\alpha,c_\alpha$,
this shows that only few switching pairs can interact with any fixed
vertex or with one another. We now formalize this observation through
the following cell decomposition.

Fix a graph $\cG\in\GNH$.
For $x,y\in\bG\setminus\bT$, write $x\simeq y$ if
\begin{equation}\label{e:cell-relation}
\dist_{\cG^{\langle\bT\rangle}}(x,y)\le \frac{\fR}{4}
\qquad\text{or}\qquad
\bigl|G^{\langle\bT\rangle}_{xy}(z)\bigr|\ge\varphi(z).
\end{equation}
For two sets $\bX,\bY\subset\bG\setminus\bT$, write $\bX\simeq \bY$
if $x\simeq y$ for some $x\in \bX$ and $y\in \bY$.

\begin{definition}[Cell structure]\label{def:cell-structure}
Starting from
\[
\{a_1,\ldots,a_\mu\},
\qquad
\{b_\alpha,c_\alpha\},
\quad \alpha\in\mathsf W_S,
\]
repeatedly merge any two sets related by $\simeq$.
The resulting sets are called cells and are denoted by
$\mathbb S_0,\mathbb S_1,\ldots,\bS_\kappa$, where
$\{a_1,\ldots,a_\mu\}\subset\mathbb S_0$. We denote the set of indices corresponding to singleton-pair cells by
\[
\mathsf U
:=
\left\{
\alpha\in\mathsf W_S:
\{b_\alpha,c_\alpha\}=\mathbb S_s
\text{ for some }s\ge1
\right\}.
\]
\end{definition}

This is the analogue of
\cite[Definition~5.13]{Huang2024Spectrum}, without the
$\cB_1$-enlargement.
In particular, for distinct cells $\mathbb S_s$ and $\mathbb S_t$,
\begin{equation}\label{e:cell-separation}
\dist_{\cG^{\langle\bT\rangle}}
(\mathbb S_s,\mathbb S_t)>\frac{\fR}{4},
\qquad
\max_{\substack{x\in\mathbb S_s\\y\in\mathbb S_t}}
\bigl|G^{\langle\bT\rangle}_{xy}(z)\bigr|<\varphi(z).
\end{equation}
To include the added vertices $l_\alpha^\star$, recall that each
$l_\alpha^\star$ is the root of a covering branch attached to
$a_\alpha$ before switching.
The covering-branch identity and the boundedness of the Dyson
weights imply
\[
\bigl|G^{\langle\bT\rangle}_{x l_\alpha^\star}(z)\bigr|
\lesssim
\bigl|G^{\langle\bT\rangle}_{x a_\alpha}(z)\bigr|,
\qquad x\in\bG\setminus\bT.
\]
Hence, 
we obtain, whenever $\{b_\alpha,c_\alpha\}=\mathbb S_s$
for some $s\ge1$,
\begin{equation}\label{e:cell-switching-support}
\max_{\substack{
x\in\{b_\alpha,c_\alpha\}\\
y\in\bW_S''\setminus\{b_\alpha,c_\alpha\}
}}
\bigl|G^{\langle\bT\rangle}_{xy}(z)\bigr|
\lesssim\varphi(z),
\end{equation}
where $
\bW_S''
:=
\cup_{\beta\in\mathsf W_S}
\{a_\beta,b_\beta,c_\beta,l_\beta^\star\}.
$

We next refine the good geometric event from Proposition~\ref{prop:distdeg} by incorporating the Green's function bounds and the resulting sparsity of the cell structure.
\begin{proposition}
\label{prop:cell-sparsity}
Fix $z\in\bC_+(\mathcal I_\delta)$ with $\eta\ge(\log N)^\fb/N$ and let $\cG\in\Omega(z)$ or $\Omega^+_o(z)$.
For every large $\fd>0$, there exists an event $F_2(\cG)\subset F_1(\cG)\subset\mathsf S(\cG)$, where $F_1(\cG)$ is
defined in Proposition~\ref{prop:distdeg}, with probability $\P_\cG(F_2(\cG))=1-\OO(N^{-\fd})$ such that the following estimates hold:
\begin{itemize}
  \item The graph has the required geometric properties: \eqref{distfinitedegb}-\eqref{tcGTdist} hold.
  \item The graph has good Green's function bounds and a sparse cell structure:
  \begin{align}
\sup_{x\in\bG\setminus\bT}
\left|
\left\{
\alpha\in\{1,\ldots,\mu\}:
x\simeq\{b_\alpha,c_\alpha\}
\right\}
\right|
\lesssim\log N\qquad\mbox{and}\qquad
\left|
\mathsf W_{\mathbf S}\setminus\mathsf U
\right|
\lesssim\log N.
\label{e:cell-exceptional}
\end{align}
\end{itemize}
\end{proposition}

\begin{proof}
Since $F_2(\cG)\subset F_1(\cG)$, the geometric properties are immediate from Proposition~\ref{prop:distdeg}, so it remains only to prove the Green's function bound and the sparsity of the cell structure.

\textit{Step 1: Imaginary part of diagonal term.}
Since $\cG\in\Omega(z)$ or $\Omega_o^+(z)$, we have $|Q-m|_\infty\lesssim\epsilon_*$.
Consequently,
\begin{equation}\label{eq.imG_xx}
  \begin{aligned}
    |\im G^{\langle\bT\rangle}_{xx}(z)|&\lesssim |\im G_{xx}\big(\Ext(\cB_r(x,\cG^{\langle\bT\rangle}),Q),z\big)|+(\log N)\epsilon\\
    &\lesssim|\im G_{xx}\big(\Ext(\cB_r(x,\cG^{\langle\bT\rangle}),m),z\big)|+|Q-m|_\infty+(\log N)\epsilon\\
    &\lesssim |\im m^{\rm ver}|_\infty+\epsilon_*+(\log N)\epsilon=4^{-2\ell}N\eta\varphi^2,
  \end{aligned}
\end{equation}
where we use Lemma \ref{p:subgraph} in the second inequality, Proposition \ref{p:tree-extension-bound} in the third inequality and the definition of $\varphi$ in the last equality.

Using the above inequality, the Ward identity \eqref{e:Ward} on the filled graph gives
\[
\sum_{y\in\bG\setminus\bT}
\bigl|G^{\langle\bT\rangle}_{xy}(z)\bigr|^2
\lesssim
4^{-2\ell}N\varphi^2.
\]

\textit{Step 2: One-cell correlation probability.}
For every $x\in\bG\setminus\bT$, set $\bV_x:=\{u\in\bG\setminus\bT:|G_{xu}^{\langle\bT\rangle}|\ge\varphi\}$. It follows directly that $|\bV_x|\lesssim 4^{-2\ell}N$.
Thus, for every $x\in\bG\setminus\bT$, we have
\begin{align*}
  \P_\cG(x\simeq\{b_\alpha,c_\alpha\})&\lesssim\P_\cG\big(\dist_{\cG^{\langle\bT\rangle}}(x,\{b_\alpha,c_\alpha\})\le\fR/4\big)+\P_\cG\big(|G^{\langle\bT\rangle}_{x b_\alpha}|\vee|G^{\langle\bT\rangle}_{x c_\alpha}|\ge\varphi\big)\\
  &\lesssim\frac{(d_{\rm max}-1)^{\fR/4}}{N}+\frac{|\bV_x|}{N}\lesssim 4^{-2\ell}.
\end{align*}

\textit{Step 3: Sparsity of the cell structure.}
We now conclude by the same union-bound argument as in
\cite[Proposition~5.14]{Huang2024Spectrum}.
For every fixed $x\in\bG\setminus\bT$, the events
\[
x\simeq\{b_\alpha,c_\alpha\},
\qquad \alpha\in[1,\mu],
\]
are independent under the local sampling. Hence, for any sufficiently
large constant $C>0$,
\[
\begin{aligned}
\P_\cG\left(
\left|
\left\{
\alpha\in[1,\mu]:
x\simeq\{b_\alpha,c_\alpha\}
\right\}
\right|
\ge C\log N
\right)
\le
\binom{\mu}{C\log N}
\bigl(C4^{-2\ell}\bigr)^{C\log N} \le
\left(
\frac{C\mu4^{-2\ell}}{\log N}
\right)^{C\log N}
\ll N^{-\fd-2},
\end{aligned}
\]
where we used $\mu\sim4^\ell\sim(\log N)^{\fa}$. A union bound over
$x\in\bG\setminus\bT$ proves the first inequality in \eqref{e:cell-exceptional} with probability
$1-\OO(N^{-\fd})$.

It remains to prove the second inequality \eqref{e:cell-exceptional}. By the definition of
the cells, an index $\alpha\in\mathsf W_S$ for which
$\{b_\alpha,c_\alpha\}$ is not itself one of the cells
$\mathbb S_s$, $s\ge1$, can only occur if
$
\alpha\in \sB_1\cup \sB_2
$
where $\sB_1$ and $\sB_2$ are defined in \eqref{distdega} and \eqref{distdegb},
or if
\[
\{b_\alpha,c_\alpha\}
\simeq
\{a_\beta\}_{\beta=1}^\mu
\bigcup_{\substack{\beta\in\mathsf W_S\\ \beta\ne\alpha}}
\{b_\beta,c_\beta\}.
\]
On the event $F_1(\cG)$ of Proposition~\ref{prop:distdeg}, we have
$
|\sB_1|+|\sB_2|\lesssim\fd.
$
For the remaining indices, the same union-bound argument as above
shows that their number is at most $\OO(\log N)$ with probability
$1-\OO(N^{-\fd})$. Therefore,
\[
\left|
\left\{
\alpha\in\mathsf W_S:
\{b_\alpha,c_\alpha\}\ne\mathbb S_s
\text{ for every }s\ge1
\right\}
\right|
\lesssim\log N,
\]
which proves \eqref{e:cell-exceptional}.

By the above argument, conditionally on $F_1(\cG)$,
\eqref{e:cell-exceptional} holds with probability
$1-\OO(N^{-\fd})$. Since
$\P_\cG(F_1(\cG))=1-\OO(N^{-\fd})$, taking $F_2(\cG)$ to be the
intersection of $F_1(\cG)$ with this event yields
\[
F_2(\cG)\subset F_1(\cG),
\qquad
\P_\cG(F_2(\cG))=1-\OO(N^{-\fd}).
\]

\end{proof}

\subsection{Stability under switching of the filled exterior graph}\label{subsec:stability-switching-filled-exterior}
We next study the effect of the admissible switchings on the filled
exterior graph $\cG^{\langle\bT\rangle}$. The goal of this subsection
is to show that the local Green-function approximation remains stable
after these switchings, and to obtain a sharper off-diagonal bound for
isolated switching cells. 

\begin{proposition}
[Stability under switching of the filled exterior graph]
\label{prop:stability-switching-filled-exterior}
Fix $z\in\bC_+(\mathcal I_\delta)$ with $\eta\ge (\log N)^\fb/N$ and let $\cG\in\Omega(z)$ or $\Omega^+_o(z)$. Then for every $\mathbf S\in F_2(\cG)$, defined in Proposition \ref{prop:cell-sparsity}, with $\widetilde G=T_{\mathbf S}(\cG)$, the following estimates hold:
\begin{itemize}
  \item For all \(i,j\in\bG\setminus\bT\),
\begin{equation}\label{eq:filled-analogue-550}
\left|
\widetilde G^{\langle\bT\rangle}_{ij}(z)
-
G_{ij}
\left(
\operatorname{Ext}
\bigl(
\mathcal B_r
(i,j,\widetilde{\mathcal G}^{\langle\bT\rangle}),Q(z)
\bigr),
z
\right)
\right|
\lesssim(\log N)
\varepsilon.
\end{equation}

\item For every $\alpha\in \mathsf U$, defined in Definition \ref{def:cell-structure}, we have
\begin{equation}
\label{eq:filled-analogue-551-553}
\left|
\widetilde G^{\langle\bT\rangle}_{i c_\alpha}(z)
\right|
\lesssim
\varphi
+
(\log N)^2\varepsilon^2,
\end{equation}
provided that 
\[
\mbox{either}\qquad
i=c_\beta
\quad\text{for some}\quad
\beta\in\mathsf W_S\setminus\{\alpha\}\qquad\mbox{or}\qquad i\not\simeq\{b_\alpha,c_\alpha\}~\mbox{and}~
\operatorname{dist}_{\widetilde{\mathcal G}^{\langle\bT\rangle}}
(i,a_\alpha)
\geq \mathfrak R/4.
\]
\end{itemize}
\end{proposition}

Recall
$
\mathsf W_S\subset\llbracket 1,\mu\rrbracket
$
denotes the admissible switching indices. For every
\(\alpha\in\mathsf W_S\), the graph
\(\mathcal G^{\langle\bT\rangle}\) contains the two edges
$
(a_\alpha,l_\alpha^\star)
$
and
$
(b_\alpha,c_\alpha).
$
The switching replaces them by
\begin{equation}
\label{eq:filled-exterior-switching}
(a_\alpha,l_\alpha^\star),
\quad
(b_\alpha,c_\alpha)
\longrightarrow
(a_\alpha,b_\alpha),
\quad
(c_\alpha,l_\alpha^\star).
\end{equation}
For \(\alpha\notin\mathsf W_S\), no switching is performed. 
Since all switchings in this step act only on the exterior
graph, the interior component $\cG^{\rm in}$ does not enter this subsection.
The exterior component after all switchings is
\(\widetilde{\mathcal G}^{\langle\bT\rangle}\), and its Green matrix is
\(\widetilde G^{\langle\bT\rangle}\).
The exterior switching construction is illustrated in
Figure~\ref{fig:filled-exterior-switching}.

\begin{figure}[H]
    \centering
    \includegraphics[width=1\linewidth]{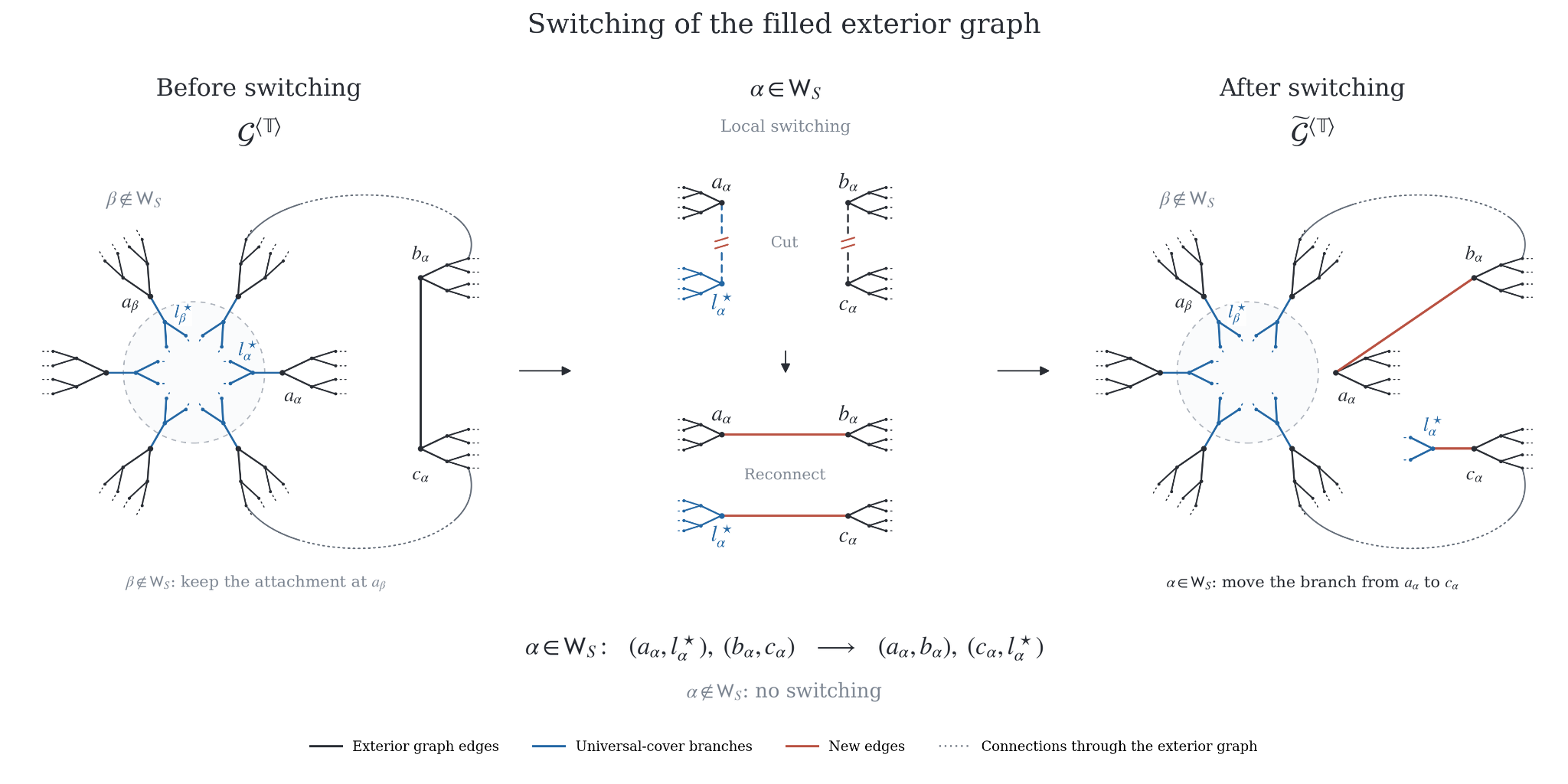}
    \caption{The filled exterior switching from
    $\cG^{\langle\bT\rangle}$ to
    $\widetilde{\cG}^{\langle\bT\rangle}$.}
    \label{fig:filled-exterior-switching}
\end{figure}

We next modify the localized comparison graph so that it contains the
same switching operations as the original graph. To this end, define
the set of original vertices involved in the admissible switchings by
\begin{equation}\label{def.bW_S}
\mathbb W_S
:=
\{a_\alpha,b_\alpha,c_\alpha:
\alpha\in\mathsf W_S\},
\end{equation}
and enlarge the original localized neighborhood
\(\mathcal B_r(\{i,j\},\mathcal G^{\langle\bT\rangle})\) to
\[
\mathcal G_o^{\mathsf W_S}
:=
\mathcal B_r
\bigl(
\{i,j\}\cup\mathbb W_S,
\mathcal G^{\langle\bT\rangle}
\bigr).
\]
Here and throughout this section, the radius-$r$ localization is
understood to apply only to the original vertices of $\mathcal G$:
all universal-cover branches introduced by the filled removal and
attached to the retained original vertices are kept in their entirety.
In particular, for every $\alpha\in\mathsf W_S$, the whole
universal-cover branch attached to $a_\alpha$ through
$l_\alpha^\star$ is contained in $\mathcal G_o^{\mathsf W_S}$ and is
not truncated by the radius-$r$ localization.

Thus, for every \(\alpha\in\mathsf W_S\), the localized graph contains
all vertices needed to realize the switching
\eqref{eq:filled-exterior-switching}. 
For notational simplicity, we slightly abuse notation and continue to
write $P^{\langle\bT\rangle}$ for the localized reference Green matrix
constructed from $\mathsf W_S$ throughout this section.
After attaching the boundary
weights \(Q\) to \(\mathcal G_o^{\mathsf W_S}\), we denote the resulting
localized Green matrix by
\begin{equation}\label{def.P<T>}
P^{\langle\bT\rangle}:=G(\Ext(\cG^{\sf W_S}_o,Q(z)),z).
\end{equation}
The approximation error before switching is defined by 
$
D^{\langle\bT\rangle}:=G^{\langle\bT\rangle}-P^{\langle\bT\rangle}.
$
We denote by
$
\widetilde{\mathcal G}_o^{\mathsf W_S}
$
the graph obtained from \(\mathcal G_o^{\mathsf W_S}\) by performing,
for every \(\alpha\in\mathsf W_S\), the same switching
\eqref{eq:filled-exterior-switching} as in
\(\mathcal G^{\langle\bT\rangle}\).
Attaching the same boundary weights \(Q(z)\) to
\(\widetilde{\mathcal G}_o^{\mathsf W_S}\), we define the corresponding
switched localized Green matrix by
\[
\widetilde P^{\langle\bT\rangle}:=G(\Ext(\widetilde \cG^{\sf W_S}_o,Q(z)),z).
\]

Again, we first calculate the refined Woodbury formula in this case.

\begin{lemma}
\label{lem:compressed-woodbury-filled-switching}
Under the same setting as Proposition \ref{prop:stability-switching-filled-exterior},
for each $\alpha\in\mathsf W_S$ define
\begin{equation}\label{eq:effective-switching-vectors}
v_\alpha^{\rm eff}
:=
-e_{b_\alpha}
-
m_{a_\alpha l_\alpha}e_{a_\alpha}\in\bC^{\bG},
~
u_\alpha
:=
e_{a_\alpha}-e_{c_\alpha}\in\bC^{\bG},
\qquad
J
:=
\bigl(
v_\alpha^{\rm eff},u_\alpha
\bigr)_{\alpha\in\mathsf W_S}
\in
\mathbb C^{\bG\times 2|\mathsf W_S|}.
\end{equation}
Define the boundary kernel
$
F_\partial\in
\mathbb C^{2|\mathsf W_S|\times 2|\mathsf W_S|}
$
cellwise by
\begin{equation}
\label{eq:F-boundary-direct-sum-refined-switching}
F_\partial
=
\bigoplus_{\alpha\in\mathsf W_S}
F_\partial^{(\alpha)},\quad F_\partial^{(\alpha)}
=
\begin{pmatrix}
\widetilde P^{\langle\bT\rangle}_{a_\alpha a_\alpha}
+
\widetilde P^{\langle\bT\rangle}_{c_\alpha c_\alpha}
&
1
-
\widetilde P^{\langle\bT\rangle}_{a_\alpha b_\alpha}
+
m_{a_\alpha l_\alpha}
\widetilde P^{\langle\bT\rangle}_{c_\alpha c_\alpha}
\\[1mm]
1
-
\widetilde P^{\langle\bT\rangle}_{b_\alpha a_\alpha}
+
m_{a_\alpha l_\alpha}
\widetilde P^{\langle\bT\rangle}_{c_\alpha c_\alpha}
&
m_{a_\alpha l_\alpha}
+
\widetilde P^{\langle\bT\rangle}_{b_\alpha b_\alpha}
+
m_{a_\alpha l_\alpha}^{\,2}
\widetilde P^{\langle\bT\rangle}_{c_\alpha c_\alpha}
\end{pmatrix}.
\end{equation}
Moreover, let
$
W\in
\mathbb C^{(\bG\setminus\bT)\times\bG}
$
be the effective transfer matrix defined by
\begin{align*}
W_{wu}
=
\delta_{wu}
-
&\sum_{\alpha\in\mathsf W_S}\Big[
\bigl(
\widetilde P^{\langle\bT\rangle}_{wb_\alpha}
+
m_{a_\alpha l_\alpha}
\widetilde P^{\langle\bT\rangle}_{wa_\alpha}
\bigr)
\mathbf 1_{\{u=a_\alpha\}}\\-
&\bigl(
\widetilde P^{\langle\bT\rangle}_{wc_\alpha}
-
\widetilde P^{\langle\bT\rangle}_{wa_\alpha}
\bigr)
\mathbf 1_{\{u=b_\alpha\}}-\bigl(
\widetilde P^{\langle\bT\rangle}_{wb_\alpha}
+
m_{a_\alpha l_\alpha}
\widetilde P^{\langle\bT\rangle}_{wc_\alpha}
\bigr)
\mathbf 1_{\{u=c_\alpha\}}\Big].
\label{eq:transfer-entry-refined-switching}
\end{align*}
Then, for any $i,j\in\bG\setminus\bT$,
\begin{equation}
\label{eq:compressed-woodbury-series-refined-switching}
\widetilde G^{\langle\bT\rangle}_{ij}
-
\widetilde P^{\langle\bT\rangle}_{ij}
=
W_i
\sum_{k=0}^{\infty}
\left(
D^{\langle\bT\rangle}|_\bG
J F_\partial J^{\mathsf T}
\right)^k
D^{\langle\bT\rangle}|_\bG
W_j^{\mathsf T},
\end{equation}
where $W_i$ denotes the row of $W$ indexed by $i$.
\end{lemma}

The derivation of the exact Woodbury identity proceeds by the same matrix-algebraic manipulations as in the proof of Lemma \ref{lem:compressed-woodbury-filled-T}. Since the argument is purely computational, we defer the proof of the above lemma to Appendix~\ref{app:woodbury}.

\begin{proof}[Proof of Proposition \ref{prop:stability-switching-filled-exterior}]
The equality \eqref{eq:F-boundary-direct-sum-refined-switching} together with $\widetilde{P}^{\langle\bT\rangle}\lesssim 1$ and $|m_{a_\alpha l_\alpha}|\lesssim 1$ directly 
yields
$
  \|F_\partial\|_\infty\lesssim 1.
$
By the analogue of Lemma \ref{claim:boundary-sum} for the switched comparison
graph and $|m_{a_\alpha l_\alpha}|\lesssim 1$, we obtain
\begin{equation}
\label{eq.sum-W}
\begin{aligned}
\sum_u |W_{wu}|
\leq
1+
\sum_{\alpha\in\mathsf W_S}
\Big[
2\left|
\widetilde P^{\langle\bT\rangle}_{w b_\alpha}
\right|
+
\bigl(1+|m_{a_\alpha l_\alpha}|\bigr)
\left(
\left|
\widetilde P^{\langle\bT\rangle}_{w a_\alpha}
\right|
+
\left|
\widetilde P^{\langle\bT\rangle}_{w c_\alpha}
\right|
\right)
\Big]
\lesssim 1 .
\end{aligned}
\end{equation}
Next we study the effective error bound of the error matrix $D^{\langle\bT\rangle}_\bG$. Proposition  \ref{prop:woodbury-stability-filled-T} gives 
\[
  |G^{\langle\bT\rangle}_{ij}-G_{ij}(\Ext(\cB_r(\{i,j\},\cG^{\langle\bT\rangle}),Q),z)|\lesssim(\log N)\varepsilon.
\]
By the same argument in the proof of Proposition \ref{prop:woodbury-stability-filled-T}, we can regard $D^{\langle\bT\rangle}|_\bG$ as satisfying the entrywise bound $|D^{\langle\bT\rangle}|_\bG|\lesssim(\log N)\varepsilon$.
Together with $
  \|F_\partial\|_\infty\lesssim 1
$ and \eqref{eq.sum-W},
repeating the argument in the proof of Proposition \ref{prop:woodbury-stability-filled-T} implies \eqref{eq:filled-analogue-550}.

 We next prove the refined estimates. We first record a consequence of
the geometric separation which will be used repeatedly. Since the
localization radius $r\le\fR/50$, whenever $\dist_{\cG^{\langle\bT\rangle}}(i,j)>\fR/4$ ($\dist_{\widetilde\cG^{\langle\bT\rangle}}(i,j)>\fR/4$), we have $P^{\langle\bT\rangle}_{ij}
=
0~(\widetilde P^{\langle\bT\rangle}_{ij}=0)$.

Fix $\alpha\in \mathsf U$.
Since \(\alpha\in U\subset \mathsf W_S\), Proposition \ref{prop:elliptic-reverse} implies that, whenever
$
i=c_\beta
$
and
$
\beta\in\mathsf W_S\setminus\{\alpha\},
$
$\dist_{\widetilde{\cG}^{\langle\bT\rangle}}(i,c_\alpha)>\fR/4$. Hence we obtain
$
\widetilde P^{\langle\bT\rangle}_{ic_\alpha}=0.
$
On the other hand, if
$
i\not\simeq\{b_\alpha,c_\alpha\},
$
we also claim that $
\operatorname{dist}_{\widetilde{\mathcal G}^{\langle\bT\rangle}}
(i,\{b_\alpha,c_\alpha\})\geq \mathfrak R/4.
$
Indeed, suppose that
\(
\operatorname{dist}_{\widetilde{\mathcal G}^{\langle\bT\rangle}}
(i,\{b_\alpha,c_\alpha\})<\mathfrak R/4
\)
and let \(\gamma\) be a corresponding path. By the isolation of
\(\alpha\in \mathsf U\subset\mathsf W_S\), the path \(\gamma\) cannot meet any
other switching cell, while it cannot enter the universal-cover branch
attached to \(c_\alpha\). Hence the only new switching edge that
\(\gamma\) could use is \(\{a_\alpha,b_\alpha\}\). This is also
impossible, since otherwise \(a_\alpha\in\gamma\) and therefore
$
\operatorname{dist}_{\widetilde{\mathcal G}^{\langle\bT\rangle}}
(i,a_\alpha)
\le |\gamma|<\mathfrak R/4,
$
contrary to the assumption. 
Consequently, we have $\widetilde P^{\langle\bT\rangle}_{ib_\alpha}=\widetilde P^{\langle\bT\rangle}_{ic_\alpha}=0$.
Therefore, in both cases under consideration, we have
\[
\widetilde G^{\langle\bT\rangle}_{ic_\alpha}
=
W_i D^{\langle\bT\rangle}\big|_{\bG} W^{\mathsf T}_{c_\alpha}+
W_i
\sum_{k=1}^{\infty}
\left(
D^{\langle\bT\rangle}|_\bG
J F_\partial J^{\mathsf T}
\right)^k
D^{\langle\bT\rangle}|_\bG
W_{c_\alpha}^{\mathsf T},
\]

We next estimate the leading order $(W D^{\langle\bT\rangle}|_\bG W^{\mathsf T})_{ic_\alpha}$.
Definition of $W_{uv}$ in Lemma \ref{lem:compressed-woodbury-filled-switching} and the block vanishing of
\(\widetilde P^{\langle\bT\rangle}\) \eqref{eq.block-vanish} imply
$
\operatorname{supp} W_{c_\beta}
\subset
\{b_\beta,c_\beta\}
$
for all $\beta\in\mathsf W_S$. 
Hence, when $i=c_\beta$ for some
$\beta\in\mathsf W_S\backslash\{\alpha\}$, only the block of
$D^{\langle\bT\rangle}|_\bG$ with rows indexed by
$\{b_\beta,c_\beta\}$ and columns indexed by
$\{b_\alpha,c_\alpha\}$ contributes to the expression under consideration.
 Since $\alpha\in\mathsf U\subset\mathsf W_S$, Proposition \ref{prop:elliptic-reverse} gives $\dist_{\cG^{\langle\bT\rangle}}(i,j)>\fR/4$ for $i\in \{b_\beta, c_\beta\}$ and $j\in\{b_\alpha,c_\alpha\}$, which implies $P^{\langle\bT\rangle}_{ij}=0$. Hence $D^{\langle\bT\rangle}_{ij}=G^{\langle\bT\rangle}_{ij}-P^{\langle\bT\rangle}_{ij}=G^{\langle\bT\rangle}_{ij}$. 
Moreover $\alpha\in \mathsf U$ gives $|G^{\langle\bT\rangle}_{ij}|\lesssim \varphi$. Therefore we can regard $D^{\langle\bT\rangle}|_\bG$ as satisfying the entrywise bound $\big|D^{\langle\bT\rangle}|_\bG\big|\lesssim\varphi$.
Together with \(\sum_u|W_{iu}|\lesssim1\), this yields
\begin{equation}\label{eq.Wphi}
\left|
W_iD^{\langle\bT\rangle}\big|_{\bG}
W_{c_\alpha}^{\mathsf T}
\right|
\lesssim\varphi.
\end{equation}
For the case $i\not\simeq\{b_\alpha,c_\alpha\}$ and
$
\operatorname{dist}_{\widetilde{\mathcal G}^{\langle\bT\rangle}}
(i,a_\alpha)\geq \mathfrak R/4
$, definition of $W_{uv}$ in Lemma \ref{lem:compressed-woodbury-filled-switching} and the block vanishing of
\(\widetilde P^{\langle\bT\rangle}\) \eqref{eq.block-vanish}, together with $\widetilde P^{\langle\bT\rangle}_{ib_\alpha}=\widetilde P^{\langle\bT\rangle}_{ic_\alpha}=0$ proved before, imply that $\supp W_i\subset\{i\}\bigcup_{\beta\neq\alpha}\{a_\beta,b_\beta,c_\beta\}$. On the other hand, $\supp W_{c_\alpha}\subset\{b_\alpha, c_\alpha\}$. Then \eqref{eq.Wphi} follows from the same argument.

It remains to estimate the higher-order terms in
\eqref{eq:compressed-woodbury-series-refined-switching}.
Here we exploit the isolation of the cell containing $\alpha$.
In the matrix products below, we suppress the restriction of
$D^{\langle\bT\rangle}$ to $\bG$.
Since $\alpha\in\mathsf U$, the cell separation implies
$P^{\langle\bT\rangle}_{xy}=0$ for
$x\in\{a_\beta,b_\beta,c_\beta\}$,
$\beta\in\mathsf W_S\setminus\{\alpha\}$, and
$y\in\{b_\alpha,c_\alpha\}$.
Hence \eqref{e:cell-switching-support} gives
$
|D^{\langle\bT\rangle}_{xy}|
=
|G^{\langle\bT\rangle}_{xy}|
\lesssim\varphi
$
on these entries.
Together with
$\supp W_{c_\alpha}\subset\{b_\alpha,c_\alpha\}$ and
$\sum_y|W_{c_\alpha y}|\lesssim1$, this implies that
$
|(D^{\langle\bT\rangle}W_{c_\alpha}^{\mathsf T})_x|
\lesssim
\varphi
$
on every switching cell different from the $\alpha$-cell, whereas
the crude bound gives
$
|(D^{\langle\bT\rangle}W_{c_\alpha}^{\mathsf T})_x|
\lesssim(\log N)\varepsilon
$
on the $\alpha$-cell.
Since $F_\partial$ is block diagonal with
$\|F_\partial\|_\infty\lesssim1$, it follows from the choice of
parameters that
\[
\|JF_\partial J^{\mathsf T}
D^{\langle\bT\rangle}W_{c_\alpha}^{\mathsf T}\|_1
\lesssim
(\log N)\varepsilon
+
4^\ell\varphi
\lesssim
(\log N)\varepsilon.
\]
Together with
$\|W_iD^{\langle\bT\rangle}\|_\infty\lesssim(\log N)\varepsilon$, this gives
$
|W_iD^{\langle\bT\rangle}JF_\partial J^{\mathsf T}
D^{\langle\bT\rangle}W_{c_\alpha}^{\mathsf T}|
\lesssim
(\log N)^2\varepsilon^2.
$
Moreover, by symmetry, we have
$
\|JF_\partial J^{\mathsf T}D^{\langle\bT\rangle}\|_{1}
=
\|D^{\langle\bT\rangle}JF_\partial J^{\mathsf T}\|_{\infty}
\lesssim
4^\ell(\log N)\varepsilon=o(1).
$
Thus, writing the $k$-th term as
$
W_iD^{\langle\bT\rangle}
(JF_\partial J^{\mathsf T}D^{\langle\bT\rangle})^{k-1}
JF_\partial J^{\mathsf T}
D^{\langle\bT\rangle}W_{c_\alpha}^{\mathsf T},
$
every additional factor
$JF_\partial J^{\mathsf T}D^{\langle\bT\rangle}$
produces a geometric gain in the $\ell^1$ norm.
Consequently,
\begin{equation}
\label{eq:higher-series-quadratic-refined}
\left|
W_i
\sum_{k=1}^{\infty}
\left(
D^{\langle\bT\rangle}|_\bG
JF_\partial J^{\mathsf T}
\right)^k
D^{\langle\bT\rangle}|_\bG
W_{c_\alpha}^{\mathsf T}
\right|
\lesssim
(\log N)^2\varepsilon^2.
\end{equation}
Combining \eqref{eq.Wphi} and
\eqref{eq:higher-series-quadratic-refined}, we conclude
\eqref{eq:filled-analogue-551-553}.
\end{proof}

\subsection{Stability after gluing back \(\bT\)}
We next glue the elliptic neighborhood $\bT$ back into the switched
filled exterior. The purpose of this subsection is to transfer the
stability estimates obtained for
$\widetilde{\cG}^{\langle\bT\rangle}$ to the full switched graph
$\widetilde{\cG}$. The resulting bounds depend on whether the two
vertices lie inside or outside $\bT$, reflecting the decay of the
Green function across the elliptic neighborhood.

\begin{proposition}[Stability after gluing back \(\bT\)]
\label{prop:woodbury-stability-glued-T}
Fix $z\in\bC_+(\mathcal I_\delta)$ with $\eta\ge (\log N)^\fb/N$ and let $\cG\in\Omega(z)$ or $\Omega^+_o(z)$. Then for every $\bFS\in F_2(\cG)$, defined in Proposition \ref{prop:cell-sparsity}, the following three estimates hold:
\begin{align}
&
\left|
\widetilde G_{ij}
-
G_{ij}\left(
\operatorname{Ext}
\bigl(
\mathcal B_r(\bT\cup\{i,j\},\widetilde{\mathcal G}),Q(z)
\bigr),z
\right)
\right|
\lesssim
\left(
\frac{1+|\ell_i-\ell_j|}
     {2^{|\ell_i-\ell_j|}}
+
\frac{(\log N)^2}
     {2^{2\ell-\ell_i-\ell_j}}
\right)(\log N)\varepsilon,
\qquad
i,j\in\bT,
\label{eq:glued-stability-TT}
\\
&
\left|
\widetilde G_{ij}
-
G_{ij}\left(
\operatorname{Ext}
\bigl(
\mathcal B_r(\bT\cup\{i,j\},\widetilde{\mathcal G}),Q(z)
\bigr),z
\right)
\right|
\lesssim
\frac{(\log N)^3\varepsilon}
     {2^{\ell-\ell_i}},\qquad
i\in\bT,\quad j\notin\bT,
\label{eq:glued-stability-T-out}
\\
&
\left|
\widetilde G_{ij}
-
G_{ij}\left(
\operatorname{Ext}
\bigl(
\mathcal B_r(\bT\cup\{i,j\},\widetilde{\mathcal G}),Q(z)
\bigr),z
\right)
\right|
\lesssim
(\log N)^3 \varepsilon,
\qquad
i,j\notin\bT,
\label{eq:glued-stability-out-out}
\end{align}
where $\ell_i$ is the elliptic level defined in \eqref{eq:elliptic-level}.
\end{proposition}

We first specify the localized exterior graph on
which the gluing operation will be performed.

Recall that
$
\widetilde a_\alpha
:=
\begin{cases}
c_\alpha,
&
\alpha\in\mathsf W_S,
\\
a_\alpha,
&
\alpha\notin\mathsf W_S,
\end{cases}
$
and 
$
\widetilde{\bA}
=
(\widetilde a_\alpha)_{\alpha=1}^\mu.
$
In the switched filled graph
$\widetilde{\mathcal G}^{\langle\bT\rangle}$,
the vertex $\widetilde a_\alpha$ is connected to the corresponding
truncated universal-cover branch through the edge
$
(\widetilde a_\alpha,l_\alpha^\star).
$
We define the localized exterior core by
\begin{equation}
\label{eq:localized-filled-core-after-switching}
\widetilde{\mathcal G}^{\langle\bT\rangle}_o
:=
\mathcal B_{r}
\bigl([
\{i,j\}\setminus\bT],
\widetilde{\mathcal G}^{\langle\bT\rangle}
\bigr)\cup\cB_{r-1}(\widetilde{\bA},\widetilde\cG^{\langle\bT\rangle}).
\end{equation}
As before, the radius-$r$ localization in
\eqref{eq:localized-filled-core-after-switching}
is understood to apply only to the original vertices $\bG$.
For every $1\leq\alpha\leq\mu$, the truncated universal-cover branch
attached to $\widetilde a_\alpha$ through
$l_\alpha^\star$ is retained and is not truncated by the
radius-$r$ localization.
This exterior core is important for the subsequent gluing operation.
Indeed, for every $\alpha$, the localized exterior graph contains
both $\widetilde a_\alpha$ and $l_\alpha^\star$, together with the
whole truncated branch rooted at $l_\alpha^\star$.
Hence all exterior vertices and edges involved in gluing $\bT$ back
are present in the localized graph.
Attaching the boundary weights $Q(z)$ to the boundary of
$\widetilde{\mathcal G}^{\langle\bT\rangle}_o$, we define the localized Green function
\[
\widetilde P^{\langle\bT\rangle}
:=
G\bigl(
\operatorname{Ext}
(
\widetilde{\mathcal G}^{\langle\bT\rangle}_o,Q(z)
),
z
\bigr).
\]

We next reintroduce the interior component $\cG^{\rm in}$ which was suppressed in the
previous section. Recall that this interior component is unaffected by
the exterior switchings. For every $1\leq\alpha\leq\mu$, it contains
the edge
$
(l_\alpha,a_\alpha^\star).
$
Recall that $G^{\rm in}$ denotes the Green matrix of this common interior
component.
We therefore define the auxiliary Green matrices before gluing by
\[
\widetilde G^{\rm aux}
:=
\widetilde G^{\langle\bT\rangle}
\oplus
G^{\rm in},
\qquad
\widetilde P^{\rm aux}
:=
\widetilde P^{\langle\bT\rangle}
\oplus
G^{\rm in}.
\]
Since the interior components of the two auxiliary graphs are
identical, their Green-function difference is supported entirely on
the exterior component. In particular,
\begin{equation}
\label{eq:tilde-D-aux-direct-sum}
\widetilde D^{\rm aux}
:=
\widetilde G^{\rm aux}
-
\widetilde P^{\rm aux}
=\widetilde D^{\langle\bT\rangle}
\oplus 0,\qquad \widetilde D^{\langle\bT\rangle}:=
\widetilde G^{\langle\bT\rangle}
-
\widetilde P^{\langle\bT\rangle}.
\end{equation}

We now glue $\bT$ back into both auxiliary graphs.
For each $1\leq\alpha\leq\mu$, the corresponding edges in the
interior and exterior components are respectively
$
(l_\alpha,a_\alpha^\star)
$
and
$
(\widetilde a_\alpha,l_\alpha^\star).
$
The gluing operation replaces them by
\begin{equation}
\label{eq:gluing-back-switching}
(l_\alpha,a_\alpha^\star),
\quad
(\widetilde a_\alpha,l_\alpha^\star)
\longrightarrow
(l_\alpha,\widetilde a_\alpha),
\quad
(l_\alpha^\star,a_\alpha^\star).
\end{equation}
We perform exactly the same switchings
\eqref{eq:gluing-back-switching} on the actual and localized
auxiliary graphs.

For the actual auxiliary graph, the edges
$(l_\alpha,\widetilde a_\alpha)$ reconnect $\bT$ to the switched
exterior graph, while the edges
$(l_\alpha^\star,a_\alpha^\star)$ glue the two complementary
truncated universal-cover branches back into the disjoint copies
$\mathcal X_\alpha$. Thus the resulting enlarged graph is
$
\widetilde{\mathcal G}
\coprod_{\alpha=1}^{\mu}\mathcal X_\alpha.
$
The gluing-back construction is illustrated in
Figure~\ref{fig:gluing-t-back}.

\begin{figure}[H]
    \centering
    \includegraphics[width=1\linewidth]{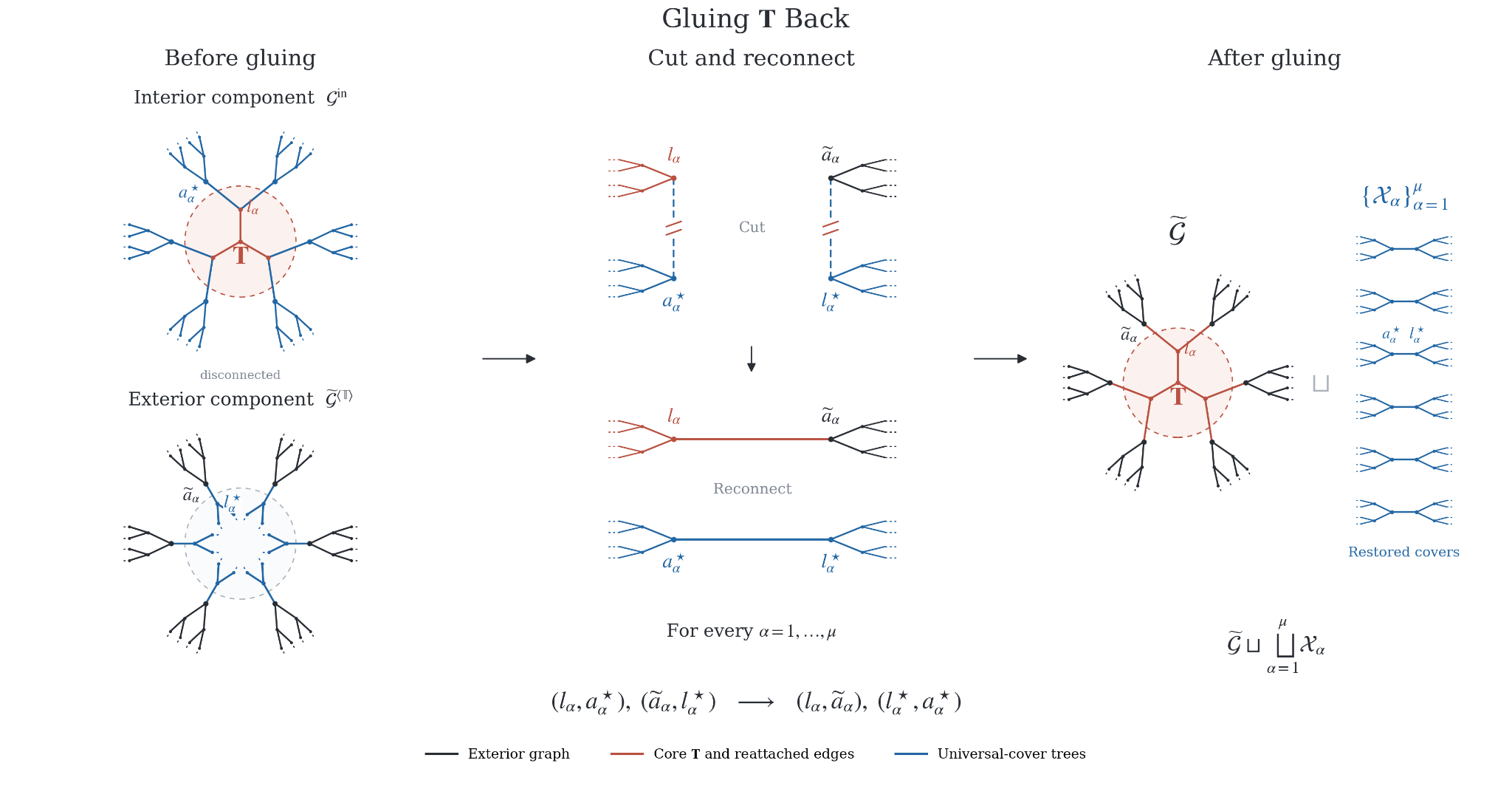}
    \caption{The gluing-back construction.}
    \label{fig:gluing-t-back}
\end{figure}

The same description applies to the localized auxiliary graph.
Indeed, by
\eqref{eq:localized-filled-core-after-switching}, its exterior core
contains the radius-$r$ neighborhoods of all
$\widetilde a_\alpha$, and the corresponding universal-cover branches
are retained in their entirety. After adjoining the common interior
component containing $\bT$ and performing
\eqref{eq:gluing-back-switching}, these pieces are joined through the
edges
$
(l_\alpha,\widetilde a_\alpha)
$
for 
$
1\leq\alpha\leq\mu.
$
Consequently, the original-vertex part of the resulting localized
graph is precisely
\begin{equation}
\label{eq:localized-core-after-gluing}
\widetilde{\mathcal G}_o
:=
\mathcal B_r
\bigl(
\{i,j\}\cup\bT,
\widetilde{\mathcal G}
\bigr),
\end{equation}
while the complementary truncated branches are reassembled into the
disjoint universal-cover copies $\mathcal X_\alpha$.

Therefore the enlarged localized graph after gluing is
$
\operatorname{Ext}
\bigl(
\widetilde{\mathcal G}_o,Q(z)
\bigr)
\coprod_{\alpha=1}^{\mu}\mathcal X_\alpha.
$
Define
\[
\widetilde P
:=
G\bigl(
\operatorname{Ext}
(
\widetilde{\mathcal G}_o,Q(z)
),
z
\bigr).
\]
Accordingly, the Green matrices of the two enlarged graphs after
gluing are
\[
\widetilde G^{\rm enl}
=
\widetilde G
\bigoplus_{\alpha=1}^{\mu}
G_{\mathcal X_\alpha},
\qquad
\widetilde P^{\rm enl}
=
\widetilde P
\bigoplus_{\alpha=1}^{\mu}
G_{\mathcal X_\alpha}.
\]

Again, we first calculate the refined Woodbury formula in this case.
For notation convenience, set
\[
\widetilde{\bA}
:=
(\widetilde a_\alpha)_{\alpha=1}^{\mu},
\qquad
\bL
:=
(l_\alpha)_{\alpha=1}^{\mu},
\qquad
\bA^\star
:=
(a_\alpha^\star)_{\alpha=1}^{\mu},
\qquad
\bL^\star
:=
(l_\alpha^\star)_{\alpha=1}^{\mu}.
\]

\begin{lemma}
\label{lem:compressed-woodbury-gluing}
Under the same setting as Proposition \ref{prop:woodbury-stability-glued-T}, define the boundary embedding
\[
J:\mathbb C^\mu\longrightarrow\mathbb C^{\bG},
\qquad
Je_\alpha=e_{\widetilde a_\alpha}.
\]
Then, whenever the corresponding Neumann series converges, for every
\(i,j\in\bG\),
\begin{equation}
\label{eq:compressed-woodbury-gluing}
\widetilde G_{ij}-\widetilde P_{ij}
=
W_i
\sum_{k=0}^{\infty}
\left(
\widetilde D^{\langle\bT\rangle}|_\bG
J\widetilde F_\partial J^{\mathsf T}
\right)^k
\widetilde D^{\langle\bT\rangle}|_\bG
W_j^{\mathsf T},
\end{equation}
where
\begin{equation}
\label{eq:F-boundary-matrix-gluing}
\widetilde F_\partial
=
\begin{pmatrix}
M_+&I
\end{pmatrix}
\begin{pmatrix}
\widetilde P_{\widetilde{\bA}\widetilde{\bA}}
&
\widetilde P_{\widetilde{\bA}\bL}
\\[1mm]
\widetilde P_{\bL\widetilde{\bA}}
&
\widetilde P_{\bL\bL}
\end{pmatrix}
\begin{pmatrix}
M_+\\ I
\end{pmatrix}
-
M_+,
\qquad
M_+
:=
\operatorname{diag}
\bigl(
m_{\widetilde a_\alpha l_\alpha}\bigr)_{\alpha=1}^\mu
.
\end{equation}
and
\begin{equation}
\label{eq:transfer-entry-gluing}
W_{xu}
=
\delta_{xu}
-
\sum_{\alpha=1}^{\mu}
\left(
\widetilde P_{xl_\alpha}
+
m_{\widetilde a_\alpha l_\alpha}
\widetilde P_{x\widetilde a_\alpha}
\right)
\mathbf 1_{\{u=\widetilde a_\alpha\}},
\qquad x,u\in\bG.
\end{equation}
\end{lemma}

The derivation of the exact Woodbury identity proceeds by the same matrix-algebraic manipulations as in the proof of Lemma \ref{lem:compressed-woodbury-filled-T}. Since the argument is purely computational, we defer the proof of the above lemma to Appendix~\ref{app:woodbury}.

The following lemma gives the boundary summation estimates for
$\widetilde D^{\langle\bT\rangle}$ needed to control the leading term
$
W\widetilde D^{\langle\bT\rangle}W^{\mathsf T}
$
in the proof of Proposition \ref{prop:woodbury-stability-glued-T}.

\begin{lemma}\label{claim:boundary sum}
  Under the same setting as Proposition \ref{prop:woodbury-stability-glued-T}, we have
\begin{equation}
\sup_{y\notin\bT}
\sum_{\alpha=1}^{\mu}
\left|
\widetilde D_{\widetilde a_\alpha y}^{\langle\bT\rangle}
\right|
\lesssim
(\log N)^2\varepsilon\qquad\mbox{and}\qquad
\sum_{\substack{\alpha,\beta=1\\ \alpha\neq\beta}}^{\mu}
\left|
\widetilde D_{\widetilde a_\alpha\widetilde a_\beta}^{\langle\bT\rangle}
\right|
\lesssim
(\log N)^3\varepsilon.
\end{equation}
\end{lemma}

\begin{proof}
From the preceding proof, we have the crude entrywise bound
$
\left|\widetilde D_{xy}^{\langle\bT\rangle}\right|
\lesssim (\log N)\varepsilon .
$
On the other hand, for $\alpha\in\mathsf U$, if
$
y\not\simeq\{b_\alpha,c_\alpha\}$
and 
$
\dist_{\widetilde\cG^{\langle\bT\rangle}}(y,a_\alpha)
\ge \fR/4,
$
then the proof of Proposition \ref{prop:stability-switching-filled-exterior} gives
$
\left|
\widetilde D_{\widetilde a_\alpha y}^{\langle\bT\rangle}
\right|
=
\left|
\widetilde G_{\widetilde a_\alpha y}^{\langle\bT\rangle}
\right|
\lesssim
\varphi+(\log N)^2\varepsilon^2.
$
By Proposition~\ref{prop:cell-sparsity} and the geometric properties
of $F_1(\cG)$, for every fixed $y\notin\bT$ there are at most
$\OO(\log N)$ indices for which the above conditions fail. Therefore,
\[
\sup_{y\notin\bT}
\sum_{\alpha=1}^{\mu}
\left|
\widetilde D_{\widetilde a_\alpha y}^{\langle\bT\rangle}
\right|
\lesssim
\mu\bigl(\varphi+(\log N)^2\varepsilon^2\bigr)
+(\log N)^2\varepsilon
\lesssim
(\log N)^2\varepsilon .
\]

Similarly, if $\alpha\ne\beta$ and at least one of
$\alpha,\beta$ belongs to $\mathsf U$, the separation property and
\eqref{eq:filled-analogue-551-553} give
$
\left|
\widetilde D_{\widetilde a_\alpha\widetilde a_\beta}^{\langle\bT\rangle}
\right|
\lesssim
\varphi+(\log N)^2\varepsilon^2.
$
Since
$
|\mathsf W_{\mathbf S}\setminus\mathsf U|\lesssim\log N,
$
only $\OO((\log N)^2)$ pairs require the crude bound. Thus
\[
\sum_{\substack{\alpha,\beta=1\\\alpha\ne\beta}}^\mu
\left|
\widetilde D_{\widetilde a_\alpha\widetilde a_\beta}^{\langle\bT\rangle}
\right|
\lesssim
\mu^2\bigl(\varphi+(\log N)^2\varepsilon^2\bigr)
+(\log N)^3\varepsilon
\lesssim
(\log N)^3\varepsilon .
\]
\end{proof}

\begin{proof}[Proof of Proposition \ref{prop:woodbury-stability-glued-T}]
Since the graphs associated with $\widetilde P$ and $P$ have the same structural property, namely, each contains at most one cycle, the same argument as in \eqref{eq:Ffrak-row-bound-rewritten} yields
\begin{equation}
\label{eq:F-boundary-row-bound-gluing}
\|\widetilde F_\partial\|_{\infty}=\max_{1\le x\le \mu}\sum_{y=1}^{\mu}|(\widetilde F_\partial)_{xy}|
\lesssim
\log_4\log N.
\end{equation}
As in the previous proof, it is enough to control the error contributed by the leading term.
Writing $W=I+\widehat W$, where $\widehat W$ is supported on the boundary indices
$\{\tilde a_\alpha\}_{\alpha=1}^{\mu}$, we have
\[
\begin{aligned}
\bigl(
W\widetilde D^{\langle\bT\rangle} W^{\mathsf T}
\bigr)_{ij}
=
\widetilde D^{\langle\bT\rangle}_{ij}
+
\sum_{\alpha=1}^{\mu}
\widehat W_{i\tilde a_\alpha}
\widetilde D^{\langle\bT\rangle}_{\tilde a_\alpha j}
+
\sum_{\beta=1}^{\mu}
\widetilde D^{\langle\bT\rangle}_{i\tilde a_\beta}
\widehat W_{j\tilde a_\beta}
+
\sum_{\alpha,\beta=1}^{\mu}
\widehat W_{i\tilde a_\alpha}
\widetilde D^{\langle\bT\rangle}_{\tilde a_\alpha\tilde a_\beta}
\widehat W_{j\tilde a_\beta}.
\end{aligned}
\]

Again we first study the effective error bound of the error matrix $\widetilde D^{\langle\bT\rangle}_\bG$. Proposition \ref{prop:stability-switching-filled-exterior} gives 
\[
  |\widetilde G^{\langle\bT\rangle}_{ij}-\widetilde G_{ij}(\Ext(\cB_r(\{i,j\},\widetilde \cG^{\langle\bT\rangle},Q(z)),z)|\lesssim(\log N)\varepsilon.
\]
By the same argument in the proof of Proposition \ref{prop:woodbury-stability-filled-T}, we can regard $\widetilde D^{\langle\bT\rangle}|_\bG$ as satisfying the entrywise bound $\big|\widetilde D^{\langle\bT\rangle}|_\bG\big|\lesssim(\log N)\varepsilon$.

Suppose first that \(i,j\in\bT\). 
Since the identity part of $W$ does not contribute on either side,
\[
\bigl(
W\widetilde D^{\langle\bT\rangle} W^{\mathsf T}
\bigr)_{ij}
=
\sum_{\alpha,\beta=1}^{\mu}
W_{i\tilde a_\alpha}
\widetilde D^{\langle\bT\rangle}_{\tilde a_\alpha\tilde a_\beta}
W_{j\tilde a_\beta}.
\]
Separating the diagonal and off-diagonal parts and applying
$|\widetilde D^{\langle\bT\rangle}|_\bG|\le (\log N)\varepsilon$ entrywise and Lemma \ref{claim:boundary sum}, we obtain
\begin{align}\label{eq:TT-zero-term-gluing}
\left|
\bigl(
W\widetilde D^{\langle\bT\rangle} W^{\mathsf T}
\bigr)_{ij}
\right|
\lesssim
(\log N)\varepsilon
\sum_{\alpha=1}^{\mu}
|W_{i\tilde a_\alpha}W_{j\tilde a_\alpha}|+
(\log N)^3\varepsilon
\max_{\alpha,\beta}
|W_{i\tilde a_\alpha}W_{j\tilde a_\beta}|.
\end{align}

In the remainder of the proof, we shall repeatedly use the elliptic-neighborhood estimates for the
transfer factors \(W\). Indeed, when $x\in\bT$,
\eqref{eq:transfer-entry-gluing} gives
\[
W_{x\widetilde a_\alpha}
=
\delta_{x\widetilde a_\alpha}
-
\widetilde P_{xl_\alpha}
-
m_{\widetilde a_\alpha l_\alpha}
\widetilde P_{x\widetilde a_\alpha}=-
\widetilde P_{xl_\alpha}
-
m_{\widetilde a_\alpha l_\alpha}
\widetilde P_{x\widetilde a_\alpha}.
\]
The vertices \(l_\alpha\) and \(\widetilde a_\alpha\) are adjacent in the
corresponding local comparison graph, which is a tree extension with at
most one cycle, and
$
|m_{\widetilde a_\alpha l_\alpha}|\lesssim 1.
$
Thus, up to the harmless diagonal term,
\(W_{x\widetilde a_\alpha}\) is a uniformly bounded linear combination
of two Green function entries on the same local tree extension.
Consequently, we can apply
Proposition \ref{prop:elliptic-neighborhood-estimates} and Lemma \ref{p:subgraph} to $W$, which implies $\sum_\alpha
|W_{i\tilde a_\alpha}W_{j\tilde a_\alpha}|
\lesssim
(1+|\ell_i-\ell_j|)2^{-|\ell_i-\ell_j|}$. Moreover $\max_{\alpha,\beta}
|W_{i\tilde a_\alpha}W_{j\tilde a_\beta}|\lesssim 2^{-(2\ell-\ell_i-\ell_j)}$ follows directly from the definition \eqref{eq:elliptic-level}.
This proves \eqref{eq:glued-stability-TT}.

Next, let \(i\in\bT\) and \(j\notin\bT\). In this case the identity part
of $W$ contributes only on the $j$-side, and hence
\[
\begin{aligned}
\bigl(
W\widetilde D^{\langle\bT\rangle} W^{\mathsf T}
\bigr)_{ij}
=
\sum_{\alpha=1}^{\mu}
W_{i\tilde a_\alpha}
\widetilde D^{\langle\bT\rangle}_{\tilde a_\alpha j}
+
\sum_{\alpha,\beta=1}^{\mu}
W_{i\tilde a_\alpha}
\widetilde D^{\langle\bT\rangle}_{\tilde a_\alpha\tilde a_\beta}
\widehat W_{j\tilde a_\beta}
=
\sum_{\alpha=1}^{\mu}
\sum_{y\notin\bT}
W_{i\tilde a_\alpha}
\widetilde D^{\langle\bT\rangle}_{\tilde a_\alpha y}
W_{jy}.
\end{aligned}
\]
We first claim 
  \begin{equation}\label{eq:blackbox-exterior-transfer-gluing}
\sup_{x\notin\bT}
\sum_{\alpha=1}^{\mu}
|W_{x\tilde a_\alpha}|
\lesssim \log N.
\end{equation}
Indeed, by Proposition \ref{pro.path-weight-decay}, Proposition \ref{prop:elliptic-neighborhood-estimates} and Lemma \ref{p:subgraph}, we always have
\begin{equation}
\sup_{x\notin\bT}
\sum_{\alpha=1}^{\mu}
|W_{x\tilde a_\alpha}|
\lesssim \sup_{x\in\partial \bT}
\sum_{\alpha=1}^{\mu}
|W_{x\tilde a_\alpha}|\lesssim\ell\lesssim\log N,
\end{equation}
which implies 
$
|W_j|_1
\lesssim \log N.
$
Then using Lemma \ref{claim:boundary sum}, the leading term
satisfies
\begin{align*}
\left|
\bigl(
W\widetilde D^{\langle\bT\rangle} W^{\mathsf T}
\bigr)_{ij}
\right|=\left|\sum_{\alpha=1}^\mu\sum_{y\notin\bT}W_{i\tilde a_\alpha}\widetilde D^{\langle\bT\rangle}_{\tilde a_\alpha y}W_{jy}\right|
\lesssim
\max_\alpha|W_{i\tilde a_\alpha}|
\sup_{y\notin\bT}
\sum_{\alpha=1}^{\mu}
\left|
\widetilde D^{\langle\bT\rangle}_{\tilde a_\alpha y}
\right|
|W_j|_1
\lesssim
(\log N)^3\varepsilon
2^{-(\ell-\ell_i)}.
\end{align*}

Finally, suppose that \(i,j\notin\bT\). In this case the identity part
of $W$ contributes on both sides, and
\[
\begin{aligned}
\bigl(
W\widetilde D^{\langle\bT\rangle} W^{\mathsf T}
\bigr)_{ij}
=
\widetilde D^{\langle\bT\rangle}_{ij}
+
\sum_{\alpha=1}^{\mu}
\widehat W_{i\tilde a_\alpha}
\widetilde D^{\langle\bT\rangle}_{\tilde a_\alpha j}
+
\sum_{\beta=1}^{\mu}
\widetilde D^{\langle\bT\rangle}_{i\tilde a_\beta}
\widehat W_{j\tilde a_\beta}
+
\sum_{\alpha,\beta=1}^{\mu}
\widehat W_{i\tilde a_\alpha}
\widetilde D^{\langle\bT\rangle}_{\tilde a_\alpha\tilde a_\beta}
\widehat W_{j\tilde a_\beta}
=
\sum_{x,y\notin\bT}
W_{ix}\widetilde D^{\langle\bT\rangle}_{xy}W_{jy}.
\end{aligned}
\]
It follows from
\eqref{eq:blackbox-exterior-transfer-gluing} that
$
|W_i|_1+|W_j|_1\lesssim\log N.
$
The triangle inequality implies
\[
\left|
\bigl(
W\widetilde D^{\langle\bT\rangle} W^{\mathsf T}
\bigr)_{ij}
\right|\lesssim(\log N)\varepsilon|W_i|_1|W_j|_1\lesssim (\log N)^3\varepsilon.
\]
\end{proof}

\subsection{Changing from \(Q\) to \(\widetilde Q\)}
\label{sec:changing-Q}
In this subsection, we  compare the self-consistent quantities before and after the
switching. Recall that the switching replaces the two edges
$
\{l_\alpha,a_\alpha\}$ and
$
\{b_\alpha,c_\alpha\}$
by
$
\{l_\alpha,c_\alpha\}
$
and
$
\{a_\alpha,b_\alpha\},
$
and denote the resulting graph by $\widetilde{\mathcal G}$. Throughout
this subsection, quantities with a tilde refer to the switched graph
$\widetilde{\mathcal G}$.

Although the switching changes the resolvent locally, its effect becomes
much smaller after averaging over each fiber. The following lemma,
based on the Ward identity, quantifies this observation and shows that
the self-consistent quantities associated with $\mathcal G$ and
$\widetilde{\mathcal G}$ differ only by a negligible error.

\begin{lemma}[Direct switching estimate]
\label{lem:direct-switching-Q}
Fix $z\in\bC_+(\mathcal I_\delta)$ with $|z|\le 2d_{\rm max}$ and $\eta\ge (\log N)^\fb/N$ and let $\cG\in\Omega(z)$ or $\Omega^+_o(z)$. Then for every $\bFS\in F_2(\cG)$, defined in Proposition \ref{prop:cell-sparsity}, the following three estimates hold. For every \(i,j\in\bG\), we have
\begin{equation}
\label{e:direct-switching-Q}
\left|
(\widetilde G-G)_{ij}
\right|
\lesssim
4^\ell
\sum_{x\in\bW_S'}
\left(
|G_{ix}|^2
+
|G_{jx}|^2
\right).
\end{equation}
where $\bW_S'
:=
\bigl\{
l_\alpha,a_\alpha,b_\alpha,c_\alpha:
\alpha\in\mathsf W_S
\bigr\}.$
Consequently, we have
\begin{equation}\label{e:bold-Q-Qtilde}
|
\widetilde{Q}^{\rm ver}
- Q^{\rm ver}
|_{\infty}
+|\widetilde p-p|_\infty+|
\widetilde{Q}
- Q
|_{\infty}\lesssim(\log N)^2\varphi^2.
\end{equation}
\end{lemma}

\begin{proof}
  Define 
$
U
:=
\bigl(
e_{l_\alpha}-e_{b_\alpha}
\bigr)_{\alpha\in\mathsf W_S}
$
and
$
V
:=
\bigl(
e_{a_\alpha}-e_{c_\alpha}
\bigr)_{\alpha\in\mathsf W_S}.
$
For the above switching, Lemma \ref{lem:F-boundary-factorization} implies that
\begin{equation}
\widetilde G-G=G\begin{pmatrix}
V&U
\end{pmatrix}
F_\partial
\begin{pmatrix}
V&U
\end{pmatrix}^{\mathsf T}G,\qquad 
\mbox{where~}F_\partial:
=
\begin{pmatrix}
U^{\mathsf T}\widetilde G U
&
I+U^{\mathsf T}\widetilde G V
\\[1mm]
I+V^{\mathsf T}\widetilde G U
&
V^{\mathsf T}\widetilde G V
\end{pmatrix}.
\end{equation}
It follows from $|\widetilde P|\lesssim 1$ and Proposition \ref{prop:woodbury-stability-glued-T} that $|\widetilde{G}|\lesssim 1$. Therefore, we can obtain $\max_{i,j}|(F_\partial)_{ij}|\lesssim 1$.
Every column of \(U\) and \(V\) is the difference of two coordinate
vectors supported on \(\bW_S'\).  Therefore, we can obtain
\begin{align*}
\left|
(\widetilde G-G)_{ij}
\right|
&\lesssim
\Big(
\sum_{x\in\bW_S'}|G_{ix}|
\Big)
\Big(
\sum_{y\in\bW_S'}|G_{yj}|
\Big)
\lesssim
|\bW_S'|
\Big(
\sum_{x\in\bW_S'}|G_{ix}|^2
\Big)^{1/2}
\Big(
\sum_{y\in\bW_S'}|G_{jy}|^2
\Big)^{1/2}.
\end{align*}
Since \(|\bW_S'|\lesssim|\sf W_S|\lesssim 4^\ell\), the elementary inequality
\(ab\leq (a^2+b^2)/2\) yields \eqref{e:direct-switching-Q}.

Next we analyze $|\widetilde Q-Q|_\infty$. 
Fix \(i\in\bH\).
Summing over the vertices of type \(i\) and using \eqref{e:direct-switching-Q}, we obtain
\begin{align}
&
\sum_{u:\pi(u)=i}
\left|
\widetilde G_{uu}-G_{uu}
\right|
\lesssim
4^\ell
\sum_{x\in\bW_S'}
\sum_{u:\pi(u)=i}
|G_{ux}|^2\le 4^\ell\sum_{x\in\bW_S'}\sum_{u\in\bG}|G_{ux}|^2=4^\ell\sum_{x\in\bW_S'}\frac{\Im G_{xx}}{\eta}.
\label{e:sum-Q-before-Ward}
\end{align}
where in the last step we used the Ward identity \eqref{e:Ward}.
Then, since $|\bW_S'|\lesssim 4^\ell$ and $\cG\in\Omega(z)$ or $\Omega_o^+(z)$, we obtain
\begin{align}
\label{e:sum-Q-after-Ward}
|\widetilde Q_i-Q_i|
&\le
\frac1N
\sum_{u:\pi(u)=i}
\left|
\widetilde G_{uu}-G_{uu}
\right|
\lesssim
\frac{4^{2\ell}}{N\eta}
\max_{x\in\bW_S'}
\Im G_{xx}
\lesssim
\frac{4^{2\ell}}{N\eta}
\left(
|\Im m^{\rm ver}|_\infty
+
\varepsilon_*
+
(\log N)^3\varepsilon
\right)
\lesssim
(\log N)^2\varphi^2,
\end{align}
where the last two inequalities follow from \eqref{eq.imG_xx} and \eqref{e:Omega-o-plus-G}.

The estimate for $\widetilde p-p$ follows in exactly the same way.
Indeed, fix a directed edge $e=(i,j)\in\vec E_\cH$.
Notice that $\widetilde p_e$ is defined using the $e$-neighbor of $u$
after the switching, which may differ from $u_e$. Since the switching
changes the $e$-neighbor of at most $\OO(\mu)$ vertices, and the Green
function entries are uniformly bounded on the bounded spectral region
under consideration, replacing the switched $e$-neighbors by the
original ones introduces an error $\OO(\mu)$. Thus, by
\eqref{e:direct-switching-Q},
\begin{align*}
N|\widetilde p_e-p_e|
\le
\sum_{\substack{u\in\bG\\ \pi(u)=i}}
\left|
\widetilde G_{u u_e}-G_{u u_e}
\right|
+\OO(\mu)
\lesssim
4^\ell
\sum_{x\in\bW_S'}
\sum_{\substack{u\in\bG\\ \pi(u)=i}}
\left(
|G_{ux}|^2+|G_{u_e x}|^2
\right)
+\OO(\mu).
\end{align*}
After dividing by $N$, the additional error is $\OO(\mu/N)$,
which is absorbed into $(\log N)^{2}\varphi^2$ by the choice of
parameters.
Then, the Ward identity gives
\[
|\widetilde p_e-p_e|
\lesssim
\frac{4^{2\ell}}{N\eta}
\max_{x\in\bW_S'}\Im G_{xx}
\lesssim
(\log N)^2\varphi^2.
\]
Hence
$
|\widetilde Q^{\rm ver}-Q^{\rm ver}|_\infty
+
|\widetilde p-p|_\infty
\lesssim
(\log N)^2\varphi^2.
$

Finally, by the definition
$
Q_e=-p_e/Q_{t(e)}
$
and
$
\widetilde Q_e=-\widetilde p_e/\widetilde Q_{t(e)},
$
and the a priori bounds
$
|Q_i|\sim |\widetilde Q_i|\sim 1$
and
$
|p_e|\lesssim 1,
$
we obtain
$
|\widetilde Q_e-Q_e|
\lesssim
|\widetilde p_e-p_e|
+
|\widetilde Q_{t(e)}-Q_{t(e)}|
\lesssim
(\log N)^2\varphi^2.
$
Taking the maximum over $e\in\vec E_\cH$ proves
\eqref{e:bold-Q-Qtilde}.

\end{proof}

\subsection{Proof of Proposition \ref{p:Omega+}}\label{sec.5.7}
Proposition~\ref{prop:stability-switching-filled-exterior} implies \eqref{e:Omega-o-plus-filled}.
On the other hand, Proposition~\ref{prop:woodbury-stability-glued-T},
together with Lemma~\ref{p:localization}, yields uniformly in
$i,j\in\bG$,
\[
\left|
\widetilde G_{ij}
-
G_{ij}\left(
\Ext\bigl(
\cB_r(\{i,j\},\widetilde{\cG}),Q
\bigr),z
\right)
\right|
\lesssim
(\log N)^3\varepsilon.
\]

It therefore remains only to replace the original empirical weight
$Q$ by the switched weight $\widetilde Q$. 
Lemma~\ref{lem:direct-switching-Q} gives
$
|\widetilde Q-Q|_\infty
\lesssim
(\log N)^2\varphi^2,
$
which implies $|\widetilde Q-m|_\infty\lesssim|Q-m|_\infty+(\log N\varphi)^2\lesssim\varepsilon_*$.
Moreover, by Lemma~\ref{p:subgraph} and the choice of parameters,
replacing $Q$ by $\widetilde Q$ in the local comparison graph
produces an error negligible compared with the corresponding
right-hand sides above. Consequently,
\begin{align*}
\left|
\widetilde G^{\langle\bT\rangle}_{ij}
-
G_{ij}\left(
\Ext\bigl(
\cB_r(\{i,j\},
\widetilde{\cG}^{\langle\bT\rangle}),\widetilde Q
\bigr),z
\right)
\right|
\lesssim
(\log N)\varepsilon,
~
\left|
\widetilde G_{ij}
-
G_{ij}\left(
\Ext\bigl(
\cB_r(\{i,j\},\widetilde{\cG}),\widetilde Q
\bigr),z
\right)
\right|
\lesssim
(\log N)^3\varepsilon.
\end{align*}

Finally, since $\cG\in\Omega(z)$,
$
|R(Q,z)|_\infty\leq\varepsilon
$.
The Dyson residual is uniformly Lipschitz in a fixed neighborhood of
the solution, and therefore
$
|R(\widetilde Q,z)|_\infty
\lesssim
|R(Q,z)|_\infty+|\widetilde Q-Q|_\infty
\lesssim
\varepsilon.
$
Thus all the defining estimates of
$\Omega_o^+(z)$ hold for $\widetilde{\cG}$.
Since Proposition~\ref{prop:cell-sparsity} gives
$
\P_\cG(F_2(\cG))
=
1-\OO(N^{-\fd}),
$
the conclusion of Proposition~\ref{p:Omega+} follows.

\section{Proof of Proposition \ref{p:Omega-}: Improved Local Approximation}
\label{sec:concentration}

In this section, we prove Proposition~\ref{p:Omega-}. We first establish concentration for the diagonal errors produced by the individual filled switchings and identify their mean. 
We then insert these estimates into the gluing expansion for $\bT$ to obtain the improved local approximation, and finally use the three-term identity on the comparison tree to derive the corresponding root estimate.

\subsection{Concentration in the switched graph}
We now prove the concentration estimate needed for the switched graph.
For a fixed vertex type $i\in \bH$, let
\[
\mathsf A_i
:=
\bigl\{
\alpha\in\qq{1,\mu}:
\pi(\widetilde a_\alpha)=i
\bigr\}.
\]
For every $\alpha$, let
$\mathcal G_{[\alpha]}^{\langle\bT\rangle}$
be the graph obtained from
$\mathcal G^{\langle\bT\rangle}$
by performing only the $\alpha$-th filled switching when
$I_\alpha=1$:
\[
\{a_\alpha,l_\alpha^\star\},
\ \{b_\alpha,c_\alpha\}
\quad\longmapsto\quad
\{a_\alpha,b_\alpha\},
\ \{c_\alpha,l_\alpha^\star\}.
\]
When $I_\alpha=0$, we set
$
\mathcal G_{[\alpha]}^{\langle\bT\rangle}
=
\mathcal G^{\langle\bT\rangle}.
$
In this definition, we use only the single-sample condition
$I_\alpha$, not the joint isolation condition $J_\alpha$.
Let $G_{[\alpha]}^{\langle\bT\rangle}$
denote the corresponding Green function, and define
\[
P_{[\alpha],c_\alpha c_\alpha}^{\langle\bT\rangle}
:=
G_{c_\alpha c_\alpha}
\left(
\Ext\big(
\cB_r\big(
\{a_\alpha, b_\alpha,c_\alpha\},
\mathcal G_{[\alpha]}^{\langle\bT\rangle}
\big),
Q
\big),z
\right).
\]
We also define
\begin{equation}
\label{e:single-switching-Z-alpha}
Z_\alpha
:=
G_{[\alpha],c_\alpha c_\alpha}^{\langle\bT\rangle}
-
P_{[\alpha],c_\alpha c_\alpha}^{\langle\bT\rangle},
\qquad
\widetilde Z_\alpha
:=
\widetilde G^{\langle\bT\rangle}_{
\widetilde a_\alpha\widetilde a_\alpha}
-
G_{\widetilde a_\alpha\widetilde a_\alpha}
\left(
\Ext\left(
\cB_r\left(
\{a_\alpha, b_\alpha, c_\alpha\},
\widetilde{\mathcal G}^{\langle\bT\rangle}
\right),
Q
\right),z
\right).
\end{equation}

\begin{proposition}[Concentration for a fixed vertex type]
\label{prop:lift-concentration}
Let $z\in\mathbb C^+$ with $|z|\le 2d_{\rm max}$ and $\eta\ge(\log N)^\fb/N$ and
assume that $\mathcal G\in\Omega_o^+(z)$.
Conditioned on the underlying graph and all randomness prior to the local sampling, let $\{w_\alpha\}_{\alpha\in\mathsf A_i}$ be deterministic weights satisfying
$
\sum_{\alpha\in\mathsf A_i}|w_\alpha|\lesssim 1.
$
Then there exists an event
$
F^i_3(\mathcal G)\subset F_2(\mathcal G)
$
such that
$
\mathbb P_{\mathcal G}\bigl(F^i_3(\mathcal G)\bigr)
=
1-\mathcal \OO(N^{-\fd}),
$
and, for every switching configuration
$\mathbf S\in F^i_3(\mathcal G)$,
\begin{equation}
\label{e:lift-concentration}
\left|
\sum_{\alpha\in\mathsf A_i}
w_\alpha\widetilde Z_\alpha
\right|
\lesssim
(\log N)^2\varepsilon
\left(
\sum_{\alpha\in\mathsf A_i}|w_\alpha|^2
\right)^{1/2}
+(\log N)|R(Q,z)|_\infty+(\log N\epsilon)^2.
\end{equation}
Define
$
F_3(\cG):=\bigcap_{i\in\bH}F_3^i(\cG).
$
Then, on the event \(F_3(\cG)\), the above concentration inequality holds simultaneously for all types \(i\in\bH\).

\end{proposition}

\begin{proof}
  Write
\[
|w|_1=\sum_{\alpha\in\mathsf A_i}|w_\alpha|,
\qquad
|w|_2=
\left(
\sum_{\alpha\in\mathsf A_i}|w_\alpha|^2
\right)^{1/2},
\qquad
|w|_\infty=\max_{\alpha\in\mathsf A_i}|w_\alpha|.
\]
We divide the proof into three steps.\\
\medskip
\noindent
\textit{Step 1. Concentration before  switching.}
Conditioned on the underlying graph and the fixed boundary data,
the samples
$\vec S_\alpha=(b_\alpha,c_\alpha)$ are independent.
The graph
$\mathcal G_{[\alpha]}^{\langle\bT\rangle}$
and its local comparison are functions of $\vec S_\alpha$ alone. 
Consequently, the random variables
$\{Z_\alpha\}_{\alpha\in\mathsf A_i}$
are conditionally independent.

We claim $|Z_\alpha|\lesssim(\log N)\epsilon$.
The claim is immediate when \(I_\alpha=0\). Otherwise, we repeat the procedure from Section
\ref{subsec:stability-switching-filled-exterior}. Define the localized
comparison Green's function
\[
P^{\langle\bT\rangle}
:=
G\left(
\Ext\left(
\cB_r\left(
\{a_\alpha,b_\alpha,c_\alpha\},
\mathcal G^{\langle\bT\rangle}
\right),
Q
\right),z
\right).
\]
For simplicity, we use the same notation
\(P^{\langle\bT\rangle}\) as in \eqref{def.P<T>}, since in both cases
it denotes a localized comparison Green's function associated with
\(\mathcal G^{\langle\bT\rangle}\). Here, however, in order to preserve
the required independence property, we take only the $\alpha$-cell
\(\{a_\alpha,b_\alpha,c_\alpha\}\) as the core of the localization,
rather than the full collection \(\bW_S\) introduced in
\eqref{def.bW_S}. The dependence on $\alpha$ is understood throughout.

Set
$
D^{\langle\bT\rangle}
:=
G^{\langle\bT\rangle}-P^{\langle\bT\rangle},
$
and
$
W_\alpha
:=
\bigl(I+P^{\langle\bT\rangle}\Xi_\alpha\bigr)^{-1}
=
I-P_{[\alpha]}^{\langle\bT\rangle}\Xi_\alpha,
$
where \(\Xi_\alpha\) denotes the finite-rank inverse-resolvent
perturbation associated with the single $\alpha$-switching.
Since Proposition \ref{p:tree-extension-bound} and Lemma \ref{p:subgraph} imply $|P_{[\alpha]}^{\langle\bT\rangle}|\lesssim 1$ entrywise, together with the fixed support of \(\Xi_\alpha\), we have
\(\|W_\alpha\|_{\infty}\lesssim1\).
Since \(\cG\in\Omega_o^+(z)\), we also have
$
\|W_\alpha D^{\langle\bT\rangle}\Xi_\alpha\|_{\infty}
\lesssim(\log N)\varepsilon=o(1).
$
Thus the resolvent identity yields
\begin{equation}\label{eq.GTa}
G_{[\alpha]}^{\langle\bT\rangle}
-
P_{[\alpha]}^{\langle\bT\rangle}
=
\bigl(I+W_\alpha D\Xi_\alpha\bigr)^{-1}
W_\alpha D W_\alpha^{\mathsf T},
\end{equation}
where the first factor is bounded by a convergent Neumann series.
Taking the entrywise maximum and using
\(\|W_\alpha^{\mathsf T}\|_{1}
=\|W_\alpha\|_{\infty}\lesssim1\),
we obtain
$
\bigl|
G_{[\alpha]}^{\langle\bT\rangle}
-
P_{[\alpha]}^{\langle\bT\rangle}
\bigr|
\lesssim
|D^{\langle\bT\rangle}|
\lesssim(\log N)\varepsilon
$ entrywise,
which implies \(|Z_\alpha|\lesssim(\log N)\varepsilon\).

Applying Hoeffding's inequality to the real and imaginary parts,
we obtain, for every $t>0$,
\begin{align}
\label{e:azuma-lift}
\mathbb P_{\mathcal G}
\left(
\left|
\sum_{\alpha\in\mathsf A_i}
w_\alpha(Z_\alpha-\E_\cG Z_\alpha)
\right|
\ge
t(\log N)\varepsilon
|w|_2
\right)
\le
4e^{-ct^2}.
\end{align}
Taking $t=\log N$, we conclude that, outside an event of probability
$\OO(e^{-c(\log N)^2})$,
\begin{equation}
\label{e:single-switching-concentration}
\left|
\sum_{\alpha\in\mathsf A_i}
w_\alpha(Z_\alpha-\E_\cG Z_\alpha)
\right|
\lesssim
(\log N)^2\varepsilon
|w|_2.
\end{equation}
We leave the conditional means $\E_\cG Z_\alpha$ unchanged.

\medskip
\noindent
\textit{Step 2. Comparison with the fully switched graph.}

Fix $\mathbf S\in F_2(\mathcal G)$ and
$\alpha\in\mathsf U$.
Then $\widetilde a_\alpha=c_\alpha$.
Recall that
$\mathcal G_{[\alpha]}^{\langle\bT\rangle}$
is obtained by performing only the $\alpha$-th switching, whereas
$\widetilde{\mathcal G}^{\langle\bT\rangle}$
contains all switchings in $\mathsf W_S$.
Hence
$\mathcal G_{[\alpha]}^{\langle\bT\rangle}$
is obtained from
$\widetilde{\mathcal G}^{\langle\bT\rangle}$
by undoing the switchings
$\beta\in\mathsf W_S\setminus\{\alpha\}$.

By exactly the same Woodbury argument as in the proof of
Proposition~\ref{prop:stability-switching-filled-exterior},
applied in the reverse direction, and using the uniform
boundedness of the corresponding switching kernel, we have
\begin{equation}
\label{e:single-full-diagonal-comparison}
\left|
\widetilde G^{\langle\bT\rangle}_{c_\alpha c_\alpha}
-
G_{[\alpha],c_\alpha c_\alpha}^{\langle\bT\rangle}
\right|
\lesssim
\sum_{\beta\in\mathsf W_S\setminus\{\alpha\}}
\sum_{x\in\{a_\beta,b_\beta,c_\beta\}}
\left|
\widetilde G^{\langle\bT\rangle}_{c_\alpha x}
\right|^2.
\end{equation}
Here the coordinates $l_\beta^\star$ have been eliminated by the
exact covering-tree relation
$
\widetilde G^{\langle\bT\rangle}_{c_\alpha l_\beta^\star}
=
-m_{a_\beta l_\beta}
\widetilde G^{\langle\bT\rangle}_{c_\alpha c_\beta}.
$

Since $\alpha\in\mathsf U$, the separation conditions for distinct
switching cells and
Proposition~\ref{prop:stability-switching-filled-exterior} imply
\[
\left|
\widetilde G^{\langle\bT\rangle}_{c_\alpha x}
\right|
\lesssim
\varphi+(\log N)^2\varepsilon^2,
\qquad
x\in\{a_\beta,b_\beta,c_\beta\},
\quad
\beta\ne\alpha.
\]
This yields
\begin{equation}
\label{e:single-full-diagonal-comparison-final}
\left|
\widetilde G^{\langle\bT\rangle}_{c_\alpha c_\alpha}
-
G_{[\alpha],c_\alpha c_\alpha}^{\langle\bT\rangle}
\right|
\lesssim
\mu
\bigl(
\varphi+(\log N)^2\varepsilon^2
\bigr)^2\lesssim 4^{\ell}
\bigl(
\varphi+(\log N)^2\varepsilon^2
\bigr)^2.
\end{equation}
Moreover, the switchings with $\beta\ne\alpha$ do not change the
radius-$r$ weighted neighborhood of $c_\alpha$. Hence the two
localized comparison Green functions coincide, and therefore
\begin{equation}
\label{e:single-full-Z-comparison}
|\widetilde Z_\alpha-Z_\alpha|
\lesssim
\mu
\bigl(
\varphi+(\log N)^2\varepsilon^2
\bigr)^2,
\qquad
\alpha\in\mathsf U.
\end{equation}
For $\alpha\notin\mathsf U$, we simply use the local approximation
to obtain
$
|\widetilde Z_\alpha-Z_\alpha|
\lesssim(\log N)\varepsilon.$
On $F_2(\mathcal G)$, we have
$|\qq{1,\mu}\setminus\mathsf U|\lesssim\log N$.
Therefore the results in Step 2 imply
\begin{align}
\label{e:single-full-weighted-comparison}
\left|
\sum_{\alpha\in\mathsf A_i}
w_\alpha(\widetilde Z_\alpha-Z_\alpha)
\right|
\lesssim{}&
\mu\bigl(\varphi+(\log N\varepsilon)^2\bigr)^2|w|_1
+
(\log N)^2\varepsilon|w|_\infty.
\end{align}

The term involving $\varphi^2$ is bounded by
$4^\ell\varphi^2$, since
$\mu\lesssim4^\ell$ and $|w|_1\lesssim1$.
For the remaining term, Cauchy--Schwarz and
$\mu(\log N\varepsilon)^2\ll 1$ give
$
\mu(\log N\varepsilon)^4|w|_1
\lesssim
\log N \varepsilon|w|_2.
$
Using also $|w|_\infty\le|w|_2$, we conclude that
\begin{equation}\label{eq.swithcing-difference}
\left|
\sum_{\alpha\in\mathsf A_i}
w_\alpha(\widetilde Z_\alpha-Z_\alpha)
\right|
\lesssim
4^\ell\varphi^2
+
(\log N)^2\varepsilon|w|_2.
\end{equation}

\medskip
\noindent
\textit{Step 3. Estimating the conditional mean.}
It remains to estimate the mean $\sum_{\alpha\in\mathsf A_i}w_\alpha\E_\cG Z_\alpha$. 
Define the
deterministic set of good boundary indices by
$
\mathsf A_i^{\rm good}
:=
\left\{
\alpha\in\mathsf A_i:
\cB_{\fR/4}(a_\alpha,\cG^{(\bT)})
\text{ is a tree}
\right\}.
$
We first establish the following conditional-mean estimate for each
good boundary index
\begin{claim}\label{lem:single-switching-concentration-mean}
  Under the same setting as Proposition~\ref{prop:lift-concentration}, for every $\alpha\in\mathsf A_i^{\rm good}$, we have
\begin{align}
\label{e:single-switching-mean-bound}
|\E_\cG Z_\alpha|
\lesssim
(\log N)|R(Q,z)|_\infty
+
(\log N\varepsilon)^2.
\end{align}
\end{claim}

On the other hand, for the deterministic exceptional boundary indices, the uniform
bound $|Z_\alpha|\lesssim(\log N)\varepsilon$ gives
\[
\left|
\sum_{\alpha\in\mathsf A_i\setminus\mathsf A_i^{\rm good}}
w_\alpha\E_\cG Z_\alpha
\right|
\lesssim
(\log N)\varepsilon
\sum_{\alpha\in\mathsf A_i\setminus\mathsf A_i^{\rm good}}
|w_\alpha|.
\]
Since $\cG\in\Omega_o^+(z)\subset\bar\Omega^+$, Definition \ref{def:barOmega1} gives
$
\operatorname{excess}\bigl(\cB_{\fR/2}(o,\cG)\bigr)
\le \fd.
$
Applying Proposition \ref{prop:boundary-tree-neighborhood}
with its radius parameter set to $R=\fR/2$, we obtain
$
|\mathsf A_i\setminus\mathsf A_i^{\rm good}|
\le 3\fd\lesssim 1.
$ 
Hence
\[
(\log N)\varepsilon
\sum_{\alpha\in\mathsf A_i\setminus\mathsf A_i^{\rm good}}
|w_\alpha|
\lesssim
(\log N)\varepsilon|w|_\infty.
\]
This application is the reason for imposing the
radius-$\fR/2$ excess bound in the definition of $\bar\Omega^+$.
Combining this with Claim~\ref{lem:single-switching-concentration-mean}
and $|w|_1\lesssim1$, we obtain
\[
\left|
\sum_{\alpha\in\mathsf A_i}w_\alpha\E_\cG Z_\alpha
\right|
\lesssim
(\log N)|R(Q,z)|_\infty
+
(\log N\varepsilon)^2
+
(\log N)\varepsilon|w|_\infty.
\]
Combining this estimate with
\eqref{e:single-switching-concentration} and \eqref{eq.swithcing-difference}
proves \eqref{e:lift-concentration}.

Finally, define $F^i_3(\mathcal G)$ as the intersection of
$F_2(\mathcal G)$ with the event in
\eqref{e:single-switching-concentration}.
The probability estimate for $F_2(\mathcal G)$ and
\eqref{e:azuma-lift} imply
$
\mathbb P_{\mathcal G}\bigl(F^i_3(\mathcal G)\bigr)
=
1-\OO(N^{-\fd}).
$
Consequently, $
\mathbb P_{\mathcal G}\bigl(F_3(\mathcal G)\bigr)
=
1-\OO(N^{-\fd}).
$
\end{proof}

\begin{proof}[Proof of Claim \ref{lem:single-switching-concentration-mean}]
Fix a deterministic index
$\alpha\in\mathsf A_i^{\rm good}$, and let
$e=(j,i)$ be the type of $(b_\alpha,c_\alpha)$.
Thus
$\pi(a_\alpha)=\pi(c_\alpha)=i$ and
$\pi(l_\alpha)=\pi(b_\alpha)=j$.
All expectations and probabilities below are conditional on
$\cG$ and the fixed boundary data.
Write
$
\vec E_e^{\rm out}
:=
\left\{
(x,y)\in\vec E_\cG:
x,y\notin\bT, (\pi(x),\pi(y))=e
\right\}
$ and $N_i^{\rm out}:=|\vec E^{\rm out}_e|$.
The sampled edge $(b_\alpha,c_\alpha)$ is uniform
on $\vec E_e^{\,\rm out}$, and
$
N_i^{\rm out}=N-\OO(|\bT|)\sim N.
$
Recall that $\bA:=\{a_\beta:1\leq\beta\leq\mu\}$ and introduce
the single-sample event
\[
\mathcal E_\alpha
:=
\left\{
I_\alpha=1,\quad
\dist_{\cG^{(\bT)}}
\bigl(\{b_\alpha,c_\alpha\},\bA\bigr)>\fR/4
\right\}.
\]

We first estimate the probability of its complement for this
fixed index.
Since
$\alpha\in A_i^{\mathrm{good}}$, the condition
$I_\alpha=1$ holds whenever $c_\alpha$ has a
radius-$\fR$ tree neighborhood in $\mathcal G$ and
$
\operatorname{dist}_{\mathcal G^{(\bT)}}
(a_\alpha,\{b_\alpha,c_\alpha\})>\fR/2+1.
$
Indeed, the induced neighborhoods
$B_{\fR/4}(a_\alpha,\mathcal G^{(\bT)})$ and
$B_{\fR/4}(\{b_\alpha,c_\alpha\},\mathcal G^{(\bT)})$
are then trees. Moreover, the distance condition
ensures that these two neighborhoods are disjoint
and that there is no edge between them. Adding
the edge $\{a_\alpha,b_\alpha\}$ therefore joins
the two trees into a tree. In particular, the
sufficient condition holds if
$
\operatorname{dist}_{\mathcal G^{(\bT)}}
(a_\alpha,\{b_\alpha,c_\alpha\})\ge\fR/2+2.
$
The geometric assumptions and the degree bound therefore give
\[
\P_\cG(I_\alpha=0)
\lesssim
\frac{
N^\fc+(d_{\rm max}-1)^{\fR/2+2}
}{N_i^{\rm out}},\qquad \P_\cG\left(
\dist_{\cG^{(\bT)}}
\bigl(\{b_\alpha,c_\alpha\},\bA\bigr)
\leq\fR/4
\right)
\lesssim
\frac{
\mu(d_{\rm max}-1)^{\fR/4}
}{N_i^{\rm out}}.
\]
Consequently, using
$\fR=(\fc/4)\log N$ and
$\mu\lesssim4^\ell=(\log N)^C$, we obtain
\begin{equation}
\label{e:single-switching-good-sample-probability}
\P_\cG(\mathcal E_\alpha^c)
\lesssim N^{-1+\fc}.
\end{equation}
On $\mathcal E_\alpha$, the canonical comparison neighborhood
at $c_\alpha$ contains no other fixed boundary vertex.
Thus, after the $\alpha$-switching, its only exact filled branch
is the $m_{\bar e}$-branch attached at $c_\alpha$.

We first estimate
\(P_{[\alpha],c_\alpha c_\alpha}^{\langle\bT\rangle}\) on $\mathcal E_\alpha$.
Let
\(\widehat P_{[\alpha],c_\alpha c_\alpha}^{\langle\bT\rangle}\)
denote the root Green function of the same canonical comparison
tree, except that the exact \(m_{\bar e}\)-branch attached to
\(c_\alpha\) is temporarily replaced by a \(Q_{\bar e}\)-branch.
Then all exterior branches carry the weight \(Q\), and
Lemma~\ref{p:fixpoint} gives
\begin{align}
\left|
\widehat P_{[\alpha],c_\alpha c_\alpha}^{\langle\bT\rangle}
-
\frac{Q_{e}}{1-Q_eQ_{\bar e}}
\right|
\lesssim\log N
\big|
R(Q,z)
\big|_\infty.
\label{e:single-switching-all-Q-fixed-point}
\end{align}
Replacing the \(Q_{\bar e}\)-branch at \(c_\alpha\) by the exact
\(m_{\bar e}\)-branch is a one-dimensional filling operation. The Schur complement formula implies
\begin{equation}
P_{[\alpha],c_\alpha c_\alpha}^{\langle\bT\rangle}
=
\frac{
\widehat P_{[\alpha],c_\alpha c_\alpha}^{\langle\bT\rangle}
}{
1+(Q_{\bar e}-m_{\bar e})
\widehat P_{[\alpha],c_\alpha c_\alpha}^{\langle\bT\rangle}
}.
\end{equation}
On the other hand, it follows that
\[
\frac{
Q_{e}/(1-Q_eQ_{\bar e})
}{
1+
(Q_{\bar e}-m_{\bar e})
Q_{e}/(1-Q_eQ_{\bar e})
}
=
\frac{Q_{e}}{1-m_{\bar e}Q_{e}}.
\]
It follows from \(\cG\in\Omega_o^+(z)\) that \(|Q-m|_\infty=o(1)\). Hence, \eqref{eq.1-memvarm} gives
$
|1-m_{\bar e}Q_{e}|
\sim
|1-m_em_{\bar e}|
\sim1.
$
Thus the corresponding filling denominator is uniformly stable. Therefore, \eqref{e:single-switching-all-Q-fixed-point} implies
\begin{align}
\left|
P_{[\alpha],c_\alpha c_\alpha}^{\langle\bT\rangle}
-
\frac{Q_{e}}{1-m_{\bar e}Q_{e}}
\right|
\lesssim
\log N
\big|
R(Q,z)
\big|_\infty.
\label{e:single-switching-filled-fixed-point}
\end{align}

We now use \eqref{eq.GTa} from Step~1.
Since
$\|W_\alpha D^{\langle\bT\rangle}\Xi_\alpha\|_\infty
\lesssim(\log N)\varepsilon=o(1)$
on the single-switching support, expanding its first factor
to first order yields
\[
Z_\alpha
=
(W_\alpha)_{c_\alpha}
D^{\langle\bT\rangle}
(W_\alpha)_{c_\alpha}^{\mathsf T}
+
\OO\left((\log N\varepsilon)^2\right).
\]
We next identify the \(c_\alpha\)-row of the transfer
factor. After the \(\alpha\)-switching, the \(c_\alpha\)-component
of the comparison graph is separated from
\(a_\alpha\) and \(b_\alpha\), and hence
$
P_{[\alpha],c_\alpha a_\alpha}^{\langle\bT\rangle}
=
P_{[\alpha],c_\alpha b_\alpha}^{\langle\bT\rangle}
=
0.
$
Using the same transfer-factor calculation as in
\eqref{eq:transfer-entry-rewritten}, together with the fact that the
universal-cover branch attached at $c_\alpha$ contributes the exact
self-energy $m_{\bar e}$, we obtain
\begin{equation}
\label{e:single-switching-transfer-row}
(W_{\alpha})_{c_\alpha}
=
\left(
1+
m_{\bar e}
P_{[\alpha],c_\alpha c_\alpha}^{\langle\bT\rangle}
\right)e_{c_\alpha}^{\mathsf T}
+
P_{[\alpha],c_\alpha c_\alpha}^{\langle\bT\rangle}
e_{b_\alpha}^{\mathsf T}.
\end{equation}
By the definition of the transfer factor in
Lemma~\ref{lem:woodbury-error-expansion}, we have the exact
intertwining identity
$
W_{\alpha}P^{\langle\bT\rangle}
=
P_{[\alpha]}^{\langle\bT\rangle}.
$
Combining this and \eqref{e:single-switching-transfer-row}, it follows from 
$
P_{[\alpha],c_\alpha b_\alpha}^{\langle\bT\rangle}=0
$
that
\begin{align}
(W_{\alpha})_{c_\alpha}
P^{\langle\bT\rangle}
(W_{\alpha})_{c_\alpha}^{\mathsf T}
=
e_{c_\alpha}^{\mathsf T}
P_{[\alpha]}^{\langle\bT\rangle}
(W_{\alpha})_{c_\alpha}^{\mathsf T}
=
P_{[\alpha],c_\alpha c_\alpha}^{\langle\bT\rangle}
\left(
1+
m_{\bar e}
P_{[\alpha],c_\alpha c_\alpha}^{\langle\bT\rangle}
\right).
\label{e:single-switching-reference-identity}
\end{align}
Substituting
\eqref{e:single-switching-filled-fixed-point}
before taking the expectation, we obtain
\begin{align}
Z_\alpha
=
\frac{
Q_e^2G_{b_\alpha b_\alpha}^{\langle\bT\rangle}
+
2Q_eG_{b_\alpha c_\alpha}^{\langle\bT\rangle}
+
G_{c_\alpha c_\alpha}^{\langle\bT\rangle}
-
Q_e
}{
(1-m_{\bar e}Q_e)^2
}
+
\OO\left(
(\log N)|R(Q,z)|_\infty
+
(\log N\varepsilon)^2
\right).
\label{e:single-switching-scalar-expansion}
\end{align}

By the same Woodbury argument as in the proof of
Lemma~\ref{lem:direct-switching-Q},
we have
$
\left|
G^{\langle\bT\rangle}_{xx}-G_{xx}
\right|
\lesssim
4^\ell
\sum_{y\in\bA\cup\bL}|G_{xy}|^2,
$
where $\bA=\{a_\alpha\}_{\alpha=1}^\mu$ and $\bL=\{l_\alpha\}_{\alpha=1}^\mu$.
Then summing over the vertices of type $i$ and applying the Ward identity \eqref{e:Ward}, we obtain
\begin{equation*}
\label{e:filled-Qver-comparison}
\frac1N
\sum_{\substack{x\notin\bT\\ \pi(x)=i}}
\left|
G^{\langle\bT\rangle}_{xx}-G_{xx}
\right|
\lesssim(\log N)^2\varphi^2.
\end{equation*}
Moreover,
$
N-N_i^{\rm out}
\leq |\bT|\lesssim 4^{\ell}\lesssim(\log N)^{\fa}
$.
Since the diagonal Green functions are uniformly bounded on
$\Omega_o^+(z)$, the vertices of type $i$ contained in $\bT$, as well as
the change of normalization from $N_i^{\rm out}$ to $N$, contribute only
$\OO((\log N)^{\fa}/N)$. Therefore,
\begin{equation*}
\label{e:average-Gcc-Qver}
\left|
\mathbb E_{\mathcal G}
G^{\langle\bT\rangle}_{c_\alpha c_\alpha}
-
Q_i
\right|
\lesssim
(\log N)^2\varphi^2
+
\frac{(\log N)^{\fa}}{N}\lesssim(\log N)^2\varphi^2.
\end{equation*}
Using the same argument, it follows that
\begin{equation}
\label{e:single-switching-sampled-block}
\E_\cG
\begin{pmatrix}
G_{c_\alpha c_\alpha}^{\langle\bT\rangle}
&
G_{c_\alpha b_\alpha}^{\langle\bT\rangle}\vspace{1mm}
\\
G_{b_\alpha c_\alpha}^{\langle\bT\rangle}
&
G_{b_\alpha b_\alpha}^{\langle\bT\rangle}
\end{pmatrix}
=
\begin{pmatrix}
Q_i & p_e\\
p_e & Q_j
\end{pmatrix}
+
\OO\left(
(\log N)^2\varphi^2
\right).
\end{equation}

The coefficients in
\eqref{e:single-switching-scalar-expansion}
are deterministic under $\E_\cG$.
Using the sampled-block estimate
\eqref{e:single-switching-sampled-block},
and removing the indicator of $\mathcal E_\alpha$ at a cost
$\OO(N^{-1+\fc})$, we obtain
\begin{align}
\E_\cG Z_\alpha
=
\frac{
Q_e^2Q_j+2Q_ep_e+Q_i-Q_e
}{
(1-m_{\bar e}Q_e)^2
}
+
\OO\left(
(\log N)|R(Q,z)|_\infty
+
(\log N\varepsilon)^2
+
(\log N)^2\varphi^2
+
N^{-1+\fc}
\right).
\label{e:single-switching-mean-computation}
\end{align}
The removal of the indicator is justified by the uniform
Green-function bounds and the bound on $Z_\alpha$ from Step~1 in the proof of Proposition \ref{prop:lift-concentration}.

Finally, the identities
$
p_e=-Q_jQ_e=-Q_iQ_{\bar e}$
and
$
R_{\bar e}(Q,z)
=
-Q_{\bar e}^{-1}+Q_j^{-1}+Q_e
$
give
\[
Q_e^2Q_j+2Q_ep_e+Q_i-Q_e
=
p_eR_{\bar e}(Q,z).
\]
Since $|p_e|\lesssim1$ and
$|1-m_{\bar e}Q_e|\sim 1$, it follows that
\[
|\E_\cG Z_\alpha|
\lesssim
(\log N)|R(Q,z)|_\infty
+
(\log N\varepsilon)^2
+
(\log N)^2\varphi^2
+
N^{-1+\fc}\lesssim (\log N)|R(Q,z)|_\infty+(\log N \varepsilon)^2.
\]
\end{proof}

\subsection{Proof of Proposition \ref{p:Omega-}}
Define $\widetilde P^{\langle\bT\rangle}
:=
G\bigl(
\operatorname{Ext}
(
\widetilde{\mathcal G}^{\langle\bT\rangle}_o,Q(z)
),
z
\bigr)$ with $
\widetilde{\mathcal G}^{\langle\bT\rangle}_o
:=
\mathcal B_{r-1}
\bigl([\{o,i\}\setminus\bT]\cup\widetilde{\bA},
\widetilde{\mathcal G}^{\langle\bT\rangle}
\bigr).
$

\textit{Proof of \eqref{eq.Omega-minus-RQ}}: This follows directly from Lemma \ref{lem:direct-switching-Q}, together with the same argument as in Section \ref{sec.5.7}.

\medskip

\textit{Proof of \eqref{e:Omega-o-minus-Q} with \(i\in\bT\)}:
It is sufficient to bound the leading term in the
Woodbury expansion in Lemma \ref{lem:compressed-woodbury-gluing}
\begin{align}
\label{eq:improved-leading-term-decomposition}
\left(
W\widetilde D^{\langle\bT\rangle}W^{\mathsf T}
\right)_{oi}
=
\sum_{\alpha=1}^{\mu}
W_{o\widetilde a_\alpha}
W_{i\widetilde a_\alpha}
\widetilde D^{\langle\bT\rangle}_{
\widetilde a_\alpha\widetilde a_\alpha}
+
\sum_{\substack{\alpha,\beta=1\\\alpha\neq\beta}}^{\mu}
W_{o\widetilde a_\alpha}
\widetilde D^{\langle\bT\rangle}_{
\widetilde a_\alpha\widetilde a_\beta}
W_{i\widetilde a_\beta}.
\end{align}
We first estimate the diagonal term. For every \(\alpha\), we can write
\begin{align}
\label{eq:split-D-diagonal-local}
\widetilde D^{\langle\bT\rangle}_{
\widetilde a_\alpha\widetilde a_\alpha}
=
\widetilde G^{\langle\bT\rangle}_{
\widetilde a_\alpha\widetilde a_\alpha}
&-
G_{\widetilde a_\alpha\widetilde a_\alpha}
\left(
\operatorname{Ext}\left(
\mathcal B_r(
\{a_\alpha, b_\alpha, c_\alpha\},
\widetilde{\mathcal G}^{\langle\bT\rangle}
),
Q
\right),z
\right)
+
G_{\widetilde a_\alpha\widetilde a_\alpha}
\left(
\operatorname{Ext}\left(
\mathcal B_r(
\{a_\alpha, b_\alpha, c_\alpha\},
\widetilde{\mathcal G}^{\langle\bT\rangle}
),
Q
\right),z
\right)\nonumber\\&-G_{\widetilde a_\alpha\widetilde a_\alpha}
\left(
\operatorname{Ext}\left(
\mathcal B_r(
\widetilde a_\alpha,
\widetilde{\mathcal G}^{\langle\bT\rangle}
),
Q
\right),z
\right)
+
G_{\widetilde a_\alpha\widetilde a_\alpha}
\left(
\operatorname{Ext}\left(
\mathcal B_r(
\widetilde a_\alpha,
\widetilde{\mathcal G}^{\langle\bT\rangle}
),
Q
\right),z
\right)
-
\widetilde P^{\langle\bT\rangle}_{
\widetilde a_\alpha\widetilde a_\alpha}.
\end{align}
The localized graph contains at most one cycle and
has diameter \(\OO(\log N)\). Hence Lemma
\ref{p:localization} and the triangle inequality give
\begin{align}
\label{eq:localization-boundary-diagonal}
|\widetilde D^{\langle\bT\rangle}_{
\widetilde a_\alpha\widetilde a_\alpha}
|\lesssim
\Big|\widetilde G^{\langle\bT\rangle}_{
\widetilde a_\alpha\widetilde a_\alpha}
-
G_{\widetilde a_\alpha\widetilde a_\alpha}
\big(
\operatorname{Ext}\big(
\mathcal B_r(
\{a_\alpha, b_\alpha, c_\alpha\},
\widetilde{\mathcal G}^{\langle\bT\rangle}
),
Q
\big),z
\big)\Big|+
\log N
\big|
R(Q,z)
\big|_\infty
+
\beta^{2r}.
\end{align}

For each type \(j\in \bH\), we apply Proposition
\ref{prop:lift-concentration} with
$
w_\alpha
=
W_{o\widetilde a_\alpha}
W_{i\widetilde a_\alpha}
$
for $\alpha\in\mathsf A_j.$
These weights are deterministic conditionally on the randomness
prior to the local sampling. Indeed, by
Proposition~\ref{prop:elliptic-reverse}(ii), for
$i\in\bT$ the comparison core is contained in
$\cB_{\fR/8}(o,\cG)$, and the type-preserving isomorphism associated
with the switching fixes $\bT$ and maps each original boundary
vertex to its switched counterpart. Hence the relevant entries of
$P$ and $W$ are identified with entries determined before the local
sampling, and are therefore fixed under the conditional randomness.
Since the weights involved in the following argument form a family of at most $|\bH||\bT|=(\log N)^{\OO(1)}$ deterministic vectors fixed before the local sampling, applying Proposition \ref{prop:lift-concentration} with $\fd$ replaced by $\fd+1$ and taking a union bound yields simultaneous concentration for all these weights with failure probability $\OO(|\bH||\bT|N^{-\fd-1})=\OO(N^{-\fd})$.
Proposition \ref{prop:elliptic-neighborhood-estimates} and
Lemma \ref{p:subgraph} give
\begin{equation}
\label{eq:W-root-i-l1}
\sum_{\alpha=1}^{\mu}
\left|
W_{o\widetilde a_\alpha}
W_{i\widetilde a_\alpha}
\right|
\lesssim
\frac{1+\ell_i}{2^{\ell_i}}
\lesssim 1.
\end{equation}
Moreover, Definition \ref{def:elliptic-neighborhood} and Proposition \ref{prop:elliptic-neighborhood-estimates} also imply
$
\max_{\alpha}
\left|
W_{o\widetilde a_\alpha}
\right|
\lesssim
2^{-\ell},
$
and
$
\sum_{\alpha=1}^{\mu}
\left|
W_{i\widetilde a_\alpha}
\right|^2
\lesssim 1.
$
Consequently, we can obtain
\begin{align}
\label{eq:W-root-i-l2}
\sum_{\alpha=1}^{\mu}
\left|
W_{o\widetilde a_\alpha}
W_{i\widetilde a_\alpha}
\right|^2
&\leq
\max_{\alpha}
\left|
W_{o\widetilde a_\alpha}
\right|^2
\sum_{\alpha=1}^{\mu}
\left|
W_{i\widetilde a_\alpha}
\right|^2
\lesssim
4^{-\ell}.
\end{align}

Thus, using Proposition \ref{prop:lift-concentration} and summing the concentration estimate over the finitely many
vertex types \(j\in\bH\), we obtain
\begin{align}
\label{eq:concentration-diagonal-W}
\left|
\sum_{\alpha=1}^{\mu}
W_{o\widetilde a_\alpha}
W_{i\widetilde a_\alpha}
\left[
\widetilde G^{\langle\bT\rangle}_{
\widetilde a_\alpha\widetilde a_\alpha}
-
G_{\widetilde a_\alpha\widetilde a_\alpha}
\left(
\operatorname{Ext}\left(
\mathcal B_r(
\{a_\alpha,b_\alpha,c_\alpha\},
\widetilde{\mathcal G}^{\langle\bT\rangle}
),
Q
\right),z
\right)
\right]
\right|
\lesssim
\frac{(\log N)^2\varepsilon}{2^\ell}
+\log N
\big|
R(Q,z)
\big|_\infty,
\end{align}
where the error terms $4^\ell\varphi^2$ and $(\log N\epsilon)^2$ are absorbed into $\OO(2^{-\ell}(\log N)^2\varepsilon)$.
Combining \eqref{eq:split-D-diagonal-local},
\eqref{eq:localization-boundary-diagonal},
\eqref{eq:W-root-i-l1} and
\eqref{eq:concentration-diagonal-W}, we conclude that
\begin{align}
\label{eq:improved-diagonal-leading}
\left|
\sum_{\alpha=1}^{\mu}
W_{o\widetilde a_\alpha}
W_{i\widetilde a_\alpha}
\widetilde D^{\langle\bT\rangle}_{
\widetilde a_\alpha\widetilde a_\alpha}
\right|
\lesssim
\frac{(\log N)^3\varepsilon}{2^\ell}
+\log N
\big|
R(Q,z)
\big|_\infty.
\end{align}
The off-diagonal contribution is treated in exactly the same way as in proof of Proposition \ref{prop:woodbury-stability-glued-T}. Indeed, Lemma
\ref{claim:boundary sum} and the elliptic decay estimate give
\begin{align}
\label{eq:improved-offdiagonal-leading}
&
\left|
\sum_{\substack{\alpha,\beta=1\\\alpha\neq\beta}}^{\mu}
W_{o\widetilde a_\alpha}
\widetilde D^{\langle\bT\rangle}_{
\widetilde a_\alpha\widetilde a_\beta}
W_{i\widetilde a_\beta}
\right|
\lesssim
\max_{\alpha,\beta}
\left|
W_{o\widetilde a_\alpha}
W_{i\widetilde a_\beta}
\right|
\sum_{\alpha\neq\beta}
\left|
\widetilde D^{\langle\bT\rangle}_{
\widetilde a_\alpha\widetilde a_\beta}
\right|
\lesssim
\frac{(\log N)^3\varepsilon}
     {2^{2\ell-\ell_i}}\lesssim\frac{(\log N)^3\varepsilon}{2^\ell}.
\end{align}
Combining \eqref{eq:improved-diagonal-leading} and 
\eqref{eq:improved-offdiagonal-leading} proves
\begin{align*}
\left|
\widetilde G_{oi}
-
G_{oi}\left(
\operatorname{Ext}
\bigl(
\mathcal B_r(\{o,i\},\widetilde{\mathcal G}), Q(z)
\bigr),z
\right)
\right|
\lesssim
\frac{(\log N)^3\varepsilon(z)}{2^\ell}
+
\log N
\big|
R(Q,z)
\big|_\infty. 
\end{align*}
It remains to change $Q$ to $\widetilde Q$. Since Lemma \ref{lem:direct-switching-Q} yields
$
|\widetilde Q-Q|_\infty\lesssim(\log N)^2\varphi^2,
$
\eqref{e:Omega-o-minus-Q}  follows directly from
Remark \ref{r:subgraph} together with
$
(\log N)^2\varphi^2\lesssim\varepsilon^2\lesssim2^{-\ell}(\log N)^3\epsilon(z).
$

\medskip
\textit{Proof of \eqref{e:Omega-o-minus-Q} with \(i\notin \bT\)}:
No concentration estimate is needed in this step. 
Indeed, set the first index in \eqref{eq:glued-stability-T-out} equal to $o$ and the second equal to the
exterior vertex $i$. 
Then, applying once again the localization principle like \eqref{eq:localization-boundary-diagonal}, followed by the same replacement of the boundary weight \(Q\) by \(\widetilde Q\), we directly obtain
\eqref{e:Omega-o-minus-Q}.

\medskip

\textit{Proof of \eqref{e:Omega-o-minus-root}}: 
The root $o$ has the radius-$\fR$ tree neighborhood required in \eqref{e:Omega-o-minus-root}.
We revisit the proof of Proposition \ref{prop:woodbury-stability-glued-T} and adopt the notation used there.
Separating the diagonal boundary terms gives
\begin{align}
\label{e:three-term-leading-expansion}
\widetilde G_{xy}-\widetilde P_{xy}
={}&
\sum_{\alpha=1}^{\mu}
W_{x\widetilde a_\alpha}
W_{y\widetilde a_\alpha}
\widetilde D^{\langle\bT\rangle}_{
\widetilde a_\alpha\widetilde a_\alpha}
+
\OO\left(
\frac{(\log N)^3\varepsilon}{4^\ell}
\right),\qquad\mbox{where}~x,y\in\{o,v\}.
\end{align}
The key point is that the leading boundary contribution in
\eqref{e:three-term-leading-expansion} satisfies an exact cancellation on the comparison tree, which follows from the following elementary identity.

\begin{claim}[The three-term identity on a tree]
\label{lem:three-term-tree}
Let $P=G(\Ext(\cG',Q),z)$ where $\cG'$ is a finite tree, and let $u\sim v$ be an edge of this tree. Then
\begin{equation}
\label{e:three-term-tree}
P_{uv}^2-P_{uv}-P_{uu}P_{vv}=0.
\end{equation}
For every vertex $x$ of the same tree, we also have
\begin{equation}
\label{e:three-term-tree-derivative}
(2P_{uv}-1)P_{ux}P_{vx}
-
P_{vv}P_{ux}^2
-
P_{uu}P_{vx}^2
=
0.
\end{equation}
\end{claim}

For \(x\in\bT\), we have
\(\delta_{x\widetilde a_\alpha}=0\), since
\(\widetilde a_\alpha\notin\bT\).
Hence the identity contribution in the gluing matrix \(W\) vanishes
on these entries. Using the formula for \(W\) together with the tree
factorization across \(\{l_\alpha,\widetilde a_\alpha\}\), we obtain
$ W_{x\widetilde a_\alpha} = \gamma_\alpha\,\widetilde P_{xl_\alpha}$ for every $x\in\bT$ and some scalar \(\gamma_\alpha\) independent of \(x\).
Consequently, \eqref{e:three-term-tree-derivative} gives,
for every $\alpha$,
\begin{align}
\label{e:three-term-boundary-cancellation}
0={}&
(2\widetilde P_{ov}-1)
W_{o\widetilde a_\alpha}
W_{v\widetilde a_\alpha}
-
\widetilde P_{vv}W_{o\widetilde a_\alpha}^2
-
\widetilde P_{oo}W_{v\widetilde a_\alpha}^2.
\end{align}

Write
$
\widetilde D_{xy}
:=
\widetilde G_{xy}-\widetilde P_{xy}.
$
Expanding the three-term expression around $\widetilde P$, and using
\eqref{e:three-term-tree} for the zeroth-order term, we obtain
\begin{align}\label{eq.three-term-G}
\widetilde G_{ov}^{\,2}
-\widetilde G_{ov}
-\widetilde G_{oo}\widetilde G_{vv}
=
(2\widetilde P_{ov}-1)\widetilde D_{ov}
-
\widetilde P_{vv}\widetilde D_{oo}
-
\widetilde P_{oo}\widetilde D_{vv}
+
\widetilde D_{ov}^{\,2}
-
\widetilde D_{oo}\widetilde D_{vv}.
\end{align}
Substituting \eqref{e:three-term-leading-expansion} into the linear
part, we obtain
\[
\begin{aligned}
&
(2\widetilde P_{ov}-1)\widetilde D_{ov}
-
\widetilde P_{vv}\widetilde D_{oo}
-
\widetilde P_{oo}\widetilde D_{vv}
\\
={}&
\sum_{\alpha}
\Big[
(2\widetilde P_{ov}-1)
W_{o\widetilde a_\alpha}
W_{v\widetilde a_\alpha}
-
\widetilde P_{vv}
W_{o\widetilde a_\alpha}^{2}
-
\widetilde P_{oo}
W_{v\widetilde a_\alpha}^{2}
\Big]
\widetilde D^{\langle\bT\rangle}_{
\widetilde a_\alpha\widetilde a_\alpha}
+
\OO\left(
\frac{(\log N)^3\varepsilon}{4^\ell}
\right).
\end{aligned}
\]
For every $\alpha$, the coefficient in square brackets vanishes by
\eqref{e:three-term-boundary-cancellation}. Hence the entire
diagonal boundary contribution cancels, and therefore
\begin{equation}\label{eq.three-term-P}
(2\widetilde P_{ov}-1)\widetilde D_{ov}
-
\widetilde P_{vv}\widetilde D_{oo}
-
\widetilde P_{oo}\widetilde D_{vv}
=
\OO\left(
\frac{(\log N)^3\varepsilon}{4^\ell}
\right).
\end{equation}

Since we have $
\max_{x,y\in\{o,v\}}
|\widetilde D_{xy}|
\lesssim
(\log N)\varepsilon$ 
in the proof of Proposition \ref{prop:woodbury-stability-glued-T},
the 
quadratic contribution
$
\widetilde D_{ov}^{\,2}
-
\widetilde D_{oo}\widetilde D_{vv}
$ satisfies
$
|
\widetilde D_{ov}^{\,2}
-
\widetilde D_{oo}\widetilde D_{vv}
|
\lesssim
(\log N)^2\varepsilon^2\lesssim 4^{-\ell}(\log N)^3\epsilon.
$
Combining this with \eqref{eq.three-term-G} and \eqref{eq.three-term-P} proves \eqref{e:Omega-o-minus-root}.
We conclude the proof.

\begin{proof}[Proof of Claim \ref{lem:three-term-tree}]
Both identities are direct consequences of the Schur complement
formula and the tree structure. Indeed, applying the Schur complement
to the block indexed by $\{u,v\}$, the effective inverse Green matrix
has the form
\[
\begin{pmatrix}
a&1\\
1&b
\end{pmatrix},
\]
which immediately gives \eqref{e:three-term-tree}. Likewise, applying
the Schur complement to the block indexed by $\{u,v,x\}$, the tree
structure implies that, according to which side of the edge
$\{u,v\}$ contains $x$, the effective inverse Green matrix has the
form
\[
\begin{pmatrix}
a&1&c\\
1&b&0\\
c&0&d
\end{pmatrix}
\qquad\text{or}\qquad
\begin{pmatrix}
a&1&0\\
1&b&c\\
0&c&d
\end{pmatrix}.
\]
A direct computation of the corresponding Green block yields
\eqref{e:three-term-tree-derivative}.
\end{proof}

\appendix

\section{Properties of the Green's functions}
\label{app:Green}

Throughout this paper, we repeatedly use some (well-known) identities for Green's functions,
which we collect in this appendix.

\subsection{Resolvent identity}

The following well-known identity is referred to the resolvent identity:
for two invertible matrices $A$ and $B$ of the same size, we have
\begin{equation} \label{e:resolv}
  A^{-1} - B^{-1} = A^{-1}(B-A)B^{-1}=B^{-1}(B-A)A^{-1}.
\end{equation}

\subsection{Schur complement formula}

Given an $N\times N$ matrix $M$ and an index set $\T \subset \bG$, recall that we denote by
$M|_\T$ the $\T \times \T$-matrix obtained by restricting $M$ to $\T$,
and that by $M^{(\T)} = M|_{\bT^c}$, where $\bT^c=\bG\setminus \T$, we denote the matrix obtained by removing
the rows and columns corresponding to indices in $\T$.

Thus, for any $\T \subset \bG$,
any real symmetric matrix $H$ can be written (up to rearrangement of indices) in the block form
\begin{equation*}
  H = \begin{bmatrix} A& B'\\ B &D  \end{bmatrix},
\end{equation*}
with $A=H|_{\T}$ and $D=H^{(\T)}$.
The Schur complement formula asserts that, for any $z\in \bC^+$,
\begin{equation} \label{e:Schur}
 G=(H-z)^{-1}= \begin{bmatrix}
   (A-z-B'\GT B)^{-1} & -(A-z-B'\GT B)^{-1}B'\GT \\
   -\GT B(A-z-B'\GT B)^{-1} & \GT+\GT B(A-z-B'\GT B)^{-1}B'\GT 
 \end{bmatrix},
\end{equation}
where $\GT=(D-z)^{-1}$.
Throughout the paper, we often use the following special cases of \eqref{e:Schur}:
\begin{align} \begin{split}\label{e:Schur1}
  G|_{\T} = (A-z-B'\GT B)^{-1},\qquad
  G|_{\T\T^c}=-G|_{\T}B'\GT,
  \end{split}
\end{align}
as well as the special case
\begin{equation} \label{e:Schurixj}
G_{ij}^{(k)} = G_{ij}-\frac{G_{ik}G_{kj}}{G_{kk}}.
\end{equation}

\subsection{Ward identity}

For any real symmetric $N\times N$ matrix $A$, its Green's function $G(z)=(A-z)^{-1}$ satisfies
the Ward identity
\begin{equation} \label{e:Ward}
  \sum_{j=1}^{N} |G_{ij}(z)|^2= \frac{\im G_{ii}(z)}{\eta},
\end{equation}
where $\eta=\Im [z]$. This identity provides a bound for the sum $\sum_{j=1}^{N} |G_{ij}(z)|^2$
in terms of the diagonal entries of the Green's function.

\subsection{Covering map}
\begin{lemma}[Lemma B.1 in \cite{Bauer2019Local}]\label{l:covering}
Given a covering $\pi:\tcG \to \cG$ of graphs,
denote the Green's function of $\tcG$ by $\tG$ and that of $\cG$ by $G$. Then
for all vertices $i,j$ in $\cG$, the Green's functions obey
\begin{align}\label{defg-bis}
G_{ij}=\sum_{y:\pi(y)=j}\tilde{G}_{xy},
\end{align}
where $x$ is any vertex of $\tcG$ with $\pi(x)=i$, provided the right hand side of \eqref{defg-bis} converges.
\end{lemma}

\section{Proofs of lemmas in section \ref{sec.tree-like}}\label{a:ext}
\begin{proof}[Proof of Lemma \ref{lem:NB-Dyson-regularity}]
It remains only to prove that, for every $i\in\bH$ and $e,f\in\vec E_\cH$,
\begin{equation*}
|m_{i}(z)|\sim \frac{1}{1+|z|}\qquad\text{and}\qquad
\Im m_i(z)\sim \Im m_e(z),
\qquad z\in\bC_+(\mathcal I_\delta),
\end{equation*}
since all the remaining assertions follow directly from \cite[Theorems 2.5-2.7 and the proof of Theorem 2.8]{huang2026nonsymmetricvectordysonequations}. 
Therefore, throughout the proof below, we shall freely use the estimates
$
|m_e(z)|\sim (1+|z|)^{-1}
$
and
$
\Im m_e(z)\sim \Im m_f(z),
$
for all \(e,f\in\vec E_\cH\).

For $i\in\bH$, set
$
h_i(z):=z+\sum_{j\sim i}m_{ij}(z).
$
Then it follows that $m_i(z)=-(h_i(z))^{-1}$.
Fix a neighbor $x\sim i$. The Dyson equation for the oriented edge
$(x,i)$ gives
\[
-\frac{1}{m_{xi}}
=
z+\sum_{\substack{j\sim i\\j\neq x}}m_{ij}.
\]
Adding the missing term \(m_{ix}\), we obtain
\begin{equation}\label{e:hi-identity}
h_i
=
-\frac{1}{m_{xi}}+m_{ix},
\qquad
1-m_{xi}m_{ix}=-m_{xi}h_i.
\end{equation}

Fix an edge \(\{u,v\}\). We first claim that for every $u\sim v$,
\begin{equation}\label{eq.1-memvarm}
  |1-m_{uv}(z)m_{vu}(z)|\sim 1,\qquad z\in\bC_+(\mathcal I_\delta).
\end{equation}
The upper bound follows directly from $|m_e(z)|\sim (1+|z|)^{-1}$. For the
lower bound, it is enough to consider \(|z|\leq K\), since for
sufficiently large \(K\), 
$
|m_{uv}(z)m_{vu}(z)|\lesssim (1+|z|)^{-2}\leq 1/2.
$
On \(|z|\leq K\), $|m_e(z)|\sim (1+|z|)^{-1}$ implies
$
|m_e(z)|\sim 1
$
for every oriented edge \(e\).

It follows from \eqref{e:hi-identity} and $|m_e(z)|\sim 1$ when $|z|\le K$ that
$
|h_u|\sim|h_um_{vu}|=|1-m_{uv}m_{vu}|.
$
Moreover, if \(y\sim x\), then we have
$
|1-m_{yx}m_{xy}|=|m_{yx}h_x|\sim |h_x|,
$
and consequently
$
|h_y|
=
|1-m_{xy}m_{yx}|/|m_{xy}|
\sim |h_x|.
$
Since \(\cH\) is connected, iterating these estimates along paths gives
\begin{equation}\label{e:smallness-propagation}
|1-m_{xy}m_{yx}|\sim |h_x|\sim|h_u| \sim|1-m_{uv}m_{vu}|
\end{equation}
for every edge \(\{x,y\}\in E_{\cH}\).

Choose a fixed simple cycle
$
\mathscr{C}=(e_0\rightsquigarrow e_1\rightsquigarrow \ldots\rightsquigarrow e_r=e_0)
$
in \(\cH\), and define
$
q_+
:=
\prod_{k=1}^{r}m_{e_k}
$
and
$
q_-
:=
\prod_{k=1}^{r}m_{\bar{e}_k}.
$
Since the cycle has fixed length and $|m_e|$ is bounded, it follows from \eqref{e:smallness-propagation} that
\begin{equation}\label{eq.1-q+q-}
|1-q_+q_-|=\left|1-\prod_{k=1}^{r}m_{e_k}m_{\bar{e}_k}\right|\lesssim\sum_{k=1}^r|1-m_{e_k}m_{\bar{e}_k}|\lesssim |1-m_{uv}m_{vu}|.
\end{equation}
On the other hand, using the cycle-weight estimate proved independently in the argument for Proposition~\ref{pro.path-weight-decay}, which relies only on the bounds for $|m_e|$, repeating the cycle in either orientation yields
$
|q_+|\leq \beta^r
$
and
$
|q_-|\leq \beta^r.
$
Therefore, combining \eqref{eq.1-q+q-} and $\beta\in(0,1)$ yields that
\[
1-C|1-m_{uv}m_{vu}|
\leq |q_+q_-|
\leq \beta^{2r}<1,
\]
where $C$ is a finite positive constant depending only on the model parameters.
Hence
$
|1-m_{uv}m_{vu}|\gtrsim (1-\beta^{2r}),
$
which proves the claim.

Combining \eqref{e:hi-identity} with the definition of \(m_{i}\), we
have
$
m_{i}(z)
=
(m_{xi}(z))/
(1-m_{xi}(z)m_{ix}(z))$.
Thus \eqref{eq.1-memvarm} and $|m_e|\sim(1+|z|)^{-1}$ imply
$
|m_i(z)|
\sim |m_{xi}(z)|
\sim (1+|z|)^{-1}.
$

Taking imaginary parts in the definitions of $m_i$ and $m_{xi}$ gives
\[
\frac{\Im m_i}{|m_i|^2}
=
\eta+\sum_{j\sim i}\Im m_{ij},
\qquad
\frac{\Im m_{xi}}{|m_{xi}|^2}
=
\eta+\sum_{j\sim i,j\neq x}\Im m_{ij}.
\]
Since the minimum degree of $\cH$ is at least three,
$\Im m_e\sim\Im m_f$ for all oriented edges, and
$|m_i|\sim|m_{xi}|\sim(1+|z|)^{-1}$, the two right-hand sides are comparable. Hence
$
\Im m_i(z)\sim \Im m_{xi}(z).
$
Finally, using again the assumed comparability of the
imaginary parts, we conclude
$
\Im m_i(z)\sim\Im m_e(z)
$
for all
$
e\in\vec E_{\cH}.
$
\end{proof}

\begin{proof}[Proof of Proposition \ref{p:tree-extension-bound}]
Set
$
\rho_0(z):=\max_{i\in\bH}\Im m_i(z).
$
In the proof below, we will use the estimates established in Lemma \ref{lem:NB-Dyson-regularity} and Proposition \ref{pro.path-weight-decay}:
\begin{equation}\label{eq.result-used}
  |m_i(z)|\sim|m_e(z)|
\sim \frac{1}{1+|z|},
\quad
\Im m_i(z)\sim\Im m_e(z)\sim\rho_0(z),\quad
|\cM(\gamma;z)|\lesssim \beta^{|\gamma|},\quad z\in\bC_+(\mathcal I_\delta).
\end{equation}
By expanding the imaginary part of a product and using $|\Im (z_1z_2)|\le |z_1||\Im z_2|+|z_2||\Im z_1|$, it follows from \eqref{eq.result-used} that for any
path $\gamma$ of length $\ell$,
\begin{equation}\label{e:Im-path-weight}
\left|
\Im\bigl(m_i(z)\cM(\gamma;z)\bigr)
\right|
\lesssim
(\ell+1)\beta^\ell \rho_0(z).
\end{equation}

There are only two cases.

\medskip
\noindent
\emph{Case 1: $\cG$ is a tree.} Then
$\TE(\cG)$ is the universal cover $\cX$. By
Proposition~\ref{greentree},
\[
G_{ij}(\TE(\cG),z)
=
m_i(z)\cM_{\cX}(i,j;z).
\]
The required bounds follow immediately from
Proposition~\ref{pro.path-weight-decay},
\eqref{e:Im-path-weight} and
$|m_i(z)|\sim(1+|z|)^{-1}$.

\medskip
\noindent
\emph{Case 2: $\cG$ contains a cycle.}
Then $\TE(\cG)$ contains a unique cycle, which we denote by
$\mathscr C$.
Let
$
p:\cX\longrightarrow \TE(\cG)
$
be the universal cover map and choose $x\in\bX$ such that
$p(x)=i$. By Lemma \ref{l:covering},
\begin{equation}\label{e:covering-expansion-unicyclic}
G_{ij}\bigl(\TE(\cG),z\bigr)
=
\sum_{y:\,p(y)=j}G_{xy}(\cX,z)
=
m_i(z)
\sum_{y:\,p(y)=j}
\cM_{\cX}(x,y;z).
\end{equation}

The geodesic from $x$ to a point $y\in p^{-1}(j)$ projects onto a
non-backtracking walk from $i$ to $j$ in $\TE(\cG)$.
Since $\TE(\cG)$ has only one cycle, all such walks can be divided
into a uniformly bounded number of families, each of which is obtained
from a reduced path from $i$ to $j$ by making additional complete
traversals of $\mathscr C$ in one of its two orientations.

More precisely, for each such family there is a reduced path
$\gamma_0:i\leadsto j$ and a cycle weight $q_{\rm cyc}(z)$ such that the
corresponding path weights have the form
$
\cM(\gamma_0;z),
\cM(\gamma_0;z)q_{\rm cyc},
\cM(\gamma_0;z)q_{\rm cyc}^2,\ldots,
$
up to modifying $\gamma_0$ by one of finitely many initial reduced
paths. Proposition \ref{pro.path-weight-decay} implies
$
\sum_{n\geq0}|q_{\rm cyc}|^n
\leq
(1-\beta^|\mathscr C|)^{-1}
\leq
(1-\beta)^{-1}.
$
Since the number of geometric families is uniformly bounded, it
follows from \eqref{e:covering-expansion-unicyclic} and Proposition \ref{pro.path-weight-decay} that
\begin{equation}\label{e:unicyclic-path-max}
\left|
G_{ij}\bigl(\TE(\cG),z\bigr)
\right|
\lesssim
|m_i(z)|
\max_{\gamma:i\leadsto j}
|\cM(\gamma;z)|
\lesssim
\max_{\gamma:i\leadsto j}
|\cM(\gamma;z)|\lesssim \beta^{\dist_{\cG}(i,j)}.
\end{equation}
For the imaginary part, by \eqref{e:Im-path-weight}, each path
contribution satisfies
\[
\left|
\Im\left(
m_i(z)\cM_{\cX}(x,y;z)
\right)
\right|
\lesssim
\bigl(|\gamma_y|+1\bigr)
\beta^{|\gamma_y|}|\im m^{\rm ver}|_\infty,
\]
where $\gamma_y$ denotes the projection of the geodesic from $x$ to
$y$.
Summing over the number of traversals of the unique cycle and using
$
\sum_{\ell\geq0}(\ell+1)\beta^\ell<\infty,
$
we obtain
$
\left|
\Im G_{ij}\bigl(\TE(\cG),z\bigr)
\right|
\lesssim
|\im m^{\rm ver}|_\infty.
$
\end{proof}

\begin{proof}[Proof of Lemma \ref{p:subgraph}]
The proof of the lemma closely follows the proof of \cite[Proposition 2.12]{Huang2024Spectrum}. We therefore only highlight the main differences.
We suppress the dependence on $z$ and set
$
P
:=
G\bigl(\Ext(\cG,m),z\bigr)
=
G\bigl(\TE(\cG),z\bigr).
$
For $x\in\bG$, define
\[
q_x
:=
\sum_{e\in\mathscr M_\cG(x)}
(q_e-m_e)\qquad\mbox{and}\qquad Q:=\diag(q_x)_{x\in\bG}.
\]
By Definition \ref{def:ext} and the resolvent identity \eqref{e:resolv}, we can obtain
\begin{align}
\label{e:random-lift-neumann}
G\bigl(\Ext(\cG,q),z\bigr)
=
P+\sum_{k\ge1}(PQ)^kP.
\end{align}
We first estimate the leading order:
\begin{equation}\label{eq.PQPleading}
  |(PQP)_{ij}|=\big|\sum_{x\in\bG}P_{ix}q_xP_{xj}\big|\le |q-m|_\infty\sum_{x\in\bG}|\mathscr M_\cG(x)||P_{ix}||P_{xj}|.
\end{equation}
When $\cG$ is a tree, let $[i,j]$ be the unique path from
$i$ to $j$.  For each $x\in\bG$, let $a$ be the point where the path
from $x$ first meets $[i,j]$.  By multiplicativity of the path
weights and Proposition \ref{greentree}, we have
\[
P_{ix}P_{xj}=m_i\cM(i,x)m_{x}\cM(x,j)
=
m_im_{x}\cM(i,j)
\cM(a,x)\cM(x,a).
\]
Then for every fixed $a\in[i,j]$, the same argument as in the proof of \eqref{eq:elliptic-boundary-pair} shows the sum over the component attached to
$a$ satisfies
\[\sum_{\substack{x\in\bG\\
x\text{ projects to }a}}
|\mathscr M_\cG(x)|
\left|
P_{ix}P_{xj}
\right|
\lesssim
|\cM(i,j)|
\sum_{\substack{x\in\bG\\
x\text{ projects to }a}}
|\mathscr M_\cG(x)|
\left|
\cM(a,x)\cM(x,a)
\right|
\lesssim 1.
\]
Therefore, putting the above relation into \eqref{eq.PQPleading} yields
\begin{equation*}
  |(PQP)_{ij}|\le |q-m|_\infty\sum_{a\in[i,j]}\sum_{\substack{x\in\bG\\
x\text{ projects to }a}}|\mathscr M_\cG(x)||P_{ix}||P_{xj}|\lesssim (1+\dist(i,j))|q-m|_\infty|\cM(i,j)|.
\end{equation*}

If $\cG$ contains one cycle, let $\mathscr C$ be its unique cycle and
$L:=|\mathscr C|$. As in the proof of
Proposition \ref{p:tree-extension-bound}, after lifting to the
universal cover, the relevant paths can be divided into a uniformly
bounded number of families, each obtained from a simple path by
making $n\ge0$ additional traversals of $\mathscr C$ in one of its two
orientations. By Proposition \ref{pro.path-weight-decay}, each
additional traversal contributes a factor at most $\beta^L$.
Hence we have 
\[
\sum_{n\ge0}\beta^{nL}\lesssim1,
\qquad
\sum_{n\ge0}(1+nL)\beta^{nL}\lesssim1,
\]
uniformly in $L$.
Applying the preceding tree estimate to each such family
and summing over the winding number therefore gives
\[
\sum_{x\in\bG}
|\mathscr M_\cG(x)||P_{ix}||P_{xj}|
\lesssim
(1+\dist_\cG(i,j))
\max_{\gamma:i\leadsto j}|\cM(\gamma)|.
\]
Here the possible longer simple path causes no additional loss:
if its length is comparable with $\dist_\cG(i,j)$ this is immediate,
while otherwise its extra length is absorbed by the exponential
decay of Proposition \ref{pro.path-weight-decay}.
Thus \eqref{eq.PQPleading} yields the same bound as in the tree case.

The higher order terms in \eqref{e:random-lift-neumann} are estimated exactly as in the proof of \cite[Proposition 2.12]{Huang2024Spectrum}.
We conclude the proof.
\end{proof}

\begin{proof}[Proof of Lemma \ref{p:fixpoint}]
Write
$
P:=G(\Ext(\cG,q),z)
$
and
$
A_P:=P^{-1}.
$
For each $x\neq o$, let $p(x)$ be its parent and let $e_x$
be the directed type of the edge from $p(x)$ to $x$.
Set $e_o:=e$.
Define a vector $v=(v_x)_{x\in\bG}$ recursively by
\[
v_o:=\frac{q_e}{1-q_eq_{\bar e}},
\qquad
v_x:=-q_{e_x}v_{p(x)},
\quad x\neq o.
\]
We claim that
\begin{equation}
\label{eq:fixpoint-exact-defect}
(A_Pv)_x=\mathbf1_{\{x=o\}}+R_{e_x}v_x,
\qquad x\in\bG.
\end{equation}
Indeed, if $x\neq o$, the definition of the extension and
the recursion for $v$ give
\[
\begin{aligned}
(A_Pv)_x
=-zv_x+v_{p(x)}+\sum_{f:\,e_x\rightsquigarrow f}(-q_f)v_{x}=
v_x\Big(
-\frac1{q_{e_x}}
-z
-\sum_{f:\,e_x\rightsquigarrow f}q_f
\Big)
=
R_{e_x}v_x.
\end{aligned}
\]
At the root, since $h(e)=\pi(o)$, we have
\[
\frac1{v_o}
=
\frac1{q_e}-q_{\bar e}
=
-z-\sum_{f:\,t(f)=\pi(o)}q_f-R_e.
\]
Consequently,
\[
(A_Pv)_o
=
-\big(
z+\sum_{f:\,t(f)=\pi(o)}q_f
\big)v_o
=
1+R_ev_o.
\]
This proves \eqref{eq:fixpoint-exact-defect}.
Multiplying \eqref{eq:fixpoint-exact-defect} by $P$ and
taking the component at $o$, we obtain the exact identity
\begin{equation}
\label{eq:fixpoint-residual-representation}
P_{oo}-v_o
=
-\sum_{x\in\bG}P_{ox}R_{e_x}v_x.
\end{equation}
It remains to estimate the coefficients in this sum.
The uniform Dyson bounds imply $|v_o|\lesssim1$.
Since $\max_f|q_f/m_f-1|\lesssim|q-m|_\infty$, we have $\Pi_{\text{$f$ in path}}|q_f/m_f|\le\exp(C\dist(o,x)|q-m|)\lesssim 1$.
Therefore we have
$
|v_x|\lesssim|\cM(o,x)|.
$
On the other hand, Lemma~\ref{p:subgraph} gives
$
|P_{ox}|\lesssim|\cM(o,x)|.
$
Using the PF path-mass estimate, by the same argument as in the proof of \eqref{eq:elliptic-boundary-L2}, we have
$
\sum_{x\in\bG}|\cM(o,x)|^2
\lesssim1+\diam(\cG).
$
We conclude from \eqref{eq:fixpoint-residual-representation}
that
\[
|P_{oo}-v_o|
\le
|R|_\infty
\sum_{x\in\bG}|P_{ox}||v_x|
\lesssim
(1+\diam(\cG))|R|_\infty.
\]
This concludes the proof.
\end{proof}

\begin{proof}[Proof of Lemma \ref{p:localization}]
If $i$ and $j$ belong to different connected components
of $\cK$, then $\dist_{\cG}(i,j)\ge2r$ and
$
G_{ij}(\Ext(\cK,q),z)=0.
$
The claim follows immediately from Lemma~\ref{p:subgraph}
and Proposition~\ref{pro.path-weight-decay}.
Otherwise, we replace $\cK$ by its connected component
containing $i$ and $j$.

We first reduce to the case in which $\cK$ is induced and
every component of $\bG\setminus\bK$ is a tree attached
to $\cK$ by exactly one edge.
Since $\cG$ contains at most one cycle, this can be achieved
by deleting at most two edges whose endpoints are at distance
at least $r$ from both $i$ and $j$.
For each deletion, the newly missing edge types are filled
with their corresponding weights $q$.
The resulting change of the inverse Green function has
bounded entries and is supported on the two endpoints.
The resolvent identity, Lemma~\ref{p:subgraph}, and path
decay therefore bound the change of the $ij$ entry by
$\OO(\beta^{2r})$.
All resulting components have diameter
$\OO(1+\diam(\cG))$.
Thus this reduction costs only $\OO(\beta^{2r})$.
We relabel the resulting graph as $\cG$.

Enumerate the boundary edges of $\cK$ by
$(l_\alpha,a_\alpha)$, where $l_\alpha\in\bK$ and
$a_\alpha\notin\bK$, and let $e_\alpha$ denote the directed type
from $l_\alpha$ to $a_\alpha$.
Let $\mathcal C_\alpha$ be the exterior tree containing
$a_\alpha$.
By the reduction above, the trees $\mathcal C_\alpha$
are pairwise distinct.
Set
$
g_\alpha
:=
G_{a_\alpha a_\alpha}
\bigl(\Ext(\mathcal C_\alpha,q),z\bigr).
$
In this fully filled extension, the missing edge from
$a_\alpha$ toward $l_\alpha$ is represented by the boundary
weight $q_{\bar e_\alpha}$.
Let $h_\alpha$ denote the root Green function of the
partially filled extension obtained from
$\Ext(\mathcal C_\alpha,q)$ by removing only this filling,
while retaining all other $q$-fillings.
These remaining fillings are exactly those inherited from
$\Ext(\cG,q)$, so $h_\alpha$ is precisely the exterior
Green function entering the Schur complement of $\cK$.
Removing the filling toward $l_\alpha$ adds
$q_{\bar e_\alpha}$ to the $a_\alpha a_\alpha$ entry of
the effective inverse Green matrix.
Hence the rank-one resolvent identity gives
$
h_\alpha^{-1}
=
g_\alpha^{-1}+q_{\bar e_\alpha}
$
and 
$
h_\alpha
=
g_\alpha/(1+q_{\bar e_\alpha}g_\alpha).
$
Lemma~\ref{p:fixpoint} implies
$
\left|
g_\alpha
-
q_{e_\alpha}/(1-q_{e_\alpha}q_{\bar e_\alpha})
\right|
\lesssim
(1+\diam(\cG))|R|_\infty.
$
Consequently, we have
$
1+q_{\bar e_\alpha}g_\alpha
=(1-q_{e_\alpha}q_{\bar e_\alpha})^{-1}
+
\OO\bigl((1+\diam(\cG))|R|_\infty\bigr).
$
Since $|R|_\infty\lesssim|q-m|_\infty$, the smallness
assumption and the uniform Dyson bounds imply that
$|1-q_{e_\alpha}q_{\bar e_\alpha}|\sim 1$ and
$1+q_{\bar e_\alpha}g_\alpha$ is bounded away from zero.
Using the exact identity
\[
h_\alpha-q_{e_\alpha}
=
\frac{1-q_{e_\alpha}q_{\bar e_\alpha}}
     {1+q_{\bar e_\alpha}g_\alpha}
\left(
g_\alpha
-
\frac{q_{e_\alpha}}
     {1-q_{e_\alpha}q_{\bar e_\alpha}}
\right),
\]
we obtain
\begin{equation}
\label{eq:localization-cavity-residual}
|h_\alpha-q_{e_\alpha}|
\lesssim
(1+\diam(\cG))|R|_\infty.
\end{equation}

Now write
$
P:=G(\Ext(\cK,q),z)
$
and
$
P':=G(\Ext(\cG,q),z)\big|_{\bK}.
$
Hence the resolvent identity \eqref{e:resolv} yields
\[
P'_{ij}-P_{ij}
=
\sum_\alpha
P_{il_\alpha}
(h_\alpha-q_{e_\alpha})
P'_{l_\alpha j}.
\]
Using the PF path-mass estimate, by the same argument as in
the proof of \eqref{eq:elliptic-boundary-pair}, we obtain
$
\sum_\alpha
|P_{il_\alpha}||P'_{l_\alpha j}|
\lesssim1.
$
Combining this with \eqref{eq:localization-cavity-residual}
gives
$
|P'_{ij}-P_{ij}|
\lesssim
(1+\diam(\cG))|R|_\infty.
$
Adding the $\OO(\beta^{2r})$ error from the initial reduction
completes the proof.
\end{proof}

\section{Effective Woodbury Identity}\label{app:woodbury}

\begin{proof}[Proof of Lemma \ref{lem:compressed-woodbury-filled-switching}]
Let \(H^{\langle \bT\rangle}\) and \(\widetilde H^{\langle\bT\rangle}\) denote the 
adjacency matrices of the exterior component before and after switching. Then
\[
\widetilde H^{\langle\bT\rangle}
-
H^{\langle\bT\rangle}
=
-\sum_{\alpha\in\mathsf W_S}
\xi_\alpha,\qquad\mbox{where}~
\xi_\alpha
:=
\Delta_{a_\alpha l_\alpha^\star}
+
\Delta_{b_\alpha c_\alpha}
-
\Delta_{a_\alpha b_\alpha}
-
\Delta_{c_\alpha l_\alpha^\star}
\]
Define
$
U
:=
\bigl(
e_{a_\alpha}-e_{c_\alpha}
\bigr)_{\alpha\in\mathsf W_S}
$
and
$
V
:=
\bigl(
e_{l_\alpha^\star}-e_{b_\alpha}
\bigr)_{\alpha\in\mathsf W_S}.
$
The same Woodbury identity as in
\eqref{eq:woodbury-error-master} gives, for every $i,j\in\bG\backslash\bT$,
\begin{equation}
\label{eq:woodbury-master-filled-switching}
\widetilde{G}^{\langle\bT\rangle}_{ij}-\widetilde{P}^{\langle\bT\rangle}_{ij}
=W_i
\sum_{k=0}^{\infty}
(D^{\langle\bT\rangle}F)^kD^{\langle\bT\rangle}
W_j^{\mathsf T},
\end{equation}
where 
\begin{equation}
\label{eq:F-factorization-filled-switching}W:=(I+P^{\langle\bT\rangle}F)|_{\bG\backslash\bT,\cdot},
\quad
F
=
\begin{pmatrix}
V&U
\end{pmatrix}F_\partial
\begin{pmatrix}
V&U
\end{pmatrix}^{\mathsf T},\quad F_\partial:=
\begin{pmatrix}
U^{\mathsf T}\widetilde P^{\langle\bT\rangle}U
&
I+
U^{\mathsf T}\widetilde P^{\langle\bT\rangle}V
\\[1mm]
I+
V^{\mathsf T}\widetilde P^{\langle\bT\rangle}U
&
V^{\mathsf T}\widetilde P^{\langle\bT\rangle}V
\end{pmatrix}.
\end{equation}

We first determine the effective action of \(U\) and \(V\). Since they always appear together with \(D^{\langle\bT\rangle}\), it suffices to analyze how they interact with \(D^{\langle\bT\rangle}\) in the relevant matrix products.
It follows from \eqref{eq:D-cavity-filled-switching} that
$
D^{\langle\bT\rangle}e_{l_\alpha^\star}
=
-
m_{a_\alpha l_\alpha}
D^{\langle\bT\rangle}e_{a_\alpha}.
$
It is therefore natural to introduce the effective restriction of
$V$ to the original vertices,
\begin{equation}
\label{eq:V-original-effective}
D^{\langle\bT\rangle}V
=
D^{\langle\bT\rangle}V^{\rm eff},
\quad
V^{\mathsf T}D^{\langle\bT\rangle}
=
(V^{\rm eff})^{\mathsf T}D^{\langle\bT\rangle},
\qquad
\mbox{where~}
V^{\rm eff}
:=
\bigl(
-e_{b_\alpha}
-
m_{a_\alpha l_\alpha}e_{a_\alpha}
\bigr)_{\alpha\in\mathsf W_S}.
\end{equation}
Since $U$ is already supported on the original vertices, we consequently have
$
D^{\langle\bT\rangle}
\begin{pmatrix}
V&U
\end{pmatrix}
=
D^{\langle\bT\rangle}
\begin{pmatrix}
V^{\rm eff}&U
\end{pmatrix}.
$

We next check the effective action of $F_\partial$. For the following calculation, write
\[
\bA=(a_\alpha)_{\alpha\in\mathsf W_S},
\quad
\bB=(b_\alpha)_{\alpha\in\mathsf W_S},
\quad
\bC=(c_\alpha)_{\alpha\in\mathsf W_S},
\quad
\bL^\star=(l_\alpha^\star)_{\alpha\in\mathsf W_S},
\quad
M_+
:=
\operatorname{diag}
\bigl(
m_{a_\alpha l_\alpha}:
\alpha\in\mathsf W_S
\bigr).
\]
After switching, each \(l_\alpha^\star\) is attached to \(c_\alpha\).
Hence, by the same Schur-complement calculation as above,
\begin{equation}
\label{eq:switched-branch-blocks}
\widetilde P^{\langle\bT\rangle}_{\bL^\star\bX}
=
-M_+
\widetilde P^{\langle\bT\rangle}_{\bC\bX},
\qquad
\widetilde P^{\langle\bT\rangle}_{\bX\bL^\star}
=
-\widetilde P^{\langle\bT\rangle}_{\bX\bC}M_+,
\qquad
\widetilde P^{\langle\bT\rangle}_{\bL^\star\bL^\star}
=
M_+
+
M_+
\widetilde P^{\langle\bT\rangle}_{\bC\bC}
M_+,
\end{equation}
whenever \(\bX\) lies outside the corresponding universal-cover
branches.
Using \eqref{eq:switched-branch-blocks}, the four blocks in
\eqref{eq:F-factorization-filled-switching} are given by
\begin{align*}
U^{\mathsf T}
\widetilde P^{\langle\bT\rangle}
U
=&
\widetilde P^{\langle\bT\rangle}_{\bA\bA}
-
\widetilde P^{\langle\bT\rangle}_{\bA\bC}
-
\widetilde P^{\langle\bT\rangle}_{\bC\bA}
+
\widetilde P^{\langle\bT\rangle}_{\bC\bC},
\\
I+
U^{\mathsf T}
\widetilde P^{\langle\bT\rangle}
V
=&
I
-
\widetilde P^{\langle\bT\rangle}_{\bA\bB}
+
\widetilde P^{\langle\bT\rangle}_{\bC\bB}
-
\widetilde P^{\langle\bT\rangle}_{\bA\bC}M_+
+
\widetilde P^{\langle\bT\rangle}_{\bC\bC}M_+,
\\
I+
V^{\mathsf T}
\widetilde P^{\langle\bT\rangle}
U
=&
I
-
\widetilde P^{\langle\bT\rangle}_{\bB\bA}
+
\widetilde P^{\langle\bT\rangle}_{\bB\bC}
-
M_+\widetilde P^{\langle\bT\rangle}_{\bC\bA}
+
M_+\widetilde P^{\langle\bT\rangle}_{\bC\bC},
\\
V^{\mathsf T}
\widetilde P^{\langle\bT\rangle}
V
=&
M_+
+
\widetilde P^{\langle\bT\rangle}_{\bB\bB}
+
\widetilde P^{\langle\bT\rangle}_{\bB\bC}M_+
+
M_+\widetilde P^{\langle\bT\rangle}_{\bC\bB}
+
M_+\widetilde P^{\langle\bT\rangle}_{\bC\bC}M_+.
\end{align*}

Next, we exploit the geometric properties of the local sampling to simplify the preceding identities.
Proposition \ref{prop:elliptic-reverse} implies, for any distinct
\(\alpha,\beta\in \mathsf W_S\),
\[
\operatorname{dist}_{\widetilde{\mathcal G}^{\langle\bT\rangle}}
\bigl(
\{a_\alpha,b_\alpha,c_\alpha\},
\{a_\beta,b_\beta,c_\beta\}
\bigr)
\geq \frac{\mathfrak R}{4}.
\]
Since the localization radius is much smaller than \(\mathfrak R\),
the corresponding localized neighborhoods belong to different
connected components. Hence, for any
\(x,y\in\{a,b,c\}\) and \(\alpha\neq\beta\), we have
$
\widetilde P^{\langle\bT\rangle}_{x_\alpha y_\beta}=0.
$

On the other hand, for each admissible index \(\alpha\in\mathsf W_S\)
it follows from \(I_\alpha=1\) that
$
\mathcal B_{\mathfrak R/4}
\bigl(
\{a_\alpha,b_\alpha,c_\alpha\},
\mathcal G^{\langle\bT\rangle}
\bigr)\cup\bigl\{\{a_\alpha,b_\alpha\}\bigr\}
$
is a tree. Notice that \(\{b_\alpha,c_\alpha\}\) is an edge of
\(\mathcal B_{\mathfrak R/4}
\bigl(
\{a_\alpha,b_\alpha,c_\alpha\},
\mathcal G^{\langle\bT\rangle}
\bigr)\). Since the switching removes the edge $\{b_\alpha, c_\alpha\}$, by the tree property,
it separates the
component containing \(c_\alpha\) from the component containing
\(\{a_\alpha,b_\alpha\}\). In particular, because the localization radius
is much smaller than \(\mathfrak R/4\), we obtain that
\begin{equation}\label{eq.block-vanish}
\widetilde P^{\langle\bT\rangle}_{a_\alpha c_\alpha}
=
\widetilde P^{\langle\bT\rangle}_{c_\alpha a_\alpha}
=
\widetilde P^{\langle\bT\rangle}_{b_\alpha c_\alpha}
=
\widetilde P^{\langle\bT\rangle}_{c_\alpha b_\alpha}
=0 .
\end{equation}
Combining the two observations above, we obtain the exact block
vanishing
\[
\widetilde P^{\langle\bT\rangle}_{\bA\bC}
=
\widetilde P^{\langle\bT\rangle}_{\bC\bA}
=
\widetilde P^{\langle\bT\rangle}_{\bB\bC}
=
\widetilde P^{\langle\bT\rangle}_{\bC\bB}
=0.
\]
More precisely, we have
\[
\widetilde P^{\langle\bT\rangle}_{\bA\bA}
=
\operatorname{diag}_{\alpha\in\mathsf W_S}
\bigl(
\widetilde P^{\langle\bT\rangle}_{a_\alpha a_\alpha}
\bigr),
\qquad
\widetilde P^{\langle\bT\rangle}_{\bB\bB}
=
\operatorname{diag}_{\alpha\in\mathsf W_S}
\bigl(
\widetilde P^{\langle\bT\rangle}_{b_\alpha b_\alpha}
\bigr),
\qquad
\widetilde P^{\langle\bT\rangle}_{\bC\bC}
=
\operatorname{diag}_{\alpha\in\mathsf W_S}
\bigl(
\widetilde P^{\langle\bT\rangle}_{c_\alpha c_\alpha}
\bigr),
\]
\[
\widetilde P^{\langle\bT\rangle}_{\bB\bA}=\widetilde P^{\langle\bT\rangle}_{\bA\bB}
=
\operatorname{diag}_{\alpha\in\mathsf W_S}
\bigl(
\widetilde P^{\langle\bT\rangle}_{a_\alpha b_\alpha}
\bigr).
\]
Consequently, it follows that
\begin{equation}
\label{eq:F-boundary-matrix-filled-switching}
F_\partial
=
\begin{pmatrix}
\widetilde P^{\langle\bT\rangle}_{\bA\bA}
+
\widetilde P^{\langle\bT\rangle}_{\bC\bC}
&
I
-
\widetilde P^{\langle\bT\rangle}_{\bA\bB}
+
\widetilde P^{\langle\bT\rangle}_{\bC\bC}M_+
\\[2mm]
I
-
\widetilde P^{\langle\bT\rangle}_{\bB\bA}
+
M_+\widetilde P^{\langle\bT\rangle}_{\bC\bC}
&
M_+
+
\widetilde P^{\langle\bT\rangle}_{\bB\bB}
+
M_+\widetilde P^{\langle\bT\rangle}_{\bC\bC}M_+
\end{pmatrix}.
\end{equation}
Moreover, \(F_\partial\) admits a further block-diagonal decomposition
with respect to the switching index. After regrouping the
coordinates of \(F_\partial\) from
$
(V_{\alpha_1},\dots,V_{\alpha_{|\mathsf W_S|}},U_{\alpha_1},\dots,U_{\alpha_{|\mathsf W_S|}})
$
into
$
(V_{\alpha_1},U_{\alpha_1};\dots;V_{\alpha_{|\mathsf W_S|}},U_{\alpha_{|\mathsf W_S|}}),
$
the matrix \(F_\partial\) can be written as a direct sum of
\(2\times2\) diagonal blocks:
\[
F_\partial
=
\bigoplus_{\alpha\in\mathsf W_S}
F_\partial^{(\alpha)},
\quad
F_\partial^{(\alpha)}
=
\begin{pmatrix}
\widetilde P^{\langle\bT\rangle}_{a_\alpha a_\alpha}
+
\widetilde P^{\langle\bT\rangle}_{c_\alpha c_\alpha}
&
1-
\widetilde P^{\langle\bT\rangle}_{a_\alpha b_\alpha}
+
\widetilde P^{\langle\bT\rangle}_{c_\alpha c_\alpha}
\,m_{a_\alpha l_\alpha}
\\[2mm]
1-
\widetilde P^{\langle\bT\rangle}_{b_\alpha a_\alpha}
+
m_{a_\alpha l_\alpha}\,
\widetilde P^{\langle\bT\rangle}_{c_\alpha c_\alpha}
&
m_{a_\alpha l_\alpha}
+
\widetilde P^{\langle\bT\rangle}_{b_\alpha b_\alpha}
+
m_{a_\alpha l_\alpha}\,
\widetilde P^{\langle\bT\rangle}_{c_\alpha c_\alpha}\,
m_{a_\alpha l_\alpha}
\end{pmatrix}.
\]
In other words, \(F_\partial\) acts independently on each switching
cell, and each such action is described by a \(2\times2\) matrix.

Next we identify the transfer factors \(W\) and \(W^{\mathsf T}\).
It follows from
$
\widetilde P^{\langle\bT\rangle}
=
P^{\langle\bT\rangle}
+
P^{\langle\bT\rangle}\Xi
\widetilde P^{\langle\bT\rangle}
$
that
$
I+P^{\langle\bT\rangle}F
=
I+\widetilde P^{\langle\bT\rangle}\Xi.
$
Hence, for \(w\in\bG\setminus\bT\),
\begin{multline}
W_{wu}
=
\delta_{wu}
+
\sum_{\alpha\in\mathsf W_S}\Big[
\bigl(
\widetilde P^{\langle\bT\rangle}_{wa_\alpha}
-
\widetilde P^{\langle\bT\rangle}_{wc_\alpha}
\bigr)
\mathbf 1_{\{u=l_\alpha^\star\}}
-
\bigl(
\widetilde P^{\langle\bT\rangle}_{wb_\alpha}
+
m_{a_\alpha l_\alpha}
\widetilde P^{\langle\bT\rangle}_{wc_\alpha}
\bigr)
\mathbf 1_{\{u=a_\alpha\}}\\
+
\bigl(
\widetilde P^{\langle\bT\rangle}_{wc_\alpha}
-
\widetilde P^{\langle\bT\rangle}_{wa_\alpha}
\bigr)
\mathbf 1_{\{u=b_\alpha\}}
+
\bigl(
\widetilde P^{\langle\bT\rangle}_{wb_\alpha}
+
m_{a_\alpha l_\alpha}
\widetilde P^{\langle\bT\rangle}_{wc_\alpha}
\bigr)
\mathbf 1_{\{u=c_\alpha\}}\Big],
\label{eq:W-exact-filled-switching}
\end{multline}
where we used
$
\widetilde P^{\langle\bT\rangle}_{w l_\alpha^\star}
=
-m_{a_\alpha l_\alpha}
\widetilde P^{\langle\bT\rangle}_{w c_\alpha}
$
in the above equality. 
In the Woodbury expansion, \(W\) only appears next to
\(D^{\langle\bT\rangle}\). Since, before switching,
\(l_\alpha^\star\) is attached to \(a_\alpha\), the Schur-complement
formula gives
\begin{equation}
\label{eq:D-cavity-filled-switching}
D^{\langle\bT\rangle}_{l_\alpha^\star\,\cdot}
=
-m_{a_\alpha l_\alpha}
D^{\langle\bT\rangle}_{a_\alpha\,\cdot},
\qquad
D^{\langle\bT\rangle}_{\cdot\,l_\alpha^\star}
=
-m_{a_\alpha l_\alpha}
D^{\langle\bT\rangle}_{\cdot\,a_\alpha}.
\end{equation}
Therefore the \(l_\alpha^\star\)-coordinate in
\eqref{eq:W-exact-filled-switching} can be absorbed into the
\(a_\alpha\)-coordinate.
Indeed, in the product \(WD^{\langle\bT\rangle}\), the contributions
from the \(l_\alpha^\star\)- and \(a_\alpha\)-coordinates are
\begin{align*}
&\bigl(
\widetilde P^{\langle\bT\rangle}_{wa_\alpha}
-
\widetilde P^{\langle\bT\rangle}_{wc_\alpha}
\bigr)
D^{\langle\bT\rangle}_{l_\alpha^\star\,\cdot}
-
\bigl(
\widetilde P^{\langle\bT\rangle}_{wb_\alpha}
+
m_{a_\alpha l_\alpha}
\widetilde P^{\langle\bT\rangle}_{wc_\alpha}
\bigr)
D^{\langle\bT\rangle}_{a_\alpha\,\cdot}
\\
={}&
-m_{a_\alpha l_\alpha}
\bigl(
\widetilde P^{\langle\bT\rangle}_{wa_\alpha}
-
\widetilde P^{\langle\bT\rangle}_{wc_\alpha}
\bigr)
D^{\langle\bT\rangle}_{a_\alpha\,\cdot}
-
\bigl(
\widetilde P^{\langle\bT\rangle}_{wb_\alpha}
+
m_{a_\alpha l_\alpha}
\widetilde P^{\langle\bT\rangle}_{wc_\alpha}
\bigr)
D^{\langle\bT\rangle}_{a_\alpha\,\cdot}
=
-
\bigl(
\widetilde P^{\langle\bT\rangle}_{wb_\alpha}
+
m_{a_\alpha l_\alpha}
\widetilde P^{\langle\bT\rangle}_{wa_\alpha}
\bigr)
D^{\langle\bT\rangle}_{a_\alpha\,\cdot},
\end{align*}
where the two terms involving
\(m_{a_\alpha l_\alpha}
\widetilde P^{\langle\bT\rangle}_{wc_\alpha}\)
cancel. The \(b_\alpha\)- and \(c_\alpha\)-coordinates are unchanged.
Thus the action of \(W\) on \(D^{\langle\bT\rangle}\) is exactly the
same as that of the matrix obtained by removing the
\(l_\alpha^\star\)-coordinates and replacing the \(a_\alpha\)-coefficient
accordingly.
More precisely,
\begin{equation}\label{eq:W-effective-action-filled-switching}
W D^{\langle\bT\rangle}
=
W^{\mathrm{eff}}D^{\langle\bT\rangle},
\qquad
D^{\langle\bT\rangle}W^{\mathsf T}
=
D^{\langle\bT\rangle}(W^{\mathrm{eff}})^{\mathsf T},
\end{equation}
where
\begin{align*}
W^{\mathrm{eff}}_{wu}
=
\delta_{wu}
-
\sum_{\alpha\in\mathsf W_S}\Big[
\bigl(
\widetilde P^{\langle\bT\rangle}_{wb_\alpha}
+
m_{a_\alpha l_\alpha}
\widetilde P^{\langle\bT\rangle}_{wa_\alpha}
\bigr)
\mathbf 1_{\{u=a_\alpha\}}
-
\bigl(
\widetilde P^{\langle\bT\rangle}_{wc_\alpha}
-
\widetilde P^{\langle\bT\rangle}_{wa_\alpha}
\bigr)
\mathbf 1_{\{u=b_\alpha\}}
-
\bigl(
\widetilde P^{\langle\bT\rangle}_{wb_\alpha}
+
m_{a_\alpha l_\alpha}
\widetilde P^{\langle\bT\rangle}_{wc_\alpha}
\bigr)
\mathbf 1_{\{u=c_\alpha\}}\Big].
\end{align*}
Consequently, we have
$
W D^{\langle\bT\rangle}W^{\mathsf T}
=
W^{\mathrm{eff}}
D^{\langle\bT\rangle}
(W^{\mathrm{eff}})^{\mathsf T},
$
and the same replacement applies to every term in
\eqref{eq:woodbury-master-filled-switching}.
Hence, in the remainder of the proof, we write \(W\) again for the
effective transfer matrix \(W^{\mathrm{eff}}\).

Finally, since the row supports of \(U\) and \(V^{\rm eff}\), as well as the column support of \(W\), are contained in \(\bG\), all matrix products above depend only on the \(\bG\times\bG\) block of \(D^{\langle\bT\rangle}\). After applying the same cellwise regrouping as above, \(\begin{pmatrix}V^{\rm eff}&U\end{pmatrix}\) is precisely the matrix \(J=(v_\alpha^{\rm eff},u_\alpha)_{\alpha\in\mathsf W_S}\) in the statement. Hence, by \eqref{eq:F-factorization-filled-switching} and \eqref{eq:V-original-effective}, for every \(k\geq0\),
\[
W_i
\bigl(
D^{\langle\bT\rangle}F
\bigr)^k
D^{\langle\bT\rangle}
W_j^{\mathsf T}
=
W_i
\left(
D^{\langle\bT\rangle}\big|_{\bG}
J F_\partial J^{\mathsf T}
\right)^k
D^{\langle\bT\rangle}\big|_{\bG}
W_j^{\mathsf T}.
\]
Substituting this into \eqref{eq:woodbury-master-filled-switching}, we obtain the desired identity.
\end{proof}  

\begin{proof}[Proof of Lemma \ref{lem:compressed-woodbury-gluing}]
Let \(\widetilde H^{\rm aux}\) and \(\widetilde H\) denote the enlarged
adjacency matrices before and after gluing, respectively. Then
\[
\widetilde H-\widetilde H^{\rm aux}=-\sum_{\alpha=1}^\mu\xi_\alpha,\qquad\mbox{where}~\xi_\alpha
:=
\Delta_{l_\alpha a_\alpha^\star}
+
\Delta_{\widetilde a_\alpha l_\alpha^\star}
-
\Delta_{l_\alpha\widetilde a_\alpha}
-
\Delta_{l_\alpha^\star a_\alpha^\star}.
\]
Lemma \ref{lem:F-boundary-factorization} and Lemma \ref{lem:woodbury-error-expansion} again imply
\begin{equation}
\label{eq:woodbury-error-master-gluing}
\widetilde G-\widetilde P
=
\bigl(
I+\widetilde P^{\rm aux}\widetilde F
\bigr)
\sum_{k=0}^{\infty}
\bigl(
\widetilde D^{\rm aux}\widetilde F
\bigr)^k
\widetilde D^{\rm aux}
\bigl(
I+\widetilde F\widetilde P^{\rm aux}
\bigr)
\end{equation}
where 
\begin{equation}
\label{eq:F-factorization-gluing}
\widetilde F
=
\begin{pmatrix}
V&U
\end{pmatrix}
\begin{pmatrix}
U^{\mathsf T}\widetilde P U
&
I+U^{\mathsf T}\widetilde P V
\\[1mm]
I+V^{\mathsf T}\widetilde P U
&
V^{\mathsf T}\widetilde P V
\end{pmatrix}
\begin{pmatrix}
V&U
\end{pmatrix}^{\mathsf T},
\qquad 
U
:=
\bigl(
e_{l_\alpha}-e_{l_\alpha^\star}
\bigr)_{\alpha=1}^{\mu},~
V
:=
\bigl(
e_{a_\alpha^\star}-e_{\widetilde a_\alpha}
\bigr)_{\alpha=1}^{\mu}.
\end{equation}

We first compress the middle Woodbury kernel. On the enlarged space,
the Schur formula \eqref{e:Schur1} implies
\begin{equation}
\label{eq:cavity-error-gluing}
\widetilde D^{\rm aux}_{\cdot\,\bL^\star}
=
-\widetilde D^{\rm aux}_{\cdot\,\widetilde{\bA}}M_+,
\qquad
\widetilde D^{\rm aux}_{\bL^\star\,\cdot}
=
-M_+
\widetilde D^{\rm aux}_{\widetilde{\bA}\,\cdot}.
\end{equation}
Together with the fact that \(\widetilde D^{\rm aux}=\widetilde D^{\langle\bT\rangle}\oplus0\) vanishes on the
remaining auxiliary component, this gives
$
\widetilde D^{\rm aux}V
=
-\widetilde D^{\rm aux}J
$
and
$
\widetilde D^{\rm aux}U
=
\widetilde D^{\rm aux}JM_+,
$
where \(J\) is understood as the corresponding boundary embedding
into the enlarged space. Hence, it follows that
\begin{equation*}
\label{eq:error-factor-compression-gluing}
\widetilde D^{\rm aux}
\begin{pmatrix}
V&U
\end{pmatrix}
=
\widetilde D^{\rm aux}J
\begin{pmatrix}
-I&M_+
\end{pmatrix},\qquad
\begin{pmatrix}
V&U
\end{pmatrix}^{\mathsf T}
\widetilde D^{\rm aux}
=
\begin{pmatrix}
-I\\ M_+
\end{pmatrix}
J^{\mathsf T}\widetilde D^{\rm aux}.
\end{equation*}
Consequently, we have
$
\widetilde D^{\rm aux}
\widetilde F
\widetilde D^{\rm aux}
=
\widetilde D^{\rm aux}
J\widetilde F_\partial J^{\mathsf T}
\widetilde D^{\rm aux},
$
where initially
\begin{equation}
\label{eq:F-partial-preliminary-gluing}
\widetilde F_\partial
:=
\begin{pmatrix}
-I&M_+
\end{pmatrix}
\begin{pmatrix}
U^{\mathsf T}\widetilde P U
&
I+U^{\mathsf T}\widetilde P V
\\[1mm]
I+V^{\mathsf T}\widetilde P U
&
V^{\mathsf T}\widetilde P V
\end{pmatrix}
\begin{pmatrix}
-I\\ M_+
\end{pmatrix}.
\end{equation}

We next identify \(\widetilde F_\partial\). 
After gluing, \(\widetilde{\bA}\) and \(\bL\) belong to the original
component, whereas \(\bA^\star\) and \(\bL^\star\) belong to the
restored universal-cover copies. Hence every Green-function block
joining these two components vanishes. Therefore,
\begin{align*}
U^{\mathsf T}\widetilde P U
=
\widetilde P_{\bL\bL}
+
\widetilde P_{\bL^\star\bL^\star},\qquad
I+U^{\mathsf T}\widetilde P V
=
I
-
\widetilde P_{\bL\widetilde{\bA}}
-
\widetilde P_{\bL^\star\bA^\star},
\\
V^{\mathsf T}\widetilde P V
=
\widetilde P_{\widetilde{\bA}\widetilde{\bA}}
+
\widetilde P_{\bA^\star\bA^\star},
\qquad
I+V^{\mathsf T}\widetilde P U
=
I
-
\widetilde P_{\widetilde{\bA}\bL}
-
\widetilde P_{\bA^\star\bL^\star}.
\end{align*}
Substituting these identities into
\eqref{eq:F-partial-preliminary-gluing}, we obtain
\begin{align*}
\widetilde F_\partial
=
\widetilde P_{\bL\bL}
+
\widetilde P_{\bL\widetilde{\bA}}M_+
+
M_+\widetilde P_{\widetilde{\bA}\bL}
+
M_+\widetilde P_{\widetilde{\bA}\widetilde{\bA}}M_+
+
\widetilde P_{\bL^\star\bL^\star}
+
\widetilde P_{\bL^\star\bA^\star}M_+
+
M_+\widetilde P_{\bA^\star\bL^\star}
+
M_+\widetilde P_{\bA^\star\bA^\star}M_+
-2M_+.
\label{eq:F-partial-expanded-gluing}
\end{align*}
On every restored universal-cover copy, the Schur-complement
identities give
\begin{equation}
\label{eq:restored-star-identities-gluing}
\widetilde P_{\bL^\star\bL^\star}
=
M_+
+
M_+\widetilde P_{\bA^\star\bA^\star}M_+,
\qquad
\widetilde P_{\bL^\star\bA^\star}
=
-M_+\widetilde P_{\bA^\star\bA^\star},
\qquad
\widetilde P_{\bA^\star\bL^\star}
=
-\widetilde P_{\bA^\star\bA^\star}M_+.
\end{equation}
Thus all terms involving the restored copies reduce to \(-M_+\).
Hence,
\[
\widetilde F_\partial
=
M_+
\widetilde P_{\widetilde{\bA}\widetilde{\bA}}
M_+
+
M_+
\widetilde P_{\widetilde{\bA}\bL}
+
\widetilde P_{\bL\widetilde{\bA}}
M_+
+
\widetilde P_{\bL\bL}
-
M_+,
\]
or equivalently,
\[
\widetilde F_\partial
=
\begin{pmatrix}
M_+&I
\end{pmatrix}
\begin{pmatrix}
\widetilde P_{\widetilde{\bA}\widetilde{\bA}}
&
\widetilde P_{\widetilde{\bA}\bL}
\\[1mm]
\widetilde P_{\bL\widetilde{\bA}}
&
\widetilde P_{\bL\bL}
\end{pmatrix}
\begin{pmatrix}
M_+\\ I
\end{pmatrix}
-M_+.
\]

It remains to identify the transfer factors. The deterministic
resolvent identity gives
$
I+\widetilde P^{\rm aux}\widetilde F
=
I+\widetilde P\Xi.
$
Using \eqref{eq:cavity-error-gluing} and restricting the resulting
action to \(\bG\), we obtain
\[
W_{xu}
=
\delta_{xu}
-
\sum_{\alpha=1}^{\mu}
\left(
\widetilde P_{xl_\alpha}
+
m_{\widetilde a_\alpha l_\alpha}
\widetilde P_{x\widetilde a_\alpha}
\right)
\mathbf 1_{\{u=\widetilde a_\alpha\}},
\qquad x,u\in\bG.
\]
Thus \(W\) is a \(\bG\times\bG\) matrix. By complex symmetry, the
corresponding right transfer factor is \(W^{\mathsf T}\).

Finally, since
$
\widetilde D^{\rm aux}
=
\widetilde D^{\langle\bT\rangle}\oplus0,
$
every occurrence of \(\widetilde D^{\rm aux}\) in the Woodbury series
which survives after restriction to \(\bG\) is exactly
\(\widetilde D^{\langle\bT\rangle}|_\bG\). Hence, for every
\(k\geq1\),
\[
\left(
\widetilde D^{\rm aux}\widetilde F
\right)^k
\widetilde D^{\rm aux}
\Big|_{\bG\times\bG}
=
\left(
\widetilde D^{\langle\bT\rangle}|_\bG
J\widetilde F_\partial J^{\mathsf T}
\right)^k
\widetilde D^{\langle\bT\rangle}|_\bG.
\]
Substituting this identity and the effective transfer factors into
\eqref{eq:woodbury-error-master-gluing}, and restricting to
\(i,j\in\bG\), concludes the proof.
\end{proof}

\bibliography{Ref-Lift}
\bibliographystyle{abbrv}

\end{document}